\documentclass{article}
\usepackage{fullpage}

\usepackage[utf8]{inputenc}
\usepackage{amsmath}
\usepackage{amsthm,amssymb}
\usepackage{nicefrac,amsfonts}
\usepackage{thmtools,mathtools,leftindex,leftidx}
\usepackage{caption,subcaption}
\usepackage{epsfig,graphicx,graphics,color,tikz,tikz-cd,tcolorbox}
\usepackage{enumerate,enumitem}
\usepackage[sort,nocompress]{cite}
\usepackage{xspace}
\usepackage{bbm}
\usepackage{hyperref}
\usepackage[capitalise,noabbrev]{cleveref}
\crefname{thm}{Theorem}{Theorems}
\usepackage{float}
\usepackage{booktabs}
\usepackage{tabularx}
\usepackage{array}
\definecolor{tableheader}{RGB}{226,232,240}
\definecolor{tablerow}{RGB}{248,250,252}
\hypersetup{
	colorlinks=true,       % false: boxed links; true: colored links
	linkcolor=blue,        % color of internal links (change box color with linkbordercolor)
	citecolor=magenta,         % color of links to bibliography
	filecolor=magenta,     % color of file links
	urlcolor=cyan,         % color of external links
	linktoc=all
}
\usepackage{titling}

\newcommand{\starref}[2]{\hyperref[#1]{\ensuremath{(\star_{#2})}}}

\newlength{\bibitemsep}
\newlength{\bibparskip}
\let\oldthebibliography\thebibliography
\renewcommand\thebibliography[1]{%
	\oldthebibliography{#1}%
	\setlength{\parskip}{\bibitemsep}%
	\setlength{\itemsep}{\bibparskip}%
}

\makeatletter
\renewcommand{\paragraph}{%
	\@startsection{paragraph}{4}%
	{\z@}{1.6ex \@plus 1ex \@minus .2ex}{-0.5em}%
	{\normalfont\normalsize\bfseries}%
}
\makeatother

\theoremstyle{plain}
\newtheorem{thm}{Theorem}[section]
\newtheorem{lem}[thm]{Lemma}
\newtheorem{cor}[thm]{Corollary}
\newtheorem{cla}[thm]{Claim}
\newtheorem{prop}[thm]{Proposition}

\theoremstyle{definition}

\newtheorem{rem}[thm]{Remark}

\newcommand*{\claimproofname}{Proof of claim.}
\newenvironment{claimproof}[1][\claimproofname]{\begin{proof}[#1]}{\end{proof}}

\makeatletter
\newcommand{\leqnomode}{\tagsleft@true}
\newcommand{\reqnomode}{\tagsleft@false}
\makeatother

\makeatletter
\newcommand{\linkdest}[1]{\Hy@raisedlink{\hypertarget{#1}{}}}
\makeatother

\def\final{0}  % set this to 1 to get a comment-free version
\def\iflong{\iffalse}
\ifnum\final=0  %namely if we allow comments in the output
\newcommand{\kristof}[1]{{\color{red}[{\tiny \textbf{Kristóf:}  #1}]\marginpar{\color{red}*}}}
\newcommand{\andris}[1]{{\color{magenta}[{\tiny \textbf{Andris:}  #1}]\marginpar{\color{magenta}*}}}
\newcommand{\adam}[1]{{\color{teal}[{\tiny \textbf{Adam:}  #1}]\marginpar{\color{teal}*}}}
\else % in this case [final=1] we don't want any comments to show
\newcommand{\kristof}[1]{}
\newcommand{\andris}[1]{}
\newcommand{\adam}[1]{}
\fi

\DeclareMathOperator{\clo}{cl}

\DeclareMathOperator{\dir}{dir}

\DeclareMathOperator\cost{cost}

\newcommand{\bN}{\mathbb{N}}

\newcommand{\bQ}{\mathbb{Q}}
\newcommand{\bR}{\mathbb{R}}
\newcommand{\bZ}{\mathbb{Z}}

\newcommand{\cA}{\mathcal{A}}
\newcommand{\cB}{\mathcal{B}}
\newcommand{\cC}{\mathcal{C}}

\newcommand{\cF}{\mathcal{F}}
\newcommand{\cG}{\mathcal{G}}

\newcommand{\cI}{\mathcal{I}}

\newcommand{\cK}{\mathcal{K}}

\newcommand{\cM}{\mathcal{M}}

\newcommand{\cP}{\mathcal{P}}

\newcommand{\cR}{\mathcal{R}}
\newcommand{\cS}{\mathcal{S}}
\newcommand{\cT}{\mathcal{T}}
\newcommand{\cU}{\mathcal{U}}

\newcommand{\cX}{\mathcal{X}}

\newcommand*\diff{\mathop{}\!\mathrm{d}}
\newcommand{\mm}{\text{\sf mm}}

\newcommand{\bmm}{\text{\sf bmm}}

\colorlet{groundcolor}{gray!5}
\colorlet{fcolor}{gray!100}
\colorlet{f0color}{gray!50}
\colorlet{f1color}{gray!62}
\colorlet{f2color}{gray!74}
\colorlet{f3color}{gray!86}

\tikzset{
	groundset/.style={thick, fill=groundcolor, rounded corners=15pt, draw=black},
	set F/.style={thick, fill=fcolor, draw=black},
	edge F/.style={thick, draw=black},
	set boundary/.style={thick, dashed, draw=black!85},
	set X/.style={thick, fill=gray!100, draw=black},
	set Y/.style={thick, fill=white, draw=black},
	label J/.style={anchor=north east, font=\Large\bfseries},
	label F/.style={font=\large\bfseries, anchor=west},
	label Fi/.style={font=\small\bfseries, color=white},
	label A/.style={font=\bfseries, color=black},
	label X/.style={font=\small\bfseries, color=white},
	label Y/.style={font=\small\bfseries, color=black}
}

\title{Measurable Matroids: Foundations and Min--Max Theorems}

\thanksmarkseries{alph}
\author{
	Kristóf Bérczi\thanks{MTA-ELTE Matroid Optimization Research Group and HUN-REN-ELTE Egerváry Research Group, Department of Operations Research, ELTE Eötvös Loránd University, and HUN-REN Alfréd Rényi Institute of Mathematics, Budapest, Hungary. Email: \texttt{kristof.berczi@ttk.elte.hu}.}
	\and
	András Imolay\thanks{MTA-ELTE Matroid Optimization Research Group, Department of Operations Research, ELTE Eötvös Loránd University, Budapest, Hungary. Email: \texttt{andras.imolay@ttk.elte.hu}.}
	\and
	Ádám Schweitzer\thanks{Department of Mathematics, KTH Royal Institute of Technology, Stockholm, Sweden. Email: \texttt{adasch@kth.se}.} 
}
\date{}

\begin{document}
	\maketitle
	\thispagestyle{empty}
	%%%%%%%%%%%%%%%%%%%%%%%%%%%%%%%%
	
	\begin{abstract}
		We develop a measure-theoretic analogue of matroid theory on standard atomless measure spaces. Motivated by the quotient-convergence framework for submodular set functions, our aim is to identify suitable measurable objects on the limit side of finite matroid theory and to develop their basic structural and optimization theory. We prove that measurable matroids admit equivalent descriptions by independent sets, bases, rank functions, and closure operators. The class includes normalized finite matroids, cycle matroids of graphings, and measurable analogues of partition, nested, lattice path, transversal, and matching matroids. We establish measurable analogues of truncation, elongation, direct sum, duality, and minors. 
		
		We prove measurable versions of Edmonds' matroid intersection theorem and the Edmonds--Fulkerson matroid union theorem, together with an attainment theorem for common bases. The intersection theorem retains the classical min--max form, with the maximum replaced in general by a supremum. Applied to partition matroids, it gives an exact Hall-deficiency formula for measurable matchings in bipartite graphings and a min--max theorem for continuous bipartite $b$-matchings, with applications to capacity-constrained transport and prescribed cross-sections. We give a rank-expansion criterion under which the supremum is attained. For bipartite graphings, this criterion specializes to the measurable perfect matching theorem of Lyons and Nazarov.
		
		For measurable union, the Edmonds--Fulkerson rank formula remains valid. As a main application of measurable union and its attainment theorem, we prove measurable Nash-Williams--Tutte theorems for graphings, characterizing approximate coverings and packings by hyperfinite essential spanning forests and obtaining exact decompositions under strengthened rank inequalities. 
		
		\medskip
		
		
		\medskip
		
	\end{abstract}
	%%%%%%%%%%%%%%%%%%%%%%%%%%%%%%%%
	\newpage
	\pagenumbering{roman}
	\tableofcontents
	\newpage
	\pagenumbering{arabic}
	\setcounter{page}{1}
	%%%%%%%%%%%%%%%%%%%%%%%%%%%%%%%%
	%%%%%%%%%%%%%%%%%%%%%%%%%%%%%%%%
	\section{Introduction}
	\label{chap:intro}
	%%%%%%%%%%%%%%%%%%%%%%%%%%%%%%%%
	
	Submodularity is one of the central convexity-type notions in discrete mathematics, with a foundational role in combinatorial optimization and game theory. A set function $f$ is \emph{submodular} if it satisfies 
	\[
	f(X)+f(Y) \geq f(X \cap Y) + f(X \cup Y)
	\]
	for all pairs of sets $X,Y$. Besides its role in combinatorial optimization and graph theory, submodularity naturally appears in information theory, potential theory, economics, machine learning, and several other fields; we refer to~\cite{frank2011connections,fujishige2005submodular,schrijver2003combinatorial} and the survey~\cite{bilmes2022submodularity}. Although the same inequality appears in both finite combinatorics and continuous analysis, the finite and continuous theories have largely developed along separate lines.
	
	In the finite setting, \emph{matroids} form the basic combinatorial model. A finite matroid can be defined as a nonnegative integer-valued monotone submodular function $r$ satisfying $r(X) \le |X|$ for every set $X$; several equivalent definitions exist, and this one follows the submodular language. Since the foundational works of Whitney~\cite{whitney1935abstract}, Rado~\cite{rado1942theorem}, and Edmonds~\cite{edmonds1970submodular}, matroids have become one of the central frameworks for understanding independence. They appear in linear algebra~\cite{whitney1935abstract}, polyhedral combinatorics~\cite{edmonds1970submodular,edmonds1965transversals}, structural matroid theory~\cite{seymour1980decomposition,oxley2011matroid}, and optimization~\cite{schrijver2000combinatorial,fujishige2005submodular}. Recent work has extended the theory into algebraic geometry as well, most notably through the Hodge-theoretic approach of Adiprasito, Huh, and Katz to long-standing log-concavity conjectures for matroids~\cite{adiprasito2018hodge}.
	
	Parallel to this, the case when the ground set is \emph{continuous} was first studied in the seminal work of Choquet~\cite{choquet1954theory} on capacities of Borel sets. This line of research led to the Choquet integral and to the study of infinite-alternating functions, analytic counterparts of coverage functions, which form a well-studied subclass of submodular functions~\cite{bhaskar2019complexity,chakrabarty2015recognizing,berczi2026monotonic}. Despite these parallel developments, the interaction between the discrete and continuous theories remained limited for several decades.
	
	The emergence of \emph{graph limit theory}~\cite{lovasz2012large} changed the way finite combinatorial objects are connected to measurable ones. Large finite graphs can often be interpreted as approximations of analytic objects, such as graphons in the dense setting and graphings in the bounded-degree setting~\cite{benjamini2001recurrence,BCLSV1,BCLSV2,lovasz2012large}. This makes it possible to study finite graph parameters, local structure, and optimization problems after passing to a measurable limit object. In a recent extensive work, Lovász developed Choquet's theory from a combinatorial viewpoint~\cite{lovasz2023submodular}, reinterpreting analytic capacities through the lens of submodularity and initiating a systematic study of problems arising from extending the theory of finite submodular functions to the analytic setting. A natural question is whether there is a comparable measurable framework for matroidal independence and submodular optimization.
	
	At first sight, the analogy with graph limits is compelling. Matroids generalize the independence structure of vector spaces and graphs, and encode many central graph-theoretic notions, including forests, spanning trees, matchings, connectivity, and planarity. Moreover, by Whitney's theorem~\cite{whitney19332}, the cycle matroid of a graph determines the graph up to elementary transformations under mild assumptions. Thus, a measurable theory of matroids should also shed light on graph limits, especially on questions where cycles, forests, and spanning structures play the central role.
	
	There is, however, a fundamental obstacle. Most graph limit theories rely on local neighborhoods, homomorphism densities, or sampling small induced subgraphs. Matroids do not have a comparable notion of local neighborhood, nor is there a canonical analogue of homomorphism density that captures their structure through fixed finite test objects. In particular, even for graphic matroids, the natural ground set is the edge set, and independence is a global acyclicity condition. This makes it unlikely that a matroid limit theory can be obtained by a direct adaptation of existing graph limit methods. A different approach to matroid limits, of bounded-branch-depth matroids representable over a fixed finite field, was studied by Kardo\v{s}, Kr\'al', Liebenau, and Mach~\cite{kardovs2017first}.
	
	A first step in this direction was taken by Bérczi, Borbényi, Lovász, and Tóth~\cite{berczi2026quotient}, who introduced \emph{quotient-convergence} for submodular set functions. Given a set function $\varphi$ on a set algebra and a finite measurable partition $\mathcal{Q}$ of the ground set, the quotient of $\varphi$ with respect to $\mathcal{Q}$ records the values of $\varphi$ on unions of parts of $\mathcal{Q}$. A sequence of set functions is quotient-convergent if these finite quotients converge, in Hausdorff distance, for every fixed number of parts. This notion is particularly natural for matroids, since it observes the asymptotic behavior of their rank functions rather than attempting to sample local neighborhoods.
	
	The quotient-convergence framework gives a way to discuss limits of finite submodular functions, but convergence alone does not identify the right objects on the limit side. In graph limit theory, graphons and graphings are not merely limiting set functions; they are structured measurable objects on which graph-theoretic concepts continue to make sense. The analogous problem for matroids is to find measurable objects that retain the most important structural features of finite matroids. The aim of this paper is to identify such objects and to develop their basic theory.
	
	We introduce \emph{measurable matroids}, analytic counterparts of finite matroids defined on standard atomless measure spaces. Informally, a measurable matroid consists of a family of Borel sets, called independent sets, which is downward closed, closed under increasing countable unions, and satisfies a measurable exchange axiom. Maximal independent subsets of a given measurable set are required to have the same measure, and this common value defines the rank. The resulting rank function is a monotone submodular set function dominated by the ground measure, and thus fits naturally into the framework of submodular set functions on $\sigma$-algebras discussed above. In \cref{sec:atomic-case}, we describe what happens if the same independence axioms are imposed on a finite standard Borel measure space with atoms. As we will see, this causes no loss of generality, since the atomic part separates from the atomless part.
	
	We show that measurable matroids carry the full basic calculus expected from matroid theory: independent sets, bases, rank functions, closure operators, duality, and minors all have measurable analogues. A motivating example comes from graphings. Lov\'asz~\cite{lovasz2024matroid} associated with a graphing $G=(J,\Sigma_J,\nu,E)$ a monotone submodular rank function $\rho_G$ on the Borel edge space $(E,\Sigma_E)$, generalizing the normalized rank of the cycle matroid of a finite graph. B\'erczi, Borb\'enyi, Lov\'asz, and T\'oth~\cite{berczi2026cycle} showed that, for graphings, hyperfinite forests and hyperfinite essential spanning forests behave as the independent sets and bases of a cycle matroid. They also showed that if a sequence of finite bounded-degree graphs converges to a graphing $G$ in the local-global sense, then the normalized rank functions of their cycle matroids converge to $\rho_G$ under quotient-convergence. Measurable matroids place such examples into a general axiomatic and optimization framework.
	
	Beyond this structural theory, we prove measurable analogues of two central results of finite matroid optimization: Edmonds' matroid intersection theorem~\cite{edmonds1970submodular} and the Edmonds--Fulkerson~\cite{edmonds1965transversals} matroid union theorem. The intersection theorem retains the classical min--max value, although the optimum need not be attained. Lov\'asz's fractional intersection theorem gives the same value when common independent sets are replaced by minorizing measures, so there is no gap between the fractional and measurable-set optimal values. 
	
	Attainment is therefore a separate question. We prove that a rank-expansion condition guarantees attainment, and in the full-rank case yields a common basis. For the partition matroids associated with a bipartite graphing, this condition specializes to the expansion condition in the measurable perfect matching theorem of Lyons and Nazarov~\cite{lyons2011perfect}. Thus, the same expansion principle extends from measurable matchings to arbitrary pairs of measurable matroids.
	
	%%%%%%%%%%%%%%%%%%%%%%%%%%%%%%%%
	\subsection{Previous Models of Infinite Matroids}
	%%%%%%%%%%%%%%%%%%%%%%%%%%%%%%%%
	
	There are three existing lines of work that are closely related to our approach, but none provides the framework needed here. 
	
	The first is the purely combinatorial theory of infinite matroids. One of the historically most natural extensions is the class of \emph{finitary matroids}, in which every dependent set contains a finite circuit. Finitary matroids, however, do not behave well with respect to duality. To overcome this difficulty, Bruhn, Diestel, Kriesell, Pendavingh, and Wollan~\cite{bruhn2013axioms} introduced an axiomatic framework for infinite matroids, later connected to earlier work of Higgs~\cite{higgs1969matroids} on $B$-matroids. This theory successfully incorporates topological cycle matroids of infinite graphs and gives a robust discrete infinite theory; see also~\cite{bruhn2011infinite,bowler2018infinite}. The analogues of matroid intersection, union, base packing, and base covering have also been studied in this framework; see, for example, \cite{bowler2015matroid,aigner2011infinite}. Nevertheless, these models are fundamentally set-theoretic rather than measure-theoretic. They do not naturally capture the analytic behavior of large finite matroids under convergence, nor do they interact directly with graphings or Choquet-type submodular functions. 
	
	The second line of work studies continuous analogues of matroids through lattices equipped with real-valued rank functions. Bj\"orner introduced the continuous partition lattice as a limit of finite partition lattices \cite{bjorner1987continuous}, and Bj\"orner and Lov\'asz developed a more general theory of pseudomodular lattices and continuous matroids \cite{bjorner1987pseudomodular}; see also Bj\"orner's later account \cite{bjorner2019continuous}. Haiman subsequently gave a realization of Bj\"orner's continuous partition lattice in terms of measurable partitions \cite{haiman1994realization}. These approaches are closely related in spirit to ours, but their basic objects are lattices or measurable partitions rather than families of measurable subsets of a fixed measure space satisfying independence axioms.
	
	The third line of work is convex-analytic. A \emph{charge} on a set algebra $(J,\Sigma)$ is a function $\alpha\colon\Sigma\to\bR$ such that $\alpha(\emptyset)=0$ and $\alpha(A\cup B)=\alpha(A)+\alpha(B)$ whenever $A,B\in\Sigma$ are disjoint. For an increasing submodular set function $\varphi$ on $(J,\Sigma)$ with $\varphi(\emptyset)=0$, Lovász~\cite{lovasz2023submodular} considered the set of minorizing charges
	\[
	\mm(\varphi)=\{\alpha\colon\Sigma\to\bR_{\geq 0} \mid \alpha\text{ is a charge and }\alpha(A)\leq\varphi(A)\text{ for all }A\in\Sigma\},
	\]
	and its basic version
	\[
	\bmm(\varphi)=\{\alpha\in\mm(\varphi) \mid \alpha(J)=\varphi(J)\}.
	\]
	These sets are analytic analogues of the independence and base polytopes of a polymatroid. Since $\varphi$ is nonnegative, it follows from~\cite[Corollary~5.3(b)]{lovasz2023submodular} that $\bmm(\varphi)$ is nonempty. In the finite case, when $\varphi$ is the rank function of a matroid, $\mm(\varphi)$ and $\bmm(\varphi)$ are precisely its independence and base polytopes, respectively; consequently, their extreme points are precisely the characteristic vectors of its independent sets and bases, respectively~\cite{frank2011connections,schrijver2003combinatorial}. It is therefore tempting to define bases of a measurable matroid as extreme points of $\bmm(\varphi)$.
	
	This approach faces two difficulties. First, the extreme points of $\bmm(\varphi)$ are not yet well understood in general. Even for the cycle matroid of a graphing, where exposed points correspond to hyperfinite essential spanning forests~\cite{berczi2026cycle}, it remains open whether all extreme points are exposed. Second, the minorizing-charge construction does not by itself encode the subcardinality property of matroid rank functions. In the finite setting, this is the condition that adding one element increases the rank by at most one. Thus, this definition seems too general to serve as a measurable analogue of matroids: it gives a generalization of polymatroids rather than matroids. These observations suggest that a measurable theory of matroids should not be based only on extreme measures. Our definition is intended to overcome these difficulties through an axiomatic approach.
	
	%%%%%%%%%%%%%%%%%%%%%%%%%%%%%%%%
	\subsection{Our Results}
	%%%%%%%%%%%%%%%%%%%%%%%%%%%%%%%%
	
	Our first contribution is to give an axiomatic foundation for matroids over standard atomless measure spaces. We define measurable matroids through independent sets and prove that the usual equivalent descriptions of finite matroids extend to the measurable setting; see \cref{sec:axioms}. Thus, measurable matroids can be characterized by their independent sets, bases, rank functions, or closure operators. However, several properties that are automatic over finite ground sets must be imposed explicitly in the measurable setting. In particular, closure under increasing limits and closure under convergence in measure play an essential role. We also treat the case of measures with atoms in \cref{sec:atomic-case}.
	
	We then discuss several basic examples of measurable matroids; see \cref{sec:examples}. Finite matroids embed into the theory after the usual normalization, and cycle matroids of graphings provide a motivating infinite example. We also define measurable analogues of partition, nested, lattice path, transversal, and matching matroids. These examples show that measurable matroids capture a variety of natural independence structures.
	
	In \cref{sec:operations}, we study basic operations on measurable matroids. We show that the class is closed under truncation, elongation, direct sum, duality, restriction, contraction, and union. The existence of duals is particularly important, since one of the difficulties with approaches based on finitary matroids is that this class is not closed under duality~\cite{bruhn2011infinite}. In our setting, the dual is obtained by taking complements of bases, just as in the finite case.
	
	The union operation reveals a basic difference between the measurable and finite settings. Given measurable matroids $M_1,\dots,M_k$ on the same ground space, the finite analogy suggests declaring a set independent if it can be written as a union of independent sets, one from each $M_i$. In contrast to the finite case, this family need not be closed under increasing unions. It is therefore natural to take its closure under such unions. We show that the resulting family is a measurable matroid.
	
	In \cref{sec:measurable}, we establish connections with the theory of infinite submodularity and with the convex-analytic viewpoint of minorizing charges. The linear extension gives a useful form of the generalized submodular inequality. The discussion of minorizing charges explains why the extreme-point theory of base polytopes alone does not recover the full matroidal structure: even for natural measurable rank functions, an extreme basic minorizing charge need not be the restriction of the ground measure to a basis.
	
	In \cref{sec:intun}, we prove the measurable matroid intersection theorem. If $M_1=(J,\Sigma,\mu,\cF_1)$ and $M_2=(J,\Sigma,\mu,\cF_2)$ are measurable matroids with rank functions $r_1$ and $r_2$, then
	\begin{equation*}
		\sup\{\mu(F)\mid F\in\cF_1\cap\cF_2\}
		=
		\min_{X\in\Sigma}\{r_1(X)+r_2(J\setminus X)\}.
	\end{equation*}
	As discussed above, this has the same optimal value as Lov\'asz's fractional intersection theorem~\cite[Theorem~5.15]{lovasz2023submodular}; see also~\cite{yu2026submodular}. The proof is not a formal extension of Edmonds' finite argument. In the measurable setting, the usual circuit-based exchange graph is unavailable, and single-element exchanges have no effect on the measure. We replace these tools by closure-based analogues of alternating reachability and measurable augmenting chains, together with a limiting argument that allows the augmentation process to be iterated. We also prove an attainment theorem under a rank-expansion condition; in the full-rank case, it yields a common basis. For partition matroids associated with bipartite graphings, the condition specializes to that of the measurable perfect matching theorem of Lyons and Nazarov~\cite{lyons2011perfect}.
	
	For measurable union, we obtain the Edmonds--Fulkerson rank formula. If $M_1,\dots,M_k$ have rank functions $r_1,\dots,r_k$, then  the rank of their measurable union is
	\begin{equation*}
		r_\vee(A)
		=
		\min_{X\subseteq A}
		\left\{
		\mu(A\setminus X)+\sum_{i=1}^k r_i(X)
		\right\}.
	\end{equation*}
	The proof follows the finite reduction of matroid union to matroid intersection, by passing to disjoint copies and intersecting the direct sum with a finite-classed measurable partition matroid.
	
	Finally, we discuss several structural applications in \cref{sec:applications}. Partition matroids and the intersection theorem give an exact Hall-deficiency formula for the supremal measure of a matching in a bipartite graphing, while the attainment results yield perfect matchings under rank expansion. Applied to the two partition matroids associated with the coordinate projections of a product space, the intersection theorem also gives an atomless min--max formula for continuous bipartite $b$-matchings. This problem can equivalently be viewed as a capacity-constrained marginal problem closely related to optimal transport. Together with weak-star compactness, the Krein--Milman theorem, and the extreme-point characterization of Korman and McCann~\cite[Proposition~3.2]{korman2013insights}, this yields capacitated Hall and Lorentz-type results for prescribed cross-sections. Related measurable arboricity questions have been considered previously~\cite{aim2022descriptive}. Applying the union theorem to cycle matroids of graphings gives measurable Nash-Williams--Tutte theorems: exact characterizations of approximate coverings and packings by hyperfinite essential spanning forests. Under strengthened rank inequalities, the union-attainment theorem gives exact partitions into hyperfinite forests and exact packings by hyperfinite essential spanning forests. We also obtain an approximate measurable version of the Greene--Magnanti basis-exchange theorem.
	
	%%%%%%%%%%%%%%%%%%%%%%%%%%%%%%%%
	\subsection{Organization and Reading Guide}
	%%%%%%%%%%%%%%%%%%%%%%%%%%%%%%%%
	
	The paper brings together ideas from matroid theory, measurable combinatorics, and submodular analysis. Since results and techniques that are standard in one of these areas may be less familiar to readers working in the others, we have tried to make the presentation self-contained and sometimes include details that specialists may consider routine. This also explains the paper's broad scope: besides developing the basic theory of measurable matroids, we discuss a range of examples, operations, min--max theorems, and applications. We rely on numerous fields in this paper, thus, it is not possible to include all the necessary background; for general background, we refer the reader to Oxley~\cite{oxley2011matroid} for matroid theory, Folland~\cite{folland1999real} for measure theory, Kechris~\cite{kechris1995classical} for standard Borel spaces and descriptive set theory, and Fujishige~\cite{fujishige2005submodular} for submodular functions and optimization.
	
	The rest of the paper is organized as follows. \Cref{sec:axioms} gives the axiomatic foundations of measurable matroids, discusses the case when the ground measure has atoms, and proves the equivalence of the independence, basis, rank, and closure formulations. The main examples of measurable matroids are discussed in \cref{sec:examples}. The basic operations, including truncation, elongation, direct sum, duality, restriction, contraction, and union, are developed in \cref{sec:operations}. \Cref{sec:measurable} discusses connections with infinite submodularity and minorizing charges. The measurable matroid intersection theorem is proved in \cref{sec:intun}; the same section also gives an expansion criterion for attainment of the intersection supremum and proves the rank formula for measurable union. Finally, structural applications to bipartite matchings in graphings, continuous bipartite $b$-matchings and prescribed cross-sections, measurable Nash-Williams--Tutte theorems, and basis exchanges are given in \cref{sec:applications}.
	
	%%%%%%%%%%%%%%%%%%%%%%%%%%%%%%%%
	\subsection{Tools and Notation}
	\label{sec:notation}
	%%%%%%%%%%%%%%%%%%%%%%%%%%%%%%%%
	
	We denote the sets of \emph{real numbers}, \emph{integers}, and \emph{positive integers} by $\bR$, $\bZ$, and $\bN$, respectively, adding the subscript $\geq 0$ when restricting to \emph{nonnegative} values. For a positive integer $n$, we use $[n]\coloneqq\{1,\dots,n\}$. For sets $A, B \subseteq J$, we write $A \triangle B = (A \setminus B) \cup (B \setminus A)$. If $Y$ consists of a single element $y$, then $X\setminus \{y\}$ and $X\cup \{y\}$ are abbreviated as $X-y$ and $X+y$, respectively. 
	
	We reserve the symbol $\cB$ for the family of bases of measurable matroids, and the Borel $\sigma$-algebra of a space is always denoted by $\Sigma$. Throughout the paper, all sets that appear as subsets of an ambient measure space are assumed to belong to the relevant $\sigma$-algebra, unless explicitly stated otherwise. Thus, when the ambient space is denoted by $(J,\Sigma,\mu)$, expressions such as $X\subseteq J$ or $F\subseteq X\subseteq J$ mean that all the sets involved belong to $\Sigma$. Throughout the paper, we omit routine checks that the sets and maps under consideration are Borel. We give details only when this requires a separate argument or a standard result from descriptive set theory.
	
	We use the following convention throughout, without further mention. Measurable sets are identified whenever they differ by a null set. Thus, unless a literal pointwise statement is explicitly intended, $A=B$ means $\mu(A\triangle B)=0$, while $A\subseteq B$ means $\mu(A\setminus B)=0$. Set-theoretic operations are interpreted in the resulting quotient, and whenever a pointwise argument is needed, we pass to suitable Borel representatives. These operations are well-defined for finite and countable families because the null sets form a $\sigma$-ideal. The resulting quotient is called the \emph{measure algebra} of $(J,\Sigma,\mu)$; see~\cite{kechris1995classical}. Since every set measurable with respect to the completion of $\mu$ differs from a Borel set by a null set, passing to the completion would give the same measure algebra. We therefore work with Borel sets throughout.
	
	A measurable space $(J,\Sigma)$ is called a \emph{standard Borel space} if there exists a complete separable metric on $J$ whose Borel $\sigma$-algebra is $\Sigma$. In this paper, a measure space $(J,\Sigma,\mu)$ is called \emph{standard} if $(J,\Sigma)$ is a standard Borel space and $\mu\colon\Sigma\to\bR_{\geq 0}$ is finite and atomless. If $\mu(J)>0$, then, by \cite[Theorem~17.41]{kechris1995classical}, the normalized space $(J,\Sigma,\mu/\mu(J))$ is measure-isomorphic to the unit interval equipped with Lebesgue measure. Thus, when measurable matroids are considered only up to measure-space isomorphism, one may work on $[0,1]$ without loss of generality. For measurable spaces $(S,\Sigma_S)$ and $(T,\Sigma_T)$, we write $\Sigma_S\otimes\Sigma_T$ for the product $\sigma$-algebra on $S\times T$. If $\nu$ and $\eta$ are measures on $S$ and $T$, respectively, then $\nu\otimes\eta$ denotes their product measure. For $X\subseteq S\times T$, we write $X_t\coloneqq\{s\in S\mid (s,t)\in X\}$ and $X^s\coloneqq\{t\in T\mid (s,t)\in X\}$ for its sections at $t\in T$ and $s\in S$, respectively.

	Let $\varphi$ be a set function on a set algebra $(J,\Sigma)$. We say that $\varphi$ is \emph{submodular} if
	\[
	\varphi(X)+\varphi(Y)\geq\varphi(X\cap Y)+\varphi(X\cup Y)
	\]
	for all $X,Y\in\Sigma$, and \emph{modular} if equality holds for all $X,Y\in\Sigma$. A \emph{charge} on $(J,\Sigma)$ is a function $\alpha\colon\Sigma\to\bR$ such that $\alpha(\emptyset)=0$ and $\alpha(A\cup B)=\alpha(A)+\alpha(B)$ whenever $A,B\in\Sigma$ are disjoint. For an increasing submodular set function $\varphi$ with $\varphi(\emptyset)=0$, let
	\[
	\mm(\varphi)\coloneqq\{\alpha\colon\Sigma\to\bR_{\geq0}\mid \alpha\text{ is a charge and }\alpha(A)\leq\varphi(A)\text{ for every }A\in\Sigma\}
	\]
	denote the set of \emph{minorizing charges} of $\varphi$, and let
	\[
	\bmm(\varphi)\coloneqq\{\alpha\in\mm(\varphi)\mid \alpha(J)=\varphi(J)\}
	\]
	denote the set of \emph{basic minorizing charges}.
	
	We will repeatedly use the following \emph{exact-subset property} of atomless measures: if $A\in\Sigma$ and $0\leq t\leq\mu(A)$, then there exists a measurable set $B\subseteq A$ with $\mu(B)=t$. For a set $X\in\Sigma$, we write $\mathbf 1_X$ for its \emph{indicator function}. Let $\cX\subseteq\Sigma$ be a family of sets. A set $X\in\cX$ is \emph{maximal in $\cX$} if $X\subseteq Y$ for some $Y\in\cX$ implies $Y\subseteq X$, and \emph{minimal in $\cX$} if $Y\subseteq X$ for some $Y\in\cX$ implies $X\subseteq Y$. For $Z \subseteq J$, we denote the restriction of $\cX$ to $Z$ by
	\[
	\cX|Z = \{X \in \cX \mid X \subseteq Z\}.
	\]
	We say that a sequence $(X_i)_{i \in \bN}$ in $\cX$ is \emph{increasing} if $X_i \subseteq X_{i+1}$ for all $i\in\bN$, and \emph{decreasing} if $X_i\supseteq X_{i+1}$ for all $i\in\bN$. The \emph{upward closure} of $\cX$ is defined by
	\begin{equation*}
		\overline{\cX} \coloneqq \Big\{\bigcup_{i=1}^\infty X_i \mid (X_i)_{i \in \bN} \text{ is increasing in } \cX\Big\}.   
	\end{equation*}
	Similarly, the \emph{downward closure} of $\cX$ is defined by
	\begin{equation*}
		\underline{\cX} \coloneqq \Big\{\bigcap_{i=1}^\infty X_i \mid (X_i)_{i \in \bN} \text{ is decreasing in } \cX\Big\}.   
	\end{equation*}
	We call $\cX$ \emph{$\sigma$-increasing} or \emph{$\sigma$-decreasing} if $\cX=\overline{\cX}$ or $\cX=\underline{\cX}$ holds, respectively.\footnote{Such families also appear in the definition of a \emph{monotone class}; see, e.g., \cite[Chapter 1]{halmos1950measure} for background.} 
	
	\begin{lem} \label{lem:upward_closure}
		Let $\cX \subseteq \Sigma$ be a family closed under countable intersections. Then $\overline{\cX}$ is $\sigma$-increasing.
	\end{lem}
	
	\begin{proof}
		Let $Z \in \overline{\overline{\cX}}$ and let $(Z_i)_{i \in \bN}$ be an increasing sequence in $\overline{\cX}$ such that $Z=\bigcup_{i=1}^\infty Z_i$. Hence, for each $i \in \bN$, there exists $X_i \in \cX$ such that $X_i \subseteq Z_i$ and $\mu(Z_i \setminus X_i) \leq \frac{1}{2^i}$. Let 
		$W_{k} \coloneqq \bigcap_{i=k}^\infty X_i \subseteq Z_k$,
		which belongs to $\cX$ for every $k \in \bN$, as it is a countable intersection of sets from $\cX$. Clearly, $(W_i)_{i \in \bN}$ is an increasing sequence. We have
		$$
		\mu(Z_k \setminus W_k) \leq \sum_{i=k}^\infty \mu(Z_k \setminus X_i) \leq  \sum_{i=k}^\infty \mu(Z_i \setminus X_i) \leq \frac{1}{2^{k-1}}.
		$$
		Consequently, 
		$$Z=\bigcup_{i=1}^\infty Z_i=\bigcup_{i=1}^\infty W_i \in \overline{\cX},$$ as desired.
	\end{proof}
	
	Unfortunately, it is not always true that for any family $\cX$ we have $\overline{\cX}=\overline{\overline{\cX}}$; that is, it is possible that $\overline{\cX}$ is not $\sigma$-increasing, as the following remark shows. However, we will only take upward closures of families that satisfy the conditions of \cref{lem:upward_closure}, which guarantees that $\overline{\cX}$ is $\sigma$-increasing.
	
	\begin{rem} \label{rem:upward_closure}
		Let $(Z_i)_{i \in \bN}$ be an increasing family with $Z_i \subsetneq Z_{i+1}$ for all $i \in \bN$, and let $A_{i,j}$ be pairwise disjoint subsets of $Z_1$ with positive measure for $i,j \in \bN$. Let $X_{k,n} \coloneqq Z_k \setminus \bigcup_{i=n}^\infty A_{k,i}$ for $k,n \in \bN$, and define $\cX \coloneqq \{ X_{k,n} \mid k,n \in \bN \}$. Then $X_{k,n} \subseteq X_{k',n'}$ if and only if $k'=k$ and $n \leq n'$. Consequently, every increasing sequence in $\cX$ is contained
		in one of the chains $(X_{k,n})_{n\in\bN}$ with $k$ fixed. The union of such a sequence is either some $X_{k,n}$, if the indices $n$ remain bounded, or $Z_k$, if
		they are unbounded. Hence
		\[
		\overline{\cX}=\cX \cup \{Z_k \mid k \in \bN\}.
		\] This is not $\sigma$-increasing, since $\bigcup_{i=1}^\infty Z_i \in \overline{\overline{\cX}} \setminus \overline{\cX}$.
	\end{rem}
	
	The following folklore lemma is a straightforward consequence of Zorn's lemma~\cite{zorn1935remark}; see also, for example, \cite[Theorem 2.4]{jech2003set}.
	
	\begin{lem} \label{lem:maximal}
		Let $(J, \Sigma, \mu)$ be a standard measure space, and let $\cX \subseteq \Sigma$ be a nonempty family. If $\cX$ is $\sigma$-increasing, then it contains at least one maximal set, while if $\cX$ is $\sigma$-decreasing, then it contains at least one minimal set.
	\end{lem}
	\begin{proof}
		We only prove the first part of the lemma; the second follows analogously. 
		
		By Zorn's Lemma, it suffices to show that any chain $\cC \subseteq \cX$ has an upper bound in $\cX$. Let $\cC$ be such a chain, and set $s \coloneqq \sup\{\mu(C) \mid C \in \cC\}$. Since $\cC$ is a chain, there exists a countable increasing subchain $(C_i)_{i\in\bN}$ in $\cC$ such that
		\[
		\lim_{i \to \infty} \mu(C_i) = s.
		\]
		Let $D \coloneqq \bigcup_{i=1}^{\infty} C_i$. Then $D \in \cX$, since $\cX$ is $\sigma$-increasing, and $\mu(D) = s$.
		
		We claim that $D$ is an upper bound of $\cC$. Suppose, for contradiction, that there exists $C \in \cC$ with $\mu(C \setminus D) > 0$. Then $C \not\subseteq C_i$ for any $i$, which implies $C_i \subseteq C$ for all $i \in \bN$. Consequently, $D \subseteq C$, so $\mu(C) >\mu(D) = s$, contradicting the definition of $s$.
		Hence, every chain in $\cX$ has an upper bound, as required.
	\end{proof}
	
	We will also need the following trivial corollary of \cref{lem:maximal}. 
	
	\begin{cor} \label{cor:unionclosed}
		Let $(J, \Sigma, \mu)$ be a standard measure space, and let $\cX \subseteq \Sigma$ be a nonempty $\sigma$-increasing family that is closed under finite unions. Then $\cX$ contains a unique maximal set $X$. Furthermore, $Y \subseteq X$ for all $Y \in \cX$. Similarly, if $\cX$ is a $\sigma$-decreasing family closed under finite intersections, then it contains a unique minimal set $X$ with $X \subseteq Y$ for all $Y \in \cX$.
	\end{cor}
	
	\begin{proof}
		We prove only the first part; the second follows analogously.
		
		By \cref{lem:maximal}, $\cX$ contains a maximal set $X$. Let $Y\in \cX$. Since $\cX$ is closed under finite unions, we have $X\cup Y\in \cX$. By the maximality of $X$, it follows that $Y\subseteq X$. In particular, if $X'$ is another maximal set in $\cX$, then $X'\subseteq X$ and $X\subseteq X'$, so $X=X'$. Thus, $X$ is the unique maximal set.
	\end{proof}
	
	Let $(J,\Sigma,\mu)$ be a finite measure space, and let $(F_i)_{i\in\bN}$ be a sequence of measurable sets. We use the usual terminology
	$$
	\liminf_{i\to\infty} F_i \coloneqq \bigcup_{n=1}^{\infty} \bigcap_{i=n}^\infty F_i \qquad \text{and} \qquad
	\limsup_{i\to\infty} F_i \coloneqq \bigcap_{n=1}^{\infty} \bigcup_{i=n}^\infty F_i.
	$$
	Furthermore, if $\liminf_{i\to\infty} F_i=\limsup_{i\to\infty} F_i$, then we say that the \emph{limit} of the sequence $(F_i)_{i\in\bN}$ exists, and we write
	$$
	\lim_{i\to\infty} F_i \coloneqq \liminf_{i\to\infty} F_i=\limsup_{i\to\infty} F_i.
	$$
	We will use the following standard form of the Borel--Cantelli lemma; see the original works of Borel~\cite{borel1909les} and Cantelli~\cite{cantelli1917sulla}, or the book of Billingsley~\cite{billingsley1995probablity} for a more recent source.
	
	\begin{prop}[Borel--Cantelli lemma] \label{prop:Borel_Cantelli}
		Let $(J,\Sigma,\mu)$ be a finite measure space, and let $(E_i)_{i\in\bN}$ be a sequence of sets. If
		$
		\sum_{i=1}^\infty \mu(E_i)<\infty
		$,
		then
		$
		\mu(\limsup_{i\to\infty} E_i)=0.
		$
	\end{prop}
	
	For convenience, we will also state the following simple corollary.
	
	\begin{cor} \label{cor:Borel_Cantelli}
		Let $(J,\Sigma,\mu)$ be a finite measure space, and let $(F_i)_{i\in\bN}$ be a sequence of sets. Assume that
		$
		\sum_{i=1}^\infty \mu(F_i\triangle F_{i+1})<\infty.
		$
		Then 
		$
		\lim_{i\to\infty} F_i$ exists and $\lim_{i \to \infty} \mu(F_i)=\mu(\lim_{i \to \infty} F_i)$.
	\end{cor}
	\begin{proof}
		Set $E_i\coloneqq F_i\triangle F_{i+1}$. By the Borel--Cantelli lemma, $\mu(\limsup_{i\to\infty} E_i)=0$. If $x\notin \limsup_{i\to\infty} E_i$, then $x\notin E_i$ for all sufficiently large $i$. Consequently, for almost every $x\in J$, the membership of $x$ in $F_i$ is eventually constant, implying that $x\in \liminf_{i\to\infty}F_i$ if and only if $x\in \limsup_{i\to\infty}F_i$. Thus, $\liminf_{i\to\infty}F_i=\limsup_{i\to\infty}F_i$, ensuring the existence of the limit $F\coloneqq\lim_{i\to\infty}F_i$. Moreover, for every $i$ we have $F_i\triangle F\subseteq \bigcup_{j=i}^{\infty}E_j$. Therefore, $\mu(F_i\triangle F)\leq \sum_{j=i}^{\infty}\mu(E_j)$, and the right-hand side tends to $0$. Hence, $\mu(F_i\triangle F)\to0$, which implies $\mu(F_i)\to\mu(F)$.
	\end{proof}
	
	%%%%%%%%%%%%%%%%%%%%%%%%%%%%%%%%
	\section{Axiomatic Characterizations}
	\label{sec:axioms}
	%%%%%%%%%%%%%%%%%%%%%%%%%%%%%%%%
	
	Finite matroids can be axiomatized through independent sets, bases, circuits, rank functions, or closure operators, all of which yield the same combinatorial structure; we refer the interested reader to~\cite{oxley2011matroid}. These different perspectives provide powerful tools for addressing both structural and algorithmic questions in finite matroid theory. We develop a similarly robust axiomatic framework for measurable matroids, which generalize matroids to standard measure spaces. 
	
	In the following, we present the independent set, basis, rank, and closure formulations. We do not include a circuit formulation among the measurable characterizations. In the atomless setting, inclusion-minimal dependent sets are not a robust analogue of finite circuits: a dependent measurable set need not contain a minimal dependent subset of positive measure.  
	
	As discussed in \cref{sec:notation}, all equalities and inclusions between measurable sets are understood in the measure algebra unless explicitly stated otherwise.
	
	%%%%%%%%%%%%%%%%%%%%%%%%%%%%%%%%
	\subsection{Independence Axioms}
	\label{sec:independence}
	%%%%%%%%%%%%%%%%%%%%%%%%%%%%%%%%
	
	We call $M=(J, \Sigma, \mu, \cF)$ a \emph{measurable matroid} if $(J, \Sigma, \mu)$ is a standard measure space and $\cF \subseteq  \Sigma$ satisfies the following \emph{independence axioms}:
	\begin{enumerate}[label=\normalfont{(I\arabic*)}, left=0pt, itemsep=0em]
		\item $\emptyset \in \cF$. \label{ax:i1}
		\item If $F_1 \subseteq  F_2 \in \cF$, then $F_1 \in \cF$. \label{ax:i2}
		\item For all $X\in\Sigma$ and all maximal sets $F_1, F_2 \in \cF | X$, we have $\mu(F_1)=\mu(F_2)$. \label{ax:i3}
		\item $\cF$ is $\sigma$-increasing, that is, for every increasing sequence $(F_i)_{i\in\bN}$ in $\cF$, we have $\bigcup_{i=1}^\infty F_i \in \cF$. \label{ax:i4}
	\end{enumerate}
	
	\begin{rem}
		If $\mu(J)>0$, normalizing $\mu$ to $\mu/\mu(J)$ does not affect the definition, so one could equivalently assume that $\mu$ is a probability measure; we stick to the original form for convenience.
	\end{rem}
	
	Following the terminology of finite matroids, we call the sets in $\cF$ \emph{independent} and the remaining sets \emph{dependent}. The maximal independent sets in $\cF$ are called \emph{bases}. For $X\in\Sigma$, define the rank $r(X)\coloneqq\mu(F)$, where $F \in \cF|X$ is a maximal independent subset. A maximal such set exists by \cref{lem:maximal}, and the value of $r(X)$ does not depend on the choice of $F$ thanks to \ref{ax:i3}. The resulting function $r \colon \Sigma \to \bR_{\geq 0}$ is called the \emph{rank function}.
	
	In certain cases, it is more convenient to work with the following variant of \ref{ax:i3}. Note that the usual finite axiom, which guarantees that an element of a larger independent set can always be added to a smaller one, has no direct analogue in the measurable setting, where it is not meaningful to speak of individual elements.
	
	\begin{enumerate}[label=\normalfont{(I\arabic*')}, left=0pt, itemsep=0em]\setcounter{enumi}{2}
		\item For all $F_1, F_2 \in \cF$ with $\mu(F_1)<\mu(F_2)$, there exists $G \in \cF$ with $F_1 \subsetneq G \subseteq  F_1 \cup F_2$. \label{ax:i3'}
	\end{enumerate} 
	
	The set of properties \ref{ax:i1}, \ref{ax:i2}, \ref{ax:i3'} and \ref{ax:i4} was previously studied in~\cite{berczi2026cycle} in the context of cycle matroids of graphings. The following proposition shows
	that, in the presence of \ref{ax:i4}, axioms \ref{ax:i3} and
	\ref{ax:i3'} are equivalent.
	
	\begin{prop} \label{prop:i3i3'}
		\{\ref{ax:i3},\ref{ax:i4}\} and \{\ref{ax:i3'},\ref{ax:i4}\} are equivalent. 
	\end{prop}
	
	\begin{proof}
		Assume first that \ref{ax:i3} holds, and let $F_1,F_2\in\cF$ satisfy
		$\mu(F_1)<\mu(F_2)$. By \ref{ax:i4} and \cref{lem:maximal}, there exists
		a maximal set $H\in\cF|(F_1\cup F_2)$ containing $F_2$. If $F_1$ were
		also maximal in $\cF|(F_1\cup F_2)$, then \ref{ax:i3} would give
		$\mu(F_1)=\mu(H)\geq\mu(F_2)$, a contradiction. Thus, $F_1$ is not
		maximal in $\cF|(F_1\cup F_2)$, so there exists $G\in\cF$ with
		$F_1\subsetneq G\subseteq F_1\cup F_2$. This proves \ref{ax:i3'}.
		
		Conversely, assume that \ref{ax:i3'} holds, and let $X\in\Sigma$ with
		maximal sets $F_1,F_2\in\cF|X$. If $\mu(F_1)\neq\mu(F_2)$, the set of
		smaller measure is not maximal in $F_1\cup F_2$ by \ref{ax:i3'}, hence
		it cannot be maximal in $X$. This proves \ref{ax:i3}.
	\end{proof}
	
	\begin{rem}
		The assumption \ref{ax:i4} cannot be omitted. Let $J=[0,1]$ with Lebesgue measure $\lambda$, set
		$
		A=[1/2,1]$ and $H_n=[0,1-2^{-n}]$ for all $n \in \bN$.
		Define
		\[
		\cF\coloneqq\{F\in\Sigma\mid F\subseteq A
		\text{ or }F\subseteq H_n\text{ for some }n\in\bN\}.
		\]
		Clearly, \ref{ax:i1} and \ref{ax:i2} hold. To verify \ref{ax:i3}, let $X\in\Sigma$. If $X\in\cF$, then $X$ is the unique maximal member of $\cF|X$. Otherwise, $X\setminus H_n$ has positive measure for all $n\in\bN$. Hence, for every $n\in\bN$, there exists $m>n$ such that $X\cap(H_m\setminus H_n)$ has positive measure. Thus, no $Y\subseteq H_n\cap X$ is maximal in $\cF|X$, since
		$Y\subsetneq Y\cup(X\cap(H_m\setminus H_n))\in\cF|X.$
		Hence, the only possible maximal member of $\cF|X$ is $X\cap A$, so \ref{ax:i3} holds. On the other hand, $A$ is maximal in $\cF$ and $\lambda(A)<\lambda(H_2)$, so \ref{ax:i3'} fails.
	\end{rem}
	
	As a corollary of \cref{prop:i3i3'}, we obtain the following measurable analogue of the Steinitz augmentation property; see~\cite[Section~5.3]{lovasz2023submodular}.
	
	\begin{enumerate}[label=\normalfont{(I\arabic*{'}{'})}, left=0pt, itemsep=0em]\setcounter{enumi}{2}
		\item For $F_1, F_2 \in \cF$ with $\mu(F_1)\leq \mu(F_2)$, there exists $G \in \cF$ with $F_1 \subseteq G \subseteq F_1 \cup F_2$ and $\mu(G) \geq \mu(F_2)$. \label{ax:i3''} 
	\end{enumerate}
	
	We show that, assuming \ref{ax:i4}, we can exchange \ref{ax:i3'} to \ref{ax:i3''}.
	
	\begin{prop} \label{prop:strong_prematroid}
		\{\ref{ax:i3'},\ref{ax:i4}\} and \{\ref{ax:i3''},\ref{ax:i4}\} are equivalent. 
	\end{prop}
	
	\begin{proof}
		Since \ref{ax:i3''} generalizes \ref{ax:i3'}, one direction is immediate.
		
		For the converse, let $\cG \coloneqq \{F \in \cF \mid F_1 \subseteq  F \subseteq  F_1 \cup F_2\}$. Then $\cG$ is nonempty and $\sigma$-increasing by \ref{ax:i4}, so by \cref{lem:maximal} it contains a maximal element $G$. If $\mu(G)<\mu(F_2)$, \ref{ax:i3'} implies that $G$ is not maximal in $\cG$, a contradiction. Hence, $\mu(G) \ge \mu(F_2)$, which proves \ref{ax:i3''}.
	\end{proof}
	
	The following consequences of \ref{ax:i2} and \ref{ax:i4} will be used repeatedly.
	
	\begin{lem}\label{lem:liminf}
		Let $(J,\Sigma,\mu)$ be a standard measure space, and let $\cF\subseteq\Sigma$ satisfy \ref{ax:i2} and \ref{ax:i4}. If $(F_i)_{i\in\bN}$ is a sequence in $\cF$, then $\liminf_{i\to\infty}F_i\in\cF$.
	\end{lem}
	
	\begin{proof}
		For every $n\in\bN$, the set $G_n\coloneqq\bigcap_{i=n}^\infty F_i$ satisfies $G_n\subseteq F_n$, hence $G_n\in\cF$ by \ref{ax:i2}. Moreover, $(G_n)_{n\in\bN}$ is an increasing sequence, and $\liminf_{i\to\infty}F_i=\bigcup_{n=1}^\infty G_n$. Thus, the lemma follows by \ref{ax:i4}.
	\end{proof}
	
	\begin{lem}\label{lem:erosebb_i4}
		Let $(J,\Sigma,\mu)$ be a standard measure space, and let $\cF\subseteq\Sigma$ satisfy \ref{ax:i2} and \ref{ax:i4}. Let $(F_i)_{i\in\bN}$ be a sequence in $\cF$, and assume that $\sum_{i=1}^\infty\mu(F_i\triangle F_{i+1})<\infty$. Then $\lim_{i\to\infty}F_i$ exists, $\lim_{i\to\infty}\mu(F_i)=\mu(\lim_{i\to\infty}F_i)$, and $\lim_{i\to\infty}F_i\in\cF$.
	\end{lem}
	
	\begin{proof}
		By \cref{cor:Borel_Cantelli}, the limit $F\coloneqq\lim_{i\to\infty}F_i$ exists and $\lim_{i\to\infty}\mu(F_i)=\mu(F)$. Since $F=\liminf_{i\to\infty}F_i$, it follows from \cref{lem:liminf} that $F\in\cF$.
	\end{proof}
	
	We will also use the following corollary.
	
	\begin{cor}\label{cor:nem_novekvo_unio}
		Let $(J,\Sigma,\mu)$ be a standard measure space, and let $\cF\subseteq\Sigma$ satisfy \ref{ax:i2} and \ref{ax:i4}. Let $(F_i)_{i\in\bN}$ be a sequence in $\cF$, and let $F\in\Sigma$ satisfy $F_i\subseteq F$ for every $i\in\bN$ and $\lim_{i\to\infty}\mu(F_i)=\mu(F)$. Then $F\in\cF$.
	\end{cor}
	
	\begin{proof}
		Choose a subsequence $(F_{i_k})_{k\in\bN}$ such that $\mu(F\setminus F_{i_k})<2^{-k}$ for every $k\in\bN$. Since $F_{i_k}\subseteq F$, we have
		\[
		\mu(F_{i_k}\triangle F_{i_{k+1}})
		\leq
		\mu(F\setminus F_{i_k})+\mu(F\setminus F_{i_{k+1}}),
		\]
		and hence $\sum_{k=1}^\infty\mu(F_{i_k}\triangle F_{i_{k+1}})<\infty$. Therefore, by \cref{lem:erosebb_i4}, the limit $F'\coloneqq\lim_{k\to\infty}F_{i_k}$ exists and belongs to $\cF$.
		
		It remains to identify the limit. Since $F_{i_k}\subseteq F$ for every $k\in\bN$, we have $F'\subseteq F$. On the other hand,
		\[
		\mu(F')
		=
		\lim_{k\to\infty}\mu(F_{i_k})
		=
		\mu(F).
		\]
		Thus, $F'=F$, and hence $F\in\cF$.
	\end{proof}
	
	Axiom \ref{ax:i4} is automatic in the finite setting but has to be imposed separately here; the preceding statements illustrate some of its useful consequences. As we will see, \ref{ax:i4} is also needed for the basis and rank characterizations of measurable matroids. However, if \ref{ax:i3''} is used, omitting \ref{ax:i4} leads to a natural weakening whose upward closure is a measurable matroid; see \cref{thm:upward_closure_prematroid}. We call $M=(J,\Sigma,\mu,\cF)$ a \emph{measurable pre-matroid} if it satisfies \ref{ax:i1}, \ref{ax:i2}, and \ref{ax:i3''}, where $(J,\Sigma,\mu)$ is a standard measure space. 
	
	\begin{rem}
		The sets of axioms \{\ref{ax:i1}, \ref{ax:i2}, \ref{ax:i3'}\} and \{\ref{ax:i1}, \ref{ax:i2}, \ref{ax:i3''}\} are non-equivalent, as shown by the following example. Let $(J,\Sigma,\mu)$ be a standard measure space with $\mu(J)=1$, and let $B\in\Sigma$ be such that $\mu(B)=\frac12$. Define
		$$\mathcal F\coloneqq\left\{A\in\Sigma\mid A\subseteq B \text{ or } \mu(A)<\frac14\right\}.
		$$
		Then $\cF$ satisfies \ref{ax:i1}, \ref{ax:i2} and \ref{ax:i3'}, however, it does not satisfy \ref{ax:i3''}.
	\end{rem}
	
	We show that the upward closure of a measurable pre-matroid is a measurable matroid. 
	
	\begin{thm}\label{thm:upward_closure_prematroid} 
		Let $(J,\Sigma,\mu,\cS)$ be a measurable pre-matroid, and let
		$\overline{\cS}$ denote the upward closure of $\cS$. Then
		$(J,\Sigma,\mu,\overline{\cS})$ is a measurable matroid.
	\end{thm}
	
	\begin{proof}
		Clearly, \ref{ax:i1} holds for $\overline{\cS}$. By
		\cref{lem:upward_closure}, $\overline{\cS}$ is
		$\sigma$-increasing, so \ref{ax:i4} holds.
		To verify \ref{ax:i2}, let $F'\subseteq F\in\overline{\cS}$, and choose
		an increasing sequence $(F_i)_{i\in\bN}$ in $\cS$ such that
		$F=\bigcup_{i=1}^{\infty}F_i$. Then $F_i\cap F'\in\cS$ for every
		$i\in\bN$, $(F_i\cap F')_{i\in\bN}$ is an increasing sequence and $F'=\bigcup_{i=1}^{\infty} F_i\cap F'$. Hence $F'\in\overline{\cS}$.
		
		It remains to prove \ref{ax:i3'}. Let
		$A,B\in\overline{\cS}$ with $\mu(A) \leq \mu(B)$. For $\mu(A) = \mu(B)$ the statement clearly holds, hence we assume that $\mu(A) < \mu(B)$. Choose increasing
		sequences $(A_i)_{i\in\bN}$ and $(B_i)_{i\in\bN}$ in $\cS$ such that
		\[
		A=\bigcup_{i=1}^{\infty}A_i
		\qquad\text{and}\qquad
		B=\bigcup_{i=1}^{\infty}B_i.
		\]
		By continuity of measure from below, there exists $j\in\bN$ such that
		$\mu(B_j)>\mu(A)$. Set $B'\coloneqq B_j$.
		
		We recursively construct sets $A_i'\in\cS$ such that
		\[
		A_i\subseteq A_i'\subseteq A_i\cup B'
		\qquad\text{and}\qquad
		\mu(A_i')\geq\mu(B')
		\]
		for every $i\in\bN$. Applying \ref{ax:i3''} to $A_1$ and $B'$ gives
		such a set $A_1'$. Suppose that $A_{i-1}'$ has already been defined.
		Since
		$
		\mu(A_i)\leq\mu(A)<\mu(B')\leq\mu(A_{i-1}'),
		$
		axiom \ref{ax:i3''}, applied to $A_i$ and $A_{i-1}'$, gives a set
		$A_i'\in\cS$ such that
		$
		A_i\subseteq A_i'
		\subseteq A_i\cup A_{i-1}'
		\subseteq A_i\cup B'
		$
		and
		$
		\mu(A_i')\geq\mu(A_{i-1}')\geq\mu(B').
		$
		
		The sequence $(A_i'\setminus A)_{i\in\bN}$ is decreasing, since
		$A_i'\subseteq A_i\cup A_{i-1}'$ and $A_i\subseteq A$. Define
		\[
		X\coloneqq\bigcap_{i=1}^{\infty}(A_i'\setminus A).
		\]
		By continuity of measure from above,
		\[
		\mu(X)
		=
		\lim_{i\to\infty}\mu(A_i'\setminus A)
		\geq
		\mu(B')-\mu(A)
		>
		0.
		\]
		Moreover, $X\subseteq B'$. For every $i\in\bN$, we have
		$A_i\cup X\subseteq A_i'$, and therefore
		$A_i\cup X\in\cS$ by \ref{ax:i2}. The sequence
		$(A_i\cup X)_{i\in\bN}$ is increasing, so
		\[
		A\cup X
		=
		\bigcup_{i=1}^{\infty}(A_i\cup X)
		\in\overline{\cS}.
		\]
		Since $X\cap A=\emptyset$, $\mu(X)>0$, and $X\subseteq B'\subseteq B$,
		we obtain
		$
		A\subsetneq A\cup X\subseteq A\cup B.
		$
		Thus \ref{ax:i3'} holds, completing the proof.
	\end{proof}
	
	The following simple criterion gives a more direct description of upward closures.

	\begin{lem}\label{lem:upward_closure_approximation}
		Let $(J,\Sigma,\mu,\cS)$ be a measurable pre-matroid, and let $\overline{\cS}$ be its upward closure. For $I\in\Sigma$, we have $I\in\overline{\cS}$ if and only if, for every $\varepsilon>0$, there exists $S\in\cS$ such that $S\subseteq I$ and $\mu(I\setminus S)<\varepsilon$.
	\end{lem}
	
	\begin{proof}
		Suppose first that $I\in\overline{\cS}$. Then $I=\bigcup_{n=1}^{\infty}S_n$ for some increasing sequence $(S_n)_{n\in\bN}$ in $\cS$. Hence, $S_n\subseteq I$ for every $n\in\bN$, and $\mu(I\setminus S_n)\to0$.
		
		Conversely, suppose that the stated approximation property holds. For every $n\in\bN$, choose $S_n\in\cS$ such that $S_n\subseteq I$ and $\mu(I\setminus S_n)<1/n$. By \cref{thm:upward_closure_prematroid}, $\overline{\cS}$ is the independent-set family of a measurable matroid. Since $S_n\in\overline{\cS}$, $S_n\subseteq I$, and $\mu(S_n)\to\mu(I)$, \cref{cor:nem_novekvo_unio} applied to $\overline{\cS}$ gives $I\in\overline{\cS}$.
	\end{proof}
	
	%%%%%%%%%%%%%%%%
	\subsection{The Atomic Case}
	\label{sec:atomic-case}
	%%%%%%%%%%%%%%%%
	
	We briefly discuss what happens if the atomlessness assumption is omitted. In this section, we write $X\cup\{y\}$ and $X\setminus\{y\}$ instead of $X+y$ and $X-y$ to emphasize that the singletons appearing below are atoms of the measure space. Recall that a set $A\in\Sigma$ is an \emph{atom} of a measure $\mu$ if $\mu(A)>0$ and every measurable set $B\subseteq A$ satisfies either $\mu(B)=0$ or $\mu(A\setminus B)=0$. Every finite Borel measure on a standard Borel space consists of an atomless part and at most countably many point masses. The decomposition into purely atomic and atomless parts is standard. Moreover, every atom of a finite Borel measure on a standard Borel space agrees up to a null set with a singleton; see \cite[Subsection~2.1.6]{kadets2018course}. 
	
	Since measurable sets are identified up to null sets throughout the paper, we may therefore represent each atom by the corresponding point. Let $(J,\Sigma)$ be a standard Borel space, let $\mu$ be a finite Borel measure on it, and suppose that $\cF\subseteq\Sigma$ satisfies the independence axioms \ref{ax:i1}--\ref{ax:i4}. Set
	\[
	D\coloneqq \{i\in J\mid \mu(\{i\})>0\}
	\qquad\text{and}\qquad
	J_0\coloneqq J\setminus D.
	\]
	Then $D$ is finite or countable and the restriction of $\mu$ to $J_0$ is atomless. Set
	\[
	T\coloneqq \{\mu(\{i\})\mid i\in D\}.
	\]
	For $t\in T$, let
	\[
	D_t\coloneqq \{i\in D\mid \mu(\{i\})=t\},
	\]
	and define
	\[
	\cI_t\coloneqq \{I\subseteq D_t\mid I\in\cF\}.
	\]
	Here every subset of $D$ belongs to $\Sigma$, since $D$ is countable.
	
	\begin{prop}
		With the notation above, the following statements hold.
		\begin{enumerate}[label=(\alph*), itemsep=0em]
			\item The system $(J_0,\Sigma|J_0,\mu|J_0,\cF|J_0)$ is a measurable matroid. \label{atom:a}
			\item For every $t\in T$, the set $D_t$ is finite and $(D_t,\cI_t)$ is a finite matroid. \label{atom:b}
			\item A set $F\in\Sigma$ belongs to $\cF$ if and only if $F\cap J_0\in\cF|J_0$ and $F\cap D_t\in\cI_t$ for every $t\in T$. \label{atom:c}
		\end{enumerate}
	\end{prop}
	
	\begin{proof}
		The restriction of $\mu$ to $J_0$ is atomless by the discussion above. The independence axioms for $\cF|J_0$ follow immediately from those for $\cF$, proving \ref{atom:a}.
		
		Let
		\[
		\cI\coloneqq \{I\subseteq D\mid I\in\cF\}.
		\]
		If $I\notin\cI$, then some finite subset of $I$ does not belong to $\cI$. Indeed, otherwise an increasing sequence of finite subsets of $I$ would give $I\in\cF$ by \ref{ax:i4}. Consequently, every dependent subset of $D$ contains a finite inclusionwise minimal dependent set $C\subseteq D$. For every $i\in C$, the set $C\setminus\{i\}$ is a maximal independent subset of $C$. Hence, for any $i,j\in C$, axiom \ref{ax:i3} gives
		$
		\mu(C\setminus\{i\})=\mu(C\setminus\{j\}),
		$
		and therefore $\mu(\{i\})=\mu(\{j\})$. Thus, every inclusionwise minimal dependent subset of $D$ is contained in $D_t$ for some $t\in T$.
		
		The set $D_t$ is finite, since its elements are pairwise distinct points of the same positive measure. Let $X\subseteq D_t$. If $I_1$ and $I_2$ are maximal members of $\cI_t|X$, then $I_1$ and $I_2$ are maximal independent subsets of $X$. By \ref{ax:i3},
		\[
		t|I_1|=\mu(I_1)=\mu(I_2)=t|I_2|,
		\]
		and hence $|I_1|=|I_2|$. Therefore, $\cI_t$ is the family of independent sets of a finite matroid. Moreover, $I\in\cI$ if and only if $I\cap D_t\in\cI_t$ for every $t\in T$: otherwise, a finite inclusionwise minimal dependent subset of $I$ would be contained in one of the sets $D_t$. This proves \ref{atom:b} and describes independence on the atomic part.
		
		It remains to show that the atomic and atomless parts do not interact. We first prove the following claim.
		
		\begin{cla}
			Let $i\in D$, and let $K,H\in\Sigma$ be such that $K$, $\{i\}$, and $H$ are pairwise disjoint and $H\subseteq J_0$. If $K\cup\{i\}\in\cF$ and $K\cup H\in\cF$, then $K\cup\{i\}\cup H\in\cF$.
		\end{cla}
		
		\begin{claimproof}
			Since $\mu|J_0$ is atomless, there is a finite measurable partition
			$
			H=H_1\cup\dots\cup H_m
			$
			such that $\mu(H_j)<\mu(\{i\})$ for every $j\in[m]$. Set $C_j\coloneqq H_1\cup\dots\cup H_j$ for $j\in[m]$, and let $C_0\coloneqq\emptyset$. We prove by induction that $K\cup\{i\}\cup C_j\in\cF$ for every $j=0,\dots,m$. The case $j=0$ follows from the assumption. Suppose that the statement holds for $j-1$. Both $K\cup C_j$ and $K\cup\{i\}\cup C_{j-1}$ are independent, and
			$
			\mu(K\cup C_j)<\mu(K\cup\{i\}\cup C_{j-1}).
			$
			By \ref{ax:i3'}, there exists an independent set $G$ such that
			\[
			K\cup C_j\subsetneq G\subseteq K\cup\{i\}\cup C_j.
			\]
			Thus, $G\setminus(K\cup C_j)$ is a non-null subset of $\{i\}$, and hence it agrees with $\{i\}$ up to a null set. Therefore,
			$
			G=K\cup\{i\}\cup C_j.
			$
			This completes the induction and proves the claim.
		\end{claimproof}
		
		Now let $F_0\in\cF|J_0$, and let $I\in\cI$. Enumerate the elements of $I$ and apply the claim successively. Every finite subset of $I$ together with $F_0$ is independent. Axiom \ref{ax:i4} then gives
		$
		F_0\cup I\in\cF.
		$
		Conversely, if $F\in\cF$, then $F\cap J_0\in\cF|J_0$ and $F\cap D_t\in\cI_t$ for every $t\in T$ by \ref{ax:i2}. This proves \emph{(c)}.
	\end{proof}
	
	Thus, the atoms of each fixed measure determine a finite matroid, and neither atoms of different measures nor the atomic and atomless parts interact. We therefore assume throughout the rest of the paper that the ground measure is atomless; as the basic theory of finite matroids is well understood, the results extend naturally to finite measure spaces with atoms.
	
	%%%%%%%%%%%%%%%%%%%%%%%%%%%%%%%%
	\subsection{Basis Axioms}
	\label{sec:basis}
	%%%%%%%%%%%%%%%%%%%%%%%%%%%%%%%%
	
	Observe that the family $\cB$ of bases, that is, maximal independent sets of a measurable matroid $(J,\Sigma,\mu,\cF)$ uniquely determines the matroid itself, since 
	\begin{equation}
		\mathcal F = \{F \in \Sigma \mid F \subseteq B \text{ for some } B \in \mathcal B\}.\label{ind:b}
	\end{equation}
	Suppose now that $(J,\Sigma,\mu)$ is a standard measure space, and let $\mathcal{B} \subseteq \Sigma$ be a given family. A natural question arises: what properties must $\mathcal{B}$ satisfy in order for the family defined in \eqref{ind:b} to form the collection of independent sets of a measurable matroid? To answer this, we show that, analogously to the finite case, measurable matroids can be described in terms of their bases. We introduce the following \emph{basis axioms}, where $(J, \Sigma, \mu)$ is a standard measure space and $\mathcal B \subseteq \Sigma$:
	\begin{enumerate}[label=\normalfont{(B\arabic*)}, left=0pt, itemsep=0em]
		\item $\cB$ is nonempty. \label{ax:b1}
		\item For all $B_1, B_2 \in \cB$ and all $X \subseteq  B_1 \setminus B_2$, there exists $Y \subseteq  B_2 \setminus B_1$ such that $\mu(X)=\mu(Y)$ and $(B_1 \setminus X) \cup Y \in \cB$. \label{ax:b2}
		\item If $(B_i)_{i\in\bN}$ is a sequence in $\cB$ and $B\in\Sigma$ such that $\lim_{n \to \infty} \mu(B_n \triangle B)=0$, then $B \in \cB$. \label{ax:b3}
	\end{enumerate}
	
	We call \ref{ax:b2} the \emph{exchange axiom}, following the terminology from the finite setting. Note that \ref{ax:b2} implies $\mu(B_1) = \mu(B_2)$ for all $B_1, B_2 \in \mathcal B$. Indeed, taking $X \coloneqq B_1 \setminus B_2$ yields a set $Y \subseteq B_2 \setminus B_1$ with $\mu(Y) = \mu(X) \le \mu(B_2 \setminus B_1)$, which gives $\mu(B_1) \le \mu(B_2)$. Reversing the roles of $B_1$ and $B_2$ gives the opposite inequality, and hence $\mu(B_1) = \mu(B_2)$.
	
	\begin{thm} \label{thm:basis}
		The bases of any measurable matroid satisfy the basis axioms. Conversely, if $(J,\Sigma,\mu)$ is a standard measure space and $\mathcal B \subseteq \Sigma$ satisfies the basis axioms, then $\cF$ defined in \eqref{ind:b} satisfies the independence axioms, and $(J,\Sigma,\mu,\mathcal F)$ is a measurable matroid with $\mathcal B$ as the family of bases.
	\end{thm}
	
	\begin{proof}
		Let $\mathcal F$ be the family of independent sets and $\mathcal B$ the family of bases of a measurable matroid. By \cref{lem:maximal}, $\mathcal F$ contains at least one maximal set, so \ref{ax:b1} holds. Next, let $B_1, B_2 \in \mathcal B$ and $X \subseteq B_1 \setminus B_2$. By \ref{ax:i3''}, there exists $G \in \mathcal F$ with $B_1 \setminus X \subseteq G \subseteq (B_1 \setminus X) \cup B_2$ and $\mu(G) \ge \mu(B_2)$. Since all bases have the same measure $r$ by \ref{ax:i3}, it follows that $G \in \mathcal B$. Consequently, \ref{ax:b2} holds by taking $Y \coloneqq G \setminus (B_1 \setminus X)$. 
		
		Finally, suppose the sequence $(B_i)_{i\in\bN}$ in $\cB$ satisfies $\lim_{i \to \infty} \mu(B_i \triangle B) = 0$ for some $B \in \Sigma$. Define $F_i \coloneqq B_i \cap B$; by \ref{ax:i2}, each $F_i$ is independent. Since $F_i \subseteq B$ and $\mu(B \setminus F_i) \leq \mu(B_i \triangle B) \to 0$, we have $B \in \mathcal F$ by \cref{cor:nem_novekvo_unio}. Moreover, $\lim_{i \to \infty} \mu(F_i) = r$, and hence $\mu(B) = r$, so $B \in \mathcal B$. This establishes \ref{ax:b3}.
		
		Now assume that $\mathcal B \subseteq \Sigma$ satisfies the basis axioms, let $r$ denote the common measure of the sets in $\mathcal B$, and define $\cF$ as in \eqref{ind:b}. It is not difficult to check that \ref{ax:i1} and \ref{ax:i2} hold for $\cF$. To verify \ref{ax:i3'}, let $F_1, F_2 \in \mathcal F$ with $\mu(F_1)<\mu(F_2)$. Choose $B_1, B_2 \in \mathcal B$ such that $F_1 \subseteq B_1$ and $F_2 \subseteq B_2$. If $\mu((B_1\cap F_2)\setminus F_1)>0$, then $F_1 \cup (B_1 \cap F_2) \subseteq B_1$ belongs to $\mathcal F$ and strictly contains $F_1$, so \ref{ax:i3'} holds. Otherwise, $B_1\cap F_2\subseteq F_1$, which is equivalent to $F_2 \setminus F_1 \subseteq B_2 \setminus B_1$. Let $X \coloneqq B_1 \setminus (F_1 \cup B_2)$; by \ref{ax:b2}, there exists $Y \subseteq B_2 \setminus B_1$ with $\mu(X) = \mu(Y)$ such that $B' \coloneqq (B_1 \setminus X) \cup Y \in \cB$. Since $B_1\setminus B_2$ is the disjoint union of $F_1\setminus B_2$ and $X$, and since $\mu(B_1\setminus B_2)=\mu(B_2\setminus B_1)$, we have $\mu(Y)=\mu(B_2\setminus B_1)-\mu(F_1\setminus B_2)$. On the other hand, $\mu(F_2\setminus F_1)>\mu(F_1\setminus F_2)\geq \mu(F_1\setminus B_2)$. It follows that $Y$ cannot be contained in $(B_2\setminus B_1)\setminus(F_2\setminus F_1)$, and hence
		\[
		\mu\left(B'\cap(F_2\setminus F_1)\right)
		=
		\mu\left(Y\cap(F_2\setminus F_1)\right)
		>0.
		\]
		Therefore, $G \coloneqq F_1 \cup (B' \cap (F_2 \setminus F_1)) \subseteq B'$ belongs to $\mathcal F$ and strictly contains $F_1$, which establishes \ref{ax:i3'}.
		
		Finally, to verify \ref{ax:i4}, let $(F_i)_{i \in \bN}$ be an increasing sequence in $\mathcal F$, and set $F \coloneqq \bigcup_{i=1}^\infty F_i$. For each $F_i$, choose a basis $B_i \in \mathcal B$ with $F_i \subseteq B_i$. Define $B'_1 \coloneqq B_1$, and for $i \ge 2$, define $B'_i$ recursively as follows. Let $X \coloneqq (B_i\setminus B'_{i-1})\setminus F_i$ and apply \ref{ax:b2} to obtain $Y \subseteq B'_{i-1} \setminus B_i$ such that 
		\begin{equation*}
			B'_i \coloneqq (B_i\setminus X) \cup Y= (B'_{i-1}\cap B_i) \cup F_i \cup Y
		\end{equation*}
		is a basis. By construction, $(B'_{i-1} \cap B_i) \cup Y \subseteq B'_{i-1}$, so $B'_i\setminus B'_{i-1} = F_i\setminus B'_{i-1}\subseteq F_i\setminus F_{i-1}$. It follows that $B'_i \setminus F_i \subseteq B'_{i-1} \setminus F_{i-1}$ for all $i \ge 2$, and hence \begin{equation*}
			\mu\left(\bigcap_{i=1}^\infty (B'_i\setminus F_i)\right)=\lim_{i \to \infty} (r-\mu(F_i))=r-\mu(F).
		\end{equation*} 
		Define
		\[
		B'\coloneqq F\cup\bigcap_{i=1}^\infty(B_i'\setminus F_i).
		\]
		Since $(F_i)_{i\in\bN}$ is increasing and $(B_i'\setminus F_i)_{i\in\bN}$ is decreasing, we have $B'=\liminf_{i\to\infty}B_i'$. Moreover, $\mu(B')=r$ and $B'\setminus B_i'\subseteq F\setminus F_i$ for every $i\in\bN$. Since $\mu(B')=\mu(B_i')=r$, it follows that
		\[
		\lim_{i\to\infty}\mu(B'\triangle B_i')=
		2 \lim_{i\to\infty}\mu(B'\setminus B_i')\leq 
		2 \lim_{i\to\infty}\mu(F\setminus F_i)=
		0.
		\]
		Hence, $B'\in\cB$ by \ref{ax:b3}. Therefore, $F\subseteq B'\in\cB$, so $F\in\cF$, proving \ref{ax:i4}.
		
		It remains to identify the bases of the constructed matroid. Let $\cF\coloneqq \{F\in\Sigma\mid F\subseteq B \text{ for some } B\in\cB\}$. First, every $B\in\cB$ is maximal in $\cF$. Indeed, if $B\subseteq F\in\cF$, then $F\subseteq B'$ for some $B'\in\cB$. Since all members of $\cB$ have the same measure, we have $\mu(B)=\mu(B')$, and hence $B=F=B'$. Conversely, if $F$ is maximal in $\cF$, then by definition $F\subseteq B$ for some $B\in\cB$. Since $B\in\cF$, the maximality of $F$ gives $F=B$. Thus, the bases of the constructed measurable matroid are precisely the members of $\cB$.
	\end{proof}
	
	By \cref{thm:basis}, a measurable matroid can be equivalently defined through its family of bases. Consequently, when it is more convenient to specify a matroid via its bases, we will write $M = (J, \Sigma, \mu, \mathcal B)$, where $\mathcal B$ denotes the family of bases.
	
	We next introduce the measurable analogue of the \emph{co-exchange axiom} from the finite setting.
	
	\begin{enumerate}[label=\normalfont{(B\arabic*')}, left=0pt, itemsep=0em]\setcounter{enumi}{1}
		\item For all $B_1, B_2 \in \cB$ and all $X \subseteq B_1 \setminus B_2$ there exists $Y \subseteq B_2 \setminus B_1$ such that $\mu(X)=\mu(Y)$ and $(B_2 \setminus Y) \cup X \in \cB$. \label{ax:b2'}
	\end{enumerate}
	
	We prove that this property also holds in measurable matroids.
	
	\begin{prop} \label{prop:coexchange}
		Let $M = (J, \Sigma, \mu, \mathcal F)$ be a measurable matroid. Then the co-exchange axiom \ref{ax:b2'} holds for its family of bases $\mathcal B$.
	\end{prop}
	
	\begin{proof}
		Let $B_1, B_2 \in \mathcal B$ and $X \subseteq B_1 \setminus B_2$. Consider a maximal independent set $A$ in $X \cup B_2$ containing $X \cup (B_1 \cap B_2)$. Such a set exists by \cref{lem:maximal}, since $X \cup (B_1 \cap B_2) \subseteq B_1$ and therefore $X \cup (B_1 \cap B_2) \in \mathcal F$. As $B_2 \in \mathcal F$ is maximal, it is also maximal in $X \cup B_2$, so \ref{ax:i3} implies $\mu(A) = \mu(B_2)$, and hence $A \in \mathcal B$. Setting $Y \coloneqq B_2 \setminus A$ completes the proof.
	\end{proof}
	
	\begin{rem}
		Replacing the exchange axiom \ref{ax:b2} with the co-exchange axiom \ref{ax:b2'} yields an equivalent characterization of measurable matroids. The proof is deferred to \cref{cor:coexchange}.
	\end{rem}
	
	%%%%%%%%%%%%%%%%%%%%%%%%%%%%%%%%
	\subsection{Rank Axioms}
	\label{sec:rank}
	%%%%%%%%%%%%%%%%%%%%%%%%%%%%%%%%
	
	The classical rank axioms can be naturally extended to set functions on a standard measure space $(J,\Sigma,\mu)$. The main subtlety lies in formulating an analogue of the subcardinality condition, which ensures that adding a single element increases the rank by at most one. In the measurable setting, this can be partially addressed by imposing $\varphi\leq\mu$, reflecting the finite constraint that the rank of a set is bounded by its size. However, this condition alone is not enough to capture the full structure of finite matroids.
	
	Just as a measurable matroid can be described in terms of its bases, it can also be described via a rank function. For a measurable matroid $(J,\Sigma,\mu,\cF)$ with rank function $r$, the family of independent sets is exactly   
	\begin{equation}
		\mathcal F = \{F \in \Sigma \mid r(F)=\mu(F)\}.\label{ind:r}
	\end{equation}
	Indeed, if $F\in\cF$, then $F$ is a maximal independent subset of itself, and hence $r(F)=\mu(F)$. Conversely, if $r(F)=\mu(F)$ and $F'\in\cF|F$ is maximal, then $\mu(F')=r(F)=\mu(F)$, and therefore $F'=F$. Thus, $F\in\cF$.
	
	Analogously to the case of bases, we are interested in identifying the conditions on $r$ under which the family defined in \eqref{ind:r} forms the collection of independent sets of a measurable matroid. To this end, we characterize measurable matroids in terms of their rank functions $r \colon \Sigma \to [0,\infty)$ on a standard measure space $(J, \Sigma, \mu)$ via the following \emph{rank axioms}:
	\begin{enumerate}[label=\normalfont{(R\arabic*)}, left=0pt, itemsep=0em]
		\item $r(\emptyset)=0$. \label{ax:r1}
		\item $r(X) \leq r(Y)$ for all $X \subseteq  Y$. \label{ax:r2}
		\item $r(X) \leq \mu(X)$ for all $X$. \label{ax:r3}
		\item $r(X)+r(Y) \geq r(X \cap Y)+r(X \cup Y)$ for all $X,Y$. \label{ax:r4}
		\item For all $X \subseteq  Y$ with $r(X)=\mu(X)$, there exists a set $Z$ with $X \subseteq  Z \subseteq  Y$ and $r(Y)=r(Z)=\mu(Z)$. \label{ax:r5}
	\end{enumerate}
	Note that \ref{ax:r4} asserts the \emph{submodularity} of the rank function. 
	
	\begin{rem}
		It is natural to ask whether axiom~\ref{ax:r5} can be omitted, or at least weakened to the following condition: for every $Y\in\Sigma$, there exists $Z\subseteq Y$ such that $\mu(Z)=r(Z)=r(Y)$. The next example shows that this weaker condition, together with \ref{ax:r1}--\ref{ax:r4}, does not imply \ref{ax:r5}.
		
		Let $J=[0,3]$ with the Lebesgue measure $\lambda$, and partition $J$ into $A=[0,1)$ and $B=[1,3)$. Let $f\colon B\to A$ be defined by $f(x)=(x-1)/2$. For $X\in\Sigma$, define
		\begin{equation}\label{eq:weak_r5_counterexample}
			r(X)\coloneqq \lambda(X\cap B)+\lambda((X\cap A)\setminus f(X\cap B)).
		\end{equation}
		Clearly, $r(\emptyset)=0$. Monotonicity in $X\cap A$ is immediate. If $X\cap B$ is extended by a set of measure $c$, then the first term in \eqref{eq:weak_r5_counterexample} increases by $c$, while the second term can decrease by at most $c/2$. Hence, $r$ is monotone. Also, $r(X)\leq \lambda(X\cap B)+\lambda(X\cap A)=\lambda(X)$, so \ref{ax:r1}, \ref{ax:r2}, and \ref{ax:r3} hold.
		
		We next verify submodularity. Let $S_i=A_i\cup B_i$, where $A_i\subseteq A$ and $B_i\subseteq B$ for $i=1,2$, and set $C_i\coloneqq f(B_i)$. The first term in \eqref{eq:weak_r5_counterexample} is modular. Thus, it is enough to prove
		\begin{equation*}
			\lambda(A_1\setminus C_1)+\lambda(A_2\setminus C_2)
			\geq
			\lambda((A_1\cap A_2)\setminus(C_1\cap C_2))
			+
			\lambda((A_1\cup A_2)\setminus(C_1\cup C_2)).
		\end{equation*}
		Using $\lambda(A_i\setminus C_i)=\lambda(A_i)-\lambda(A_i\cap C_i)$ and $\lambda(A_1)+\lambda(A_2)=\lambda(A_1\cap A_2)+\lambda(A_1\cup A_2)$,
		this is equivalent to
		\[
		\lambda((A_1\cup A_2)\cap(C_1\cup C_2))
		+\lambda(A_1\cap A_2\cap C_1\cap C_2)
		\geq
		\lambda(A_1\cap C_1)+\lambda(A_2\cap C_2).
		\]
		This is immediate, since $A_1\cap C_1$ and $A_2\cap C_2$ are both contained in $(A_1\cup A_2)\cap(C_1\cup C_2)$, and their intersection is $A_1\cap A_2\cap C_1\cap C_2$. Hence, \ref{ax:r4} holds.
		
		The weakened form of \ref{ax:r5} also holds. Indeed, for $Y\in\Sigma$, set
		$
		Z\coloneqq (Y\cap B)\cup ((Y\cap A)\setminus f(Y\cap B)).
		$
		Then $Z\subseteq Y$, and by definition
		$
		\lambda(Z)=r(Z)=r(Y).
		$
		
		On the other hand, \ref{ax:r5} fails. We have $\lambda(A)=r(A)=1$ and $r(J)=2$. Suppose that there exists $Z$ with $A\subseteq Z\subseteq J$ and $r(Z)=\lambda(Z)=r(J)$. Write $C\coloneqq Z\cap B$. Since $A\subseteq Z$, we have
		$
		r(Z)=\lambda(C)+\lambda(A\setminus f(C))=\lambda(C)+1-\lambda(f(C))=1+\lambda(C)/2,
		$
		while $\lambda(Z)=1+\lambda(C)$. Thus, $r(Z)=\lambda(Z)$ implies $\lambda(C)=0$, and then $r(Z)=1$, contradicting $r(Z)=r(J)=2$. Hence, \ref{ax:r5} does not follow from the weakened condition.
	\end{rem}
	
	\begin{rem}
		As \ref{ax:r5} is a rather strong assumption, one might wonder whether \ref{ax:r4} is necessary. However, the following example shows that the submodularity constraint cannot be dismissed. Let $J=[0,3]$ with the Lebesgue measure $\lambda$. Partition $J$ into $A=[0,1)$, $B=[1,2)$, and $C=[2,3]$. For any measurable $X \subseteq J$, define
		\begin{equation*}
			r(X) \coloneqq \lambda(X\cap A) + \lambda(X\cap C) + \min\{\lambda(X\cap A),\lambda(X\cap B),\lambda(X\cap C)\}.
		\end{equation*}
		Since $r(\emptyset)=0$, axiom \ref{ax:r1} holds. As both $\lambda$ and the $\min$ function are monotone, $r$ is non-decreasing, proving \ref{ax:r2}. Moreover, by definition $r(X)\le \lambda(X\cap A)+\lambda(X\cap B)+\lambda(X\cap C)=\lambda(X)$, so \ref{ax:r3} also holds. To verify \ref{ax:r5}, observe that $r(X)=\lambda(X)$ if and only if $\lambda(X\cap B) \le \min\{\lambda(X\cap A),\lambda(X\cap C)\}$. By the non-atomicity of $\lambda$, for any $X \subseteq Y$ with $r(X)=\lambda(X)$ we can choose a set $Z$ such that $Z \cap A = Y \cap A$, $Z \cap C = Y \cap C$, and $Z \cap B$ is a subset of $Y \cap B$ containing $X \cap B$ with measure $\min\{\lambda(Y\cap A),\lambda(Y\cap B),\lambda(Y\cap C)\}$. Then $X \subseteq Z \subseteq Y$ and $r(Y)=r(Z)=\lambda(Z)$. However, \ref{ax:r4} fails. Indeed, let $X=A\cup B$ and $Y=B\cup C$. Then $X\cap Y=B$ and $X\cup Y=J$, and therefore $r(X)+r(Y)=2 < 3 = r(X\cap Y)+r(X\cup Y)$.
	\end{rem}
	
	We next show that measurable matroids can be characterized by their rank functions.
	
	\begin{thm} \label{thm:rank}
		The rank function of any measurable matroid satisfies the rank axioms. Conversely, if $(J,\Sigma,\mu)$ is a standard measure space and $r \colon \Sigma \to \bR_{\geq 0}$ satisfies the rank axioms, then $\cF$ defined in \eqref{ind:r} satisfies the independence axioms, and $(J,\Sigma,\mu,\cF)$ is a measurable matroid with $r$ as the rank function.
	\end{thm}
	\begin{proof}
		Let $\mathcal F$ be the family of independent sets and $r$ the rank function of a measurable matroid. Axioms \ref{ax:r1}, \ref{ax:r2}, and \ref{ax:r3} follow directly from the definition. To verify \ref{ax:r4}, let $X$ and $Y$ be arbitrary sets. Choose $F_1 \in \mathcal F$ with $F_1 \subseteq X \cap Y$ and $\mu(F_1)=r(X \cap Y)$. Consider a maximal independent set $F_2$ in $X \cup Y$ containing $F_1$. Such a set exists by \cref{lem:maximal}, and $F_2 \cap (X \cap Y)=F_1$, since $F_1$ is maximal in $\mathcal F|(X \cap Y)$. Consequently, 
		\begin{equation*}
			r(X \cap Y)+r(X \cup Y)=\mu(F_1)+\mu(F_2)=\mu(F_2 \cap X)+\mu(F_2 \cap Y) \leq r(X)+r(Y),
		\end{equation*}
		as required. Finally, \ref{ax:r5} follows from \cref{lem:maximal}. Indeed, the family $\cG \coloneqq \{G \in \cF | X \subseteq  G \subseteq  Y\}$ is $\sigma$-increasing, hence it contains a maximal set $Z$. This set is also maximal in $\mathcal F|Y$, and therefore $r(Y)=r(Z)=\mu(Z)$.
		
		Conversely, assume that $r \colon \Sigma \to \bR_{\geq 0}$ satisfies the rank axioms and define $\cF$ as in \eqref{ind:r}. Clearly, $\cF$ satisfies \ref{ax:i1}. Let $F_1 \subseteq F_2 \in \cF$. Using \ref{ax:r3} and \ref{ax:r4}, we obtain
		\begin{equation*}
			\mu(F_1) \ge r(F_1) \ge r(F_2)+r(\emptyset)-r(F_2 \setminus F_1)
			\ge \mu(F_2)-\mu(F_2 \setminus F_1)=\mu(F_1),
		\end{equation*}
		hence $r(F_1)=\mu(F_1)$, and therefore $F_1 \in \cF$, establishing \ref{ax:i2}.
		
		To verify \ref{ax:i4}, let $(F_i)_{i\in\bN}$ be an increasing chain in $\cF$, and set $F \coloneqq \bigcup_{i=1}^\infty F_i$. By continuity of $\mu$ from below, 
		\begin{equation*}
			\lim_{i \to \infty} r(F_i)=\lim_{i \to \infty} \mu(F_i)=\mu(F).
		\end{equation*}
		By \ref{ax:r2}, we have $r(F_i) \le r(F)$ for all $i$, hence $\mu(F) \le r(F)$. Together with \ref{ax:r3}, this implies $r(F)=\mu(F)$, proving \ref{ax:i4}.
		
		Finally, we prove \ref{ax:i3} and at the same time show that the rank function of the constructed matroid is the original function $r$. Let $X\in\Sigma$, and let $F$ be a maximal member of $\cF|X$, which exists by \cref{lem:maximal}. By definition of $\cF$, we have $r(F)=\mu(F)$. Applying \ref{ax:r5} to $F\subseteq X$, there exists a set $Z$ with $F\subseteq Z\subseteq X$ and
		$
		r(X)=r(Z)=\mu(Z).
		$
		Thus, $Z\in\cF|X$. By the maximality of $F$, we get $Z=F$, and hence $\mu(F)=r(X)$. Thus, every maximal member of $\cF|X$ has measure $r(X)$. This proves \ref{ax:i3}, and also shows that the rank function of the measurable matroid $(J,\Sigma,\mu,\cF)$ is exactly $r$.
	\end{proof}
	
	By \cref{thm:rank}, a measurable matroid can be defined through its rank function. When it is more convenient to define a matroid via its rank function, we will write $M=(J,\Sigma,\mu,r)$, where $r$ denotes the rank function.
	
	The rank function is submodular; hence the theory of submodular set functions on $\sigma$-algebras applies directly. We explore this connection in detail in \cref{sec:measurable}. For now, we state a useful smoothness lemma of Lovász~\cite{lovasz2023submodular} after introducing the necessary definitions.
	
	We say that a set function $\varphi \colon \Sigma \to \bR$ on a $\sigma$-algebra $(J,\Sigma)$ is \emph{continuous from above} if, for every decreasing sequence $(A_i)_{i\in\bN}$ in $\Sigma$, we have $\lim_{i\to\infty}\varphi(A_i)=\varphi(\bigcap_{i=1}^\infty A_i)$. Similarly, we say that $\varphi$ is \emph{continuous from below} if, for every increasing sequence $(A_i)_{i\in\bN}$ in $\Sigma$, we have $\lim_{i\to\infty}\varphi(A_i)=\varphi(\bigcup_{i=1}^\infty A_i)$. 
	
	The following simple consequence of the smoothness results of Lovász~\cite[Section~6.1]{lovasz2023submodular} will be useful.
	
	\begin{prop}\label{lem:smooth}
		Let $\varphi\colon\Sigma\to\bR_{\geq 0}$ be an increasing subadditive set function on a $\sigma$-algebra $(J,\Sigma)$ with $\varphi(\emptyset)=0$. Assume that $\varphi\leq\mu$ for some finite measure $\mu$ on $(J,\Sigma)$. Then $\varphi$ is continuous from below and continuous from above.
	\end{prop}
	
	\begin{proof}
		Let $(A_i)_{i\in\bN}$ be an increasing sequence in $\Sigma$ and set $A\coloneqq\bigcup_{i\in\bN}A_i$. By subadditivity and monotonicity,
		\[
		0\leq\varphi(A)-\varphi(A_i)\leq\varphi(A\setminus A_i)\leq\mu(A\setminus A_i),
		\]
		and the right-hand side converges to $0$.
		
		Similarly, let $(A_i)_{i\in\bN}$ be a decreasing sequence in $\Sigma$ and set $A\coloneqq\bigcap_{i\in\bN}A_i$. Then
		\[
		0\leq\varphi(A_i)-\varphi(A)\leq\varphi(A_i\setminus A)\leq\mu(A_i\setminus A),
		\]
		and again the right-hand side converges to $0$.
	\end{proof}
	
	Since rank functions are increasing and submodular, hence subadditive, we immediately obtain the following.
	
	\begin{cor}\label{cor:smooth}
		Let $M=(J,\Sigma,\mu,r)$ be a measurable matroid. Then $r$ is continuous from below and continuous from above.
	\end{cor}
	
	% \begin{prop} \label{lem:smooth}
		%     Let $\varphi \colon \Sigma \to \bR$ be an increasing submodular set function on a $\sigma$-algebra $(J,\Sigma)$ with $\varphi(\emptyset)=0$. Assume there exists a finite measure $\mu$ on $(J,\Sigma)$ such that $\varphi(X) \le \mu(X)$ for all $X \in \Sigma$. Then $\varphi$ is continuous from below and continuous from above.
		% \end{prop}
	
	% As a corollary, we get the following.
	
	% \begin{cor} \label{cor:smooth}
		%     Let $M=(J, \Sigma, \mu, r)$ be a measurable matroid. Then $r$ is continuous from below and continuous from above. 
		% \end{cor}
	
	As a further consequence, minorizing charges are automatically measures in our setting.
	
	\begin{cor}\label{cor:mm-measures}
		Let $\varphi\colon\Sigma\to\bR_{\geq 0}$ be an increasing submodular set function on a $\sigma$-algebra $(J,\Sigma)$ with $\varphi(\emptyset)=0$. Suppose that $\varphi\leq\mu$ for some finite measure $\mu$. Then every element of $\mm(\varphi)$ is a finite measure dominated by $\mu$. In particular, this holds when $\varphi$ is the rank function of a measurable matroid.
	\end{cor}
	
	\begin{proof}
		By \cref{lem:smooth}, $\varphi$ is continuous from above. Hence, every $\alpha\in\mm(\varphi)$ is countably additive by~\cite[Lemma~7.1]{lovasz2023submodular}. Moreover, $\alpha\leq\varphi\leq\mu$.
	\end{proof}
	
	Accordingly, whenever $\varphi$ satisfies the assumptions of \cref{cor:mm-measures}, we refer to the elements of $\mm(\varphi)$ and $\bmm(\varphi)$ as minorizing measures and basic minorizing measures, respectively.
	
	%%%%%%%%%%%%%%%%%%%%%%%%%%%%%%%%
	\subsection{Closure Axioms}
	\label{sec:closure}
	%%%%%%%%%%%%%%%%%%%%%%%%%%%%%%%%
	
	Finally, we can describe measurable matroids in terms of a closure operator. It is not immediate that a meaningful closure operator can be defined in the measurable setting. Let $M=(J,\Sigma,\mu,\mathcal F)$ be a measurable matroid with rank function $r$. For any $X\in\Sigma$, define
	\begin{equation*}
		\cA_X \coloneqq \{ A \in \Sigma \mid X \subseteq A,\ r(X)=r(A) \}.
	\end{equation*}
	If $A_1, A_2 \in \mathcal A_X$, then monotonicity and submodularity of $r$ yield
	\begin{equation*}
		r(X)+r(X)=r(A_1)+r(A_2) \geq r(A_1 \cap A_2)+r(A_1 \cup A_2) \geq r(X)+r(X),
	\end{equation*}
	hence $A_1 \cap A_2$ and $A_1 \cup A_2$ are in $\mathcal A_X$, showing that $\mathcal A_X$ is closed under both intersections and unions. By \cref{cor:smooth}, it is also $\sigma$-increasing. Hence, by \cref{cor:unionclosed}, $\mathcal A_X$ contains a unique maximal set. We denote this set by $\clo(X)$ and call it the \emph{closure} of $X$. We say that $X$ \emph{spans} $Y$ if $Y\subseteq\clo(X)$. We call $\clo(\emptyset)$ the \emph{loop set} of $M$, and say that $M$ is \emph{loopless} if $\clo(\emptyset)=\emptyset$. Since the loop set has rank zero, $r(X\setminus\clo(\emptyset))=r(X)$ for every $X\in\Sigma$.
	
	The independent sets can be recovered from the closure operator. Namely, the independent sets of a measurable matroid $(J,\Sigma,\mu,\cF)$ with closure operator $\clo$ are exactly the sets
	\begin{equation}
		\cF \coloneqq \{F \in \Sigma \mid \mu(X\setminus \clo(F\setminus X))>0 \text{ for every } X\subseteq F \text{ with } \mu(X)>0\}.\label{ind:cl}
	\end{equation}
	Indeed, if $F$ is independent and $X\subseteq F$ has positive measure, then $X\subseteq\clo(F\setminus X)$ would imply $r(F)=r(F\setminus X)=\mu(F\setminus X)<\mu(F)=r(F)$, a contradiction. Conversely, if $F$ is dependent, let $I$ be a maximal independent subset of $F$ and set $X\coloneqq F\setminus I$. Then $\mu(X)>0$ and $r(I)=r(F)$, so $F\subseteq\clo(I)$, and hence $\mu(X\setminus\clo(F\setminus X))=0$.
	
	Analogously to the cases of bases and rank functions, we are interested in identifying the conditions on a closure operator under which the family defined above corresponds to a measurable matroid. If $\clo\colon\Sigma\to\Sigma$ is a map, then a set $F\in\Sigma$ is called \emph{$\clo$-independent} if it satisfies the condition in \eqref{ind:cl}. We consider maps $\clo\colon\Sigma\to\Sigma$ satisfying the following closure axioms:
	\begin{enumerate}[label=\normalfont{(CL\arabic*)}, left=0pt, itemsep=0em]
		\item $X\subseteq\clo(X)$ for every $X\in\Sigma$. \label{ax:cl1}
		\item If $X\subseteq Y$, then $\clo(X)\subseteq\clo(Y)$. \label{ax:cl2}
		\item $\clo(\clo(X))=\clo(X)$ for every $X\in\Sigma$. \label{ax:cl3}
		\item If $F$ is $\clo$-independent and $F\subseteq\clo(A)$, then $\mu(F)\leq \mu(A)$. \label{ax:cl4}
		\item If $F$ is $\clo$-independent and $\mu(A\setminus \clo(F))>0$, then there exists $H\subseteq A\setminus \clo(F)$ with $\mu(H)>0$ such that $F\cup H$ is $\clo$-independent. \label{ax:cl5}
	\end{enumerate}
	
	The first three axioms are the usual closure axioms. The last two axioms connect closure with the measure. Axiom \ref{ax:cl4} says that an independent set spanned by $A$ cannot have measure larger than $A$. Axiom \ref{ax:cl5} is the corresponding augmentation axiom: if a positive-measure part of $A$ is not spanned by $F$, then some positive-measure subset of it can be added to $F$.
	
	\begin{thm}\label{thm:closure}
		The closure operator of a measurable matroid satisfies the closure axioms. Conversely, if $(J,\Sigma,\mu)$ is a standard measure space and $\clo\colon\Sigma\to\Sigma$ satisfies the closure axioms, then $\cF$ defined as in \eqref{ind:cl} satisfies the independence axioms, and $(J,\Sigma,\mu,\cF)$ is a measurable matroid with $\clo$ as the closure operator.
	\end{thm}
	
	\begin{proof}
		Let $M=(J,\Sigma,\mu,\cF)$ be a measurable matroid, let $r$ be its rank function, and let $\clo$ be its closure operator. By definition, $X\subseteq\clo(X)$, so \ref{ax:cl1} holds. If $X\subseteq Y$, then submodularity gives
		\[
		r(\clo(X))+r(\clo(Y))
		\geq r(\clo(X)\cap\clo(Y))+r(\clo(X)\cup\clo(Y))
		\geq r(X)+r(Y).
		\]
		Since $r(\clo(X))=r(X)$ and $r(\clo(Y))=r(Y)$, equality holds throughout. Thus, $r(\clo(X)\cup\clo(Y))=r(Y)$, and the maximality of $\clo(Y)$ gives $\clo(X)\subseteq\clo(Y)$. This proves \ref{ax:cl2}. Axiom \ref{ax:cl3} follows directly from the definition of closure.
		
		We first check that \eqref{ind:cl} indeed recovers the independent sets of $M$. If $F$ is independent and $X\subseteq F$ has positive measure, then
		$
		r(F\setminus X)\leq \mu(F\setminus X)<\mu(F)=r(F).
		$
		Hence, $F$ cannot be contained in $\clo(F\setminus X)$, and in particular $\mu(X\setminus\clo(F\setminus X))>0$. Conversely, suppose that $F$ is not independent. Let $F'\subseteq F$ be a maximal independent subset of $F$. Then $\mu(F')=r(F)<\mu(F)$, so $X\coloneqq F\setminus F'$ has positive measure. Since $r(F')=r(F)$, we have $F\subseteq\clo(F')=\clo(F\setminus X)$, contradicting the condition in \eqref{ind:cl}.
		
		Now let $F$ be $\clo$-independent. By the previous paragraph, $F$ is independent in $M$. If $F\subseteq\clo(A)$, then
		$
		\mu(F)=r(F)\leq r(\clo(A))=r(A)\leq\mu(A),
		$
		proving \ref{ax:cl4}. Finally, suppose that $\mu(A\setminus\clo(F))>0$. Set $B\coloneqq A\setminus\clo(F)$. If $F$ were maximal independent in $F\cup B$, then $r(F\cup B)=r(F)$, and hence $B\subseteq\clo(F)$, a contradiction. Thus, $F$ is not maximal independent in $F\cup B$, and there exists $H\subseteq B$ with $\mu(H)>0$ such that $F\cup H$ is independent. By the preceding paragraph again, $F\cup H$ is $\clo$-independent. This proves \ref{ax:cl5}.
		
		Conversely, assume that $\clo$ satisfies the closure axioms, and let $\cF$ be the family of $\clo$-independent sets. Clearly, $\emptyset\in\cF$. If $F_1\subseteq F_2$ and $F_2\in\cF$, then $F_1\in\cF$; otherwise there is $X\subseteq F_1$ with $\mu(X)>0$ and $\mu(X\setminus\clo(F_1\setminus X))=0$. By \ref{ax:cl2}, this implies $\mu(X\setminus\clo(F_2\setminus X))=0$, contradicting $F_2\in\cF$. Thus, \ref{ax:i1} and \ref{ax:i2} hold.
		
		We prove \ref{ax:i4}. Let $(F_i)_{i\in\bN}$ be an increasing sequence in $\cF$, and set $F\coloneqq\bigcup_{i=1}^{\infty} F_i$ . Suppose that $F\notin\cF$. Then there is $X\subseteq F$ with $\mu(X)>0$ such that $\mu(X\setminus\clo(F\setminus X))=0$. By \ref{ax:cl1}, we also have $F\setminus X\subseteq\clo(F\setminus X)$, and therefore $F\subseteq\clo(F\setminus X)$. Hence, $F_i\subseteq\clo(F\setminus X)$ for every $i$. By \ref{ax:cl4}, applied to the $\clo$-independent set $F_i$, we have $\mu(F_i)\leq\mu(F\setminus X)$ for every $i$. Letting $i\to\infty$ gives $\mu(F)\leq\mu(F\setminus X)$, a contradiction. Hence, $F\in\cF$, proving \ref{ax:i4}.
		
		It remains to prove \ref{ax:i3}. Let $Y\in\Sigma$, and let $F_1,F_2$ be maximal members of $\cF|Y$. Such maximal members exist by \cref{lem:maximal}. If $\mu(F_2\setminus\clo(F_1))>0$, then \ref{ax:cl5}, applied with $A=F_2$, gives a set $H\subseteq F_2\setminus\clo(F_1)$ of positive measure such that $F_1\cup H\in\cF$. This contradicts the maximality of $F_1$ in $Y$. Hence, $F_2\subseteq\clo(F_1)$. By \ref{ax:cl4}, we get $\mu(F_2)\leq\mu(F_1)$. By symmetry, $\mu(F_1)\leq\mu(F_2)$, and therefore $\mu(F_1)=\mu(F_2)$. This proves \ref{ax:i3}.
		
		Finally, we show that the closure operator of the measurable matroid just constructed is $\clo$. Let $Y\in\Sigma$, and let $F$ be a maximal $\clo$-independent subset of $Y$, which exists by \cref{lem:maximal}. By \ref{ax:cl5}, maximality implies $Y\subseteq\clo(F)$, and therefore $\clo(Y)=\clo(F)$ by \ref{ax:cl2} and \ref{ax:cl3}. Also, every $\clo$-independent subset of $\clo(F)$ has measure at most $\mu(F)$ by \ref{ax:cl4}, while $F\subseteq\clo(F)$. Hence, the rank of $\clo(Y)$ in the constructed matroid is $\mu(F)$, which is also the rank of $Y$.
		
		Now let $Z$ be any measurable set with $Y\subseteq Z$ and with the same rank as $Y$. Let $I$ be a maximal $\clo$-independent subset of $Z$. Then $\mu(I)=\mu(F)$. If $\mu(I\setminus\clo(F))>0$, then \ref{ax:cl5} gives a positive-measure set $H\subseteq I\setminus\clo(F)$ such that $F\cup H$ is $\clo$-independent. Since $F\cup H\subseteq Z$, this contradicts the fact that the rank of $Z$ is $\mu(F)$. Hence, $I\subseteq\clo(F)$. By maximality of $I$ in $Z$, again using \ref{ax:cl5}, we have $Z\subseteq\clo(I)$, and therefore $Z\subseteq\clo(F)=\clo(Y)$. Thus, $\clo(Y)$ is the maximal same-rank superset of $Y$, so it is the closure of $Y$ in the constructed measurable matroid.
	\end{proof}
	
	The next lemma shows that, for independent sets of equal measure, containment in the closure is already enough to force equality of closures. We will use this observation in the alternating-chain construction for the proof of the intersection theorem in \cref{sec:intun}.
	
	\begin{lem}\label{lem:closure_same}
		Let $M=(J, \Sigma, \mu, \cF)$ be a measurable matroid and let $F_1, F_2 \in \cF$ with $\mu(F_1)=\mu(F_2)$. Assume that $F_2 \subseteq \clo(F_1)$. Then $\clo(F_2)=\clo(F_1)$. 
	\end{lem}
	
	\begin{proof}
		Define 
		\begin{equation*}
			\cA_{F_2} \coloneqq \{ A \in \Sigma \mid F_2 \subseteq A,\ r(F_2)=r(A) \}.
		\end{equation*}
		By definition, $\clo(F_2)$ is the unique maximal set in this family. As $F_1,F_2\in\cF$, we have
		$$
		r(F_2)=\mu(F_2)=\mu(F_1)=r(F_1)=r(\clo(F_1)).
		$$
		Since $F_2 \subseteq \clo(F_1)$, this implies $\clo(F_1) \in \cA_{F_2}$. Furthermore, by the definition of $\clo(F_1)$, for any $X \supsetneq \clo(F_1)$, it holds that $r(X)>\mu(F_1)$. Consequently, $\clo(F_1)$ is a maximal set in $\cA_{F_2}$. Since $\clo(F_2)$ is the unique maximal set in $\cA_{F_2}$, we get $\clo(F_1)=\clo(F_2)$.
	\end{proof}
	
	%%%%%%%%%%%%%%%%%%%%%%%%%%%%%%%%
	\section{Classes of Measurable Matroids}
	\label{sec:examples}
	%%%%%%%%%%%%%%%%%%%%%%%%%%%%%%%%
	
	In this section, we show that our framework includes normalized finite matroids, cycle matroids of graphings, and measurable analogues of partition, nested, lattice path, transversal, and matching matroids.
	
	%%%%%%%%%%%%%%%%%%%%%%%%%%%%%%%%
	\subsection{Finite Matroids}
	\label{sec:finite}
	%%%%%%%%%%%%%%%%%%%%%%%%%%%%%%%%
	
	First, we show that finite matroids fit naturally into our model, even if we restrict to only atomless measures:  every finite matroid can be embedded into a measurable matroid.
	
	\begin{prop}
		Let $M_0=([n],r_0)$ be a finite matroid, and let $(J,\Sigma,\mu)$ be a standard probability space partitioned into sets $J_1,\dots,J_n$ of equal measure $\mu(J_i)=1/n$. Define $r\colon \Sigma\to \bR_{\geq 0}$ by
		\[
		r(X)\coloneqq \min\left\{ \frac{r_0(A)}{n}+\mu\big(X\setminus \bigcup_{a\in A}J_a\big)\,\middle|\, A\subseteq [n]\right\}.
		\]
		Then $M=(J,\Sigma,\mu,r)$ is a measurable matroid with $r_0(A)=n\cdot r\big(\bigcup_{a\in A}J_a\big)$ for all $A\subseteq [n]$.
	\end{prop}
	
	\begin{proof}
		For ease of discussion, for any $A\subseteq[n]$, let us define $J_A \coloneqq\bigcup_{a \in A} J_a$. We shall use the following equivalent description of independence. For $X\in\Sigma$,
		\[
		r(X)=\mu(X)
		\quad\Longleftrightarrow\quad
		\mu(X\cap J_A)\leq \frac{r_0(A)}{n}
		\text{ for every } A\subseteq[n].
		\]
		Indeed, the inequality $r(X)\leq\mu(X)$ is obtained by taking $A=\emptyset$ in the definition of $r$. Equality holds exactly when every term in the minimum defining $r(X)$ is at least $\mu(X)$, which is equivalent to the given condition.
		
		Axioms \ref{ax:r1} and \ref{ax:r2} clearly hold, while choosing $A=\emptyset$ in the definition of $r$ shows \ref{ax:r3}. To prove \ref{ax:r4}, let $X_1,X_2\in\Sigma$ and assume that $A_1,A_2\subseteq[n]$ are the sets for which
		\[
		r(X_i)=\frac{r_0(A_i)}{n}+\mu(X_i\setminus J_{A_i})
		\]
		for $i=1,2$. We have
		\[
		\mu(X_1\setminus J_{A_1})+\mu(X_2\setminus J_{A_2})\geq\mu((X_1\cap X_2)\setminus J_{A_1\cap A_2})+\mu((X_1\cup X_2)\setminus J_{A_1\cup A_2}).
		\]
		This inequality follows by checking it separately on each part $J_i$: if $i$ belongs to neither or both of $A_1$ and $A_2$, it follows from modularity of $\mu$, while the two remaining cases follow from monotonicity of $\mu$. Hence, by the submodularity of $r_0$,
		\begin{align*}
			r(X_1)+r(X_2)
			&=\frac{r_0(A_1)}{n}+\mu(X_1\setminus J_{A_1})+\frac{r_0(A_2)}{n}+\mu(X_2\setminus J_{A_2})\\
			&\geq \frac{r_0(A_1\cap A_2)}{n}+\mu((X_1\cap X_2)\setminus J_{A_1\cap A_2})+\frac{r_0(A_1\cup A_2)}{n}+\mu((X_1\cup X_2)\setminus J_{A_1\cup A_2})\\
			&\geq r(X_1\cap X_2)+r(X_1\cup X_2),
		\end{align*}
		proving submodularity.
		
		Finally, we prove \ref{ax:r5}. Let $X\subseteq Y$ with $r(X)=\mu(X)$. We construct a sequence of sets $Z_0 \subseteq Z_1 \subseteq \dots \subseteq Z_n$ as follows. Let $Z_0 \coloneqq X$. For each $i \in [n]$, define $Z_i \coloneqq Z_{i-1} \cup W_i$, where $W_i \subseteq (Y \setminus X) \cap J_i$ is chosen so that
		\[
		\mu(W_i)=\min\left\{\mu\left({(Y\setminus X)\cap J_i}\right),\ \min\left\{\frac{r_0(A)}{n}-\mu(Z_{i-1}\cap J_A)\,\middle|\, i\in A\subseteq[n]\right\}\right\}.
		\] 
		We claim that $\mu(Z_i) = r(Z_i)$ for every $i=0,\dots,n$. By the equivalent description above, it suffices to verify that $\mu(Z_i \cap J_A) \le \tfrac{r_0(A)}{n}$ holds for all $A \subseteq [n]$. We argue by induction on $i$. The claim is immediate for $i=0$. Suppose it holds for $i-1$. Fix $A \subseteq [n]$. If $i \notin A$, then $Z_i \cap J_A = Z_{i-1} \cap J_A$, and the desired inequality follows from the induction hypothesis. Thus, a violation could only occur when $i \in A$. In that case, $\mu(Z_i\cap J_A)=\mu(Z_{i-1}\cap J_A)+\mu(W_i)$. However, by the choice of $W_i$, we have $\mu(W_i)\leq \frac{r_0(A)}{n}-\mu(Z_{i-1}\cap J_A)$, and hence $\mu(Z_i \cap J_A) \le \frac{r_0(A)}{n}$, as required.
		
		Set $Z\coloneqq Z_n$. By construction, $X \subseteq Z \subseteq Y$, and the above implies that $\mu(Z)=r(Z)$. It remains to show that $r(Z)=r(Y)$. Since $Z \subseteq Y$, monotonicity implies $r(Z) \leq r(Y)$. For the reverse, let 
		\begin{equation*}
			S \coloneqq \{ i \in [n] \mid \mu(W_i) < \mu((Y \setminus X) \cap J_i) \}.
		\end{equation*} 
		For each $i \in S$, the second term in the definition of $\mu(W_i)$ is the minimum; thus there exists $A_i \subseteq [n]$ with $i \in A_i$ such that $\mu(Z_i \cap J_{A_i}) = \frac{r_0(A_i)}{n}$. As $\mu(Z_j \cap J_{A_i})$ is non-decreasing in $j$ and bounded by $\frac{r_0(A_i)}{n}$, we have $\mu(Z \cap J_{A_i}) = \frac{r_0(A_i)}{n}$. Let $A' \coloneqq \bigcup_{i \in S} A_i$. Note that if $A, B \subseteq [n]$ satisfy $\mu(Z \cap J_A) = \frac{r_0(A)}{n}$ and $\mu(Z \cap J_B) = \frac{r_0(B)}{n}$, then the submodularity of $r_0$, the modularity of $A\mapsto \mu(Z\cap J_A)$, and $\mu(Z)=r(Z)$ imply
		\begin{align*}
			\mu(Z\cap J_A)+\mu(Z\cap J_B)
			&=\frac{r_0(A)}{n}+\frac{r_0(B)}{n}\\
			&\geq \frac{r_0(A \cup B)}{n} + \frac{r_0(A \cap B)}{n}\\
			&\geq \mu(Z \cap J_{A \cup B}) + \mu(Z \cap J_{A \cap B})\\
			&=\mu(Z\cap J_A)+\mu(Z\cap J_B).
		\end{align*}
		Therefore, equality must hold for both $A \cup B$ and $A \cap B$. Thus, $A'$ is tight, that is, $\mu(Z \cap J_{A'}) = \frac{r_0(A')}{n}$. Furthermore, for any $i \notin S$, we have $W_i = (Y \setminus X) \cap J_i$, implying $Z \cap J_i = Y \cap J_i$. It follows that $Y \setminus Z \subseteq J_{S} \subseteq J_{A'}$, which yields $Y \setminus J_{A'} = Z \setminus J_{A'}$. Therefore,
		\[
		r(Y) \leq \frac{r_0(A')}{n} + \mu(Y \setminus J_{A'}) = \mu(Z \cap J_{A'}) + \mu(Z \setminus J_{A'}) = \mu(Z) = r(Z),
		\]
		concluding the proof of \ref{ax:r5}.
		
		Finally, let $A\subseteq[n]$. Taking $B=A$ in the definition gives $r(J_A)\leq r_0(A)/n$. Conversely, for every $B\subseteq[n]$, monotonicity and the fact that adding a single element increases the rank by at most one give
		\[
		r_0(A)\leq r_0(A\cup B)\leq r_0(B)+|A\setminus B|.
		\]
		Hence
		\[
		\frac{r_0(A)}{n}
		\leq
		\frac{r_0(B)}{n}+\mu(J_A\setminus J_B)
		\]
		for every $B\subseteq[n]$, and therefore $r(J_A)\geq r_0(A)/n$. These together imply $r(J_A)=r_0(A)/n$ as stated.
	\end{proof}
	
	%%%%%%%%%%%%%%%%
	\subsection{Cycle Matroids}
	\label{sec:graphic}
	%%%%%%%%%%%%%%%%
	
	For a finite graph $G=(V,E)$, the \emph{cycle matroid} $M(G)=(E,\cI)$ has ground set $E$, where a set of edges is independent if it is a forest; that is,
	\[
	\cI=\{F\subseteq E\mid F\text{ contains no cycle}\}.
	\]
	If $G$ is connected, then the bases of $M(G)$ are precisely the spanning trees of $G$, and its rank is $|V|-1$; see~\cite{oxley2011matroid} for further details.
	
	The corresponding measurable examples arise from graphings; see, e.g.,
	\cite{lovasz2012large} for background. Let $(J,\Sigma_J,\nu)$ be a standard measure space. A \emph{Borel graph} on $J$ is a graph whose edge relation
	\[
	\{(x,y)\in J\times J\mid \{x,y\}\in E\}
	\]
	is a Borel subset of $J\times J$.  A \emph{graphing} $G=(J,\Sigma_J,\nu,E)$ is a Borel graph on $V(G)=J$ with bounded degrees such that, for any Borel sets $A,B\subseteq J$,
	\[
	\int_A \deg(B,x)\diff\nu(x)=\int_B \deg(A,y)\diff\nu(y),
	\]
	where $\deg(S,x)$ is the number of neighbors of $x$ in $S$. This symmetry allows one to define an edge measure $\eta_{G}$ on Borel subsets $F \subseteq E$ by setting 
	\[
	\eta_{G}(F) = \frac12 \int_J \deg_F(x) \diff\nu(x),
	\]
	where $\deg_F(x)$ counts edges in $F$ incident to $x$. The bounded-degree assumption ensures that this integral is finite. 
	
	Throughout the graphing examples, the ground space of the corresponding measurable matroid is the Borel edge space equipped with the edge measure. More precisely, if $G=(J,\Sigma_J,\nu,E)$ is a graphing, then we denote by $\Sigma_E$ the Borel $\sigma$-algebra of the unoriented edge space and equip $(E,\Sigma_E)$ with the finite edge measure $\eta_G$. In the bounded-degree graphing setting considered here, $\eta_G$ is atomless whenever $\nu$ is atomless. Thus, the edge space satisfies the standing convention for ground spaces of measurable matroids. As everywhere in the paper, measurable edge sets are identified up to $\eta_G$-null sets.
	
	For the rest of this section, we normalize the vertex measure so that $\nu(J)=1$. Thus, all graphings considered below are probability-measure-preserving graphings unless stated otherwise. This is only a normalization: for a finite nonzero vertex measure, replacing $\nu$ by $\nu/\nu(J)$ rescales both the edge measure and the rank function by the same factor. Then, following the work of Lovász~\cite{lovasz2024matroid}, the \emph{cycle matroid} of a graphing is defined via its rank function. For a Borel set $F \subseteq E$, the rank is given by 
	\begin{equation} \label{eq:cycle_matroid_rank}
		\rho_{G}(F)=1 - \int_J \frac{1}{|V(G^F_x)|} \diff\nu(x),
	\end{equation}
	where $V(G^F_x)$ is the vertex set of the connected component containing $x$, and we adopt the convention $1/\infty = 0$. The integrand is Borel measurable by \cite[Lemmas~2.2 and~4.1]{lovasz2024matroid}. Intuitively, this generalizes the finite cycle matroid rank to the measurable setting. Lovász~\cite{lovasz2024matroid} showed that $\rho_G$ is monotone, submodular, and satisfies $\rho_G \le \eta_G$, providing a canonical example of a measurable set function with matroid-like properties.
	
	A graphing $G$ is \emph{hyperfinite} if for every $\varepsilon>0$ there exists a Borel set $S \subseteq E$ with $\eta_G(S) \le \varepsilon$ such that all components of $E \setminus S$ are finite. A \emph{forest} is an acyclic $T \subseteq E$, and an \emph{essential spanning forest} is a forest spanning each finite component and containing only infinite components within infinite parts of $G$. Sets are considered forests or essential spanning forests up to measure-zero modifications. Bérczi, Borbényi, Lovász, and Tóth~\cite{berczi2026cycle} showed that, in a graphing, hyperfinite forests and hyperfinite essential spanning forests play the roles of independent sets and bases of a cycle matroid. We recall their results in this subsection as the main motivating example. 
	
	By \cref{cor:mm-measures}, $\mm(\rho_G)$ consists of finite measures dominated by $\eta_G$. A measure $\alpha\in\mm(\rho_G)$ is an \emph{exposed point} if there exists a function $w\in L^1(\eta_G)$ such that
	\[
	\int_E w\diff\alpha>\int_E w\diff\beta
	\]
	for every $\beta\in\mm(\rho_G)$ with $\beta\neq\alpha$. These integrals are well defined because $\alpha$ and $\beta$ are dominated by $\eta_G$. Exposed points of $\bmm(\rho_G)$ are defined analogously. 
	
	\begin{thm}[Bérczi, Borbényi, Lovász, Tóth]
		Let $G=(J,\Sigma_J,\nu,E)$ be a graphing, and let $F\in\Sigma_E$. Then the following are equivalent:
		\begin{enumerate}[label=(\roman*)]\itemsep0em
			\item $F$ is a hyperfinite forest in $G$,
			\item $F$ is a subset of a hyperfinite essential spanning forest,
			\item $\rho_G(F) = \eta_G(F)$,
			\item the measure $\eta_G|_F$ is an exposed point of $\mm(\rho_G)$.
		\end{enumerate}
		Similarly, the following are equivalent:
		\begin{enumerate}[label=(\roman*')]\itemsep0em
			\item $F$ is a hyperfinite essential spanning forest in $G$,
			\item $\rho_G(F) = \eta_G(F) = \rho_G(E)$,
			\item the measure $\eta_G|F$ is an exposed point of $\bmm(\rho_G)$.
		\end{enumerate}
	\end{thm}
	
	Independence axioms for Borel edge sets, compatible with those used here, were previously considered in~\cite{berczi2026cycle} in the specific setting of graphings, where they were shown to be satisfied by hyperfinite forests.
	
	\begin{thm}[Bérczi, Borbényi, Lovász, Tóth]
		Let $G=(J,\Sigma_J,\nu,E)$ be a graphing. Then the family of hyperfinite forests of $G$ satisfies the independence axioms \ref{ax:i1}--\ref{ax:i4} with respect to the edge measure $\eta_G$.
	\end{thm}
	
	Therefore, hyperfinite forests form the independent sets of a measurable matroid. Its bases are precisely the hyperfinite essential spanning forests, and its rank function is $\rho_G$.
	
	%%%%%%%%%%%%%%%%
	\subsection{Partition Matroids}
	\label{sec:partition}
	%%%%%%%%%%%%%%%%
	
	Let $S=S_1\cup \dots\cup S_q$ be a partition of a finite set $S$, and let $k_i$ be an integer with $0\leq k_i\leq |S_i|$ for every $i\in[q]$. The corresponding \emph{partition matroid} has bases
	\[
	\cB=\{X\subseteq S\mid |X\cap S_i|=k_i\text{ for every }i\in[q]\},
	\]
	see~\cite{oxley2011matroid}. The sets $S_1,\dots,S_q$ are called \emph{partition classes}. Partition matroids are among the simplest and most useful classes of finite matroids, since many combinatorial conditions can be expressed by imposing bounds on the partition classes. For example, the matchings of a bipartite graph are the common independent sets of two partition matroids corresponding to the degree constraints on the two sides. As we will see in \cref{sec:bipartite}, the same description applies to measurable matchings in bipartite graphings with Borel bipartition.
	
	There are several ways to define measurable counterparts of partition matroids. The simplest one replaces the cardinality bounds on countably many partition classes by bounds on their measures. Let $(J,\Sigma,\mu)$ be a standard measure space, let $\{J_i\}_{i\in\bN}$ be a partition of $J$ into countably many measurable parts, and let $b_i\in\bR_{\geq 0}$ for every $i\in\bN$. Define
	\begin{equation}\label{eq:finitary_partition_matroid}
		\cF\coloneqq \{X\in\Sigma \mid \mu(X\cap J_i)\leq b_i \text{ for all } i\in\bN\}.
	\end{equation}
	We call the resulting system a \emph{finitary measurable partition matroid}.
	
	\begin{prop}\label{prop:finitary_partition_matroid}
		Let $(J,\Sigma,\mu)$ be a standard measure space, let $\{J_i\}_{i\in\bN}$ be a partition of $J$ into countably many measurable parts, and let $b_i\in\bR_{\geq 0}$ for every $i\in\bN$. If $\cF$ is defined by \eqref{eq:finitary_partition_matroid}, then $M=(J,\Sigma,\mu,\cF)$ is a measurable matroid with rank function $r(X)=\sum_{i=1}^\infty \min\{\mu(X\cap J_i),b_i\}$ for all $X\in\Sigma$.
	\end{prop}
	\begin{proof}
		Axioms \ref{ax:i1} and \ref{ax:i2} are immediate. To prove \ref{ax:i4}, let $(F_n)_{n\in\bN}$ be an increasing sequence in $\cF$, and set $F\coloneqq \bigcup_{n=1}^{\infty}F_n$. For every $i\in\bN$, continuity of measure from below gives $\mu(F\cap J_i)=\lim_{n\to\infty}\mu(F_n\cap J_i)\leq b_i$. Hence, $F\in\cF$.
		
		It remains to verify \ref{ax:i3}. Let $X\in\Sigma$, and let $I$ be a maximal member of $\cF|X$. We claim that, for every $i\in\bN$,
		\[
		\mu(I\cap J_i)=\min\{\mu(X\cap J_i),b_i\}.
		\]
		The inequality $\mu(I\cap J_i)\leq \min\{\mu(X\cap J_i),b_i\}$ is clear. If $\min\{\mu(X\cap J_i),b_i\}=0$, then equality follows immediately. Otherwise, if the inequality were strict for some $i$, then both $\mu((X\cap J_i)\setminus I)>0$ and $b_i-\mu(I\cap J_i)>0$. Since $\mu$ is atomless, there would exist a measurable set $A\subseteq (X\cap J_i)\setminus I$ with $0<\mu(A)\leq b_i-\mu(I\cap J_i)$. Then $I\cup A\in\cF|X$, contradicting the maximality of $I$. This proves the above equality. Summing over the countable partition $\{J_i\}_{i\in\bN}$ gives
		\[
		r(X)=\mu(I)=\sum_{i=1}^{\infty}\mu(I\cap J_i)=\sum_{i=1}^{\infty}\min\{\mu(X\cap J_i),b_i\}.
		\]
		
		If $I_1$ and $I_2$ are maximal members of $\cF|X$, then the claim gives $\mu(I_1\cap J_i)=\mu(I_2\cap J_i)$ for every $i\in\bN$. Since the sets $J_i$ form a countable partition of $J$, countable additivity yields $\mu(I_1)=\mu(I_2)$. Hence, \ref{ax:i3} holds, and $M=(J,\Sigma,\mu,\cF)$ is a measurable matroid.
	\end{proof}
	
	A different natural construction uses the partition of a product space into the fibers of one of the coordinate projections. 
	
	\begin{prop}\label{prop:fiberwise_partition_matroid}
		Let $(S,\Sigma_S,\nu)$ and $(T,\Sigma_T,\eta)$ be standard measure spaces, and set $J\coloneqq S\times T$, $\Sigma\coloneqq\Sigma_S\otimes\Sigma_T$, and $\mu\coloneqq\nu\otimes\eta$. Let $b\colon T\to\bR_{\geq 0}$ be Borel measurable. Define
		\begin{equation*}\label{eq:fiberwise_partition_matroid}
			\cF\coloneqq
			\{F\in\Sigma\mid
			\nu(F_t)\leq b(t)
			\text{ for }\eta\text{-almost every }t\in T\}.
		\end{equation*}
		Then $
		M=(J,\Sigma,\mu,\cF)
		$
		is a measurable matroid. Moreover, for every $Z\in\Sigma$, its rank is given by
		\[
		r(Z)
		=
		\int_T
		\min\{\nu(Z_t),b(t)\}
		\diff\eta(t).
		\]
	\end{prop}
	
	\begin{proof}
		By Fubini's theorem, see, e.g., \cite[Theorem~2.37]{folland1999real}, the map
		$
		t\longmapsto \nu(X_t)
		$
		is measurable for every $X\in\Sigma$. Axioms \ref{ax:i1} and \ref{ax:i2} are immediate. To prove \ref{ax:i4}, let $(F_i)_{i\in\bN}$ be an increasing sequence in $\cF$, and set
		$
		F\coloneqq \bigcup_{i=1}^{\infty}F_i.
		$
		Outside the union of the exceptional $\eta$-null sets for the sets $F_i$, continuity of measure from below gives
		\[
		\nu(F_t)
		=
		\lim_{i\to\infty}\nu((F_i)_t)
		\leq b(t).
		\]
		Hence $F\in\cF$.
		
		It remains to verify \ref{ax:i3} and the rank formula. Let $Z\in\Sigma$, and let $I$ be a maximal member of $\cF|Z$, which exists by \cref{lem:maximal}. Define
		\[
		d(t)
		\coloneqq
		\min\{\nu(Z_t),b(t)\}-\nu(I_t).
		\]
		Since $I\subseteq Z$ and $I\in\cF$, we have $d(t)\geq 0$ for $\eta$-almost every $t\in T$. We claim that $d(t)=0$ for $\eta$-almost every $t\in T$. Suppose, for contradiction, that this is not the case. Then there exists $\varepsilon>0$ such that the measurable set $Y\coloneqq \{t\in T\mid d(t)>\varepsilon\}$ has positive $\eta$-measure. Set $E\coloneqq Z\setminus I$. For $\eta$-almost every $t\in Y$, we have $\nu(E_t)=\nu(Z_t)-\nu(I_t)>\varepsilon$ and $b(t)-\nu(I_t)>\varepsilon$. Choose $m\in\bN$ such that
		$
		\frac{\nu(S)}{m}<\varepsilon.
		$
		Since $\nu$ is atomless, there exists a measurable partition
		$
		S=Q_1\cup\dots\cup Q_m
		$
		such that
		$
		\nu(Q_j)=\frac{\nu(S)}{m}
		$
		for every $j\in[m]$. Using Fubini's theorem again,
		\[
		\mu(E\cap(S\times Y))
		=
		\int_Y\nu(E_t)\diff\eta(t)
		>
		\varepsilon\eta(Y)
		>
		0.
		\]
		Hence, for some $j\in[m]$, the set
		$
		A\coloneqq E\cap(Q_j\times Y)
		$
		has positive measure. For every $t\in Y$, we have
		\[
		\nu(A_t)
		\leq
		\nu(Q_j)
		=
		\frac{\nu(S)}{m}
		<
		\varepsilon,
		\]
		while $A_t=\emptyset$ for every $t\notin Y$. It follows that
		$
		\nu((I\cup A)_t)\leq b(t)
		$
		for $\eta$-almost every $t\in T$. Thus $I\cup A\in\cF|Z$, contradicting the maximality of $I$.
		
		Consequently,
		\[
		\mu(I)
		=
		\int_T\nu(I_t)\diff\eta(t)
		=
		\int_T
		\min\{\nu(Z_t),b(t)\}
		\diff\eta(t)
		\]
		for every maximal member $I\in\cF|Z$. Hence \ref{ax:i3} holds, and the rank formula follows.
	\end{proof}
	
	We call the measurable matroid provided by \cref{prop:fiberwise_partition_matroid} a \emph{fiberwise measurable partition matroid}.
	
	Before turning to finite-classed examples, we present a simple selection principle for finite Borel equivalence relations. A \emph{Borel equivalence relation} on a standard Borel space $(J,\Sigma)$ is an equivalence relation $R\subseteq J\times J$ which is Borel as a subset of the product space. We write $xRy$ to mean $(x,y)\in R$, and
	\[
	[x]_R \coloneqq \{y\in J\mid xRy\}
	\]
	for the \emph{$R$-equivalence class} of $x$. The relation $R$ is \emph{finite} if every $R$-class is finite. A Borel set $T\subseteq J$ is a \emph{partial transversal} for $R$ if it meets every $R$-class in at most one point. A partial transversal is a \emph{transversal} if it meets every
	$R$-class. A set $C\subseteq J$ is called \emph{$R$-invariant} if it is a union of $R$-classes. The \emph{$R$-saturation} of a set $A\subseteq J$ is
	\[
	[A]_R\coloneqq\{y\in J\mid xRy\text{ for some }x\in A\}.
	\]
	A Borel bijection $\varphi\colon A\to B$ between Borel sets is called \emph{$R$-invariant} if $xR\varphi(x)$ for every $x\in A$. If $\mu$ is a measure on $(J,\Sigma)$, then a Borel equivalence relation $R$ is called \emph{measure-preserving} if every $R$-invariant Borel bijection between Borel sets preserves $\mu$.
	We will use the following standard facts about finite Borel equivalence relations.
	
	\begin{lem}\label{lem:finite_class_selection}
		Let $(J,\Sigma,\mu)$ be a standard measure space with a finite Borel equivalence relation $R$ on $J$.
		\begin{enumerate}[label=(\roman*)]\itemsep0em
			\item Every Borel set $A\subseteq J$ is a countable disjoint union of Borel partial transversals for $R$. In particular, if $\mu(A)>0$, then $A$ contains a Borel partial transversal of positive measure. \label{it:borel1}
			\item The relation $R$ is smooth; that is, it admits a Borel
			transversal. \label{it:borel2}
			\item Assume in addition that $R$ is measure-preserving. Let $C\subseteq J$ be a Borel $R$-invariant set with $\mu(C)>0$, and let $A\subseteq C$ be Borel such that $A\cap [x]_R\neq\emptyset$ for every $x\in C$. Then $A$ contains a Borel partial transversal of positive measure.\label{it:borel3}
		\end{enumerate}
	\end{lem}
	
	\begin{proof}
		Fix a linear order $\prec$ on $J$ whose graph $\{(x,y)\in J\times J\mid x\prec y\}$ is Borel. Such an order exists on every standard Borel space, for instance by pulling back the usual order from a Borel embedding into $[0,1]$; see~\cite[Section~15]{kechris1995classical}. For a Borel set $A\subseteq J$ and $j\in\bN$, let $A_j$ be the set of points $x\in A$ which are the $j$-th element, with respect to $\prec$, of the finite ordered set $A\cap[x]_R$. An immediate application of \cite[Lemma~18.12]{kechris1995classical} shows that the sets $A_j$ are Borel. They are pairwise disjoint partial transversals and $A=\bigcup_{j=1}^\infty A_j$. This proves \ref{it:borel1}.
		
		Taking $A=J$, the set $A_1$ is a Borel transversal for $R$, proving
		\ref{it:borel2}.
		
		For \ref{it:borel3}, let $A_1\subseteq A$ be the set defined from $A$ in the proof of \ref{it:borel1}. Since $A$ meets every $R$-class contained in $C$, the set $A_1$ is a Borel partial transversal and its $R$-saturation is $C$. If $\mu(A_1)=0$, then the $R$-saturation of $A_1$ is null: indeed, by decomposing the finite relation according to class sizes and ordered positions within each class, this saturation is a countable union of images of subsets of $A_1$ under $R$-invariant Borel bijections, each of which preserves $\mu$. This contradicts $\mu(C)>0$. Hence, $\mu(A_1)>0$, proving \ref{it:borel3}.
	\end{proof}
	
	For every Borel set $X\subseteq J$, the function
	$x\mapsto |X\cap[x]_R|$ is Borel measurable by the same application
	of \cite[Lemma~18.12]{kechris1995classical} as in the proof of
	\cref{lem:finite_class_selection}.
	
	Then another useful variant of measurable partition matroids arises from a finite, measure-preserving Borel equivalence relation $R$ on $J$. Define
	\begin{equation}\label{eq:finite_classed_partition}
		\cF \coloneqq \{X \in \Sigma \mid \mu(\{x\in J\mid |X \cap [x]_R|>1\})=0\}.
	\end{equation}
	Equivalently, a set in $\mathcal{F}$ contains at most one element from almost every $R$-equivalence class. This construction provides a measurable analogue of a partition matroid: each equivalence class serves as a ``part'', with a bound of one element per part. We call the resulting system a \emph{finite-classed measurable partition matroid}.
	
	\begin{prop} \label{prop:finite_classed_partition}
		Let $(J,\Sigma,\mu)$ be a standard measure space, and let $R$ be a finite Borel equivalence relation on $J$ that is measure-preserving with respect to $\mu$. If $\cF$ is defined by \eqref{eq:finite_classed_partition}, then $M=(J,\Sigma,\mu,\cF)$ is a measurable matroid.    
	\end{prop}
	
	\begin{proof}
		Axioms \ref{ax:i1} and \ref{ax:i2} are immediate. To prove \ref{ax:i4}, let $(F_i)_{i\in\bN}$ be an increasing sequence in $\cF$, and set $F\coloneqq \bigcup_{i=1}^{\infty}F_i$. Let $N_i\coloneqq\{x\in J\mid |F_i\cap[x]_R|>1\}$. Then $\mu(N_i)=0$ for every $i$. If $|F\cap[x]_R|>1$, then, since $[x]_R$ is finite and the sequence is increasing, the same violation already occurs in some $F_i$. Hence, $\{x\in J\mid |F\cap[x]_R|>1\}\subseteq\bigcup_{i=1}^{\infty}N_i$, and so $F\in\cF$.
		
		It remains to verify \ref{ax:i3}. Let $Y\in\Sigma$, and let $F_1,F_2\in\cF|Y$ be maximal. Suppose that, for some $i\in\{1,2\}$, the set
		\[
		Y_i\coloneqq \{x\in Y\mid F_i\cap [x]_R=\emptyset\}
		\]
		has positive measure. By \cref{lem:finite_class_selection}\ref{it:borel1}, $Y_i$ contains a Borel partial transversal $A$ of positive measure. Then $F_i\cup A\subseteq Y$ and $F_i\cup A\in\cF$, contradicting the maximality of $F_i$. 
		
		Thus, $\mu(Y_i)=0$ for $i\in\{1,2\}$. Since $R$ is measure-preserving, the proof of \cref{lem:finite_class_selection}\ref{it:borel3} shows that the $R$-saturation $[Y_i]_R$ is null. Consequently, outside a null set, both $F_1$ and $F_2$ meet every $R$-class that meets $Y$. Since $F_1,F_2\in\cF$, they meet each such class in exactly one point. Hence, the relation
		\[
		\{(x,y)\in F_1\times F_2\mid xRy\}
		\]
		defines, outside null sets, an $R$-invariant Borel bijection $\varphi\colon F_1\to F_2$. Since $R$ is measure-preserving, it follows that $\mu(F_1)=\mu(F_2)$. This proves \ref{ax:i3}.
	\end{proof}
	
	While finite-classed measurable partition matroids are not particularly interesting on their own, when employed as one of the measurable matroids in a measurable matroid intersection, they can give rise to nice applications. A particularly important example of this class is the following.
	
	Let $G=(J,\Sigma_J,\nu,E)$ be a bipartite graphing with Borel bipartition $J=S\cup T$, let $\Sigma_E$ denote the Borel $\sigma$-algebra of the edge space, and let $\eta_G$ be the associated edge measure. Define
	\begin{equation}\label{eq:bipartite_partition_matroid}
		\cF\coloneqq\{F\in\Sigma_E\mid \deg_F(s)\leq 1\text{ for $\nu$-almost every }s\in S\}.
	\end{equation}
	
	\begin{prop}\label{prop:bipartite_partition_matroid}
		Let $G=(J,\Sigma_J,\nu,E)$ be a bipartite graphing on a standard atomless probability space with Borel bipartition $J = S \cup T$, and let $\eta_G$ be the associated edge measure on $E$. If $\cF$ is defined by \eqref{eq:bipartite_partition_matroid}, then $M = (E,\Sigma_E,\eta_G,\mathcal{F})$ is a measurable matroid. Moreover, for every $X \in \Sigma_E$, its rank is given by
		\[
		r(X)=\nu\left(\{s \in S \mid \deg_X(s)\geq 1\}\right).
		\]    
	\end{prop}
	
	\begin{proof}
		Let $R$ be the Borel equivalence relation on $E$ defined by declaring two edges equivalent if they are incident to a common vertex in $S$. Since $G$ has bounded degrees, $R$ is a finite Borel equivalence relation on the edge space. Moreover, $R$ is measure-preserving with respect to $\eta_G$. Indeed, for every Borel edge set $A\subseteq E$, the graphing identity gives
		\[
		\int_S\deg_A(s)\,\diff\nu(s)
		=
		\int_T\deg_A(t)\,\diff\nu(t).
		\]
		Consequently, the definition of the edge measure yields
		\[
		\eta_G(A)
		=
		\frac{1}{2}\int_J\deg_A(x)\,\diff\nu(x)
		=
		\int_S\deg_A(s)\,\diff\nu(s).
		\]
		If $\varphi\colon A\to B$ is an $R$-invariant Borel bijection between Borel edge sets, then $\varphi$ preserves the $S$-fiber of every edge, and hence
		$
		\deg_A(s)=\deg_B(s)
		$
		for every $s\in S$. Therefore, $\eta_G(A)=\eta_G(B)$.
		
		The condition defining $\cF$ in \eqref{eq:bipartite_partition_matroid} is exactly the finite-classed partition condition associated with this equivalence relation: a set $F\subseteq E$ belongs to $\cF$ if and only if it contains at most one edge from almost every $R$-class. Hence, $M=(E,\Sigma_E,\eta_G,\cF)$ is a measurable matroid by \cref{prop:finite_classed_partition}.
		
		It remains to compute the rank. Let $X\in\Sigma_E$, and let $F$ be a maximal member of $\cF|X$. We claim that $F$ contains one edge from almost every nonempty $S$-fiber of $X$. Suppose not. Then the set
		\[
		U\coloneqq \{s\in S\mid \deg_X(s)\geq 1 \text{ and } \deg_F(s)=0\}
		\]
		has positive $\nu$-measure. Let
		\[
		A_0\coloneqq \{e\in X\mid e \text{ is incident to a vertex of } U\cap S\}.
		\]
		Then $\eta_G(A_0)>0$. By \cref{lem:finite_class_selection}, $A_0$ contains a Borel partial transversal $A$ for $R$ with positive $\eta_G$-measure. Since $A$ uses at most one edge from each $S$-fiber and all these fibers are disjoint from $F$, we have $F\cup A\in\cF$. Also $F\cup A\subseteq X$ and $\eta_G(A)>0$, contradicting the maximality of $F$ in $\cF|X$.
		
		Thus, up to a null set, $F$ contains exactly one edge from each $S$-fiber of $X$ that is nonempty. Therefore
		\[
		r(X)=\eta_G(F)
		=
		\int_S \deg_F(s)\,\diff\nu(s)
		=
		\nu\left(\{s\in S\mid \deg_X(s)\geq 1\}\right),
		\]
		as claimed.
	\end{proof}
	
	The motivation for this construction is that measurable matchings in a bipartite graphing are precisely the common independent sets of the two partition matroids obtained from the two sides of the bipartition. This will be used in \cref{sec:bipartite}.
	
	\begin{rem}
		The aim of this subsection is to present concrete examples that are easy to describe and will be used directly later, rather than to develop the constructions in their greatest possible generality. For instance, the partition-matroid models above can be placed in a common framework using disintegration of measures, while the finite-classed construction extends naturally by assigning an arbitrary finite matroid, rather than only a rank-one partition matroid, to each equivalence class in a Borel way. We leave the study of these more general constructions for future work.
	\end{rem}
	
	%%%%%%%%%%%%%%%%%%%%%%%%%%%%%%%%
	\subsection{Nested Matroids}
	\label{sec:nested}
	%%%%%%%%%%%%%%%%%%%%%%%%%%%%%%%%
	
	Let $S$ be a finite set, let $\cC\subseteq 2^S$ be a chain under inclusion, that is, for every $C,D\in\cC$, either $C\subseteq D$ or $D\subseteq C$, and let $b\colon\cC\to\bZ_{\geq 0}$. Then
	\begin{equation*}
		\cI\coloneqq\{I\subseteq S\mid |I\cap C|\leq b(C)\text{ for every }C\in\cC\}
	\end{equation*}
	is the family of independent sets of a matroid. The matroids that admit such a description are precisely the \emph{nested matroids}. Equivalently, these are the transversal matroids that admit a nested presentation. They are also known as generalized Catalan, freedom, or Schubert matroids~\cite{fife2019generalized}. Nested matroids form a proper subclass of lattice path matroids, which we discuss in the next subsection~\cite{bonin2003lattice}.
	
	We use the chain description to define a measurable analogue. A chain $\cC\subseteq\Sigma$ is \emph{closed} if it is both $\sigma$-increasing and $\sigma$-decreasing.
	
	\begin{thm}\label{thm:nested_matroids}
		Let $\cC\subseteq\Sigma$ be a closed chain containing $\emptyset$, and let $b\colon\cC\to\bR_{\geq 0}$ satisfy $b(\emptyset)=0$. Define
		\begin{equation}\label{eq:nested_matroid}
			\cF\coloneqq\{F\in\Sigma\mid \mu(F\cap C)\leq b(C)\text{ for every }C\in\cC\}.
		\end{equation}
		Then $M=(J,\Sigma,\mu,\cF)$ is a measurable matroid. Moreover, for every $X\in\Sigma$, its rank is given by
		\begin{equation}\label{eq:nested_rank}
			r(X)=\inf_{C\in\cC}\{b(C)+\mu(X\setminus C)\}.
		\end{equation}
	\end{thm}
	
	\begin{proof}
		Axioms \ref{ax:i1} and \ref{ax:i2} are immediate. To prove \ref{ax:i4}, let $(F_i)_{i\in\bN}$ be an increasing sequence in $\cF$, and set $F\coloneqq\bigcup_{i=1}^{\infty}F_i$. For every $C\in\cC$, continuity of measure from below gives $\mu(F\cap C)=\lim_{i\to\infty}\mu(F_i\cap C)\leq b(C)$. Hence, $F\in\cF$.
		
		We first replace $b$ by an effective bound. For $C\in\cC$, let
		\begin{equation*}
			\beta(C)\coloneqq\inf_{D\in\cC}\{b(D)+\mu(C\setminus D)\}.
		\end{equation*}
		Since $b(\emptyset)=0$, we have $\beta(\emptyset)=0$. If $C,C'\in\cC$ with $C\subseteq C'$, then $0\leq\beta(C')-\beta(C)\leq\mu(C'\setminus C)$. In particular, if $(C_i)_{i\in\bN}$ is monotone in $\cC$ and converges to $C\in\cC$ by union or intersection, then $\beta(C_i)\to\beta(C)$.
		
		The functions $b$ and $\beta$ define the same family $\cF$. Indeed, if $F$ satisfies the constraints in \eqref{eq:nested_matroid}, then, for every $C,D\in\cC$,
		\begin{equation*}
			\mu(F\cap C)\leq\mu(F\cap D)+\mu(C\setminus D)\leq b(D)+\mu(C\setminus D).
		\end{equation*}
		Taking the infimum over $D$ gives $\mu(F\cap C)\leq\beta(C)$. Conversely, $\beta(C)\leq b(C)$ follows by taking $D=C$ in the definition of $\beta$.
		
		It remains to prove \ref{ax:i3} and the rank formula. Let $X\in\Sigma$, and let $I$ be a maximal member of $\cF|X$, which exists by \cref{lem:maximal}. For $C\in\cC$, set $s(C)\coloneqq\beta(C)-\mu(I\cap C)$. For $\varepsilon>0$, let
		\begin{equation*}
			\cC_\varepsilon\coloneqq\{C\in\cC\mid s(C)\leq\varepsilon\}.
		\end{equation*}
		The family $\cC_\varepsilon$ is nonempty, since it contains $\emptyset$. The continuity of $\beta$ and of measure along monotone sequences implies that $\cC_\varepsilon$ is both $\sigma$-increasing and $\sigma$-decreasing. Thus, \cref{lem:maximal} gives a maximal member $C'_\varepsilon$ of $\cC_\varepsilon$. Since $\cC_\varepsilon$ is a chain, every member of $\cC_\varepsilon$ is contained in $C'_\varepsilon$.
		
		For $i\in\bN$, set $C_i\coloneqq C'_{1/i}$. The sequence $(C_i)_{i\in\bN}$ is decreasing. Since $\cC$ is $\sigma$-decreasing, the set $C_I\coloneqq\bigcap_{i=1}^{\infty}C_i$ belongs to $\cC$. The continuity of $\beta$ and of measure along this sequence gives $s(C_I)=\lim_{i\to\infty}s(C_i)=0$.
		
		We claim that $X\setminus I\subseteq C_I$. Suppose otherwise, and set $P\coloneqq(X\setminus I)\setminus C_I$. Then $\mu(P)>0$. Since $C_i$ decreases to $C_I$, there is an $i\in\bN$ such that $\mu(P\setminus C_i)>0$. By atomlessness, there is a set $A\subseteq P\setminus C_i$ satisfying $0<\mu(A)<1/i$. Let $C\in\cC$. If $s(C)\leq1/i$, then $C\subseteq C_i$, and hence $A\cap C=\emptyset$. If $s(C)>1/i$, then $\mu(A\cap C)\leq\mu(A)<s(C)$. In either case, $\mu((I\cup A)\cap C)\leq\beta(C)$. Thus, $I\cup A\in\cF|X$, contradicting the maximality of $I$. This proves the claim.
		
		Since $s(C_I)=0$ and $X\setminus C_I\subseteq I$, we obtain
		\begin{equation*}
			\mu(I)=\beta(C_I)+\mu(X\setminus C_I).
		\end{equation*}
		On the other hand, for every $C\in\cC$,
		\begin{equation*}
			\mu(I)=\mu(I\cap C)+\mu(I\setminus C)\leq\beta(C)+\mu(X\setminus C).
		\end{equation*}
		Therefore,
		\begin{equation*}
			\mu(I)=\min_{C\in\cC}\{\beta(C)+\mu(X\setminus C)\}.
		\end{equation*}
		In particular, all maximal members of $\cF|X$ have the same measure, proving \ref{ax:i3}.
		
		It remains to express the right-hand side in terms of the original function $b$. Since $\beta(C)\leq b(C)$ for every $C\in\cC$, the last minimum is at most the right-hand side of \eqref{eq:nested_rank}. Conversely, for every $C\in\cC$,
		\begin{equation*}
			\beta(C)+\mu(X\setminus C)=\inf_{D\in\cC}\{b(D)+\mu(C\setminus D)+\mu(X\setminus C)\}\geq\inf_{D\in\cC}\{b(D)+\mu(X\setminus D)\}.
		\end{equation*}
		Taking the minimum over $C$ gives the reverse inequality. Hence, \eqref{eq:nested_rank} follows.
	\end{proof}
	
	Although no regularity assumption on $b$ is needed in \cref{thm:nested_matroids}, lower semicontinuity with respect to convergence in measure yields a minimum in the rank formula.
	
	\begin{cor}\label{cor:nested_rank_minimum}
		In the setting of \cref{thm:nested_matroids}, suppose that $b$ is lower semicontinuous with respect to convergence in measure on $\cC$, meaning that $b(C)\leq\liminf_{i\to\infty}b(C_i)$ whenever $C,C_i\in\cC$ and $\mu(C_i\triangle C)\to0$. Then, for every $X\in\Sigma$, its rank is given by
		\begin{equation*}
			r(X)=\min_{C\in\cC}\{b(C)+\mu(X\setminus C)\}.
		\end{equation*}
	\end{cor}
	
	\begin{proof}
		Fix $X\in\Sigma$, and choose a sequence $(C_i)_{i\in\bN}$ in $\cC$ such that $b(C_i)+\mu(X\setminus C_i)$ converges to the right-hand side of \eqref{eq:nested_rank}. Since $\cC$ is a chain, $(C_i)_{i\in\bN}$ has an increasing or decreasing subsequence. Passing to this subsequence, we may assume that $(C_i)_{i\in\bN}$ is monotone. Set $C\coloneqq\bigcup_{i\in\bN}C_i$ if the sequence is increasing and $C\coloneqq\bigcap_{i\in\bN}C_i$ if it is decreasing. Since $\cC$ is closed, $C\in\cC$, and continuity of measure gives $\mu(C_i\triangle C)\to0$. Consequently,
		\begin{equation*}
			b(C)+\mu(X\setminus C)\leq\liminf_{i\to\infty}\left(b(C_i)+\mu(X\setminus C_i)\right)=\inf_{D\in\cC}\{b(D)+\mu(X\setminus D)\}.
		\end{equation*}
		Thus, the infimum is attained at $C$.
	\end{proof}
	
	We call the matroids obtained from \cref{thm:nested_matroids} \emph{measurable nested matroids}. Closed chains arise naturally from continuous parametrizations. An increasing map $Z\colon[0,1]\to\Sigma$ is \emph{continuous in measure} if $\mu(Z(t_i)\triangle Z(t))\to0$ whenever $t_i\to t$. Suppose that $Z(0)=\emptyset$ and $Z(1)=J$, and set $\cC\coloneqq\{Z(t)\mid t\in[0,1]\}$. Then $\cC$ is a closed chain. Indeed, suppose that $(Z(t_i))_{i\in\bN}$ is increasing, and set $u_i\coloneqq\max\{t_1,\dots,t_i\}$. Then $(u_i)_{i\in\bN}$ is increasing and $Z(u_i)=Z(t_i)$ for every $i$. If $u\coloneqq\lim_{i\to\infty}u_i$, continuity in measure gives $\bigcup_{i=1}^{\infty}Z(t_i)=Z(u)$. For a decreasing sequence, the same argument applies with $u_i\coloneqq\min\{t_1,\dots,t_i\}$. Thus, every function $b\colon\cC\to\bR_{\geq 0}$ with $b(\emptyset)=0$ defines a measurable nested matroid by \cref{thm:nested_matroids}.
	
	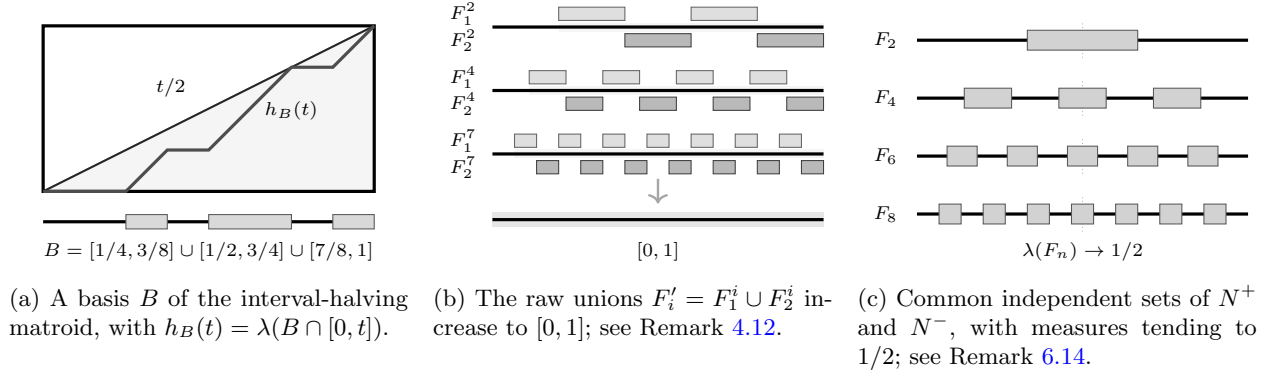
\begin{figure}[t]
		
		\begin{subfigure}[t]{0.32\textwidth}
			\centering
			
			\begin{tikzpicture}[scale=0.73]
				
				% Feasible region below t/2
				\fill[gray!8]
				(0,0) --
				(6,0) --
				(6,3) --
				cycle;
				
				% Frame
				\draw[very thick] (0,0) rectangle (6,3);
				
				% Upper bound t/2
				\draw[thick, draw=black!85]
				(0,0) -- (6,3);
				
				% Cumulative function of the basis
				\draw[very thick, draw=black!70]
				(0,0) --
				(1.5,0) --
				(2.25,0.75) --
				(3,0.75) --
				(4.5,2.25) --
				(5.25,2.25) --
				(6,3);
				
				% Labels
				\node[font=\scriptsize, above left] at (2.75,1.55) {$t/2$};
				\node[font=\scriptsize, below right] at (3.85,1.80) {$h_B(t)$};
				
				% Basis B on [0,1]
				\draw[very thick] (0,-0.55) -- (6,-0.55);
				
				\fill[gray!30, draw=black!65]
				(1.5,-0.68) rectangle (2.25,-0.42);
				\fill[gray!30, draw=black!65]
				(3,-0.68) rectangle (4.5,-0.42);
				\fill[gray!30, draw=black!65]
				(5.25,-0.68) rectangle (6,-0.42);
				
				\node[font=\scriptsize, align=center] at (3,-1.08)
				{$B=[1/4,3/8]\cup[1/2,3/4]\cup[7/8,1]$};
				
			\end{tikzpicture}
			
			\caption{A basis $B$ of the interval-halving matroid, with $h_B(t)=\lambda(B\cap[0,t])$.}
			\label{fig:halving1}
			
		\end{subfigure}
		\hfill
		\begin{subfigure}[t]{0.32\textwidth}
			\centering
			
			\begin{tikzpicture}[scale=0.73]
				
				\def\len{6}
				
				% Three approximating raw unions
				\foreach \row/\i in {0/2,1/4,2/7} {
					
					\pgfmathsetmacro{\y}{-\row*1.15}
					\pgfmathsetmacro{\start}{\len/(2*\i+1)}
					
					% Raw union F'_i
					\fill[gray!8]
					(\start,\y-0.08) rectangle (\len,\y+0.08);
					
					% Ground interval
					\draw[very thick]
					(0,\y) -- (\len,\y);
					
					% F_1^i
					\foreach \k in {1,...,7} {
						\pgfmathsetmacro{\a}{\len*(2*\k-1)/(2*\i+1)}
						\pgfmathsetmacro{\b}{\len*(2*\k)/(2*\i+1)}
						\ifdim \a pt<\len pt
						\fill[gray!25, draw=black!60]
						(\a,\y+0.12)
						rectangle
						({min(\b,\len)},\y+0.36);
						\fi
					}
					
					% F_2^i
					\foreach \k in {1,...,7} {
						\pgfmathsetmacro{\a}{\len*(2*\k)/(2*\i+1)}
						\pgfmathsetmacro{\b}{\len*(2*\k+1)/(2*\i+1)}
						\ifdim \a pt<\len pt
						\fill[gray!55, draw=black!75]
						(\a,\y-0.36)
						rectangle
						({min(\b,\len)},\y-0.12);
						\fi
					}
					
					% Labels
					\node[font=\scriptsize, anchor=east] at (-0.12,\y+0.24) {$F_1^\i$};
					\node[font=\scriptsize, anchor=east] at (-0.12,\y-0.24) {$F_2^\i$};
				}
				
				% Arrow to the limit
				\draw[gray!70, line width=0.9pt, ->]
				(3,-2.75) -- (3,-3.20);
				
				% Limit interval
				\fill[gray!20]
				(0,-3.62) rectangle (6,-3.36);
				\draw[very thick]
				(0,-3.49) -- (6,-3.49);
				
				\node[font=\scriptsize, below] at (3,-3.72) {$[0,1]$};
				
			\end{tikzpicture}
			
			\caption{The raw unions $F'_i=F_1^i\cup F_2^i$ increase to $[0,1]$; see \cref{rem:unionnotm}.}
			\label{fig:halving2}
			
		\end{subfigure}
		\hfill
		\begin{subfigure}[t]{0.32\textwidth}
			\centering
			
			\begin{tikzpicture}[scale=0.73]
				
				\def\len{6}
				
				% Four common independent sets F_n
				\foreach \row/\n in {0/2,1/4,2/6,3/8} {
					
					\pgfmathsetmacro{\y}{-\row*1.05}
					\pgfmathtruncatemacro{\den}{2*\n-1}
					\pgfmathtruncatemacro{\num}{\n-1}
					\pgfmathtruncatemacro{\last}{\n-1}
					
					% Ground interval
					\draw[very thick]
					(0,\y) -- (\len,\y);
					
					% Reflection axis
					\draw[gray!55, densely dotted]
					(3,\y-0.32) -- (3,\y+0.32);
					
					% F_n
					\foreach \k in {1,...,\last} {
						\pgfmathsetmacro{\a}{\len*(2*\k-1)/\den}
						\pgfmathsetmacro{\b}{\len*(2*\k)/\den}
						\fill[gray!35, draw=black!65]
						(\a,\y-0.18)
						rectangle
						(\b,\y+0.18);
					}
					
					\node[font=\scriptsize, left] at (-0.18,\y)
					{$F_{\n}$};
					
					%    \node[font=\scriptsize, right] at (6.15,\y)
					%       {$\frac{\num}{\den}$};
				}
				
				% Indicate convergence of the measures
				\node[font=\scriptsize] at (3,-3.82)
				{$\lambda(F_n)\to 1/2$};
				
			\end{tikzpicture}
			
			\caption{Common independent sets of $N^+$ and $N^-$, with measures tending to $1/2$; see \cref{rem:intersection_sup_not_attained}.}
			\label{fig:halving3}
			
		\end{subfigure}
		
		\caption{The interval-halving matroid: a basis, the need for upward closure in measurable union, and nonattainment in measurable intersection.}
		\label{fig:halving}
		
	\end{figure}

	A particularly important example is obtained on the unit interval. Let $(J,\Sigma,\lambda)$ be $[0,1]$ with Lebesgue measure, set $Z(t)=[0,t]$, and let $b([0,t])=t/2$. Thus,
	\begin{equation}\label{eq:interval_halving}
		\cF\coloneqq\left\{F\in\Sigma\,\middle|\,\lambda(F\cap[0,t])\leq\frac{t}{2}\text{ for every }t\in[0,1]\right\}.
	\end{equation}
	We call the resulting measurable nested matroid the \emph{interval-halving matroid}; see \cref{fig:halving1}. Since $b$ is continuous with respect to convergence in measure on the chain, \cref{cor:nested_rank_minimum} gives
	\begin{equation*}
		r(X)=\min_{t\in[0,1]}\left\{\frac{t}{2}+\lambda(X\setminus[0,t])\right\}.
	\end{equation*}
	In particular, $r([0,t])=t/2$ for every $t\in[0,1]$, and the matroid has rank $1/2$. The following observation will be used several times.
	
	\begin{lem}\label{lem:halving_measure}
		Let $\alpha$ be a finite Borel measure on $[0,1]$. If $\alpha([0,t])=t/2$ for every $t\in[0,1]$, then $\alpha=\lambda/2$. Consequently, there is no Borel set $Y\subseteq[0,1]$ such that $\lambda(Y\cap[0,t])=t/2$ for every $t\in[0,1]$.
	\end{lem}
	
	\begin{proof}
		The first statement follows from the fact that a finite Borel measure on $[0,1]$ is determined by its values on the initial intervals; see, e.g., \cite[Theorem~1.16]{folland1999real}.
		
		For the second statement, suppose that such a set $Y$ exists. Define the finite Borel measure $\alpha$ by $\alpha(C)\coloneqq\lambda(C\cap Y)$. By the first statement, $\alpha=\lambda/2$. Taking $t=1$ gives $\lambda(Y)=1/2$, whereas applying the equality $\alpha=\lambda/2$ to $Y$ gives $\lambda(Y)=\alpha(Y)=\lambda(Y)/2$, a contradiction.
	\end{proof}
	
	%%%%%%%%%%%%%%%%%%%%%%%%%%%%%%%%
	\subsection{Lattice Path Matroids}
	\label{sec:lattice}
	%%%%%%%%%%%%%%%%%%%%%%%%%%%%%%%%
	
	Lattice path matroids were introduced by Bonin, de Mier, and Noy~\cite{bonin2003lattice}; see also Bonin and de Mier~\cite{bonin2006lattice} for their structural properties. Let $P$ and $Q$ be lattice paths from $(0,0)$ to $(n,k)$ using east and north steps, and suppose that $P$ never goes above $Q$. For $i\in\{0,1,\dots,n+k\}$, let $p_i$ and $q_i$ denote the heights of $P$ and $Q$, respectively, after the first $i$ steps.
	
	Every lattice path $L$ between $P$ and $Q$ determines a $k$-element subset of $[n+k]$, namely the set of positions at which $L$ takes a step to the north; see \cref{fig:lattice_path_finite}. These sets are the bases of the corresponding lattice path matroid. Equivalently, the family of bases is
	\[
	\cB\coloneqq\left\{B\subseteq [n+k]\mid p_i\leq \left|B\cap [i]\right|\leq q_i\ \text{for every}\ i\in[n+k]\right\}.
	\]
	Note that $|B|=k$ for every $B\in\cB$ follows from $p_{n+k}=q_{n+k}=k$. Nested matroids correspond to the special case when one of the two bounding paths imposes a trivial constraint.
	
	We first give a measurable version in which constraints are given at an arbitrary collection of points. We work with the ground space $[0,1]$ equipped with the Borel $\sigma$-algebra $\Sigma$ and Lebesgue measure $\lambda$.
	
	\begin{thm}\label{thm:lattice_path_chain}
		Let $T\subseteq[0,1]$ with $\{0,1\}\subseteq T$, and let $0\leq R\leq 1$. Suppose that $a,b\colon T\to[0,R]$ are functions such that $a(0)=b(0)=0$, $a(1)=b(1)=R$, and $a(t)\leq b(t)$ for every $t\in T$. Assume further that $a$ and $b$ are nondecreasing and $1$-Lipschitz, that is, $0\leq a(t)-a(s)\leq t-s$ and $0\leq b(t)-b(s)\leq t-s$ for all $s,t\in T$ with $s\leq t$. Then, if the set
		\begin{equation*}
			\cB_T(a,b)\coloneqq\{B\subseteq [0,1]\mid a(t)\leq \lambda(B\cap[0,t])\leq b(t)\ \text{for every}\ t\in T\}
		\end{equation*}
		is nonempty, then it is the family of bases of a measurable matroid on $[0,1]$.
	\end{thm}
	\begin{proof}
		We verify the basis axioms. Axiom \ref{ax:b1} is precisely the assumption that $\cB_T(a,b)$ is nonempty.
		
		To prove~\ref{ax:b2}, let $B_1,B_2\in\cB_T(a,b)$ and let
		$X\subseteq B_1\setminus B_2$. Set
		$
		V\coloneqq B_1\setminus B_2$ and 
		$W\coloneqq B_2\setminus B_1
		$.
		For $t\in[0,1]$, define
		\[
		v(t)\coloneqq\lambda(V\cap[0,t]),
		\qquad
		w(t)\coloneqq\lambda(W\cap[0,t]).
		\]
		Since $\lambda(B_1)=\lambda(B_2)=R$, we have
		$
		m\coloneqq\lambda(V)=\lambda(W).
		$
		If $m=0$, then $\lambda(X)=0$, and taking $Y\coloneqq\emptyset$
		gives $(B_1\setminus X)\cup Y\in\cB_T(a,b)$. Thus, we may assume
		that $m>0$.
		
		For $x\in[0,m]$, define the generalized inverse functions
		\[
		\psi_V(x)
		\coloneqq
		\inf\{t\in[0,1]\mid v(t)\geq x\},
		\qquad
		\psi_W(x)
		\coloneqq
		\inf\{t\in[0,1]\mid w(t)\geq x\}.
		\]
		Since $v$ and $w$ are continuous, we have
		$
		v(\psi_V(x))=x$ and 
		$w(\psi_W(x))=x
		$
		for every $x\in[0,m]$. In particular, $\psi_V$ and $\psi_W$ are
		strictly increasing Borel maps. Moreover,
		\[
		\lambda\left(\{x\in[0,m]\mid\psi_V(x)\leq t\}\right)
		=
		v(t)
		=
		\lambda(V\cap[0,t])
		\]
		for every $t\in[0,1]$, and analogously for $W$. Since finite Borel measures on $[0,1]$ are determined by their values on the initial intervals~\cite[Theorem~1.16]{folland1999real}, the pushforward of Lebesgue measure on $[0,m]$ under $\psi_V$ is the restriction of Lebesgue measure to $V$, and analogously for $\psi_W$ and $W$. Consequently, the set
		\[
		D
		\coloneqq \psi_V^{-1}(V)\cap \psi_W^{-1}(W)=
		\{x\in[0,m]\mid \psi_V(x)\in V
		\text{ and }\psi_W(x)\in W\}
		\]
		is Borel and has full measure in $[0,m]$. Set
		\[
		V_0\coloneqq\psi_V(D),
		\qquad
		W_0\coloneqq\psi_W(D).
		\]
		By the Lusin--Souslin theorem \cite[Theorem~15.1]{kechris1995classical}, $V_0$ and $W_0$ are Borel, and the
		restrictions of $\psi_V$ and $\psi_W$ to $D$ are Borel
		isomorphisms onto $V_0$ and $W_0$, respectively. Furthermore,
		$V_0$ and $W_0$ have full measure in $V$ and $W$. Consequently,
		\[
		\phi
		\coloneqq
		\psi_W\circ(\psi_V|_D)^{-1}
		\colon V_0\to W_0
		\]
		is an order-preserving, measure-preserving Borel bijection.
		
		Set
		\[
		X_0\coloneqq X\cap V_0,
		\qquad
		Y\coloneqq\phi(X_0),
		\qquad
		B'\coloneqq(B_1\setminus X)\cup Y.
		\]
		Then $Y\subseteq B_2\setminus B_1$, and
		$
		\lambda(Y)=\lambda(X_0)=\lambda(X).
		$
		For $t\in[0,1]$, define
		\[
		x(t)\coloneqq\lambda(X\cap[0,t]),
		\qquad
		y(t)\coloneqq\lambda(Y\cap[0,t]), \qquad 
		Z
		\coloneqq
		(\psi_V|_D)^{-1}(X_0)
		=
		(\psi_W|_D)^{-1}(Y).
		\]
		Then
		$
		x(t)=\lambda(Z\cap[0,v(t)])$ and
		$y(t)=\lambda(Z\cap[0,w(t)]).$
		Comparing the two initial intervals gives
		\[
		\min\{0,w(t)-v(t)\}
		\leq
		y(t)-x(t)
		\leq
		\max\{0,w(t)-v(t)\}.
		\]
		
		For $i\in\{1,2\}$, let $h_i(t)\coloneqq \lambda(B_i\cap[0,t])$. As $h_2(t)-h_1(t)=w(t)-v(t)$ and $\lambda(B'\cap[0,t])-h_1(t)=y(t)-x(t)$, we get
		\[
		\min\{h_1(t),h_2(t)\}\leq\lambda(B'\cap[0,t])\leq\max\{h_1(t),h_2(t)\}
		\]
		for every $t\in[0,1]$. Therefore,
		\[
		a(t)\leq\lambda(B'\cap[0,t])\leq b(t)
		\]
		for every $t\in T$. Moreover, $\lambda(B')=R$, yielding $B'\in\cB_T(a,b)$. 
		
		Finally, let $(B_i)_{i\in\bN}$ be a sequence in $\cB_T(a,b)$ and suppose that $\lim_{i\to\infty} \lambda(B_i\triangle B)=0$ for some $B\subseteq[0,1]$. Then $\lambda(B)=R$, and
		\[  
		|\lambda(B_i\cap[0,t])-\lambda(B\cap[0,t])|\leq\lambda(B_i\triangle B)
		\]
		for every $t\in[0,1]$. This means that $\lim_{i\to\infty}\lambda(B_i\cap[0,t])=\lambda(B\cap[0,t])$ for every $t\in[0,1]$, implying
		\[
		a(t)\leq\lambda(B\cap[0,t])\leq b(t)
		\]
		for every $t\in T$. Thus, $B\in\cB_T(a,b)$, proving \ref{ax:b3}.
	\end{proof}
	
	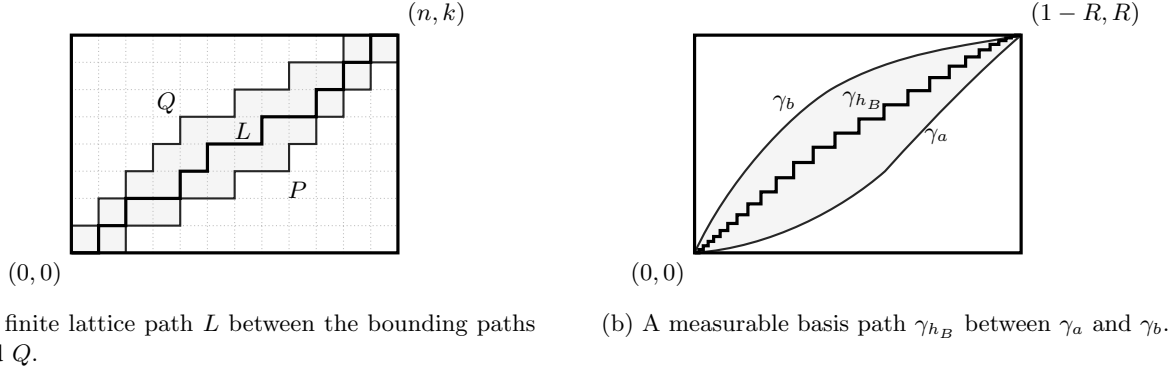
\begin{figure}[t]
		
		\begin{subfigure}[t]{0.48\textwidth}
			\centering
			
			\begin{tikzpicture}[
				x=0.72cm,
				y=0.72cm,
				curve/.style={thick, draw=black!85},
				axis parallel/.style={very thick, draw=black}
				]
				
				% Region between the bounding paths
				\fill[gray!8]
				(0,0) --
				(0,0.5) --
				(0.5,0.5) --
				(0.5,1) --
				(1,1) --
				(1,1.5) --
				(1.5,1.5) --
				(1.5,2) --
				(2,2) --
				(2,2.5) --
				(2.5,2.5) --
				(3,2.5) --
				(3,3) --
				(3.5,3) --
				(4,3) --
				(4,3.5) --
				(4.5,3.5) --
				(5,3.5) --
				(5,4) --
				(5.5,4) --
				(6,4) --
				(6,3.5) --
				(5.5,3.5) --
				(5.5,3) --
				(5,3) --
				(5,2.5) --
				(4.5,2.5) --
				(4.5,2) --
				(4,2) --
				(4,1.5) --
				(3.5,1.5) --
				(3,1.5) --
				(3,1) --
				(2.5,1) --
				(2,1) --
				(2,0.5) --
				(1.5,0.5) --
				(1,0.5) --
				(1,0) --
				(0.5,0) --
				cycle;
				
				% Fine dotted grid
				\draw[gray!50, densely dotted, line width=0.4pt, step=0.5]
				(0,0) grid (6,4);
				
				% Lower bounding path P
				\draw[curve]
				(0,0) --
				(0.5,0) --
				(1,0) --
				(1,0.5) --
				(1.5,0.5) --
				(2,0.5) --
				(2,1) --
				(2.5,1) --
				(3,1) --
				(3,1.5) --
				(3.5,1.5) --
				(4,1.5) --
				(4,2) --
				(4.5,2) --
				(4.5,2.5) --
				(5,2.5) --
				(5,3) --
				(5.5,3) --
				(5.5,3.5) --
				(6,3.5) --
				(6,4);
				
				% Upper bounding path Q
				\draw[curve]
				(0,0) --
				(0,0.5) --
				(0.5,0.5) --
				(0.5,1) --
				(1,1) --
				(1,1.5) --
				(1.5,1.5) --
				(1.5,2) --
				(2,2) --
				(2,2.5) --
				(2.5,2.5) --
				(3,2.5) --
				(3,3) --
				(3.5,3) --
				(4,3) --
				(4,3.5) --
				(4.5,3.5) --
				(5,3.5) --
				(5,4) --
				(5.5,4) --
				(6,4);
				
				% A basis path L
				\draw[axis parallel]
				(0,0) --
				(0.5,0) --
				(0.5,0.5) --
				(1,0.5) --
				(1,1) --
				(1.5,1) --
				(2,1) --
				(2,1.5) --
				(2.5,1.5) --
				(2.5,2) --
				(3,2) --
				(3.5,2) --
				(3.5,2.5) --
				(4,2.5) --
				(4.5,2.5) --
				(4.5,3) --
				(5,3) --
				(5,3.5) --
				(5.5,3.5) --
				(5.5,4) --
				(6,4);
				
				% Outer rectangle
				\draw[very thick] (0,0) rectangle (6,4);
				
				% Labels
				\node[font=\small, below left] at (0,0) {$(0,0)$};
				\node[font=\small, above right] at (6,4) {$(n,k)$};
				
				\node[font=\small] at (1.75,2.75) {$Q$};
				\node[font=\small] at (4.15,1.15) {$P$};
				\node[font=\small\bfseries] at (3.15,2.25) {$L$};
				
			\end{tikzpicture}
			
			\caption{A finite lattice path $L$ between the bounding paths $P$ and $Q$.}
			\label{fig:lattice_path_finite}
			
		\end{subfigure}
		\hfill
		\begin{subfigure}[t]{0.48\textwidth}
			\centering
			
			\begin{tikzpicture}[
				x=0.72cm,
				y=0.72cm,
				curve/.style={thick, draw=black!85},
				axis parallel/.style={very thick, draw=black}
				]
				
				% Region between the bounding paths
				\fill[gray!8]
				(0,0)
				.. controls (0.6,1.3) and (1.7,2.5) .. (2.5,3.0)
				.. controls (3.7,3.7) and (5.0,3.8) .. (6,4)
				.. controls (5.3,3.4) and (4.4,2.5) .. (3.5,1.5)
				.. controls (2.7,0.8) and (1.4,0.1) .. (0,0)
				-- cycle;
				
				% Upper bounding path gamma_b
				\draw[curve]
				(0,0)
				.. controls (0.6,1.3) and (1.7,2.5) .. (2.5,3.0)
				.. controls (3.7,3.7) and (5.0,3.8) .. (6,4);
				
				% Lower bounding path gamma_a
				\draw[curve]
				(0,0)
				.. controls (1.4,0.1) and (2.7,0.8) .. (3.5,1.5)
				.. controls (4.4,2.5) and (5.3,3.4) .. (6,4);
				
				\draw[axis parallel]
				(0,0) --
				(0.08,0) --
				(0.08,0.06) --
				(0.16,0.06) --
				(0.16,0.13) --
				(0.25,0.13) --
				(0.25,0.21) --
				(0.35,0.21) --
				(0.35,0.30) --
				(0.47,0.30) --
				(0.47,0.41) --
				(0.61,0.41) --
				(0.61,0.54) --
				(0.78,0.54) --
				(0.78,0.70) --
				(0.98,0.70) --
				(0.98,0.90) --
				(1.22,0.90) --
				(1.22,1.12) --
				(1.50,1.12) --
				(1.50,1.38) --
				(1.82,1.38) --
				(1.82,1.66) --
				(2.18,1.66) --
				(2.18,1.94) --
				(2.58,1.94) --
				(2.58,2.20) --
				(3.02,2.20) --
				(3.02,2.46) --
				(3.48,2.46) --
				(3.48,2.72) --
				(3.92,2.72) --
				(3.92,2.98) --
				(4.32,2.98) --
				(4.32,3.22) --
				(4.68,3.22) --
				(4.68,3.43) --
				(4.98,3.43) --
				(4.98,3.60) --
				(5.23,3.60) --
				(5.23,3.73) --
				(5.43,3.73) --
				(5.43,3.83) --
				(5.59,3.83) --
				(5.59,3.90) --
				(5.72,3.90) --
				(5.72,3.95) --
				(5.82,3.95) --
				(5.82,3.98) --
				(5.90,3.98) --
				(5.90,4.00) --
				(6,4);
				
				% Outer rectangle
				\draw[very thick] (0,0) rectangle (6,4);
				
				% Labels
				\node[font=\small, below left] at (0,0) {$(0,0)$};
				\node[font=\small, above right] at (6,4) {$(1-R,R)$};
				
				\node[font=\small] at (1.65,2.75) {$\gamma_b$};
				\node[font=\small] at (4.45,2.15) {$\gamma_a$};
				\node[font=\small\bfseries] at (3.1,2.80) {$\gamma_{h_B}$};
				
			\end{tikzpicture}
			
			\caption{A measurable basis path $\gamma_{h_B}$ between $\gamma_a$ and $\gamma_b$.}
			\label{fig:lattice_path_measurable}
			
		\end{subfigure}
		
		\caption{Finite and measurable lattice path matroids.}
		\label{fig:gamma}
		
	\end{figure}
	
	We call the matroids obtained from \cref{thm:lattice_path_chain} \emph{measurable lattice path matroids}. Taking $T=[0,1]$ gives the continuous version of the construction.
	
	\begin{cor}\label{cor:lattice_path}
		Let $0\leq R\leq 1$, and let $a,b\colon[0,1]\to[0,R]$ be nondecreasing $1$-Lipschitz functions such that $a(0)=b(0)=0$, $a(1)=b(1)=R$, and $a(t)\leq b(t)$ for every $t\in[0,1]$. Then, if the set
		\begin{equation*}
			\cB(a,b)\coloneqq\{B\subseteq [0,1]\mid a(t)\leq \lambda(B\cap[0,t])\leq b(t)\ \text{for every}\ t\in [0,1]\}
		\end{equation*}    
		is nonempty, then it is the family of bases of a measurable matroid on $[0,1]$.
	\end{cor}
	
	Now we explain the path interpretation of the construction.
	For an absolutely continuous function $h\colon[0,1]\to[0,R]$, consider the path $\gamma_h(t)\coloneqq(t-h(t),h(t))$. Since $h$ is differentiable almost everywhere, we may define its \emph{direction function} by $\dir_h(t)\coloneqq h'(t)$ almost everywhere. Then
	\begin{equation*}
		\gamma_h'(t)=\left(1-\dir_h(t),\dir_h(t)\right)
	\end{equation*}
	almost everywhere. Thus, the values $0$ and $1$ correspond to moving eastward and northward, respectively, while intermediate values correspond to fractional directions.
	
	For $B\in\cB(a,b)$, let $h_B(t)\coloneqq\lambda(B\cap[0,t])$. Since $h_B(t)=\int_0^t\mathbf{1}_B(s)\diff s$, we have $\dir_{h_B}=\mathbf{1}_B$ almost everywhere. Hence, $\gamma_{h_B}$ moves north on $B$ and east outside $B$. Moreover, the inequalities $a(t)\leq h_B(t)\leq b(t)$ say exactly that $\gamma_{h_B}$ remains between the paths $\gamma_a(t)\coloneqq(t-a(t),a(t))$ and $\gamma_b(t)\coloneqq(t-b(t),b(t))$; see \cref{fig:lattice_path_measurable}.
	
	Conversely, suppose that $h\colon[0,1]\to[0,R]$ is absolutely continuous, satisfies $h(0)=0$, $h(1)=R$, and $a(t)\leq h(t)\leq b(t)$ for every $t\in[0,1]$, and has $\dir_h(t)\in\{0,1\}$ almost everywhere. Let \[B\coloneqq\{t\in[0,1]\mid\dir_h(t)=1\}.\] By the fundamental theorem of calculus for absolutely continuous functions~\cite[Theorem~3.35]{folland1999real},
	\begin{equation*}
		h(t)=\int_0^t\dir_h(s)\diff s=\lambda(B\cap[0,t]).
	\end{equation*}
	Therefore, $B\in\cB(a,b)$. Thus, the bases correspond exactly to the paths between the two bounding paths whose direction is north or east almost everywhere. 
	
	\begin{rem}
		The construction contains the interval versions of measurable nested matroids, discussed in \cref{sec:nested}. For a set $B\subseteq[0,1]$ of measure $R$, the inequality $\max\{0,t+R-1\}\leq \lambda(B\cap[0,t])$ holds automatically. Thus, taking $a(t)=\max\{0,t+R-1\}$ leaves only the upper-bound constraints. The general closed-chain construction of \cref{sec:nested}, however, is not restricted to chains induced by the linear order on $[0,1]$. In particular, the interval-halving matroid from \cref{sec:nested} is a measurable lattice path matroid. Indeed, take $R=1/2$, $a(t)=\max\{0,t-1/2\}$ and $b(t)=t/2$. Then, for a set $B$ of measure $1/2$, the lower bound automatically holds, while the upper bound is precisely the condition in the definition of the interval-halving matroid. 
	\end{rem}

	\begin{rem} \label{rem:lattice_nonempty}
		The nonemptiness assumption in \cref{thm:lattice_path_chain} and \cref{cor:lattice_path} cannot be omitted. For example, let $R=1/2$ and set $a(t)=b(t)=t/2$ for every $t\in[0,1]$. A basis would have to satisfy $\lambda(B\cap[0,t])=t/2$ for every $t\in[0,1]$, but no such Borel set exists by \cref{lem:halving_measure}. 
		However, if $a(t)<b(t)$ for every $t\in(0,1)$, then $\cB(a,b)$ is nonempty. We omit the proof, as it is technical and will not be used later.
	\end{rem}
	
	%%%%%%%%%%%%%%%%%%%%%%%%%%%%%%%%
	\subsection{Transversal Matroids}
	\label{sec:transversal}
	%%%%%%%%%%%%%%%%%%%%%%%%%%%%%%%%
	
	Transversal matroids form another basic class of finite matroids. If $G=(S,T;E)$ is a finite bipartite graph, then the corresponding transversal matroid $M=(S,\cI)$ has independent sets
	\[
	\cI=\{X\subseteq S\mid X\text{ can be covered by a matching}\},
	\]
	see~\cite{oxley2011matroid}. We use the analogous notion for bipartite graphings. In the measurable setting, however, one has to distinguish between sets covered by an actual measurable matching and increasing limits of such sets.
	
	A \emph{measurable matching} in a graphing $G=(J,\Sigma_J,\nu,E)$ is a Borel set $N\subseteq E$ such that every vertex is incident to at most one edge of $N$. We denote by $V(N)$ the set of vertices incident to an edge of $N$. For $X\subseteq J$, we say that $N$ \emph{covers} $X$ if $\nu(X\setminus V(N))=0$.
	
	We will use the following standard facts about bounded-degree Borel graphs and graphings.
	
	\begin{lem}\label{lem:graphing_matching_tools}
		Let $G=(J,\Sigma_J,\nu,E)$ be a graphing.
		\begin{enumerate}[label=(\roman*), itemsep=0em]
			\item If $H$ is a bounded-degree Borel graph on $J$ and $A\subseteq J$ is Borel, then the union of the connected components of $H$ meeting $A$ is Borel. \label{it:gri}
			\item If $N\subseteq E$ is a measurable matching, then $V(N)$ is Borel, and the partner map $\tau_N\colon V(N)\to V(N)$, which sends each vertex to its unique $N$-neighbour, is a measure-preserving Borel involution. \label{it:grii}
		\end{enumerate}
	\end{lem}
	
	\begin{proof}
		For \ref{it:gri}, the neighborhood of every Borel set in a bounded-degree Borel graph is Borel by \cite[Theorem~18.2]{lovasz2012large}. Starting with $A$ and iterating the neighborhood operation, the union of the resulting sets is precisely the union of the connected components of $H$ meeting $A$. Hence, this union is Borel.
		
		For \ref{it:grii}, the same result, applied to the Borel graph with edge set $N$, shows that $V(N)$ is Borel. Since the graph of $\tau_N$ is the edge relation of $N$, which is Borel, the map $\tau_N$ is Borel by \cite[Theorem~14.12]{kechris1995classical}, and it is clearly an involution. Moreover, the Borel subgraph of $G$ with edge set $N$ is a graphing by \cite[Lemma~18.19]{lovasz2012large}. Applying its graphing identity to $A$ and $\tau_N(A)$ for any Borel set $A\subseteq V(N)$ gives
		$
		\nu(A)=\nu(\tau_N(A)).
		$
		Thus, $\tau_N$ is measure-preserving.
	\end{proof}
	
	The following lemma is a measurable version of the standard alternating-path exchange argument for finite matchings. The same componentwise switching argument works here, since \cref{lem:graphing_matching_tools} ensures that the relevant sets are Borel and that the partner maps preserve measure.
	
	\begin{lem}\label{lem:measurable_matching_exchange}
		Let $G=(J,\Sigma_J,\nu,E)$ be a graphing. Let $N_1,N_2$ be measurable matchings, and let $X_i\in\Sigma_J$ be covered by $N_i$ for $i\in\{1,2\}$. If $\nu(X_1)\leq\nu(X_2)$, then there exists a measurable matching covering a Borel set $X$ such that
		\[
		X_1\subseteq X\subseteq X_1\cup X_2
		\quad\text{and}\quad
		\nu(X)\geq\nu(X_2).
		\]
	\end{lem}
	
	\begin{proof}
		Replacing $X_i$ by $X_i\cap V(N_i)$ changes it only by a null set, so we may assume that $X_i\subseteq V(N_i)$ holds pointwise for $i\in\{1,2\}$. By deleting unnecessary edges, we may also assume that every edge of $N_i$ is incident to a vertex of $X_i$.
		
		Set
		\[
		A\coloneqq X_2\setminus X_1,\qquad
		A_0\coloneqq X_2\setminus V(N_1),\qquad
		B_0\coloneqq X_1\setminus V(N_2).
		\]
		Thus, $A_0\subseteq A$ and $B_0\subseteq X_1\setminus X_2$. Consider the Borel graph $H$ with edge set $N_1\triangle N_2$. Every nontrivial component of $H$ is a finite path, a ray, a double ray, or a cycle, with edges alternating between $N_1$ and $N_2$. Here, a ray is a one-way infinite path, while a double ray is a two-way infinite path.
		
		Let $W$ be the union of those components of $H$ that meet $A_0$ but do not meet $B_0$. Equivalently, $W$ is the difference between the union of the $H$-components meeting $A_0$ and the union of the $H$-components meeting $B_0$. Hence, $W$ is Borel by \cref{lem:graphing_matching_tools}\ref{it:gri}. Define a new matching $N$ by switching from $N_1$ to $N_2$ on the components contained in $W$, that is, set
		\[
		N
		\coloneqq
		\left(N_1\setminus E(H[W])\right)
		\cup
		\left(N_2\cap E(H[W])\right).
		\]
		Since $W$ is a union of connected components of $H$, this componentwise switch preserves the matching property. 
		
		\begin{cla}\label{cla:first}
			Every vertex of $X_1$ is covered by $N$. 
		\end{cla}
		\begin{claimproof}
			If $x\in X_1\setminus W$, then the $N_1$-edge covering $x$ was not removed. If $x\in X_1\cap W$, then $x\notin B_0$, since no component contained in $W$ meets $B_0$. Hence, $x\in V(N_2)$, and after the switch $x$ remains covered.
		\end{claimproof}
		
		We next bound the measure of the vertices covered in $X_2\setminus X_1$. The only vertices of $A=X_2\setminus X_1$ that may fail to be covered by $N$ are those in $A_0\setminus W$. 
		
		\begin{cla}\label{cla:second}
			$\nu(A_0\setminus W)\leq\nu(B_0)$.
		\end{cla}
		
		\begin{claimproof}
			Let $C$ be an $H$-component meeting $A_0\setminus W$. Since $C$ is not contained in $W$, it also meets $B_0$. A vertex of $A_0$ is covered by $N_2$ but not by $N_1$, and hence is an endpoint of its alternating component, incident to an $N_2$-edge. Similarly, a vertex of $B_0$ is covered by $N_1$ but not by $N_2$, and hence is an endpoint incident to an $N_1$-edge. Since a path has at most two endpoints, every component meeting $A_0\setminus W$ is a finite alternating path with exactly one endpoint in $A_0$ and exactly one endpoint in $B_0$.
			
			Define $\theta\colon A_0\setminus W\to B_0$ by mapping each vertex of $A_0\setminus W$ to the other endpoint of its alternating path. For each $m\in\bN$, the set of vertices for which this path has length $m$ is Borel, and on this set $\theta$ is a composition of restrictions of the partner maps of $N_1$ and $N_2$. Thus, $\theta$ is a Borel injection. Moreover, it is measure-preserving on each such set by \cref{lem:graphing_matching_tools}\ref{it:grii}. Partitioning any Borel set $D\subseteq A_0\setminus W$ according to the length of the corresponding alternating path and using countable additivity gives $\nu(\theta(D))=\nu(D)$. Therefore, $\nu(A_0\setminus W)\leq\nu(B_0)$.
		\end{claimproof}
		
		Let
		\[
		X\coloneqq(X_1\cup X_2)\cap V(N).
		\]
		Then $X$ is covered by $N$. \Cref{cla:first} gives $X_1\subseteq X\subseteq X_1\cup X_2$. Moreover, by \cref{cla:second} and the inclusion $B_0\subseteq X_1\setminus X_2$,
		\[
		\nu(X)
		\geq
		\nu(X_1)+\nu(A)-\nu(A_0\setminus W)
		\geq
		\nu(X_1)+\nu(A)-\nu(B_0)
		\geq
		\nu(X_1)+\nu(A)-\nu(X_1\setminus X_2)
		=
		\nu(X_2).
		\]
		This concludes the proof of the lemma.
	\end{proof}
	
	Let $G=(J,\Sigma_J,\nu,E)$ be a bipartite graphing with Borel bipartition $J=S\cup T$, and let $\Sigma_S$ denote the Borel $\sigma$-algebra induced on $S$. Since $S$ is Borel, $(S,\Sigma_S,\nu|S)$ is a standard measure space. We say that $X\in\Sigma_S$ is \emph{transversal} if there exists a measurable matching covering $X$, and denote the family of transversal sets by
	\begin{equation}\label{eq:transversal_prematroid}
		\cT_G\coloneqq\{X\in\Sigma_S\mid\text{there exists a measurable matching covering }X\}.
	\end{equation}
	
	Now we show that transversal sets define a measurable pre-matroid.
	
	\begin{prop}\label{prop:transversal_strong_prematroid}
		The system $(S,\Sigma_S,\nu|S,\cT_G)$ is a measurable pre-matroid.
	\end{prop}
	
	\begin{proof}
		Axioms \ref{ax:i1} and \ref{ax:i2} are immediate. To prove \ref{ax:i3''}, let $X_1,X_2\in\cT_G$ with $\nu(X_1)\leq\nu(X_2)$. Choose measurable matchings $N_1$ and $N_2$ covering $X_1$ and $X_2$, respectively. After replacing $X_i$ by $X_i\cap V(N_i)$ and restricting $N_i$ to the edges whose $S$-endpoint belongs to $X_i$, we may assume that $V(N_i)\cap S=X_i$ for $i\in\{1,2\}$. The result now follows from \cref{lem:measurable_matching_exchange}.
	\end{proof}
	
	\cref{prop:transversal_strong_prematroid} gives the measurable analogue of the finite exchange property. However, $\cT_G$ need not be closed under increasing unions, as shown by the following example. 
	
	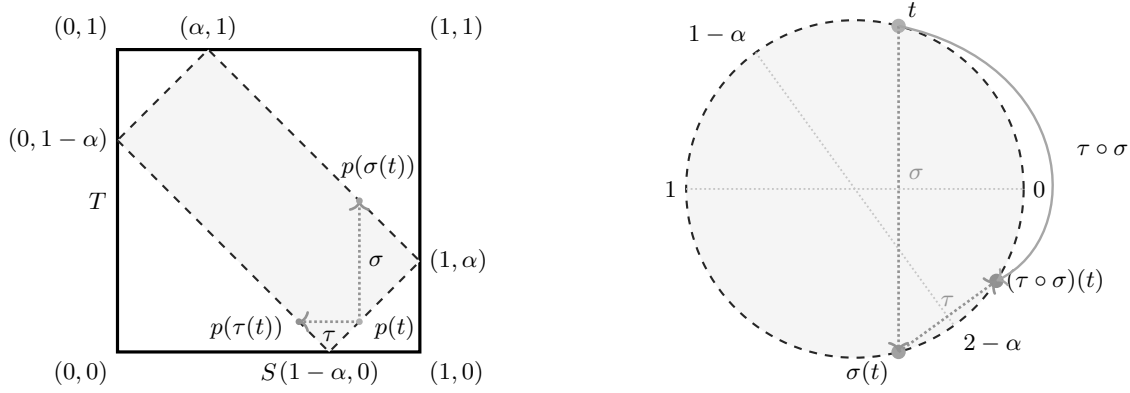
\begin{figure}[t]
		\begin{subfigure}{0.48\textwidth}
			\centering
			
			\begin{tikzpicture}[scale=0.4]
				
				% Ground square
				\draw[very thick] (0,0) rectangle (10,10);
				
				% Coordinate copies
				\node[label A] at (5,-0.65) {\small $S$};
				\node[label A] at (-0.65,5) {\small $T$};
				
				% Corners
				\node[label A,below left] at (0,0) {\small $(0,0)$};
				\node[label A,below right] at (10,0) {\small $(1,0)$};
				\node[label A,above left] at (0,10) {\small $(0,1)$};
				\node[label A,above right] at (10,10) {\small $(1,1)$};
				
				% Boundary of the tilted rectangle
				% Schematic choice alpha = 0.3
				\fill[gray!8]
				(0,7) --
				(7,0) --
				(10,3) --
				(3,10) --
				cycle;
				
				\draw[set boundary]
				(0,7) --
				(7,0) --
				(10,3) --
				(3,10) --
				cycle;
				
				% Vertices
				\node[label A,left] at (0,7) {\small $(0,1-\alpha)$};
				\node[label A,below] at (7,0) {\small $(1-\alpha,0)$};
				\node[label A,right] at (10,3) {\small $(1,\alpha)$};
				\node[label A,above] at (3,10) {\small $(\alpha,1)$};
				
				% A generic point and its two partners
				\coordinate (p) at (8,1);
				\coordinate (ps) at (8,5);
				\coordinate (pt) at (6,1);
				
				% Coordinate fibres
				\draw[gray!85, densely dotted, line width=1pt, ->]
				(p) -- (ps);
				\draw[gray!85, densely dotted, line width=1pt, ->]
				(p) -- (pt);
				
				% Points
				\fill[f1color] (p) circle (0.11);
				\fill[f2color] (ps) circle (0.11);
				\fill[f3color] (pt) circle (0.11);
				
				% Labels
				\node[font=\small, below right] at (8.2,1.5) {$p(t)$};
				\node[font=\small, above right] at (7.1,5.4) {$p(\sigma(t))$};
				\node[font=\small, below left] at (5.8,1.5) {$p(\tau(t))$};
				
				% Map labels
				\node[label A, right] at (8,3.0) {\small $\sigma$};
				\node[label A, below] at (7.0,1) {\small $\tau$};
				
			\end{tikzpicture}
			
			\caption{Dotted vertical and horizontal segments illustrate the involutions $\sigma$ and $\tau$: $p(\sigma(t))$ has the same first coordinate as $p(t)$, while $p(\tau(t))$ has the same second coordinate.}
			\label{fig:laczkovich1}
		\end{subfigure}
		\hfill
		\begin{subfigure}{0.48\textwidth}
			\centering
			
			\begin{tikzpicture}[scale=0.95]
				
				\def\rad{2.35}
				
				% Parameter circle
				\fill[gray!8] (0,0) circle (\rad);
				\draw[set boundary] (0,0) circle (\rad);
				
				% Reflection axes
				% The placement is schematic, corresponding to alpha approximately 0.3.
				\draw[gray!45, densely dotted, line width=0.7pt]
				(180:\rad) -- (0:\rad);
				
				\draw[gray!45, densely dotted, line width=0.7pt]
				(126:\rad) -- (306:\rad);
				
				% Fixed points of the two involutions
				\node[font=\small, right] at (0:\rad) {$0$};
				\node[font=\small, left] at (180:\rad) {$1$};
				
				\node[font=\small, above left] at (126:\rad) {$1-\alpha$};
				\node[font=\small, below right] at (306:\rad) {$2-\alpha$};
				
				% A generic point and its successive images
				\coordinate (t) at (75:\rad);
				\coordinate (st) at (285:\rad);
				\coordinate (tst) at (327:\rad);
				
				\fill[f1color] (t) circle (0.10);
				\fill[f2color] (st) circle (0.10);
				\fill[f3color] (tst) circle (0.10);
				
				\node[font=\small, above right] at (t) {$t$};
				\node[font=\small, below left] at (st) {$\sigma(t)$};
				\node[font=\small, right] at (tst) {$(\tau\circ\sigma)(t)$};
				
				% The two reflections
				\draw[gray!85, densely dotted, line width=1pt, ->]
				(t) -- node[midway, above right, font=\small] {$\sigma$} (st);
				
				\draw[gray!85, densely dotted, line width=1pt, ->]
				(st) -- node[midway, above, font=\small] {$\tau$} (tst);
				
				% Their composition
				\draw[gray!70, line width=0.9pt, ->]
				(t) .. controls (3.05,1.75) and (3.25,-0.65) .. (tst);
				
				\node[font=\small] at (3.45,0.55) {$\tau\circ\sigma$};
				
			\end{tikzpicture}
			
			\caption{Under the parametrization $p\colon\bR/2\bZ\to C$, the involutions $\sigma$ and $\tau$ are reflections, and their composition is the irrational rotation $\tau\circ\sigma(t)=t+2(1-\alpha)$.}
			\label{fig:laczkovich2}
		\end{subfigure}
		\caption{Illustration of Laczkovich's construction.}
		\label{fig:laczkovich}
	\end{figure}
	
	\begin{rem}\label{rem:laczkovich_matching}
		Consider the following example of Laczkovich~\cite{laczkovich1988closed}. Fix an irrational number $\alpha\in(0,1)$, let $R\subseteq[0,1]^2$ be the rectangle with vertices $(0,1-\alpha)$, $(1-\alpha,0)$, $(1,\alpha)$, $(\alpha,1)$, and let $C$ denote its boundary. Let $S$ and $T$ be disjoint copies of $[0,1]$, and join $x\in S$ to $y\in T$ whenever $(x,y)\in C$. Equip both copies with the Lebesgue measure. This defines a bipartite graphing $G$ that is $2$-regular apart from finitely many vertices.
		
		We claim that $G$ has no measurable perfect matching. Let $p\colon\bR/2\bZ\to C$ be the parametrization obtained by traversing $C$ counterclockwise at constant speed $\sqrt{2}$, starting from $p(0)=(0,1-\alpha)$, where $\bR/2\bZ$ is regarded as the circle of length $2$, equipped with its usual Lebesgue measure. A point $p(t)=(x,y)$ represents the edge joining $x\in S$ and $y\in T$. Define $\sigma(t)=-t$ and $\tau(t)=2(1-\alpha)-t$ in $\bR/2\bZ$. For all but finitely many $t$, the points $p(\sigma(t))$ and $p(\tau(t))$ are the other points of $C$ having the same first and second coordinates as $p(t)$, respectively; see \cref{fig:laczkovich1}.
		
		Suppose that $N$ is a measurable perfect matching, and let $A\subseteq\bR/2\bZ$ be the set of parameters corresponding to the edges of $N$. For almost every $t$, exactly one of $p(t)$ and $p(\sigma(t))$ belongs to $N$, and exactly one of $p(t)$ and $p(\tau(t))$ belongs to $N$. Hence, $\sigma(A)$ and $\tau(A)$ both differ from $(\bR/2\bZ)\setminus A$ by a null set. It follows that $A$ is invariant up to a null set under the rotation
		\[
		\tau\circ\sigma(t)=t+2(1-\alpha),
		\]
		see \cref{fig:laczkovich2}. Since $\alpha$ is irrational, every measurable set invariant up to a null set under this rotation has either zero or full measure; see~\cite{laczkovich1988closed}. However, $\sigma(A)$ differs from the complement of $A$ by a null set, and $\sigma$ preserves measure. Thus, $A$ and its complement have the same measure, a contradiction.
		
		As an easy consequence of a result of Elek and Lippner~\cite[Proposition~1.1]{elek2009borel}, for every $\varepsilon>0$ there is a measurable matching covering a subset $X\subseteq S$ such that $\nu(S\setminus X)<\varepsilon$. If $\cT_G$ were $\sigma$-increasing, then \cref{cor:nem_novekvo_unio} would give $S\in\cT_G$. Let $N$ be a measurable matching covering $S$. Since $\nu(S)=\nu(T)$ and the partner map $\tau_N$ is measure-preserving by \cref{lem:graphing_matching_tools}\ref{it:grii}, the matching $N$ also covers $T$. This would give a measurable perfect matching, a contradiction.
	\end{rem}
	
	We therefore define the measurable transversal matroid by taking the upward closure of $\cT_G$:
	\begin{equation}\label{eq:transversal_matroid}
		\overline{\cT}_G
		\coloneqq
		\left\{
		\bigcup_{n=1}^{\infty}X_n
		\,\middle|\,
		(X_n)_{n\in\bN}\text{ is increasing in }\cT_G
		\right\}.
	\end{equation}
	
	\begin{cor}\label{cor:transversal_matroid}
		The system $(S,\Sigma_S,\nu|S,\overline{\cT}_G)$ 
		is a measurable matroid.
	\end{cor}
	
	\begin{proof}
		By \cref{prop:transversal_strong_prematroid}, the system $(S,\Sigma_S,\nu|S,\cT_G)$ is a measurable pre-matroid. The result follows from \cref{thm:upward_closure_prematroid}.
	\end{proof}
	
	We call $(S,\Sigma_S,\nu|S,\overline{\cT}_G)$ the \emph{measurable transversal matroid} of the bipartite graphing $G$.
	
	%%%%%%%%%%%%%%%%
	\subsection{Matching Matroids}
	\label{sec:matching}
	%%%%%%%%%%%%%%%%
	
	We next recall the definition of a finite matching matroid. Let $G=(V,E)$ be a finite graph. A set $X\subseteq V$ is called \emph{matchable} if there exists a matching of $G$ covering every vertex of $X$. The family
	\[
	\cI_G
	\coloneqq
	\{X\subseteq V\mid X\text{ is matchable in }G\}
	\]
	is the family of independent sets of a matroid on $V$, called the \emph{matching matroid} of $G$; see~\cite{frank2011connections}. If $G$ is bipartite with bipartition $V=S\cup T$, then the restriction of this matroid to $S$ is the transversal matroid associated with $G$.
	
	We now make the analogous definition for graphings. Let $G=(J,\Sigma_J,\nu,E)$ be a graphing. We say that $X\in\Sigma_J$ is \emph{matchable} if there exists a measurable matching covering $X$. Let
	\[
	\cM_G
	\coloneqq
	\{X\in\Sigma_J\mid X\text{ is matchable in }G\}.
	\]
	
	\begin{prop}\label{prop:matching_prematroid}
		The system $(J,\Sigma_J,\nu,\cM_G)$ is a measurable pre-matroid.
	\end{prop}
	
	\begin{proof}
		Axioms \ref{ax:i1} and \ref{ax:i2} are immediate. To prove \ref{ax:i3''}, let $X_1,X_2\in\cM_G$ with $\nu(X_1)\leq\nu(X_2)$. Choose measurable matchings covering $X_1$ and $X_2$, respectively. The result follows from \cref{lem:measurable_matching_exchange}.
	\end{proof}
	
	We call $(J,\Sigma_J,\nu,\cM_G)$ provided by \cref{prop:matching_prematroid} the \emph{matching pre-matroid} of $G$. 
	
	\begin{rem}\label{rem:gen_matching}
		As for transversal matroids, the family $\cM_G$ need not be closed under increasing unions. Indeed, consider the graphing $G$ from \cref{rem:laczkovich_matching}. For every $\varepsilon>0$, there is a measurable matching $N$ covering a set $X\subseteq S$ such that $\nu(S\setminus X)<\varepsilon/2$. By replacing $X$ with $V(N)\cap S$, we may assume that $X=V(N)\cap S$. Since $\tau_N$ is measure-preserving and $\nu(S)=\nu(T)$, we have
		\[
		\nu(J\setminus V(N))=2\nu(S\setminus X)<\varepsilon.
		\]
		If $\cM_G$ were $\sigma$-increasing, then \cref{cor:nem_novekvo_unio} would give $J\in\cM_G$, contradicting the fact that $G$ has no measurable perfect matching.
	\end{rem}
	
	We therefore define the measurable matching matroid of $G$ by taking the upward closure
	\[
	\overline{\cM}_G
	\coloneqq
	\left\{
	\bigcup_{i=1}^{\infty}X_i
	\,\middle|\,
	(X_i)_{i\in\bN}\text{ is increasing in }\cM_G
	\right\}.
	\]
	
	\begin{cor}\label{cor:matching_matroid}
		The system $(J,\Sigma_J,\nu,\overline{\cM}_G)$ is a measurable matroid.
	\end{cor}
	
	\begin{proof}
		By \cref{prop:matching_prematroid}, the system $(J,\Sigma_J,\nu,\cM_G)$ is a measurable pre-matroid. The result follows from \cref{thm:upward_closure_prematroid}.
	\end{proof}
	
	We call $(J,\Sigma_J,\nu,\overline{\cM}_G)$ the \emph{measurable matching matroid} of $G$. If $G$ is bipartite with bipartition $J=S\cup T$, then the restrictions of its matching pre-matroid to $S$ and $T$ are precisely the corresponding transversal pre-matroids. The same holds for the measurable matroids obtained by taking upward closures. Indeed, let $X\subseteq S$. If $X=\bigcup_{i=1}^{\infty}X_i$ for an increasing sequence $(X_i)_{i\in\bN}$ in $\cM_G$, then $X_i\subseteq S$ for every $i$, and hence $X_i\in\cT_G$. Conversely, every increasing sequence in $\cT_G$ is also an increasing sequence in $\cM_G$. Thus, the restriction of $\overline{\cM}_G$ to $S$ is $\overline{\cT}_G$, and the same argument applies to $T$.
	
	%%%%%%%%%%%%%%%%%%%%%%%%%%%%%%%%
	\section{Basic Operations}
	\label{sec:operations}
	%%%%%%%%%%%%%%%%%%%%%%%%%%%%%%%%
	
	Operations such as truncation, elongation, direct sum, dualization, minors, and union are important tools in structural and optimization problems for matroids. In this section, we extend these operations to measurable matroids.
	
	%%%%%%%%%%%%%%%%%%%%%%%%%%%%%%%%
	\subsection{Truncation and Elongation}
	\label{sec:truncation}
	%%%%%%%%%%%%%%%%%%%%%%%%%%%%%%%%
	
	Let $M=(J,\Sigma,\mu,\cF)$ be a measurable matroid and let $g\in\bR_{\geq 0}$. Define \begin{equation}\label{eq:truncation} 
		\cF_g\coloneqq \{F\in\cF \mid \mu(F)\leq g\}. 
	\end{equation} 
	We call $M_g=(J,\Sigma,\mu,\cF_g)$ the $g$-\emph{truncation} of $M$. 
	
	\begin{prop}\label{prop:truncation} 
		If $\cF_g$ is defined by \eqref{eq:truncation}, then $M_g=(J,\Sigma,\mu,\cF_g)$ is a measurable matroid with rank function \[ r_g(X)=\min\{r(X),g\}. \] 
	\end{prop}
	
	\begin{proof}
		Independence axioms \ref{ax:i1} and \ref{ax:i2} clearly hold, while \ref{ax:i4} follows from the continuity of $\mu$ from below. Let $X\subseteq J$ and let $F_1,F_2\in\cF_g|X$ be maximal sets. If $r(X)\leq g$, then every member of $\cF|X$ has measure at most $r(X)\leq g$, and hence belongs to $\cF_g|X$. Therefore, $F_1$ and $F_2$ are also maximal in $\cF|X$, so they both have measure $r(X)$.
		
		Assume now that $r(X)>g$, and let $F\in\cF_g|X$ be maximal. If $\mu(F)<g$, extend $F$ to a maximal member $B$ of $\cF|X$; such a set exists by \cref{lem:maximal}, applied to the family of members of $\cF|X$ containing $F$. Then $\mu(B)=r(X)>g$, so $\mu(B\setminus F)>0$. By the atomlessness of $\mu$, choose $A\subseteq B\setminus F$ with $0<\mu(A)\leq g-\mu(F)$. Since $F\cup A\subseteq B$, we have $F\cup A\in\cF$, and the choice of $A$ gives $\mu(F\cup A)\leq g$. Thus, $F\cup A\in\cF_g|X$, contradicting the maximality of $F$. Hence, every maximal member of $\cF_g|X$ has measure $g$.
		
		Therefore, \ref{ax:i3} holds, and $r_g(X)=\min\{r(X),g\}$.
	\end{proof}
	
	Let $M=(J,\Sigma,\mu,\cF)$ be a measurable matroid and let $g\in\bR_{\geq 0}$. Define 
	\begin{equation}\label{eq:elongation} 
		\cF^g\coloneqq \{X\in\Sigma \mid \text{there exists } Y\subseteq X \text{ with } \mu(Y)\leq g \text{ and } X\setminus Y\in\cF\}. 
	\end{equation} 
	We call $M^g=(J,\Sigma,\mu,\cF^g)$ the $g$-\emph{elongation} of $M$. 
	
	\begin{prop}\label{prop:elongation}  
		If $\cF^g$ is defined by \eqref{eq:elongation}, then $M^g=(J,\Sigma,\mu,\cF^g)$ is a measurable matroid with rank function \[ r^g(X)=\min\{r(X)+g,\mu(X)\}. \] 
	\end{prop}
	
	\begin{proof}
		For $X\in\Sigma$, set $q(X)\coloneqq\mu(X)-r(X)$. We first observe that
		\begin{equation*}
			X\in\cF^g
			\quad\Longleftrightarrow\quad
			q(X)\leq g.
		\end{equation*}
		Indeed, if $X\in\cF^g$, then there exists $Y\subseteq X$ with $\mu(Y)\leq g$ and $X\setminus Y\in\cF$. Hence, $r(X)\geq\mu(X\setminus Y)\geq\mu(X)-g$, and therefore $q(X)\leq g$. Conversely, if $q(X)\leq g$ and $F$ is a maximal independent subset of $X$, then $\mu(X\setminus F)=\mu(X)-r(X)=q(X)\leq g$, so $X\in\cF^g$.
		
		The function $q$ is monotone. Indeed, if $X\subseteq Y$, then submodularity and \ref{ax:r3} give
		\begin{equation*}
			r(Y)\leq r(X)+r(Y\setminus X)\leq r(X)+\mu(Y\setminus X),
		\end{equation*}
		and hence $q(X)\leq q(Y)$. Since $q(\emptyset)=0$, axioms \ref{ax:i1} and \ref{ax:i2} hold for $\cF^g$.
		
		To verify \ref{ax:i4}, let $(X_i)_{i\in\bN}$ be an increasing sequence in $\cF^g$, and set $X\coloneqq\bigcup_{i=1}^{\infty}X_i$. Since $\mu$ is continuous from below and $r$ is continuous from below by \cref{cor:smooth}, we have $q(X)=\lim_{i\to\infty}q(X_i)\leq g$. Hence, $X\in\cF^g$.
		
		It remains to verify \ref{ax:i3} and to determine the rank function. Let $X\in\Sigma$, and let $I$ be a maximal member of $\cF^g|X$. Let $F$ be a maximal independent subset of $I$, and extend $F$ to a maximal independent subset $K$ of $X$. Since $\mu(I\setminus F)=q(I)\leq g$ and $F\subseteq K$, we have $\mu(I\setminus K)\leq g$. Thus, $I\cup K\in\cF^g|X$, since deleting $I\setminus K$ leaves the independent set $K$. By the maximality of $I$, it follows that $K\subseteq I$. Consequently, $\mu(K)=r(X)$ and $\mu(I)\leq r(X)+g$.
		
		Suppose that $\mu(I)<\min\{r(X)+g,\mu(X)\}$. Then both $\mu(X\setminus I)$ and $r(X)+g-\mu(I)$ are positive. By the atomlessness of $\mu$, there exists an $A\subseteq X\setminus I$ such that
		\begin{equation*}
			0<\mu(A)\leq r(X)+g-\mu(I).
		\end{equation*}
		Since $K\subseteq I$, we have
		\begin{equation*}
			\mu\left((I\cup A)\setminus K\right)
			=
			\mu(I)-r(X)+\mu(A)
			\leq g.
		\end{equation*}
		Hence, $I\cup A\in\cF^g|X$, contradicting the maximality of $I$. Therefore, $\mu(I)=\min\{r(X)+g,\mu(X)\}$. Thus, all maximal members of $\cF^g|X$ have the same measure, so \ref{ax:i3} holds, and $r^g(X)=\min\{r(X)+g,\mu(X)\}$.
	\end{proof}
	
	%%%%%%%%%%%%%%%%%%%%%%%%%%%%%%%%
	\subsection{Direct Sums}
	\label{sec:direct}
	%%%%%%%%%%%%%%%%%%%%%%%%%%%%%%%%
	
	Let $(M_i)_{i\in\bN}$ be a sequence of measurable matroids, where $M_i=(J_i,\Sigma_i,\mu_i,\cF_i)$, the ground spaces $J_i$ are pairwise disjoint, and $\sum_{i=1}^{\infty}\mu_i(J_i)<\infty$. Set $J\coloneqq\bigcup_{i=1}^{\infty}J_i$,
	\[
	\Sigma\coloneqq\{A\subseteq J\mid A\cap J_i\in\Sigma_i\text{ for every }i\in\bN\},
	\]
	and define $\mu(A)\coloneqq\sum_{i=1}^{\infty}\mu_i(A\cap J_i)$ for $A\in\Sigma$. Then $(J,\Sigma,\mu)$ is a standard measure space. Finally, define
	\begin{equation}\label{eq:direct_sum}
		\cF\coloneqq\{A\in\Sigma\mid A\cap J_i\in\cF_i\text{ for every }i\in\bN\}.
	\end{equation}
	We call
	\[
	\bigoplus_{i\in\bN}M_i\coloneqq(J,\Sigma,\mu,\cF)
	\]
	the \emph{direct sum} of the matroids $(M_i)_{i\in\bN}$. Finite direct sums are obtained by taking $J_i=\emptyset$ for all but finitely many $i\in\bN$; in particular, the direct sum of two measurable matroids is denoted by $M_1\oplus M_2$.
	
	\begin{prop}\label{prop:direct_sum}\label{prop:countable_direct_sum}
		If $\cF$ is defined by \eqref{eq:direct_sum}, then $\bigoplus_{i\in\bN}M_i=(J,\Sigma,\mu,\cF)$ is a measurable matroid. Its rank function is given by
		\[
		r(A)=\sum_{i=1}^{\infty}r_i(A\cap J_i),
		\]
		where $r_i$ is the rank function of $M_i$.
	\end{prop}
	
	\begin{proof}
		Axioms \ref{ax:i1} and \ref{ax:i2} are immediate. To prove \ref{ax:i4}, let $(A_k)_{k\in\bN}$ be an increasing sequence in $\cF$, and set $A\coloneqq\bigcup_{k=1}^{\infty}A_k$. Then $A\cap J_i=\bigcup_{k=1}^{\infty}(A_k\cap J_i)$ for every $i\in\bN$. Since each $\cF_i$ is $\sigma$-increasing, we get $A\cap J_i\in\cF_i$ for every $i\in\bN$, and hence $A\in\cF$.
		
		It remains to verify \ref{ax:i3}. Let $X\in\Sigma$, and let $A,B$ be maximal members of $\cF|X$. For every $i\in\bN$, the sets $A\cap J_i$ and $B\cap J_i$ are maximal members of $\cF_i|(X\cap J_i)$. Indeed, suppose that $A\cap J_i$ is not maximal. Then there exists $C\in\cF_i|(X\cap J_i)$ such that $A\cap J_i\subseteq C$ and $\mu_i(C\setminus A)>0$. The set $(A\setminus J_i)\cup C$ belongs to $\cF|X$ and contradicts the maximality of $A$. The same argument applies to $B$. Thus, $\mu_i(A\cap J_i)=\mu_i(B\cap J_i)$ for every $i\in\bN$. Summing these equalities gives $\mu(A)=\mu(B)$, proving \ref{ax:i3}.
		
		The rank formula follows from the same argument. If $A$ is a maximal member of $\cF|X$, then $A\cap J_i$ is a maximal member of $\cF_i|(X\cap J_i)$ for every $i\in\bN$. Therefore,
		\[
		r(X)=\mu(A)=\sum_{i=1}^{\infty}r_i(X\cap J_i).\qedhere
		\]
	\end{proof}
	
	%%%%%%%%%%%%%%%%%%%%%%%%%%%%%%%%
	\subsection{Dual Matroid}
	\label{sec:dual}
	%%%%%%%%%%%%%%%%%%%%%%%%%%%%%%%%
	
	Let $M=(J,\Sigma,\mu,\cB)$ be a measurable matroid defined by its bases. Define
	\begin{equation}\label{eq:dual_bases}
		\cB^*\coloneqq\{B\in\Sigma\mid J\setminus B\in\cB\}.
	\end{equation}
	We call $M^*=(J,\Sigma,\mu,\cB^*)$ the \emph{dual} of $M$. Clearly, we have $(M^*)^*=M$.
	
	\begin{prop}\label{prop:dual}
		Let $M=(J,\Sigma,\mu,\cB)$ be a measurable matroid defined by its bases. If $\cB^*$ is defined by \eqref{eq:dual_bases}, then $M^*=(J,\Sigma,\mu,\cB^*)$ is a measurable matroid. Its rank function is given by
		\begin{equation}\label{eq:dual_rank}
			r^*(S)=\mu(S)-r(J)+r(J\setminus S)
		\end{equation}
		for every $S\in\Sigma$.
	\end{prop}
	
	\begin{proof}
		Axioms \ref{ax:b1} and \ref{ax:b3} clearly hold. We now prove \ref{ax:b2}. Let $B^*_1,B^*_2\in\cB^*$ and let $X\subseteq B^*_1\setminus B^*_2$. Define $B_1\coloneqq J\setminus B^*_1$ and $B_2\coloneqq J\setminus B^*_2$. As $B_1,B_2\in\cB$ and $X\subseteq B_2\setminus B_1$, we can apply \ref{ax:b2'} in $M$, which holds by \cref{prop:coexchange}, to obtain a set $Y\subseteq B_1\setminus B_2$ such that $(B_1\setminus Y)\cup X\in\cB$ and $\mu(Y)=\mu(X)$. It follows that $Y\subseteq B^*_2\setminus B^*_1$ and
		\[
		J\setminus\left((B_1\setminus Y)\cup X\right)=(B^*_1\setminus X)\cup Y\in\cB^*,
		\]
		which is exactly \ref{ax:b2}.
		
		It remains to prove the rank formula. Let $S\in\Sigma$. Since the bases of $M^*$ are the complements of the bases of $M$, we have $r^*(S)=\mu(S)-\min\{\mu(B\cap S)\mid B\in\cB\}$. We now show that this minimum is attained and determine its value. Let $I$ be a maximal independent subset of $J\setminus S$, so $\mu(I)=r(J\setminus S)$. Extend $I$ to a basis $B$ of $M$. By the maximality of $I$ in $J\setminus S$, we have $\mu(B\cap(J\setminus S))=\mu(I)=r(J\setminus S)$, and hence $\mu(B\cap S)=r(J)-r(J\setminus S)$. Conversely, for every basis $B$ of $M$, the set $B\cap(J\setminus S)$ is independent in $J\setminus S$, so $\mu(B\cap(J\setminus S))\leq r(J\setminus S)$ and therefore $\mu(B\cap S)\geq r(J)-r(J\setminus S)$. Thus, the minimum is $r(J)-r(J\setminus S)$, proving \eqref{eq:dual_rank}.
	\end{proof}
	
	As an easy consequence, we show that measurable matroids can be characterized using the co-exchange axiom \ref{ax:b2'} instead of the exchange axiom \ref{ax:b2}.
	
	\begin{cor} \label{cor:coexchange}
		The two sets of axioms \{\ref{ax:b1}, \ref{ax:b2}, \ref{ax:b3}\} and \{\ref{ax:b1}, \ref{ax:b2'}, \ref{ax:b3}\} are equivalent.
	\end{cor}
	
	\begin{proof}
		We already proved in \cref{prop:coexchange} that \ref{ax:b2'}  holds in measurable matroids. For the other direction, let $M^*=(J, \Sigma, \mu, \cB^*)$ be a system where $\cB^*$ satisfies \ref{ax:b1}, \ref{ax:b2'}  and \ref{ax:b3}. Let $M=(J, \Sigma, \mu, \cB)$ where  
		$$\cB \coloneqq \{ B \in \Sigma \mid J \setminus B \in \cB^*\}.$$
		As in the proof of \cref{prop:dual}, we can see that $\cB$ satisfies \ref{ax:b1}, \ref{ax:b2} and \ref{ax:b3}. Consequently, $M$ is a measurable matroid. Applying \cref{prop:dual}, we obtain that the dual of $M$, which is $M^*$, is also a measurable matroid, hence \ref{ax:b2} holds.
	\end{proof}
	
	%%%%%%%%%%%%%%%%%%%%%%%%%%%%%%%%
	\subsection{Restriction and Contraction}
	\label{sec:contraction}
	%%%%%%%%%%%%%%%%%%%%%%%%%%%%%%%%
	
	Let $M=(J,\Sigma,\mu,\cF)$ be a measurable matroid and let $J' \in \Sigma$. Let $M|J'=(J',\Sigma|J',\mu|J',\cF|J')$ be the \emph{restriction} of $M$ to $J'$, where $\mu|J'$ is the restriction of $\mu$ to $\Sigma|J'$. Clearly, $M|J'$ is a measurable matroid with the same rank function as $M$, restricted to $J'$. If $Z \coloneqq J \setminus J'$, we also say that $M \backslash Z \coloneqq M|J'$ is the \emph{deletion} of $Z$ from $M$.
	
	Now we introduce, in a certain sense, the dual operation. Again, let $M=(J,\Sigma,\mu,r)$ be a measurable matroid and let $J' \in \Sigma$. Let $Z \coloneqq J \setminus J'$ and define $r' \colon \Sigma|J' \to \bR_{\geq 0}$ by
	\[
	r'(X) \coloneqq r(X \cup Z)-r(Z).
	\]
	
	\begin{prop}\label{prop:contraction} 
		$r'$ is the rank function of a measurable matroid.
	\end{prop}
	
	\begin{proof}
		We verify the rank axioms for $r'$. Clearly, $r'(\emptyset)=r(Z)-r(Z)=0$, so \ref{ax:r1} holds. Monotonicity \ref{ax:r2} follows immediately from the monotonicity of $r$. Furthermore, for every $X\in\Sigma|J'$, we have
		\[
		r'(X)=r(X\cup Z)-r(Z)\leq r(X)\leq\mu(X),
		\]
		where the first inequality follows from submodularity applied to $X$ and $Z$. Thus, \ref{ax:r3} holds.
		
		To prove submodularity, let $X_1,X_2\in\Sigma|J'$. Then
		\[
		\begin{aligned}
			r'(X_1)+r'(X_2)
			&=r(X_1\cup Z)+r(X_2\cup Z)-2r(Z)\\
			&\geq r((X_1\cap X_2)\cup Z)+r((X_1\cup X_2)\cup Z)-2r(Z)\\
			&=r'(X_1\cap X_2)+r'(X_1\cup X_2),
		\end{aligned}
		\]
		where we used the submodularity of $r$ for the sets $X_1\cup Z$ and $X_2\cup Z$. Hence, \ref{ax:r4} holds.
		
		It remains to verify \ref{ax:r5}. Let $X\subseteq Y\subseteq J'$ with $r'(X)=\mu(X)$. Since \[
		r'(X)=r(X\cup Z)-r(Z)\leq r(X)\leq\mu(X),
		\]
		we have $r(X)=\mu(X)$. Applying \ref{ax:r5} to $X\subseteq X\cup Z$, we obtain a set $X'$ such that $X\subseteq X'\subseteq X\cup Z$ and $r(X')=\mu(X')=r(X\cup Z)$. Applying \ref{ax:r5} once more, now to $X'\subseteq Y\cup Z$, we obtain a set $X''$ such that $X'\subseteq X''\subseteq Y\cup Z$ and $r(X'')=\mu(X'')=r(Y\cup Z)$.
		
		We claim that $\mu(X''\cap Z)=r(Z)$. First,
		\[
		\mu(X'\cap Z)=\mu(X')-\mu(X)=r(X\cup Z)-r'(X)=r(Z).
		\]
		Since $X''$ is independent, the set $X''\cap Z$ is independent as well, and hence $\mu(X''\cap Z)\leq r(Z)$. On the other hand, $X'\cap Z\subseteq X''\cap Z$, so $\mu(X''\cap Z)\geq\mu(X'\cap Z)=r(Z)$. Therefore, $\mu(X''\cap Z)=r(Z)$.
		
		Set $W\coloneqq X''\cap J'$. Then $X\subseteq W\subseteq Y$. Moreover,
		\[
		\mu(W)=\mu(X'')-\mu(X''\cap Z)=r(Y\cup Z)-r(Z)=r'(Y).
		\]
		Finally, since $X''\subseteq W\cup Z\subseteq Y\cup Z$ and $r(X'')=r(Y\cup Z)$, we have $r(W\cup Z)=r(Y\cup Z)$. Consequently,
		\[
		r'(W)=r(W\cup Z)-r(Z)=r(Y\cup Z)-r(Z)=r'(Y)=\mu(W).
		\]
		Thus, $X\subseteq W\subseteq Y$ and $r'(W)=\mu(W)=r'(Y)$, proving \ref{ax:r5}. Hence, $r'$ satisfies the rank axioms.
	\end{proof}
	
	Hence, $M' \coloneqq (J',\Sigma|J',\mu|J',r')$ is a measurable matroid. We say that the matroid $M'$ is obtained from $M$ by the \emph{contraction} of $Z$, and we will use the notation $M/Z$ for it. Denote the family of independent sets of $M'$ by $\cF'$. We prove the following characterization of independent sets; both the statement and the proof are analogous to the discrete case.
	
	\begin{prop} \label{prop:contraction_char}
		Let $M=(J,\Sigma,\mu,\cF)$ be a measurable matroid, let $J'\in\Sigma$, and set $Z\coloneqq J\setminus J'$. Then, for every $F\in\Sigma|J'$, the following are equivalent.
		\begin{enumerate}[label=(\roman*)]\itemsep0em
			\item $F$ is independent in $M/Z$. \label{cont:i}
			\item For every maximal set $I\in\cF|Z$, we have $I\cup F\in\cF$. \label{cont:ii}
			\item There exists a maximal set $I\in\cF|Z$ such that $I\cup F\in\cF$. \label{cont:iii}
		\end{enumerate}
	\end{prop}
	
	\begin{proof}\mbox{}
		\medskip
		
		\noindent \ref{cont:i}$\Rightarrow$\ref{cont:ii}. Let $I$ be a maximal set in $\cF|Z$. Extend $I$ to a maximal independent subset $H$ of $F\cup Z$. Since $H\cap Z$ is an independent subset of $Z$ containing $I$, the maximality of $I$ implies $H\cap Z=I$. Hence, $H=I\cup F_0$ for some $F_0\subseteq F$. By \ref{cont:i}, the set $F$ is independent in $M/Z$, so $r(F\cup Z)-r(Z)=\mu(F)$. Since $I$ is maximal in $\cF|Z$, we have $\mu(I)=r(Z)$. As $H$ is maximal in $\cF|(F\cup Z)$, it follows that
		$
		\mu(H)=r(F\cup Z)=\mu(F)+r(Z)=\mu(F)+\mu(I)=\mu(F\cup I).
		$
		But $H\subseteq F\cup I$, hence $H=F\cup I$. Therefore, $F\cup I\in\cF$, proving \ref{cont:ii}.
		\medskip
		
		\noindent \ref{cont:ii}$\Rightarrow$\ref{cont:iii}. The implication is immediate, since $\cF|Z$ contains a maximal set by \cref{lem:maximal}.
		\medskip
		
		\noindent \ref{cont:iii}$\Rightarrow$\ref{cont:i}. Let $I$ be a maximal set in $\cF|Z$ such that $I\cup F\in\cF$. Then $r(Z)=\mu(I)$ and $r(F\cup Z)\geq r(F\cup I)=\mu(F\cup I)$. Therefore
		$
		r'(F)=r(F\cup Z)-r(Z)\geq \mu(F\cup I)-\mu(I)=\mu(F).
		$
		Since $r'(F)\leq \mu(F)$ by \ref{ax:r3}, equality holds. Thus, $r'(F)=\mu(F)$, so $F$ is independent in $M/Z$. This proves \ref{cont:i}.
	\end{proof}
	
	We are now ready to define \emph{minors}. Minors play a central role in finite matroid theory, particularly in structural results and excluded-minor characterizations; see, e.g., \cite{oxley2011matroid}. Although we introduce the basic minor operations here, their study is left as a future direction. 
	
	As in finite matroid theory, a minor is obtained by deletions and contractions. In the measurable setting, it is more convenient to contract and delete measurable subsets at once. The following proposition shows that successive operations of the same type can be combined and that deletion and contraction commute.
	
	\begin{prop}\label{prop:delete_contract_commute}
		Let $M=(J,\Sigma,\mu,r)$ be a measurable matroid.
		\begin{enumerate}[label=(\alph*)]\itemsep0em
			\item If $D_1,D_2\in\Sigma$ are disjoint, then $(M\backslash D_1)\backslash D_2=M\backslash(D_1\cup D_2)$. \label{minor:a}
			\item If $C_1,C_2\in\Sigma$ are disjoint, then $(M/C_1)/C_2=M/(C_1\cup C_2)$. \label{minor:b}
			\item If $C,D\in\Sigma$ are disjoint, then $(M/C)\backslash D=(M\backslash D)/C$. \label{minor:c}
		\end{enumerate}
	\end{prop}
	
	\begin{proof}
		Part \ref{minor:a} follows immediately from the definition of deletion.
		
		To prove \ref{minor:b}, let $J'\coloneqq J\setminus(C_1\cup C_2)$ and let $X\in\Sigma|J'$. The rank of $X$ in $(M/C_1)/C_2$ is
		\[
		\left(r(X\cup C_1\cup C_2)-r(C_1)\right)
		-\left(r(C_1\cup C_2)-r(C_1)\right)=r(X\cup C_1\cup C_2)-r(C_1\cup C_2),
		\]
		which is its rank in $M/(C_1\cup C_2)$.
		
		Finally, let $C,D\in\Sigma$ be disjoint and set $J'\coloneqq J\setminus(C\cup D)$. Both matroids in \ref{minor:c} are defined on the measure space $(J',\Sigma|J',\mu|J')$. For every $X\in\Sigma|J'$, its rank in either matroid is $r(X\cup C)-r(C)$. Thus, their rank functions coincide.
	\end{proof}
	
	By \cref{prop:delete_contract_commute}, every finite sequence of deletions and contractions can be replaced by one contraction and one deletion. We call
	\[
	(M/C)\backslash D=(M\backslash D)/C
	\]
	a \emph{minor} of $M$, where $C,D\in\Sigma$ are disjoint. The same conclusion extends to countable sequences in the following sense.
	
	\begin{cor}\label{cor:countable_delete_contract}
		Let $(C_i)_{i\in\bN}$ and $(D_i)_{i\in\bN}$ be sequences of subsets of $J$ such that all the sets $C_i$ and $D_i$ are pairwise disjoint. Set $C\coloneqq\bigcup_{i\in\bN}C_i$ and $D\coloneqq\bigcup_{i\in\bN}D_i$. For every $k\in\bN$, let $M_k$ be the matroid obtained from $M$ by contracting $C_1,\dots,C_k$ and deleting $D_1,\dots,D_k$, in any order, and let $r_k$ denote its rank function. Then, for every $X\in\Sigma|(J\setminus(C\cup D))$,
		\[
		\lim_{k\to\infty}r_k(X)
		=
		r_{(M/C)\backslash D}(X).
		\]
	\end{cor}
	
	\begin{proof}
		Set $C^{(k)}\coloneqq\bigcup_{i=1}^k C_i$ and $D^{(k)}\coloneqq\bigcup_{i=1}^k D_i$. By \cref{prop:delete_contract_commute}, we have $M_k=(M/C^{(k)})\setminus D^{(k)}$, independently of the order of the operations. Hence, for every $X\in\Sigma|(J\setminus(C\cup D))$,
		\[
		r_k(X)=r(X\cup C^{(k)})-r(C^{(k)}).
		\]
		Since $(C^{(k)})_{k\in\bN}$ increases to $C$ and $r$ is continuous from below by \cref{cor:smooth}, we have
		\[
		\lim_{k\to\infty}r_k(X)
		=
		r(X\cup C)-r(C)
		=
		r_{(M/C)\backslash D}(X).
		\qedhere
		\]
	\end{proof}
	
	Having introduced duality, deletion, and contraction, we settle the standard relations between these operations. In the finite case, truncation is dual to elongation, while deletion is dual to contraction. The next proposition shows that the same statements hold for measurable matroids.
	
	\begin{prop}\label{prop:dual_operations}
		Let $M=(J,\Sigma,\mu,\cF)$ be a measurable matroid with rank function $r$, and let $r^*$ denote the rank function of $M^*$.
		\begin{enumerate}[label=(\roman*)]\itemsep0em
			\item For every $g\in\bR_{\geq 0}$ with $g\leq r(J)$, we have $(M^*)^g=(M_{r(J)-g})^*$. For every $g\in\bR_{\geq 0}$ with $g\leq r^*(J)$, we have $(M^*)_g=(M^{r^*(J)-g})^*$.\label{dual:i}
			\item For every $Z\in\Sigma$, we have $(M\backslash Z)^*=M^*/Z$ and $(M/Z)^*=M^*\backslash Z$.\label{dual:ii}
		\end{enumerate}
	\end{prop}
	
	\begin{proof}
		We first prove \ref{dual:i}. Let $r_1$ and $r_2$ denote the rank functions of $(M^*)^g$ and $(M_{r(J)-g})^*$, respectively. For $S\in\Sigma$, we have $r_1(S)=\min\{r^*(S)+g,\mu(S)\}=\min\{\mu(S)-r(J)+r(J\setminus S)+g,\mu(S)\}$. On the other hand, $r_2(S)=\mu(S)-(r(J)-g)+\min\{r(J\setminus S),r(J)-g\}$. The last expression is equal to $\min\{\mu(S)-r(J)+g+r(J\setminus S),\mu(S)\}$. Hence, $r_1(S)=r_2(S)$ for every $S\in\Sigma$, and therefore $(M^*)^g=(M_{r(J)-g})^*$. Applying this equality to the dual matroid $M^*$, and using $(M^*)^*=M$, gives $(M^*)_g=(M^{r^*(J)-g})^*$ for every $g\leq r^*(J)$.
		
		We now prove \ref{dual:ii}. Let $Z\in\Sigma$, set $J'\coloneqq J\setminus Z$, and let $r_1$ and $r_2$ denote the rank functions of $(M\backslash Z)^*$ and $M^*/Z$, respectively. For $S\in\Sigma|J'$, the dual rank formula gives $r_1(S)=\mu(S)-r(J')+r(J'\setminus S)$. Since the rank function of $M^*$ is $r^*(A)=\mu(A)-r(J)+r(J\setminus A)$, the contraction rank formula gives $r_2(S)=r^*(S\cup Z)-r^*(Z)$. Expanding this expression gives $r_2(S)=\mu(S)-r(J')+r(J'\setminus S)$, which is exactly $r_1(S)$. Thus, $(M\backslash Z)^*=M^*/Z$. The statement $(M/Z)^*=M^*\backslash Z$ follows by applying the previous equality to $M^*$ in place of $M$ and using $(M^*)^*=M$. This proves the proposition.
	\end{proof}
	
	%%%%%%%%%%%%%%%%%%%%%%%%%%%%%%%%
	\subsection{Matroid Union}
	\label{sec:union}
	%%%%%%%%%%%%%%%%%%%%%%%%%%%%%%%%
	
	Given $k$ measurable matroids $M_i = (J,\Sigma,\mu,\cF_i)$ on a common ground space, a natural analogue of the \emph{union} or \emph{sum} operation from finite matroid theory is to define a family $\cF$ by declaring a set $F \in \Sigma$ to be independent if it can be written as $F = F_1 \cup \dots \cup F_k$ with $F_i \in \cF_i$ for each $i \in [k]$. In the finite setting, this construction yields a matroid, and Edmonds and Fulkerson~\cite{edmonds1965transversals} showed that the corresponding rank function is given by $r_\vee(Z) = \min\big\{\sum_{i=1}^k r_i(X) + |Z \setminus X| \mid X \subseteq Z\big\}$, together with an algorithm to find maximum independent sets and their decompositions. In contrast, as we show below, this natural definition does not in general yield a measurable matroid, even in the case $k=2$, see \cref{rem:unionnotm}.
	
	Let $(J,\Sigma,\mu)$ be a standard measure space, and let $\cS_1,\cS_2\subseteq \Sigma$. We define
	$$
	\cS_1+\cS_2
	\coloneqq
	\{S\in\Sigma \mid S=S_1\cup S_2 \text{ for some } S_i\in\cS_i,\ i=1,2\}.
	$$
	More generally, for a finite or countable collection $(\cS_i)_{i\in I}$ of families of measurable sets, we write
	$$
	\sum_{i\in I}\cS_i
	\coloneqq
	\left\{
	S\in\Sigma \,\middle|\,
	S=\bigcup_{i\in I} S_i
	\text{ for some } S_i\in\cS_i \text{ for all } i\in I
	\right\}.
	$$
	
	\begin{thm} \label{thm:union}
		Let $M_1=(J, \Sigma, \mu, \cF_1)$ and $M_2=(J, \Sigma, \mu, \cF_2)$ be measurable matroids.
		Then $M=(J, \Sigma, \mu, \cF_1 + \cF_2)$ is a measurable pre-matroid.
	\end{thm}
	
	\begin{proof}
		Let $\cF \coloneqq \cF_1 + \cF_2$. We verify \ref{ax:i1}, \ref{ax:i2}, and \ref{ax:i3''}. Axiom \ref{ax:i1} is immediate. To prove \ref{ax:i2}, let $F\in\cF$ and let $F'\subseteq F$. Choose $A\in\cF_1$ and $B\in\cF_2$ such that $F=A\cup B$. Then $F'=(A\cap F')\cup(B\cap F')$, and $A\cap F'\in\cF_1$, $B\cap F'\in\cF_2$ by \ref{ax:i2} for $M_1$ and $M_2$. Hence, $F'\in\cF$.
		
		It remains to prove \ref{ax:i3''}. Let $S,S'\in\cF$ with $\mu(S)\leq \mu(S')$. If $\mu(S)=\mu(S')$, then $S$ itself is a suitable augmentation. We may therefore assume that $\mu(S)<\mu(S')$. Choose representations $S=A\cup B$ and $S'=A'\cup B'$, where $A,A'\in\cF_1$ and $B,B'\in\cF_2$. After replacing $B$ with $B\setminus A$ and $B'$ with $B'\setminus A'$, we may assume that $A\cap B=\emptyset$ and $A'\cap B'=\emptyset$. Set $a'\coloneqq\mu(A')$, $b'\coloneqq\mu(B')$, and $R\coloneqq\mu(S')=a'+b'$. Since $\mu(S)<\mu(S')$, after possibly interchanging the two matroids, we may assume that $\mu(A)\leq a'$.
		
		We construct alternating augmentations. Put $A_0\coloneqq A$ and $B_0\coloneqq B$, and define $P\coloneqq A\cap B'$ and $Q\coloneqq A'\cap B$. The idea is to modify $A_i$ and $B_i$ while keeping their union nondecreasing, until the first and second parts have measures $a'$ and $b'$, respectively, and no overlap remains. Since $A_i\subseteq A\cup A'$ and $B_i\subseteq B\cup B'$, any overlap can occur only in $P\cup Q$. We therefore alternate between the two regions: in the $A$-step, we remove the current overlap in $P$ from the first part and augment towards $A'$ to restore its measure to $a'$; in the $B$-step, we do the analogous operation in $Q$, augmenting towards $B'$ to restore the second part to measure $b'$. If at any point the union has measure at least $R$, we are done. Otherwise, the parts in $P$ and $Q$ evolve monotonically, which allows us to pass to disjoint limits of measures $a'$ and $b'$ whose union still contains $S$. Note that $P\cap Q=P\cap A'=Q\cap B'=\emptyset$ and $(A\cup A')\cap(B\cup B')=P\cup Q$. At each completed stage $i\geq 1$ at which the construction has not terminated, we maintain $A_i\in\cF_1$, $B_i\in\cF_2$, $\mu(A_i)=a'$, and $\mu(B_i)=b'$, as well as
		\[
		A_i\subseteq A\cup A',\qquad B_i\subseteq B\cup B',\qquad S\subseteq A_i\cup B_i\subseteq S\cup S',\qquad \mu(A_i\cup B_i)<R.
		\]
		The construction will also ensure the monotonicity relations used below.
		
		Suppose that $A_{i-1}$ and $B_{i-1}$ have already been constructed. We first define $L_i\coloneqq A_{i-1}\setminus(B_{i-1}\cap P)$. Then $L_i\subseteq A_{i-1}$, hence $L_i\in\cF_1$ and $\mu(L_i)\leq a'$ for $i\geq2$, while for $i=1$ this follows from $\mu(A)\leq a'$. Moreover,
		\[
		L_i\subseteq(A_{i-1}\setminus B_{i-1})\cup A',
		\]
		since any point of $A_{i-1}\cap B_{i-1}$ which is not in $P$ must lie in $Q\subseteq A'$. Applying \ref{ax:i3''} in $M_1$ to $L_i$ and $A'$, and then using atomlessness and downward closure to trim if necessary, we obtain $A_i\in\cF_1$ such that
		\[
		L_i\subseteq A_i\subseteq L_i\cup A',\qquad \mu(A_i)=a'.
		\]
		In particular, $A_i\subseteq A\cup A'$, and, since $P\cap A'=\emptyset$, we have
		\[
		A_i\cap P\subseteq(A_{i-1}\cap P)\setminus B_{i-1},\qquad A_{i-1}\setminus P\subseteq A_i\setminus P.
		\]
		Thus, $A_i\cup B_{i-1}$ contains $A_{i-1}\cup B_{i-1}$. If $\mu(A_i\cup B_{i-1})\geq R$, then $G\coloneqq A_i\cup B_{i-1}$ belongs to $\cF$, satisfies $S\subseteq G\subseteq S\cup S'$, and has $\mu(G)\geq\mu(S')$, so we are done.
		
		Suppose otherwise that $\mu(A_i\cup B_{i-1})<R$. Since $A_i\subseteq A\cup A'$ and $B_{i-1}\subseteq B\cup B'$, we have $A_i\cap B_{i-1}\subseteq P\cup Q$. The first monotonicity relation above gives $A_i\cap B_{i-1}\cap P=\emptyset$, and hence $A_i\cap B_{i-1}\subseteq Q$. Therefore, if we define $H_i\coloneqq B_{i-1}\setminus(A_i\cap Q)$, then $H_i=B_{i-1}\setminus A_i$. In particular,
		\[
		\mu(H_i)=\mu(A_i\cup B_{i-1})-\mu(A_i)<R-a'=b'.
		\]
		Applying \ref{ax:i3''} in $M_2$ to $H_i$ and $B'$, and trimming if necessary, we obtain $B_i\in\cF_2$ such that
		\[
		H_i\subseteq B_i\subseteq H_i\cup B',\qquad \mu(B_i)=b'.
		\]
		In particular, $B_i\subseteq B\cup B'$, and, since $Q\cap B'=\emptyset$, we have
		\[
		B_i\cap Q\subseteq(B_{i-1}\cap Q)\setminus A_i,\qquad B_{i-1}\setminus Q\subseteq B_i\setminus Q.
		\]
		Moreover, $A_i\cup B_i$ contains $A_i\cup B_{i-1}$. Hence, if $\mu(A_i\cup B_i)\geq R$, then $G\coloneqq A_i\cup B_i$ has the required properties. Otherwise, the stated invariants hold at the end of stage $i$, and the construction continues.
		
		We may therefore assume that the construction never stops. Then we obtain sequences $(A_i)_{i\geq0}$ in $\cF_1$ and $(B_i)_{i\geq0}$ in $\cF_2$ such that $\mu(A_i)=a'$, $\mu(B_i)=b'$, and $\mu(A_i\cup B_i)<R$ for every $i\geq1$. Moreover, $A_i\subseteq A\cup A'$ and $B_i\subseteq B\cup B'$ for all $i$. Since $A\cap B=\emptyset$ and $A'\cap B'=\emptyset$, possible overlaps between the two sequences can occur only in $P\cup Q$.    
		
		We now pass to the limit. The relations established during the construction give that $(A_i\cap P)_{i\geq0}$ and $(B_i\cap Q)_{i\geq0}$ are decreasing, while $(A_i\setminus P)_{i\geq0}$ and $(B_i\setminus Q)_{i\geq0}$ are increasing. Define
		\[
		A_\infty\coloneqq\left(\bigcap_{i=0}^{\infty}(A_i\cap P)\right)\cup\left(\bigcup_{i=0}^{\infty}(A_i\setminus P)\right)
		\]
		and
		\[
		B_\infty\coloneqq\left(\bigcap_{i=0}^{\infty}(B_i\cap Q)\right)\cup\left(\bigcup_{i=0}^{\infty}(B_i\setminus Q)\right).
		\]
		For every $n\geq0$, the set $(\bigcap_{i=0}^{\infty}(A_i\cap P))\cup(A_n\setminus P)$ is a subset of $A_n$ and hence belongs to $\cF_1$. These sets form an increasing sequence with union $A_\infty$, so \ref{ax:i4} gives $A_\infty\in\cF_1$. The same argument gives $B_\infty\in\cF_2$. By continuity of measure along the monotone pieces,
		\[
		\mu(A_\infty)=\lim_{i\to\infty}\mu(A_i)=a',\qquad \mu(B_\infty)=\lim_{i\to\infty}\mu(B_i)=b'.
		\]
		
		We claim that $A_\infty$ and $B_\infty$ are disjoint. Since $A_i\subseteq A\cup A'$ and $B_i\subseteq B\cup B'$ for every $i$, their possible intersections are contained in $P\cup Q$. On $P$, the preceding inclusions give
		\[
		A_\infty\cap B_\infty\cap P=\left(\bigcap_{i=0}^{\infty}(A_i\cap P)\right)\cap\left(\bigcup_{i=0}^{\infty}(B_i\cap P)\right)\subseteq\bigcup_{i=0}^{\infty}(A_{i+1}\cap B_i\cap P)=\emptyset.
		\]
		Indeed, any point in the intersection on the left belongs to $B_i\cap P$ for some $i$ and to $A_{i+1}\cap P$, while the first monotonicity relation implies $A_{i+1}\cap B_i\cap P=\emptyset$. Similarly, on $Q$ we have
		\[
		A_\infty\cap B_\infty\cap Q=\left(\bigcup_{i=0}^{\infty}(A_i\cap Q)\right)\cap\left(\bigcap_{i=0}^{\infty}(B_i\cap Q)\right)\subseteq\bigcup_{i=0}^{\infty}(A_i\cap B_i\cap Q)=\emptyset.
		\]
		Here $A_0\cap B_0=\emptyset$, and for $i\geq1$ the second monotonicity relation implies $A_i\cap B_i\cap Q=\emptyset$. Thus, $A_\infty\cap B_\infty=\emptyset$.
		
		It remains to see that $S\subseteq A_\infty\cup B_\infty$. We have $A\setminus P\subseteq A_i\setminus P$ and $B\setminus Q\subseteq B_i\setminus Q$ for every $i$, hence $A\setminus P\subseteq A_\infty$ and $B\setminus Q\subseteq B_\infty$. On $P$, we have $P\subseteq(A_i\cap P)\cup(B_i\cap P)$ for every $i$. Therefore,
		\[
		P\setminus\bigcup_{i=0}^{\infty}(B_i\cap P)\subseteq\bigcap_{i=0}^{\infty}(A_i\cap P),
		\]
		which gives $P\subseteq A_\infty\cup B_\infty$.
		The same argument on $Q$ gives
		\[
		Q\setminus\bigcup_{i=0}^{\infty}(A_i\cap Q)\subseteq\bigcap_{i=0}^{\infty}(B_i\cap Q),
		\]
		and hence $Q\subseteq A_\infty\cup B_\infty$. Therefore, $A\cup B\subseteq A_\infty\cup B_\infty$.
		
		Finally, $A_\infty\subseteq A\cup A'$ and $B_\infty\subseteq B\cup B'$, so $S\subseteq A_\infty\cup B_\infty\subseteq S\cup S'$. Since $A_\infty$ and $B_\infty$ are disjoint and have measures $a'$ and $b'$, respectively, we have $\mu(A_\infty\cup B_\infty)=a'+b'=R=\mu(S')$. Thus, $G\coloneqq A_\infty\cup B_\infty$ belongs to $\cF$ and satisfies
		\[
		S\subseteq G\subseteq S\cup S',\qquad \mu(G)=\mu(S').
		\]
		This proves \ref{ax:i3''}, and hence $M=(J,\Sigma,\mu,\cF)$ is a measurable pre-matroid.
	\end{proof}
	
	\begin{rem}\label{rem:unionnotm}
		The measurable pre-matroid $M$ in \Cref{thm:union} need not itself be a measurable matroid. To see this, consider the interval-halving matroid $N=([0,1], \Sigma, \lambda, \cF)$ introduced in \Cref{sec:nested}, and define $\cF'$ by declaring that $F' \in \cF'$ if and only if $F' = F_1 \cup F_2$ for some $F_1, F_2 \in \cF$. We show that $[0,1] \notin \cF'$, while there exists an increasing sequence $(F'_i)_{i\in\bN}$ in $\cF'$ with $\lim_{i\to\infty}\lambda(F'_i)=1$. This implies that \ref{ax:i4} fails for $\cF'$.
		
		For $i\in\bN$, let $F_1^{i}\coloneqq\bigcup_{j=1}^i [\frac{2j-1}{2i+1},\frac{2j}{2i+1}]$ and $F_2^{i}\coloneqq\bigcup_{j=1}^i [\frac{2j}{2i+1},\frac{2j+1}{2i+1}]$, and set $F'_i\coloneqq F_1^{i}\cup F_2^{i}=[\frac{1}{2i+1},1]$. Then $F_1^{i},F_2^{i}\in\cF$, so $F'_i\in\cF'$, and the sequence $(F'_i)_{i\in\bN}$ is increasing with $\lim_{i\to\infty}\lambda(F'_i)=1$; see \cref{fig:halving2}.
		
		Finally, suppose for contradiction that $F_1\cup F_2=[0,1]$ for some $F_1,F_2\in\cF$. For every $t\in[0,1]$, independence gives $\lambda(F_i\cap[0,t])\leq t/2$ for $i=1,2$. Since $F_1\cup F_2=[0,1]$, both inequalities must be equalities. This contradicts \cref{lem:halving_measure}.
	\end{rem}
	
	We define the \emph{union} of two measurable matroids $M_1=(J,\Sigma,\mu,\cF_1)$ and $M_2=(J,\Sigma,\mu,\cF_2)$ by $M_1\vee M_2\coloneqq(J,\Sigma,\mu,\overline{\cF_1+\cF_2})$. By \cref{thm:union,thm:upward_closure_prematroid}, this is a measurable matroid. We next show that the union operation is associative for finite unions.
	
	\begin{prop}\label{prop:finite_union_associativity}
		Let $M_i=(J,\Sigma,\mu,\cF_i)$ be measurable matroids for $i\in[k]$. Every iterated binary union of $M_1,\dots,M_k$, with any parenthesization, has independent-set family $\overline{\cF_1+\dots+\cF_k}$. In particular, the result is independent of the parenthesization.
	\end{prop}
	
	\begin{proof}
		For every nonempty set $A\subseteq[k]$, write $\cS_A\coloneqq\sum_{i\in A}\cF_i$. Each family $\cS_A$ is downward closed and therefore closed under countable intersections. Hence, $\overline{\cS_A}$ is $\sigma$-increasing by \cref{lem:upward_closure}.
		
		\begin{cla}\label{cla:finite_union_closure}
			For disjoint nonempty sets $A,B\subseteq[k]$,
			\[
			\overline{\overline{\cS_A}+\overline{\cS_B}}=\overline{\cS_{A\cup B}}.
			\]
		\end{cla}
		
		\begin{claimproof}
			Since $\cS_{A\cup B}=\cS_A+\cS_B\subseteq\overline{\cS_A}+\overline{\cS_B}$, taking upward closures gives
			\[
			\overline{\cS_{A\cup B}}\subseteq\overline{\overline{\cS_A}+\overline{\cS_B}}.
			\]
			For the reverse inclusion, let $S=X\cup Y$, where $X\in\overline{\cS_A}$ and $Y\in\overline{\cS_B}$. Choose increasing sequences $(X_n)_{n\in\bN}$ in $\cS_A$ and $(Y_n)_{n\in\bN}$ in $\cS_B$ such that $X=\bigcup_{n=1}^\infty X_n$ and $Y=\bigcup_{n=1}^\infty Y_n$. Then $(X_n\cup Y_n)_{n\in\bN}$ is an increasing sequence in $\cS_{A\cup B}$ with union $S$, and hence $S\in\overline{\cS_{A\cup B}}$. Thus,
			\[
			\overline{\cS_A}+\overline{\cS_B}\subseteq\overline{\cS_{A\cup B}}.
			\]
			Taking upward closures and using that $\overline{\cS_{A\cup B}}$ is $\sigma$-increasing proves the claim.
		\end{claimproof}
		
		We now prove by induction on $|A|$ that every iterated binary union of the matroids $M_i$, for $i\in A$, has independent-set family $\overline{\cS_A}$. If $A=\{i\}$, then $\overline{\cS_A}=\overline{\cF_i}=\cF_i$ by \ref{ax:i4}. For the induction step, consider a parenthesization whose outermost operation separates the indices into disjoint nonempty sets $B$ and $C$ with $A=B\cup C$. By induction, the two matroids being unioned have independent-set families $\overline{\cS_B}$ and $\overline{\cS_C}$. Hence, their binary union has independent-set family
		\[
		\overline{\overline{\cS_B}+\overline{\cS_C}}=\overline{\cS_A}
		\]
		by \cref{cla:finite_union_closure}. This completes the induction.
	\end{proof}
	
	In view of \cref{prop:finite_union_associativity}, we denote the resulting measurable matroid by $M_1\vee\dots\vee M_k$.
	
	\begin{rem}
		The preceding results do not by themselves imply that the raw sum
		$
		\cF_1+\cdots+\cF_k
		$
		of finitely many measurable matroids is a measurable pre-matroid for $k \geq 3$. This statement is nevertheless true, and can be proved by extending the argument of \cref{thm:union}. However, since the resulting proof is even more technical and the statement will not be used later, we omit the details. 
		Interestingly, the analogous statement fails for pre-matroids: the raw sum of two measurable pre-matroids need not be a measurable pre-matroid. We leave the study of raw sums of measurable pre-matroids for future work.
	\end{rem}
	
	The following proposition gives the corresponding approximation by disjoint independent sets.
	
	\begin{prop}\label{prop:finite_union_raw_description}
		Let $M_i=(J,\Sigma,\mu,\cF_i)$ be measurable matroids for $i\in[k]$. A set $I\in\Sigma$ is independent in $M_1\vee\dots\vee M_k$ if and only if, for every $\varepsilon>0$, there exist pairwise disjoint sets $F_i\in\cF_i$, for $i\in[k]$, such that $\bigcup_{i=1}^kF_i\subseteq I$ and $\mu\left(I\setminus\bigcup_{i=1}^kF_i\right)<\varepsilon$.
	\end{prop}
	
	\begin{proof}
		Set $\cS\coloneqq\cF_1+\dots+\cF_k$ and $\cU\coloneqq\overline{\cS}$. By \cref{prop:finite_union_associativity}, $\cU$ is the independent-set family of $M_1\vee\dots\vee M_k$.
		
		Suppose first that $I\in\cU$. Then $I=\bigcup_{n=1}^\infty S_n$ for some increasing sequence $(S_n)_{n\in\bN}$ in $\cS$. Given $\varepsilon>0$, choose $n$ such that $\mu(I\setminus S_n)<\varepsilon$, and write $S_n=\bigcup_{i=1}^kG_i$ with $G_i\in\cF_i$. For $i\in[k]$, set $F_i\coloneqq G_i\setminus\bigcup_{j<i}G_j$. Then the sets $F_i$ are pairwise disjoint, $F_i\in\cF_i$ for every $i\in[k]$, and $\bigcup_{i=1}^kF_i=S_n$. Thus, $\bigcup_{i=1}^kF_i\subseteq I$ and $\mu\left(I\setminus\bigcup_{i=1}^kF_i\right)<\varepsilon$.
		
		Conversely, suppose that the stated approximation property holds. For every $\varepsilon>0$, choose pairwise disjoint sets $F_i^\varepsilon\in\cF_i$, for $i\in[k]$, such that $\bigcup_{i=1}^kF_i^\varepsilon\subseteq I$ and $\mu\left(I\setminus\bigcup_{i=1}^kF_i^\varepsilon\right)<\varepsilon$. Then $\bigcup_{i=1}^kF_i^\varepsilon$ belongs to $\cS\subseteq\cU$. Applying \cref{cor:nem_novekvo_unio} finishes the proof.
	\end{proof}
	
	To pass from finite to countable unions, we use the following observation.
	
	\begin{prop}\label{prop:increasing_limit_rank}
		Let $M_i=(J,\Sigma,\mu,\cF_i)$ be measurable matroids with rank functions $r_i$ for $i\in\bN$, and assume that $r_i(X)\leq r_{i+1}(X)$ for every $X\in\Sigma$ and every $i\in\bN$. Set $\cS\coloneqq\bigcup_{i\in\bN}\cF_i$ and define $r\colon\Sigma\to\bR_{\geq 0}$ by $r(X)\coloneqq\lim_{i\to\infty}r_i(X)$. Then $(J,\Sigma,\mu,\cS)$ is a measurable pre-matroid, and its upward closure is a measurable matroid with rank function $r$.
	\end{prop}
	
	\begin{proof}
		For every $X\in\Sigma$, the sequence $(r_i(X))_{i\in\bN}$ is increasing and bounded above by $\mu(X)$. Hence, $r(X)=\lim_{i\to\infty}r_i(X)$ is well-defined.
		Since $r_i\leq r_{i+1}$ and both rank functions are dominated by $\mu$, we have $\cF_i\subseteq\cF_{i+1}$ for every $i\in\bN$. Indeed, if $F\in\cF_i$, then $\mu(F)=r_i(F)\leq r_{i+1}(F)\leq\mu(F)$, and hence $F\in\cF_{i+1}$.
		
		Now we prove that $(J,\Sigma,\mu,\cS)$ is a measurable pre-matroid. Axioms \ref{ax:i1} and \ref{ax:i2} are immediate. To verify \ref{ax:i3''}, let $F_1,F_2\in\cS$ with $\mu(F_1)\leq\mu(F_2)$. Choose $i\in\bN$ such that $F_1,F_2\in\cF_i$. Since $M_i$ is a measurable matroid, \ref{ax:i3''} gives a set $G\in\cF_i\subseteq\cS$ such that $F_1\subseteq G\subseteq F_1\cup F_2$ and $\mu(G)\geq\mu(F_2)$, proving \ref{ax:i3''}.
		
		By \cref{thm:upward_closure_prematroid}, the upward closure $\overline{\cS}$ defines a measurable matroid. Let $\rho$ be its rank function. We show that $\rho=r$. Let $X\in\Sigma$. If $F\in\cS|X$, then $F\in\cF_i|X$ for some $i\in\bN$, so $\mu(F)\leq r_i(X)\leq r(X)$. Now let $F\in\overline{\cS}|X$. By the definition of the upward closure, there is an increasing sequence $(F_i)_{i\in\bN}$ in $\cS$ such that $F=\bigcup_{i=1}^\infty F_i$. Since $\mu(F_i)\leq r(X)$ for every $i\in\bN$, we have $\mu(F)\leq r(X)$. Therefore, $\rho(X)\leq r(X)$.
		
		Conversely, for every $i\in\bN$, there exists $F_i\in\cF_i|X$ with $\mu(F_i)=r_i(X)$. Since $\cF_i\subseteq\overline{\cS}$, we have $\rho(X)\geq r_i(X)$ for every $i\in\bN$. Letting $i\to\infty$ gives $\rho(X)\geq r(X)$. Hence, $\rho=r$.
	\end{proof}
	
	\begin{rem}
		The upward closure in \cref{prop:increasing_limit_rank} cannot in general be omitted. Let $M=(J,\Sigma,\mu,\cF)$ be a measurable matroid of positive rank $R$, and let $(g_i)_{i\in\bN}$ be an increasing sequence with $g_i<R$ and $\lim_{i\to\infty}g_i=R$. For each $i\in\bN$, let $M_i$ be the $g_i$-truncation of $M$. Then
		\[
		\bigcup_{i\in\bN}\cF_i
		=
		\{F\in\cF\mid \mu(F)<R\}.
		\]
		If $B$ is a basis of $M$, then, by atomlessness, there exists an increasing sequence $(B_i)_{i\in\bN}$ of subsets of $B$ such that $\mu(B_n)=g_n$ for every $n\in\bN$. Thus, $B_n\in\bigcup_{i\in\bN}\cF_i$ for every $n\in\bN$, while $\bigcup_{i\in\bN}B_i$ has measure $R$ and does not belong to $\bigcup_{i\in\bN}\cF_i$. Hence, this union need not be closed under increasing unions.
	\end{rem}
	
	\begin{rem}
		There is no directly analogous useful statement for decreasing sequences. Indeed, the intersection of a decreasing sequence of measurable matroids may collapse to the zero matroid. Let $J\coloneqq [0,1)$ with its Borel $\sigma$-algebra and Lebesgue measure $\lambda$, and let $I_{n,k}\coloneqq [k2^{-n},(k+1)2^{-n})$ for $n\in\bN$ and $k=0,\dots,2^n-1$. Define $M_n=(J,\Sigma,\lambda,\cF_n)$ to be the finitary measurable
		partition matroid from \cref{prop:finitary_partition_matroid}, with
		\[
		\cF_n \coloneqq
		\{F\in\Sigma \mid \lambda(F\cap I_{n,k})\leq 2^{-(n+1)} \text{ for } k=0,\dots,2^n-1\}.
		\]
		Then $\cF_{n+1}\subseteq \cF_n$ for every $n\in\bN$, since each interval $I_{n,k}$ is the union of two intervals of the next dyadic level.
		
		We claim that $\bigcap_{n\in\bN}\cF_n$ contains no set of positive measure. Suppose that $\lambda(F)>0$. By the Lebesgue density theorem~\cite[Chapter~3]{folland1999real}, almost every point of $F$ is a density point of $F$. In particular, for some $x\in F$ and some sufficiently large $n$, the dyadic interval $I_{n,k}$ containing $x$ satisfies $\lambda(F\cap I_{n,k})>\frac12\lambda(I_{n,k})=2^{-(n+1)}$. Hence, $F\notin\cF_n$. Consequently, $\bigcap_{n\in\bN}\cF_n$ consists only of null sets and hence defines the zero matroid. On the other hand, the rank function $r_n$ of $M_n$ satisfies $r_n(J)=1/2$ for every $n\in\bN$. Thus, the rank function of the intersection is not the pointwise limit of the rank functions.
	\end{rem}
	
	We now extend the union operation to countably many measurable matroids.
	
	\begin{thm}\label{thm:countable_union}
		Let $M_i=(J,\Sigma,\mu,\cF_i)$ be measurable matroids for $i\in\bN$, and let $\cR\coloneqq\sum_{i\in\bN}\cF_i$. Then $(J,\Sigma,\mu,\overline{\cR})$ is a measurable matroid. Moreover, a set $I\in\Sigma$ belongs to $\overline{\cR}$ if and only if, for every $\varepsilon>0$, there exist a finite set $A\subseteq\bN$ and pairwise disjoint sets $F_i\in\cF_i$, for $i\in A$, such that $\bigcup_{i\in A}F_i\subseteq I$ and $\mu\left(I\setminus\bigcup_{i\in A}F_i\right)<\varepsilon$. In particular, the resulting measurable matroid is invariant under every permutation of the sequence $(M_i)_{i\in\bN}$.
	\end{thm}
	
	\begin{proof}
		For every $n\in\bN$, let $M^{(n)}\coloneqq M_1\vee\dots\vee M_n$, and write $M^{(n)}=(J,\Sigma,\mu,\cF^{(n)})$. By \cref{prop:finite_union_associativity}, we have $\cF^{(n)}=\overline{\cF_1+\dots+\cF_n}$. Since $\cF_1+\dots+\cF_n\subseteq\cF_1+\dots+\cF_{n+1}$, we have $\cF^{(n)}\subseteq\cF^{(n+1)}$. Hence, the corresponding rank functions are increasing pointwise. Set $\cS\coloneqq\bigcup_{n=1}^{\infty}\cF^{(n)}$. By \cref{prop:increasing_limit_rank}, $(J,\Sigma,\mu,\cS)$ is a measurable pre-matroid, and $(J,\Sigma,\mu,\overline{\cS})$ is a measurable matroid.
		
		We claim that $\overline{\cS}=\overline{\cR}$. The family $\cR$ is downward closed and therefore closed under countable intersections. Hence, $\overline{\cR}$ is $\sigma$-increasing by \cref{lem:upward_closure}. For every $n\in\bN$, we have $\cF_1+\dots+\cF_n\subseteq\cR$. Therefore, $\cF^{(n)}\subseteq\overline{\cR}$, and hence $\cS\subseteq\overline{\cR}$. Since $\overline{\cR}$ is $\sigma$-increasing, it follows that $\overline{\cS}\subseteq\overline{\cR}$.
		
		Conversely, let $R\in\cR$, and write $R=\bigcup_{i\in\bN}F_i$ with $F_i\in\cF_i$. For every $n\in\bN$, set $R_n\coloneqq\bigcup_{i=1}^nF_i$. Then $R_n\in\cF_1+\dots+\cF_n\subseteq\cF^{(n)}\subseteq\cS$, and $(R_n)_{n\in\bN}$ is increasing with union $R$. Hence, $R\in\overline{\cS}$. Thus $\cR\subseteq\overline{\cS}$. Since $\overline{\cS}$ is the independent-set family of a measurable matroid, it is $\sigma$-increasing, and therefore $\overline{\cR}\subseteq\overline{\cS}$. This proves the claim. 
		
		It remains to prove the approximation characterization. Suppose first that $I\in\overline{\cR}=\overline{\cS}$, and fix $\varepsilon>0$. By \cref{lem:upward_closure_approximation}, there exists $S\in\cS$ such that $S\subseteq I$ and $\mu(I\setminus S)<\varepsilon/2$. Since $S\in\cF^{(n)}$ for some $n\in\bN$, \cref{prop:finite_union_raw_description} gives pairwise disjoint sets $F_i\in\cF_i$, for $i\in[n]$, such that $\bigcup_{i=1}^nF_i\subseteq S$ and $\mu\left(S\setminus\bigcup_{i=1}^nF_i\right)<\varepsilon/2$. It follows that $\mu\left(I\setminus\bigcup_{i=1}^nF_i\right)\leq\mu(I\setminus S)+\mu\left(S\setminus\bigcup_{i=1}^nF_i\right)<\varepsilon$.
		
		Conversely, suppose that the stated approximation property holds. Given $\varepsilon>0$, choose a finite set $A\subseteq\bN$ and corresponding sets $F_i$. Choose $n\in\bN$ such that $A\subseteq[n]$, and set $F_i\coloneqq\emptyset$ for $i\in[n]\setminus A$. Then $\bigcup_{i\in A}F_i\in\cF_1+\dots+\cF_n\subseteq\cF^{(n)}\subseteq\cS$. Since such a set approximates $I$ from below for every $\varepsilon>0$, \cref{lem:upward_closure_approximation} gives $I\in\overline{\cS}=\overline{\cR}$.
		
		Finally, the approximation characterization is unchanged under every permutation of $(M_i)_{i\in\bN}$. Hence, the resulting measurable matroid is invariant under every permutation of $(M_i)_{i\in\bN}$. 
	\end{proof}
	
	We call $(J,\Sigma,\mu,\overline{\cR})$ the \emph{countable measurable union} of the matroids $M_i$, for $i\in\bN$, and denote it by $\bigvee_{i\in\bN}M_i$. The rank function of such matroids is determined in \cref{thm:countable_union_rank_formula}.
	
	%%%%%%%%%%%%%%%%%%%%%%%%%%%%%%%%
	\section{Connections to Infinite Submodularity}
	\label{sec:measurable}
	%%%%%%%%%%%%%%%%%%%%%%%%%%%%%%%%
	
	Before turning to the intersection and union theorems, we relate measurable matroids to the theory of submodular set functions on $\sigma$-algebras developed by Lovász~\cite{lovasz2023submodular}. We first discuss the Choquet extension, the measurable analogue of the Lovász extension, and then turn to minorizing measures, which provide infinite analogues of independent-set and base polytopes.
	
	%%%%%%%%%%%%%%%%%%%%%%%%%%%%%%%%
	\subsection{Linear Extension}
	\label{sec:linear}
	%%%%%%%%%%%%%%%%%%%%%%%%%%%%%%%%
	
	The finite Lovász extension and its measurable analogue, the Choquet extension, are instances of the same construction. We will refer to both as the \emph{Choquet--Lovász extension}. We first recall the finite version in a slightly more general form. Let $E$ be a finite set, and let $\varphi\colon 2^E\to\bR$ be a set function with $\varphi(\emptyset)=0$. For $w\in\bR^E$, there is a unique representation
	\[
	w=c\mathbf 1_E+\sum_{i=1}^k a_i\mathbf 1_{A_i},
	\]
	where $c\in\bR$, $a_i>0$, and $\emptyset\subsetneq A_1\subsetneq\dots\subsetneq A_k\subsetneq E$,
	with $A_1,\dots,A_k$ being the nontrivial upper level sets of $w$. Here, $k=0$ is allowed when $w$ is constant. The \emph{Lovász extension} of $\varphi$ is defined by
	\begin{equation}\label{eq:finite_lovasz_extension}
		\widehat\varphi(w)
		\coloneqq
		c\varphi(E)+\sum_{i=1}^k a_i\varphi(A_i).
	\end{equation}
	It extends $\varphi$, since $\widehat\varphi(\mathbf 1_X)=\varphi(X)$ for every $X\subseteq E$. It is linear on each cone determined by an ordering of the coordinates of $w$. 
	
	If $\varphi$ is submodular, then $\widehat\varphi$ is convex~\cite{lovasz1983submodular}; equivalently, it is subadditive and positively homogeneous. In particular, for every $A_1,\dots,A_m\subseteq E$ and every $a_1,\dots,a_m\in\bR_{\geq 0}$,
	\begin{equation}\label{eq:finite_generalized_submodular}
		\widehat\varphi\left(\sum_{i=1}^m a_i\mathbf 1_{A_i}\right)
		\leq
		\sum_{i=1}^m a_i\varphi(A_i).
	\end{equation}
	We will refer to \eqref{eq:finite_generalized_submodular} as the \emph{generalized submodular inequality}. For $m=2$ and $a_1=a_2=1$, it gives the usual submodular inequality, since $\mathbf 1_{A_1}+\mathbf 1_{A_2}$ has upper level sets $A_1\cup A_2$ and $A_1\cap A_2$. The form \eqref{eq:finite_generalized_submodular} is often useful when several sets have to be handled simultaneously; for example, it appears naturally in strengthenings of the basis-exchange axiom and in list-coloring theorems for matroids~\cite{schrijver2003combinatorial}.
	
	We now consider the measurable setting; see~\cite{lovasz2023submodular} for a detailed discussion. Let $(J,\Sigma)$ be a standard Borel space, and let $\varphi\colon\Sigma\to\bR$ be a bounded submodular set function with $\varphi(\emptyset)=0$. As in~\cite{lovasz2023submodular}, we use the shorthand $\{f\ge t\}$ for the level set $\{x\in J\mid f(x)\ge t\}$. For a bounded measurable function $w\colon J\to\bR$, choose $c\leq \inf{w}$, and define
	\begin{equation}\label{eq:measurable_lovasz_extension}
		\widehat\varphi(w)
		\coloneqq
		\int_c^\infty \varphi(\{w\geq t\})\diff t
		+
		c\varphi(J).
	\end{equation}
	The expression is well defined~\cite[Lemma~3.18]{lovasz2023submodular}, and its value is independent of the choice of $c$. This is the Choquet--Lovász extension of $\varphi$ in the measurable setting. If $\varphi$ is a measure, then $\widehat\varphi$ is the usual integral with respect to $\varphi$.
	
	It follows directly from the definition that $\widehat\varphi$ is positively homogeneous, that is,
	\[
	\widehat\varphi(af)=a\widehat\varphi(f)
	\]
	for every bounded measurable function $f\colon J\to\bR$ and every $a\in\bR_{\geq 0}$. In particular,
	\[
	\widehat\varphi(a\mathbf 1_X)=a\varphi(X)
	\]
	for every $X\in\Sigma$ and $a\in\bR_{\geq 0}$. Moreover, if $\varphi$ is increasing, then $\widehat\varphi$ is monotone: if $f\leq g$, then $\widehat\varphi(f)\leq\widehat\varphi(g)$.
	
	The following fundamental property of the Choquet--Lovász extension is due to Lovász~\cite[Section~4.3]{lovasz2023submodular}.
	
	\begin{prop}\label{prop:choquet_convex}
		Let $\varphi\colon\Sigma\to\bR$ be a bounded submodular set function with $\varphi(\emptyset)=0$. Then its Choquet--Lovász extension $\widehat\varphi$ is convex.
	\end{prop}
	
	Together with positive homogeneity, convexity implies that $\widehat\varphi$ is subadditive. This gives the following generalized form of the submodular inequality.
	
	\begin{cor}\label{cor:generalized_submodular_inequality}
		Let $\varphi\colon\Sigma\to\bR$ be a bounded submodular set function with $\varphi(\emptyset)=0$. Then, for every $A_1,\dots,A_m\in\Sigma$ and $a_1,\dots,a_m\in\bR_{\geq 0}$,
		\begin{equation}
			\widehat\varphi\left(\sum_{i=1}^m a_i\mathbf 1_{A_i}\right)
			\leq
			\sum_{i=1}^m a_i\varphi(A_i).\label{eq:generalized_submodular_inequality}
		\end{equation}
	\end{cor}
	
	\begin{proof}
		By convexity and positive homogeneity, $\widehat\varphi$ is subadditive. Hence,
		\[
		\widehat\varphi\left(\sum_{i=1}^m a_i\mathbf 1_{A_i}\right)
		\leq
		\sum_{i=1}^m\widehat\varphi(a_i\mathbf 1_{A_i})
		=
		\sum_{i=1}^m a_i\varphi(A_i).\qedhere
		\]
	\end{proof}
	
	We now discuss the interpretation for rank functions of measurable matroids. For finite matroids, this is the usual greedy algorithm. In the measurable setting one should be slightly careful: the value is still the correct supremum, but a greedy construction along a chain need not produce a basis. 
	
	\begin{prop}\label{prop:rank_lovasz_weight}
		Let $M=(J,\Sigma,\mu,\cF)$ be a measurable matroid with rank function $r$, let $\cB$ denote its family of bases, and let $\widehat r$ be the extension defined by \eqref{eq:measurable_lovasz_extension}. Then, for every bounded measurable function $w\colon J\to\bR$,
		\begin{equation}\label{eq:lovasz_basis_weight}
			\widehat r(w)=\sup_{B\in\cB}\int_B w\diff\mu.
		\end{equation}
	\end{prop}
	
	\begin{proof}
		First assume that $w\geq0$ takes only finitely many values. Then we can write
		$
		w=\sum_{i=1}^k a_i\mathbf 1_{A_i}
		$
		with $a_i>0$ and
		$
		\emptyset\subsetneq A_1\subsetneq\dots\subsetneq A_k\subseteq J.
		$
		Choose a maximal independent set $I_1$ in $A_1$. Having chosen $I_{i-1}$, extend it to a maximal independent set $I_i$ in $A_i$. Then $I_k$ is independent. Moreover, for every $i$, the set $I_k\cap A_i$ is an independent subset of $A_i$ containing $I_i$. Since $I_i$ is maximal in $A_i$, we have
		$
		\mu(I_k\cap A_i)=r(A_i).
		$
		Hence
		$
		\int_{I_k} w\diff \mu
		=
		\sum_i a_i r(A_i)
		=
		\widehat r(w).
		$
		Extend $I_k$ to a basis $B$. Since $w\geq0$, we have $\int_B w\diff\mu\geq\widehat r(w)$. Conversely, if $B'\in\cB$, then $\mu(B'\cap A_i)\leq r(A_i)$ for every $i$, and therefore $\int_{B'} w\diff \mu\leq\widehat r(w)$. This proves \eqref{eq:lovasz_basis_weight} for nonnegative functions taking finitely many values.
		
		Now let $w\geq0$ be an arbitrary bounded measurable function. For every $n$, choose nonnegative functions $u_n,v_n$ taking only finitely many values such that
		$
		u_n\leq w\leq v_n
		$
		and
		$
		\|v_n-u_n\|_\infty\leq 1/n.
		$
		By the case already proved,
		\[
		\widehat r(u_n)
		=
		\sup_{B\in\cB}\int_B u_n\diff \mu
		\leq
		\sup_{B\in\cB}\int_B w\diff \mu
		\leq
		\sup_{B\in\cB}\int_B v_n\diff \mu
		=
		\widehat r(v_n).
		\]
		Since $v_n-(1/n)\mathbf 1_J\leq u_n\leq v_n$, monotonicity and the translation property of $\widehat r$ give
		\[
		\widehat r(v_n)-\frac{r(J)}{n}
		=
		\widehat r\left(v_n-\frac{1}{n}\mathbf 1_J\right)
		\leq \widehat r(u_n)
		\leq \widehat r(v_n).
		\]
		Hence $0\leq \widehat r(v_n)-\widehat r(u_n)\leq r(J)/n$, and therefore $\widehat r(v_n)-\widehat r(u_n)\to0$ as $n\to\infty$. Since $u_n\leq w\leq v_n$, monotonicity also gives $\widehat r(u_n)\leq\widehat r(w)\leq\widehat r(v_n)$. Thus, \eqref{eq:lovasz_basis_weight} holds for every bounded measurable $w\geq0$.
		
		Finally, let $w$ be arbitrary, and choose $c\leq\inf w$. Then $w-c\mathbf 1_J\geq0$, and
		$
		\widehat r(w)=\widehat r(w-c\mathbf 1_J)+c\,r(J).
		$
		By the nonnegative case, and since every basis has measure $r(J)$,
		\[
		\widehat r(w)
		=
		\sup_{B\in\cB}\int_B(w-c\mathbf 1_J)\diff \mu+c\,r(J)
		=
		\sup_{B\in\cB}\int_B w\diff \mu,
		\]
		proving \eqref{eq:lovasz_basis_weight}.
	\end{proof}
	
	\begin{rem}
		For a finite matroid, the Lovász extension of the rank function has the familiar greedy interpretation, namely
		\[
		\widehat r(w)=\max_{B\in\cB}\sum_{e\in B}w(e).
		\]
		Indeed, after ordering the elements by decreasing weight, the upper level sets of $w$ form a chain, and the Lovász extension is determined by the successive rank increases along this chain, exactly as in the greedy algorithm.
		
		A measure-valued version of this interpretation survives in the measurable setting. Lovász proved that, for an increasing submodular set function that is continuous from above, every chain admits a minorizing measure that agrees with the set function on every member of the chain; see~\cite[Corollary~7.2]{lovasz2023submodular}. Applied to the chain of level sets of $w$, this gives a basic minorizing measure that is tight on these sets and hence realizes the Choquet--Lovász extension as the optimal fractional weight. This is the measure-theoretic analogue of the greedy construction in~\cite[Corollary~5.6]{lovasz2023submodular}. For a measurable matroid, however, such an optimal minorizing measure need not be induced by a basis, and, as the example below shows, the supremum over bases in \cref{prop:rank_lovasz_weight} need not be attained.
		
		For example, consider the interval-halving matroid from \cref{sec:nested}. Along the chain of initial intervals we have $r([0,t])=t/2$, so the unique minorizing measure that is tight on this chain is $\lambda/2$ by \cref{lem:halving_measure}. This is a basic minorizing measure, but it is not induced by any basis. More concretely, for $w(x)=1-x$,
		\[
		\widehat r(w)
		=
		\int_0^1 r([0,1-t])\diff t
		=
		\int_0^1\frac{1-t}{2}\diff t
		=
		\frac14.
		\]
		Hence, by \cref{prop:rank_lovasz_weight},
		\[
		\sup_{B\in\cB}\int_B w\diff\lambda=\frac14.
		\]
		If a basis $B$ attained this value, then
		\[
		\int_B(1-x)\diff\lambda
		=
		\int_0^1\lambda(B\cap[0,t])\diff t
		\leq
		\int_0^1\frac{t}{2}\diff t
		=
		\frac14.
		\]
		Equality implies $\lambda(B\cap[0,t])=t/2$ for almost every $t$. Since $t\mapsto\lambda(B\cap[0,t])$ is continuous, the equality holds for every $t$, contradicting \cref{lem:halving_measure}. Thus, the fractional greedy optimum is attained by a basic minorizing measure, while the corresponding supremum over bases need not be attained.
	\end{rem}

	%%%%%%%%%%%%%%%%%%%%%%%%%%%%%%%%
	\subsection{Extreme Points of Minorizing Measures}
	\label{sec:extreme_points}
	%%%%%%%%%%%%%%%%%%%%%%%%%%%%%%%%
	
	We now turn to the minorizing measures of the rank function. In the general theory of submodular functions, the compact convex set of basic minorizing charges is often studied through its exposed and extreme points. For measurable matroids, we show that the measures induced by independent sets are weak-star dense in $\mm(r)$ and that the measures induced by bases are weak-star dense in $\bmm(r)$.
	
	Let $M=(J,\Sigma,\mu,\cF)$ be a measurable matroid with rank function $r$. Recall that $\mm(r)$ and $\bmm(r)$ denote the sets of minorizing charges and basic minorizing charges of $r$, respectively. By \cref{cor:mm-measures}, every element of $\mm(r)$ is a finite measure dominated by $\mu$. Accordingly, throughout this section we refer to the elements of $\mm(r)$ and $\bmm(r)$ as minorizing measures and basic minorizing measures, respectively.
	
	By the Radon--Nikodym theorem, each $\alpha\in\mm(r)$ has a density $f_\alpha\coloneqq d\alpha/d\mu$; see, e.g., \cite[Chapter~3]{folland1999real}. Since $\alpha\leq\mu$, we have $0\leq f_\alpha\leq 1$ almost everywhere. Thus, we may identify $\mm(r)$ with a subset of the order interval $\{f\in L^\infty(\mu)\mid 0\leq f\leq 1\}$. We use the weak-star topology on $L^\infty(\mu)$ coming from the duality with $L^1(\mu)$. Thus, convergence in $\mm(r)$ means convergence of the corresponding densities against all functions in $L^1(\mu)$. 
	
	Whenever we speak of exposed points below, this is with respect to the weak-star topology. If $C\subseteq\mm(r)$ is convex, then $\alpha\in C$ is an \emph{exposed point} of $C$ if there exists $h\in L^1(\mu)$ such that
	\[
	\int_J h\diff\alpha>\int_J h\diff\beta
	\]
	for every $\beta\in C\setminus\{\alpha\}$. A basic weak-star neighbourhood of $\alpha$ is obtained by choosing $h_1,\dots,h_k\in L^1(\mu)$ and $\varepsilon>0$, and requiring
	\[
	\left|\int h_\ell\diff \beta-\int h_\ell\diff \alpha\right|<\varepsilon
	\]
	for every $\ell\in[k]$; see, e.g.,~\cite[Sections~3.8 and~3.14]{rudin1991functional}. 
	
	For $I\in\cF$, let $\mu_I$ denote the measure defined by $\mu_I(X)\coloneqq \mu(I\cap X)$. Then $\mu_I\in\mm(r)$, since $I\cap X$ is an independent subset of $X$. If $B$ is a basis, then $\mu_B\in\bmm(r)$.
	
	The next lemma shows that the rank inequalities are sufficient for realizing prescribed measures on the parts of a finite partition by an independent set.
	
	\begin{lem}\label{lem:finite_partition_realization}
		Let $M=(J,\Sigma,\mu,\cF)$ be a measurable matroid with rank function $r$, and let $P_1,\dots,P_q$ be a finite measurable partition of $J$. Let $\alpha_1,\dots,\alpha_q\in\bR_{\geq 0}$. Then there exists $F\in\cF$ such that $\mu(F\cap P_j)=\alpha_j$ for every $j\in[q]$ if and only if, for every $Q\subseteq[q]$,
		\begin{equation}\label{eq:finite_partition_realization}
			\sum_{j\in Q}\alpha_j
			\leq
			r\left(\bigcup_{j\in Q}P_j\right).
		\end{equation}
	\end{lem}
	
	\begin{proof}
		The necessity is immediate. Indeed, if such an independent set $F$ exists, then $F\cap \bigcup_{j\in Q}P_j$ is independent for every $Q\subseteq[q]$, and therefore
		$
		\sum_{j\in Q}\alpha_j
		=
		\mu\left(F\cap\bigcup_{j\in Q}P_j\right)
		\leq
		r\left(\bigcup_{j\in Q}P_j\right).
		$
		
		We prove the converse by induction on $q$. If some $\alpha_j$ is zero, then we delete the part $P_j$, apply the induction hypothesis to the restriction of $M$ to the union of the remaining parts, and take no element from $P_j$. Thus, we may assume that $\alpha_j>0$ for every $j$. Set $g\coloneqq \sum_{j=1}^q \alpha_j$. By replacing $M$ with its $g$-truncation, we may assume that $r(J)=g$. The inequalities in \eqref{eq:finite_partition_realization} remain valid for the truncated rank function $r_g(X)=\min\{r(X),g\}$, and every independent set in the truncation is independent in $M$. The case $q=1$ then follows from the definition of rank. Assume now that $q\geq 2$. We distinguish two cases.
		\medskip
		
		\noindent \textbf{Case 1.} There is a nonempty proper set $Q\subsetneq[q]$ such that
		$
		\sum_{j\in Q}\alpha_j
		=
		r\left(\bigcup_{j\in Q}P_j\right).
		$
		
		By induction, applied to the restriction of $M$ to $\bigcup_{j\in Q}P_j$, there exists an independent set $F_Q$ such that $\mu(F_Q\cap P_j)=\alpha_j$ for every $j\in Q$. Since $\mu(F_Q)=r\left(\bigcup_{j\in Q}P_j\right)$, the set $F_Q$ spans $\bigcup_{j\in Q}P_j$. Let $R=[q]\setminus Q$, and consider the contraction $M/F_Q$. Fix $S\subseteq R$, and set $U_Q\coloneqq \bigcup_{j\in Q}P_j$ and $U_S\coloneqq \bigcup_{j\in S}P_j$. Since $F_Q$ spans $U_Q$, we have
		\[
		r_{M/F_Q}(U_S)
		=
		r(F_Q\cup U_S)-\mu(F_Q)
		=
		r(U_Q\cup U_S)-r(U_Q).
		\]
		Applying \eqref{eq:finite_partition_realization} to $Q\cup S$ and using the tightness of $Q$, we obtain
		$
		\sum_{j\in S}\alpha_j
		\leq
		r(U_Q\cup U_S)-r(U_Q)
		=
		r_{M/F_Q}(U_S).
		$
		This proves that the inequalities required for applying the induction hypothesis hold in the contraction. By induction, there exists an independent set $F_R$ of $M/F_Q$ such that $\mu(F_R\cap P_j)=\alpha_j$ for every $j\in R$. Then $F_Q\cup F_R$ is independent in $M$ and has the required measures.
		\medskip
		
		\noindent \textbf{Case 2.} For every nonempty proper set $Q\subsetneq[q]$, we have $\sum_{j\in Q}\alpha_j<r\left(\bigcup_{j\in Q}P_j\right)$. 
		
		We work with the last part $P_q$. Let $\mathcal H$ be the family of all independent sets $H\subseteq P_q$ such that $\mu(H)\leq\alpha_q$ and, for every $S\subseteq[q-1]$,
		\begin{equation}\label{eq:H_feasible}
			\sum_{j\in S}\alpha_j
			\leq
			r\left(H\cup\bigcup_{j\in S}P_j\right)-\mu(H).
		\end{equation}
		The family $\mathcal H$ is nonempty, since $\emptyset\in\mathcal H$. It is also $\sigma$-increasing. Indeed, if $(H_n)_{n\in\bN}$ is increasing in $\mathcal H$ and $H=\bigcup_{n \in \bN} H_n$, then $H$ is independent, $\mu(H)\leq\alpha_q$, and \eqref{eq:H_feasible} follows from the continuity from below of $\mu$ and $r$. Hence, by \cref{lem:maximal}, $\mathcal H$ contains a maximal member $H$.
		
		If $\mu(H)=\alpha_q$, then \eqref{eq:H_feasible} says exactly that the remaining vector $(\alpha_1,\dots,\alpha_{q-1})$ satisfies the required inequalities in the contraction $M/H$. By induction, applied in $M/H$ to the parts $P_1,\dots,P_{q-1}$, there exists an independent set $F'$ of $M/H$ such that $\mu(F'\cap P_j)=\alpha_j$ for every $j<q$. Then $H\cup F'$ is independent in $M$ and has the required measures.
		
		We may therefore assume that $\mu(H)<\alpha_q$. Set $M'\coloneqq M/H$, and define residual targets $\alpha'_q\coloneqq \alpha_q-\mu(H)$ and $\alpha'_j\coloneqq \alpha_j$ for $j<q$. The vector $\alpha'$ satisfies the inequalities in \eqref{eq:finite_partition_realization} in $M'$. For sets not containing $q$, this is precisely \eqref{eq:H_feasible}. For sets containing $q$, it follows from the original inequalities, since $H\subseteq P_q$.
		
		We claim that, in $M'$, some nonempty proper subset of $[q]$ is tight with respect to the residual targets. Suppose not. Then every nonempty proper subset has positive slack. In particular, among the nonempty sets $S\subseteq[q-1]$, the minimum of
		$
		r_{M'}\left(\bigcup_{j\in S}P_j\right)
		-
		\sum_{j\in S}\alpha_j
		$
		is positive. Moreover, the inequality for the singleton $\{q\}$ gives $r_{M'}(P_q\setminus H)\geq \alpha'_q>0$. Thus, there exists an independent set $H'\subseteq P_q\setminus H$ in $M'$ of positive measure. Since $\mu$ is atomless, we may choose $H'$ so that $\mu(H')\leq\alpha'_q$ and smaller than the minimum slack above. Then $H\cup H'$ is independent in $M$ and has measure at most $\alpha_q$. For every nonempty $S\subseteq[q-1]$, monotonicity gives
		\[
		r\left(H\cup H'\cup\bigcup_{j\in S}P_j\right)-\mu(H)-\mu(H')
		\geq
		r\left(H\cup\bigcup_{j\in S}P_j\right)-\mu(H)-\mu(H')
		\geq
		\sum_{j\in S}\alpha_j .
		\]
		For $S=\emptyset$, the same inequality holds because $H\cup H'$ is independent. Hence, $H\cup H'$ belongs to $\mathcal H$, contradicting the maximality of $H$. This proves the claim.
		
		We can now apply the tight-set case, Case 1 above, in the contracted matroid $M'$ to the partition
		$
		P_1,\dots,P_{q-1}, P_q\setminus H
		$
		and to the residual targets $\alpha'_1,\dots,\alpha'_q$. We obtain an independent set $F'$ of $M'$ such that $\mu(F'\cap P_j)=\alpha_j$ for every $j<q$ and $\mu(F'\cap(P_q\setminus H))=\alpha'_q$. Then $H\cup F'$ is independent in $M$, and its intersection with each $P_j$ has measure $\alpha_j$.
	\end{proof}
	
	\begin{prop}\label{prop:minorizing_closure}
		The set $\mm(r)$ is the weak-star closure of
		$
		\{\mu_I\mid I\in\cF\}.
		$
	\end{prop}
	
	\begin{proof}
		The set $\mm(r)$ is weak-star closed. Indeed, the order interval $0\leq f\leq 1$ in $L^\infty(\mu)$ is weak-star closed. For every $X\in\Sigma$, the map $\alpha\mapsto\alpha(X)$ is weak-star continuous, since it is evaluation against $\mathbf 1_X\in L^1(\mu)$. Hence, the inequalities $\alpha(X)\leq r(X)$ define weak-star closed halfspaces, and $\mm(r)$ is the intersection of the order interval with these halfspaces.
		
		We have already observed that $\mu_I\in\mm(r)$ for every $I\in\cF$. It remains to show that every $\alpha\in\mm(r)$ lies in the weak-star closure of these measures. Let $h_1,\dots,h_k\in L^1(\mu)$ and let $\varepsilon>0$. By the density of simple functions in $L^1(\mu)$~\cite[Proposition~6.7]{folland1999real}, we may choose simple functions $s_1,\dots,s_k$ such that $\|h_\ell-s_\ell\|_1<\varepsilon$ for every $\ell\in[k]$. By taking a common refinement of their level-set partitions, we may assume that there is a finite measurable partition $P_1,\dots,P_q$ of $J$ such that every $s_\ell$ is constant on each part $P_j$.
		
		Set $\alpha_j\coloneqq \alpha(P_j)$ for $j\in[q]$. Since $\alpha\in\mm(r)$, for every $Q\subseteq[q]$ we have
		\[
		\sum_{j\in Q}\alpha_j
		=
		\alpha\left(\bigcup_{j\in Q}P_j\right)
		\leq
		r\left(\bigcup_{j\in Q}P_j\right).
		\]
		By \cref{lem:finite_partition_realization}, there exists $I\in\cF$ such that $\mu(I\cap P_j)=\alpha_j$ for every $j\in[q]$. Thus, $\mu_I$ and $\alpha$ agree on every part of the partition, and therefore
		$
		\int s_\ell\diff \mu_I=\int s_\ell\diff \alpha
		$
		for every $\ell\in[k]$. Since both $\mu_I$ and $\alpha$ are dominated by $\mu$, we get
		\[
		\left|\int h_\ell\diff \mu_I-\int h_\ell\diff \alpha\right|
		\leq
		2\|h_\ell-s_\ell\|_1
		<
		2\varepsilon
		\]
		for every $\ell\in[k]$. This proves that every weak-star neighbourhood of $\alpha$ contains some $\mu_I$ with $I\in\cF$.
	\end{proof}
	
	\begin{cor}\label{cor:basic_minorizing_closure}
		The set $\bmm(r)$ is the weak-star closure of
		$
		\{\mu_B\mid B \text{ is a basis of } M\}.
		$
	\end{cor}
	
	\begin{proof}
		Let $\alpha\in\bmm(r)$. In the proof of \cref{prop:minorizing_closure}, apply \cref{lem:finite_partition_realization} 
		with $\alpha_j=\alpha(P_j)$. The independent set $I$ obtained there satisfies
		$
		\mu(I)=\sum_j\alpha_j=\alpha(J)=r(J).
		$
		Hence, $I$ is a basis. The same approximation argument therefore uses only measures of the form $\mu_B$, where $B$ is a basis.
		
		The reverse inclusion follows from $\mu_B\in\bmm(r)$ for every basis $B$, and from the weak-star closedness of $\bmm(r)$.
	\end{proof}
	
	\begin{rem}\label{rem:weak_star_compactness}
		The order interval
		$
		\{f\in L^\infty(\mu)\mid 0\leq f\leq 1\}
		$
		is weak-star compact and weak-star sequentially compact. Indeed, it is weak-star closed in the closed unit ball of
		$L^\infty(\mu)=(L^1(\mu))^*$, and hence is weak-star compact by the Banach--Alaoglu theorem~\cite[Theorem~V.3.1]{conway1990course}. Moreover, since $L^1(\mu)$ is separable, the weak-star topology on the closed unit ball is metrizable~\cite[Theorem~V.5.1]{conway1990course}. Thus, compactness and sequential compactness agree on this ball.
		
		As observed in the proof of \cref{prop:minorizing_closure}, $\mm(r)$ is weak-star closed. Moreover, $\bmm(r)$ is weak-star closed in $\mm(r)$, since the map $\alpha\mapsto\alpha(J)$ is weak-star continuous. Consequently, both $\mm(r)$ and $\bmm(r)$ are weak-star compact and weak-star sequentially compact.
	\end{rem}
	
	Thus every basic minorizing measure is a weak-star limit of base measures. This does not mean, however, that every extreme point of $\bmm(r)$ is itself a base measure.
	
	For brevity, write $\bmm(M)\coloneqq\bmm(r)$. Every base measure is an
	extreme point of $\bmm(M)$. Indeed, suppose that $\mu_B=(\alpha+\beta)/2$ for some $\alpha,\beta\in\bmm(M)$. The densities of $\alpha$ and $\beta$ with respect to $\mu$ take values in $[0,1]$, while their average is $\mathbf 1_B$. It follows that both densities are equal to
	$1$ on $B$ and to $0$ outside $B$. Hence $\alpha=\beta=\mu_B$. The converse is false. For the interval-halving matroid introduced in \cref{sec:nested} and illustrated in \cref{fig:halving1}, there is an extreme basic minorizing measure that is not a base measure.
	
	\begin{prop}
		Let $N$ be the interval-halving matroid, and let $\nu$ 
		be the measure defined by $\nu(X) = \lambda(X)/2$ for all $X \in \Sigma$. Then $\nu$ is an extreme point of $\bmm(N)$.
	\end{prop}
	\begin{proof}
		Let $r$ denote the rank function of $N$. We first show that $\nu\in\bmm(N)$. Let $X\in\Sigma$; we need to show that $r(X)\geq\nu(X)=\lambda(X)/2$. It suffices to find a set $X'\subseteq X$ with $X'\in\cF$ and $\lambda(X')=\lambda(X)/2$.
		
		Define $f(t)\coloneqq\lambda(X\cap[0,t])$ for $t\in[0,1]$. If $0\leq s\leq t\leq1$, then $0\leq f(t)-f(s)=\lambda(X\cap(s,t])\leq t-s$, so $f$ is continuous. Let $c\coloneqq\sup\{t\in[0,1]\mid f(t)\leq\lambda(X)/2\}$. Since $f(0)=0$, $f(1)=\lambda(X)$, and $f$ is continuous, we have $f(c)=\lambda(X)/2$. Moreover, $f(c)\leq c$, and hence $c\geq\lambda(X)/2$.
		
		Set $X'\coloneqq X\cap[c,1]$. Then we have $\lambda(X')=\lambda(X)-f(c)=\lambda(X)/2$. To verify that $X'\in\cF$, let $t\in[0,1]$. If $t\leq c$, then $\lambda(X'\cap[0,t])=0\leq t/2$. If $t\geq c$, then
		\[
		\lambda(X'\cap[0,t])
		=
		f(t)-f(c)
		=
		f(t)-\frac{\lambda(X)}{2}
		\leq
		\frac{f(t)}{2}
		\leq
		\frac{t}{2}.
		\]
		Thus, $X'\in\cF$, proving that $r(X)\geq\nu(X)$.
		
		Finally, suppose that $\nu=(\alpha+\beta)/2$ for some $\alpha,\beta\in\bmm(N)$. For every $t\in[0,1]$, we have
		\[
		\alpha([0,t])\leq r([0,t])=\frac{t}{2},
		\qquad
		\beta([0,t])\leq r([0,t])=\frac{t}{2}.
		\]
		Since their average is $\nu([0,t])=t/2$, equality holds in both inequalities. Thus, $\alpha([0,t])=\beta([0,t])=t/2$ for every $t\in[0,1]$. By \cref{lem:halving_measure}, $\alpha=\beta=\nu$. Therefore, $\nu$ is an extreme point of $\bmm(N)$.
	\end{proof}
	
	%%%%%%%%%%%%%%%%%%%%%%%%%%%%%%%%
	\section{Measurable Matroid Intersection and Union Theorems}
	\label{sec:intun}
	%%%%%%%%%%%%%%%%%%%%%%%%%%%%%%%%
	
	This section proves the measurable analogues of the two central min--max theorems of finite matroid optimization: matroid intersection and matroid union. The intersection theorem may also be viewed as an integral counterpart of Lovász's fractional intersection theorem~\cite[Theorem~5.15]{lovasz2023submodular}. Lovász's theorem gives the same min--max value over nonnegative measures minorized by both rank functions, whereas we optimize over restrictions $\mu|_F$ to common measurable independent sets. Thus, there is no integrality gap, although the integral supremum need not be attained. In fact, for measurable matroid rank functions, Lovász's theorem follows from ours: if $(F_i)_{i\in\bN}$ is a sequence of common independent sets whose measures converge to the supremum, then a weak-star limit point of the measures $\mu|_{F_i}$ is minorized by both rank functions and attains the same optimal value.
	
	The central and technically deepest result of the paper is the intersection theorem, stated in \cref{thm:intersection}. Its proof is the most technical argument in the paper, and we believe that it is difficult to follow without further motivation. We therefore devote \cref{sec:intersection_guide} to explaining the connections, motivations, and main ideas behind the proof. Although this overview makes the formal argument substantially easier to follow, no result depends on it, so readers interested only in the precise proof may safely skip it. In \cref{sec:preparations}, we present the technical lemmas on which the proof is based, while \cref{sec:intersection} contains the proof of the intersection theorem. In \cref{sec:expansion_attainment}, we give an expansion condition under which the supremum in the intersection theorem is attained. Finally, in \cref{sec:union_rank}, we use the intersection theorem to prove our other main result, the union theorem.
	
	Recall that all equalities and inclusions between measurable sets are understood in the measure algebra unless explicitly stated otherwise.
	
	%%%%%%%%%%%%%%%%%%%%%%%%%%%%%%%%
	\subsection{Outline of the Proof}
	\label{sec:intersection_guide}
	%%%%%%%%%%%%%%%%%%%%%%%%%%%%%%%%
	
	The main statement of this section, proved in \cref{thm:intersection}, is the following. If $M_1=(J,\Sigma,\mu,\cF_1)$ and $M_2=(J,\Sigma,\mu,\cF_2)$ are measurable matroids with rank functions $r_1$ and $r_2$, then
	\[
	\sup\{\mu(F)\mid F\in\cF_1\cap\cF_2\}
	=
	\min_{X\in\Sigma}\{r_1(X)+r_2(J\setminus X)\}.
	\]
	The expression on the right is indeed a minimum by \cref{lem:minimum}. It is easy to see that it is an upper bound on the measure of any common independent set. The theorem establishes the reverse inequality: for every $\varepsilon>0$, there exists a common independent set whose measure is within $\varepsilon$ of the right-hand side. However, the supremum on the left need not be attained, as shown in \cref{rem:intersection_sup_not_attained}.
	
	This distinction is already important in the most basic application, namely measurable bipartite matchings. Approximate Hall-type statements, which produce matchings leaving a set of arbitrarily small measure uncovered, follow from the local-improvement method of Elek and Lippner~\cite{elek2009borel} and are now considered standard. Finding a matching that covers all vertices up to a null set is a genuinely stronger and much more delicate problem. Hall's condition alone does not suffice, even for regular acyclic graphings, as shown by Laczkovich~\cite{laczkovich1988closed} (see also \cref{rem:laczkovich_matching}) and Kun~\cite{kun2021measurable}. For higher even degrees, Conley and Kechris~\cite[Section~6]{conley2013measurable} constructed $d$-regular bipartite Borel graphs with no Borel perfect matching for every even $d\geq 4$. Positive results typically impose additional structure, most notably expansion~\cite{lyons2011perfect,grabowski2022measurable}. Consequently, the supremum form of the intersection theorem captures the correct general min--max value, but the attainment question remains essential.
	
	The arguments below build on the classical augmenting-path proof for finite bipartite matchings~\cite{frobenius1917zerlegbare,konig1931graphok,hall1935representatives}. Let $G=(S\cup T,E)$ be a finite bipartite graph and let $M$ be a matching. Starting from the unmatched vertices of $S$, one follows paths whose edges alternate between $E\setminus M$ and $M$. If such a path reaches an unmatched vertex of $T$, then switching the status of every edge along the path increases the size of the matching by one. If no such path exists, the sets reachable by alternating paths encode an obstruction explaining why the matching cannot be enlarged. This obstruction yields the dual certificates underlying Hall's theorem and Kőnig's min--max theorem.
	
	Edmonds' finite matroid intersection algorithm~\cite{edmonds1970submodular} follows the same general strategy, although a matroid does not initially come with any local structure. Let $M_1=(E,\cF_1)$ and $M_2=(E,\cF_2)$ be finite matroids, and let $F\in\cF_1\cap\cF_2$. The missing local structure is manufactured from $F$ by an auxiliary directed graph. Set 
	\[
	S_1\coloneqq\{x\in E\setminus F\mid F+x \in\cF_1\},
	\qquad
	S_2\coloneqq\{x\in E\setminus F\mid F+x\in\cF_2\}.
	\]
	For $x\in E\setminus F$ and $y\in F$, put an arc from $x$ to $y$ if $F-y+x\in\cF_2$, and an arc from $y$ to $x$ if $F-y+x\in\cF_1$. If the auxiliary graph contains a directed path
	$
	x_0,y_1,x_1,\dots,y_n,x_n
	$
	from $S_1$ to $S_2$, where $x_i\notin F$ and $y_i\in F$, then a shortest such path is augmenting: a standard fundamental-circuit argument shows that
	$
	\widehat F
	=
	F\setminus\{y_1,\dots,y_n\}\cup\{x_0,\dots,x_n\}
	$ 
	is independent in both matroids and has cardinality $|F|+1$. If no directed path from $S_1$ to $S_2$ exists, let $R$ be the set of vertices reachable from $S_1$. It is not difficult to prove that
	\[
	r_1(E\setminus R)+r_2(R)
	=
	|F\setminus R|+|F\cap R|
	=
	|F|,
	\]
	which is exactly the dual certificate in the finite matroid intersection theorem, namely
	\[
	\max\{|I| \mid I \in \cF_1 \cap \cF_2\}
	=
	\min_{X \subseteq E}\{r_1(X)+r_2(E\setminus X)\}.
	\]
	
	Edmonds' algorithm does not generalize easily to the measurable setting. First, (A) there is no evident way to construct a local structure for measurable matroids: the finite argument uses the possible exchanges through fundamental circuits, while our axiomatic framework provides no corresponding circuit structure. Second, (B) exchanging a single element has no effect; an augmentation must exchange measurable sets and produce a positive net gain. Third, (C) even if one can always find a common independent set of slightly larger measure, simply repeating this step would not be enough: the sequence of sets obtained need not converge almost everywhere to a common independent set.
	
	The first key observation addresses (A). Given two measurable matroids and a common independent set $F$, although we cannot construct a measurable exchange graph, the existence of alternating paths of bounded length can still be expressed using the closure operators. In the finite setting, one can describe recursively, using closure, the sets of elements that can be connected to $S_1$ by an alternating path of a given length. We imitate this construction in the measurable setting and denote the resulting sets by
	$
	A_0,F_1,A_1,F_2,A_2,\dots.
	$
	These sets should be viewed as the measurable analogues of the reachability layers in the finite exchange graph: they record the existence of alternating paths of bounded length without requiring us to construct the paths themselves or any underlying local graph structure. Moreover, if $\mu(F)$ is smaller than the right-hand side of the intersection formula, then this recursive construction must terminate after finitely many steps.
	
	The layers alone do not yet give an augmentation, so (B) remains. The proof of \cref{lem:short_augmentation} works backwards through the layers and constructs pairwise disjoint positive-measure sets
	$X_n,Y_n,X_{n-1},Y_{n-1},\dots,\allowbreak Y_1,X_0$.
	The sets $X_i$ lie outside $F$, the sets $Y_i$ lie inside $F$, with $\mu(X_i)=\mu(Y_i)$ for all $i\in[n]$, and $\left(F\cup\bigcup_{i=0}^nX_i\right)\setminus\bigcup_{i=1}^nY_i$ is a common independent set of measure strictly larger than $\mu(F)$. Conceptually, this is a measurable augmenting chain that plays the role of a single path and resolves (B).
	
	For iteration, merely obtaining a larger set is not enough, leaving (C) open. The same issue already appears for measurable matchings. Elek and Lippner showed that, for every fixed $q$, one can construct a Borel matching with no augmenting path containing at most $q$ matched edges~\cite{elek2009borel}. Under Hall's condition, a standard argument implies that at most a proportion of order $1/q$ of the vertices can remain unmatched. This yields the standard approximate Hall theorem. Their construction, however, relies on the local structure of graphs, and hence does not extend directly to matroids. Attainment requires an additional convergence mechanism. A key idea in the Lyons--Nazarov argument~\cite{lyons2011perfect} is to control how much the matching changes during each improvement. Under expansion, the total measure of these changes can be made summable. The Borel--Cantelli lemma then implies that almost every edge changes its status only finitely many times, so the sequence has a measurable limiting matching. Our proof uses the same idea.
	
	The raw measurable augmenting chain constructed above may replace a large portion of $F$ in exchange for a very small gain, so it does not yet provide the control required in (C). The second key step is therefore a controlled thinning of the chain. After the raw chain has been constructed backwards, the proof chooses, in the forward order, subsets
	$
	X'_0,Y'_1,X'_1,\dots,Y'_n,X'_n
	$
	having the same nonzero measure such that
	\[
	\widehat F
	\coloneqq
	\Big(F\cup\bigcup_{i=0}^nX'_i\Big)\setminus\bigcup_{i=1}^nY'_i.
	\]
	Consequently, the total amount changed is bounded by the length of the augmenting chain times the actual gain. 
	
	Fix $L\in\bN$. As long as the current common independent set is more than $\mu(J)/L$ below the right-hand side of the intersection formula, there exists an augmenting chain of length at most $L$. The small-change estimate allows us to iterate these augmentations transfinitely, using the Borel--Cantelli lemma at countable limit stages to obtain an almost-everywhere limit that remains common independent. Since the measure strictly increases at every successor stage, the process must terminate at a countable stage, yielding a common independent set with no augmenting chain of length at most $L$. Such a set is within $\mu(J)/L$ of the right-hand side of the intersection formula. Letting $L$ tend to infinity yields the measurable matroid intersection theorem.
	
	%%%%%%%%%%%%%%%%%%%%%%%%%%%%%%%%
	\subsection{Auxiliary Lemmas}
	\label{sec:preparations}
	%%%%%%%%%%%%%%%%%%%%%%%%%%%%%%%%
	
	We first collect the general tools that will be used repeatedly in the construction.
	
	The next lemma allows us to identify, inside an independent set $F$, the smallest part that is needed to span a prescribed set $A$.
	
	\begin{lem} \label{lem:closure_intersection}
		Let $M=(J, \Sigma, \mu, \cF)$ be a measurable matroid with closure operator $\clo$. Let $F_1,F_2 \in \cF$ and $A \in \Sigma$ be such that $F_1 \cup F_2 \in \cF$ and $A \subseteq \clo(F_1) \cap \clo(F_2)$. Then $A \subseteq \clo(F_1 \cap F_2)$.
	\end{lem}
	
	\begin{proof}
		Let $r$ be the rank function of $M$. Since $A \subseteq \clo(F_i)$ for $i=1,2$, we have $r(F_i\cup A)=r(F_i)=\mu(F_i)$. Also, $A\subseteq \clo(F_1\cup F_2)$, and $F_1\cup F_2$ is independent by assumption. Applying submodularity to $F_1\cup A$ and $F_2\cup A$ gives
		\[
		\mu(F_1)+\mu(F_2)
		=
		r(F_1\cup A)+r(F_2\cup A)
		\geq
		r((F_1\cup F_2)\cup A)+r((F_1\cap F_2)\cup A)
		=
		\mu(F_1\cup F_2)+r((F_1\cap F_2)\cup A).
		\]
		Hence, $r((F_1\cap F_2)\cup A)\leq \mu(F_1\cap F_2)$. Since $F_1\cap F_2$ is independent, the reverse inequality follows from monotonicity. Thus, $r((F_1\cap F_2)\cup A)=r(F_1\cap F_2)$, which means exactly that $A\subseteq \clo(F_1\cap F_2)$.
	\end{proof}
	
	The next lemma allows us to define an operation $\widehat{\clo}$ which, in a sense, reverses closure by selecting the smallest part of an independent set needed to span a prescribed set.
	
	\begin{lem} \label{lem:inverse_closure}
		Let $M=(J, \Sigma, \mu, \cF)$ be a measurable matroid with closure operator $\clo$. Let $F \in \cF$ and $A \in \Sigma$ be such that $A \subseteq \clo(F)$. Then the family
		\[
		\cG \coloneqq \{F'\subseteq F \mid A\subseteq \clo(F')\}
		\]
		has a unique minimal member $G$, for which $G \subseteq X$ for all $X \in \cG$.
	\end{lem}
	
	\begin{proof}
		The family $\cG$ is nonempty, since $F\in\cG$. If $F_1,F_2\in\cG$, then $F_1\cap F_2\in\cG$ by \cref{lem:closure_intersection}, because $F_1,F_2$ and $F_1\cup F_2$ are all independent as subsets of $F$.
		
		We claim that $\cG$ is $\sigma$-decreasing. Let $(F_i)_{i\in\bN}$ be a decreasing sequence in $\cG$, and set $F_0\coloneqq\bigcap_{i=1}^\infty F_i$. Since $A\subseteq\clo(F_i)$, we have $r(F_i\cup A)=r(F_i)=\mu(F_i)$ for every $i\in\bN$. Hence, $r(F_0\cup A)\leq r(F_i\cup A)=\mu(F_i)$ for every $i$. Letting $i\to\infty$ gives $r(F_0\cup A)\leq\mu(F_0)$. The reverse inequality follows from $F_0\subseteq F_0\cup A$ and the independence of $F_0$. Thus, $r(F_0\cup A)=r(F_0)$, and so $A\subseteq\clo(F_0)$. Hence, $F_0\in\cG$, proving the claim.
		
		The family $\cG$ is therefore $\sigma$-decreasing and closed under finite intersections. By \cref{cor:unionclosed}, the lemma follows.
	\end{proof}

	For $F\in\cF$ and $A\in\Sigma$ with $A\subseteq\clo(F)$, we denote the unique minimal member from \cref{lem:inverse_closure} by $\widehat{\clo}(F,A)$.
	
	The next lemma is the basic simultaneous augmentation step. It says that if a positive-measure set lies outside the closure of several independent sets, then a common positive-measure part of it can be added to all of them at once.
	
	\begin{lem} \label{lem:extend_outside_closure}
		Let $M_i=(J, \Sigma, \mu, \cF_i)$ be measurable matroids for $i\in[n]$, with closure operators $\clo_i$. Let $F_i\in\cF_i$ for $i\in[n]$, and let $Y\in\Sigma$ satisfy $\mu(Y)>0$ and $\mu(Y\cap \clo_i(F_i))=0$ for every $i\in[n]$. Then there exists $Y'\subseteq Y$ with $\mu(Y')>0$ such that $F_i\cup Y'\in\cF_i$ for every $i\in[n]$.
	\end{lem}
	
	\begin{proof}
		We construct sets $Y=Y_0\supseteq Y_1\supseteq \dots \supseteq Y_n$ with positive measure such that $F_i\cup Y_j\in\cF_i$ for all $1\leq i\leq j$. Suppose that $Y_0,\dots,Y_k$ have already been defined for some $0\leq k<n$. Since $\mu(Y_k)>0$ and $\mu(Y_k\cap \clo_{k+1}(F_{k+1}))=0$, we have $Y_k\not\subseteq \clo_{k+1}(F_{k+1})$. Hence, $r_{k+1}(F_{k+1}\cup Y_k)>r_{k+1}(F_{k+1})$.
		
		Extend $F_{k+1}$ to a maximal independent set $F'_{k+1}$ in $F_{k+1}\cup Y_k$, and set $Y_{k+1}\coloneqq F'_{k+1}\setminus F_{k+1}$. Then $\mu(Y_{k+1})>0$, $Y_{k+1}\subseteq Y_k$, and $F_{k+1}\cup Y_{k+1}\in\cF_{k+1}$. Moreover, for $1\leq i\leq k$, we have $F_i\cup Y_{k+1}\subseteq F_i\cup Y_k\in\cF_i$, so $F_i\cup Y_{k+1}\in\cF_i$ by \ref{ax:i2}. Thus, the induction continues.
		
		Taking $Y'\coloneqq Y_n$ completes the proof.
	\end{proof}
	
	We will also use the following consequence of the characterization of contraction.
	
	\begin{lem}\label{lem:contraction}
		Let $M=(J,\Sigma,\mu,\cF)$ be a measurable matroid, and let $Z_0\subseteq Z_1$ be measurable sets. Let $K_1,K_2\subseteq Z_1\setminus Z_0$. Assume that there exists a maximal set $I_0\in\cF|Z_0$ such that $I_0\cup K_1$ and $I_0\cup K_2$ are maximal sets in $\cF|Z_1$. Then, for every maximal set $I\in\cF|Z_0$ and every $H\subseteq J\setminus Z_1$, we have $I\cup K_1\cup H\in\cF$ if and only if $I\cup K_2\cup H\in\cF$.
	\end{lem}
	
	\begin{proof}
		Let $N\coloneqq M/Z_0$. We first show that $K_1$ and $K_2$ are maximal independent sets of $N|(Z_1\setminus Z_0)$. By \cref{prop:contraction_char}, both $K_1$ and $K_2$ are independent in $N$, since $I_0\cup K_1$ and $I_0\cup K_2$ are independent in $M$. If, say, $K_1$ were not maximal in $N|(Z_1\setminus Z_0)$, then there would be an independent set $K\in N|(Z_1\setminus Z_0)$ with $K_1\subsetneq K$. Applying \cref{prop:contraction_char} again, $I_0\cup K$ would be independent in $M$ and contained in $Z_1$, contradicting the maximality of $I_0\cup K_1$ in $\cF|Z_1$. The same argument applies to $K_2$.
		
		Now let $H\subseteq J\setminus Z_1$. Since $K_1$ and $K_2$ are maximal independent sets of $N|(Z_1\setminus Z_0)$, \cref{prop:contraction_char} applied to the contracted matroid $N/(Z_1\setminus Z_0)$ gives that $K_1\cup H$ is independent in $N$ if and only if $H$ is independent in $N/(Z_1\setminus Z_0)$, and this holds if and only if $K_2\cup H$ is independent in $N$. Thus, $K_1\cup H$ is independent in $N$ if and only if $K_2\cup H$ is independent in $N$. Finally, applying \cref{prop:contraction_char} to $N$, we get that, for every maximal set $I\in\cF|Z_0$ and for $j=1,2$, the set $K_j\cup H$ is independent in $N$ if and only if $I\cup K_j\cup H$ is independent in $M$. Therefore, $I\cup K_1\cup H\in\cF$ if and only if $I\cup K_2\cup H\in\cF$.
	\end{proof}
	
	Finally, we discuss a consequence of submodularity and continuity from above and below that was communicated to us by Lovász~\cite{lovasz2026personal}. We include the proof for completeness.
	
	\begin{lem} \label{lem:minimum}
		Let $\varphi \colon \Sigma \to \bR$ be a bounded submodular set function on a $\sigma$-algebra $(J,\Sigma)$ that is continuous from below and continuous from above. Then $\varphi$ attains its minimum; that is, there exists $X\in\Sigma$ such that $\varphi(X)\leq\varphi(Y)$ for all $Y\in\Sigma$.
	\end{lem}
	
	\begin{proof}
		Let $m\coloneqq \inf\{\varphi(X)\mid X\in\Sigma\}$, and choose sets $X_i\in\Sigma$ such that $\varphi(X_i)<m+\varepsilon_i$, where $\varepsilon_i>0$ and $\sum_{i=1}^\infty \varepsilon_i<\infty$.
		
		For $k\leq N$, set $D_{k,N}\coloneqq \bigcap_{i=k}^N X_i$. We claim that $\varphi(D_{k,N})\leq m+\sum_{i=k}^N\varepsilon_i$. This is clear for $N=k$. If it holds for $N-1$, then submodularity gives
		\[
		\varphi(D_{k,N})+\varphi(D_{k,N-1}\cup X_N)
		\leq
		\varphi(D_{k,N-1})+\varphi(X_N).
		\]
		Since $\varphi(D_{k,N-1}\cup X_N)\geq m$, the induction hypothesis implies $\varphi(D_{k,N})\leq m+\sum_{i=k}^N\varepsilon_i$.
		
		Now set $D_k\coloneqq\bigcap_{i=k}^\infty X_i$. By continuity from above, $\varphi(D_k)=\lim_{N\to\infty}\varphi(D_{k,N})$, and hence $\varphi(D_k)\leq m+\sum_{i=k}^\infty\varepsilon_i$. The sets $D_k$ form an increasing sequence. Let $X\coloneqq\bigcup_{k=1}^\infty D_k$. By continuity from below, $\varphi(X)=\lim_{k\to\infty}\varphi(D_k)=m$. Thus, $X$ attains the minimum.
	\end{proof}
	
	%%%%%%%%%%%%%%%%%%%%%%%%%%%%%%%%
	\subsection{Intersection Theorem}
	\label{sec:intersection}
	%%%%%%%%%%%%%%%%%%%%%%%%%%%%%%%%
	
	We now prove the measurable analogue of Edmonds' matroid intersection theorem. The main step is an augmenting-chain lemma. It shows that whenever a common independent set has measure strictly below the optimum value, an alternating construction produces a common independent set of slightly larger measure. Moreover, the amount by which the set changes can be controlled in terms of the length of the chain. The alternating layers used in the proof are illustrated in \cref{fig:XY}.
	
	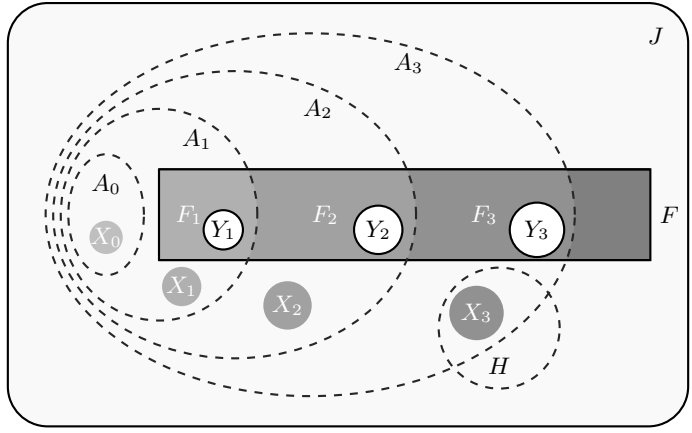
\begin{figure}[t]
		\centering
		
		\begin{tikzpicture}[scale=1]
			
			% 1. Groundset J
			\draw[groundset] (-0.5, -2.8) rectangle (8.5, 2.8);
			\node[label J] at (8.3, 2.6) {\small $J$};
			
			% 2. Set F (Horizontal strip)
			\draw[set F] (1.5, -0.6) rectangle (8, 0.6);
			\node[label F] at (8, 0) {\small $F$};
			
			% 3. Intersections F_i = A_i \cap F (Shaded in progressive gray steps)
			\begin{scope}
				\clip (1.5, -0.6) rectangle (8, 0.6); 
				
				\fill[f3color] (3.5, 0) ellipse (3.5cm and 2.4cm);
				\fill[f2color] (2.5, 0) ellipse (2.4cm and 1.9cm);
				\fill[f1color] (1.5, 0) ellipse (1.3cm and 1.4cm);
			\end{scope}
			
			% Redraw top and bottom edges of F over fills
			\draw[edge F] (1.5, -0.6) -- (8, -0.6);
			\draw[edge F] (1.5, 0.6) -- (8, 0.6);
			\draw[edge F] (1.5, -0.6) -- (1.5, 0.6);

			% 4. Sets A_i (Uniform style across all set boundaries)
			\draw[set boundary] (3.5, 0) ellipse (3.5cm and 2.4cm);
			\node[label A] at (4.8, 2) {\small $A_3$};
			
			\draw[set boundary] (2.5, 0) ellipse (2.4cm and 1.9cm);
			\node[label A] at (3.6, 1.4) {\small $A_2$};
			
			\draw[set boundary] (1.5, 0) ellipse (1.3cm and 1.4cm);
			\node[label A] at (2, 1) {\small $A_1$};
			
			\draw[set boundary] (0.8, 0) ellipse (0.5cm and 0.8cm);
			\node[label A] at (0.8, 0.4) {\small $A_0$};
			
			\draw[set boundary] (6, -1.5) ellipse (0.8cm and 0.8cm);
			\node[label A] at (6, -2) {\small $H$};

			% 5. Intersection Labels
			\node[label Fi] at (1.9, 0) {$F_1$};
			\node[label Fi] at (3.7, 0) {$F_2$};
			\node[label Fi] at (5.8, 0) {$F_3$};

			% X_0 \subseteq A_0
			\fill[f0color] (0.8, -0.3) ellipse (0.22cm and 0.22cm);
			\node[label X] at (0.8, -0.3) {\small $X_0$};
			
			% X_1 \subseteq A_1 \ (A_0 \cup F)
			\fill[f1color] (1.8, -0.95) ellipse (0.26cm and 0.26cm);
			\node[label X] at (1.8, -0.95) {\small $X_1$};
			
			% X_2 \subseteq A_2 \ (A_1 \cup F)
			\fill[f2color] (3.2, -1.2) ellipse (0.32cm and 0.32cm);
			\node[label X] at (3.2, -1.2) {\small $X_2$};
			
			% X_3 \subseteq A_3 \cap H
			\fill[f3color] (5.7, -1.3) ellipse (0.36cm and 0.36cm);
			\node[label X] at (5.7, -1.3) {\small $X_3$};
			
			% Y_1 \subseteq F_1
			\draw[set Y] (2.35, -0.2) ellipse (0.26cm and 0.26cm);
			\node[label Y] at (2.35, -0.2) {\small $Y_1$};
			
			% Y_2 \subseteq F_2 \ F_1
			\draw[set Y] (4.4, -0.2) ellipse (0.32cm and 0.32cm);
			\node[label Y] at (4.4, -0.2) {\small $Y_2$};
			
			% Y_3 \subseteq F_3 \ F_2
			\draw[set Y] (6.5, -0.2) ellipse (0.36cm and 0.36cm);
			\node[label Y] at (6.5, -0.2) {\small $Y_3$};
			
		\end{tikzpicture}
		\caption{Illustration of the alternating layers for $n=3$. The sets $X_i$ are added and the sets $Y_i$ are removed in the backward construction.}
		\label{fig:XY}
	\end{figure}
	
	\begin{lem}\label{lem:short_augmentation}
		Let $M_1=(J,\Sigma,\mu,\cF_1)$ and $M_2=(J,\Sigma,\mu,\cF_2)$ be measurable matroids with rank functions $r_1$ and $r_2$ and closure operators $\clo_1$ and $\clo_2$. Let $z\coloneqq \min_{X\in\Sigma}\{r_1(X)+r_2(J\setminus X)\}$. Let $F\in\cF_1\cap\cF_2$, and suppose that $\delta_F\coloneqq z-\mu(F)>0$. Then there exist an integer $n\leq \mu(J)/\delta_F$, a number $\eta>0$, and a set $\widehat F\in\cF_1\cap\cF_2$ such that $\mu(\widehat F)=\mu(F)+\eta$ and $\mu(\widehat F\triangle F)\leq (2n+1)\eta$.
	\end{lem}
	
	\begin{proof}
		We first construct the alternating layers. Let $A_0\coloneqq J\setminus\clo_1(F)$ and let $H\coloneqq J\setminus\clo_2(F)$. If $\mu(A_0\cap H)>0$, then by \cref{lem:extend_outside_closure} there exists a set $X_0\subseteq A_0\cap H$ with $\mu(X_0)>0$ such that $F\cup X_0\in\cF_1\cap\cF_2$. In this case, take $n=0$, set $\eta\coloneqq\mu(X_0)$, and let $\widehat F\coloneqq F\cup X_0$. Then $\mu(\widehat F)=\mu(F)+\eta$ and $\mu(\widehat F\triangle F)=\eta$, so the conclusion follows.
		
		Hence, we may assume that $\mu(A_0\cap H)=0$. Set $F_0\coloneqq\emptyset$. Suppose that $A_0,\dots,A_k$ and $F_0,\dots,F_k$ have already been defined and that $\mu(A_k\cap H)=0$. Since $\mu(A_k\cap H)=0$, we have $A_k\subseteq\clo_2(F)$. Define
		\[
		F_{k+1}\coloneqq \widehat{\clo}_2(F,A_k)
		\quad\text{and}\quad
		A_{k+1}\coloneqq J\setminus\clo_1(F\setminus F_{k+1}),
		\]
		where $\widehat{\clo}_2$ is defined after \cref{lem:inverse_closure}.

		As long as the construction continues, we have $F_k\subseteq F_{k+1}$ and $A_k\subseteq A_{k+1}$. Indeed, $F_0\subseteq F_1$ and $A_0\subseteq A_1$. If $A_{k-1}\subseteq A_k$, then every subset of $F$ spanning $A_k$ also spans $A_{k-1}$, so the minimality of $F_k=\widehat{\clo}_2(F,A_{k-1})$ gives $F_k\subseteq F_{k+1}$. Hence, $F\setminus F_{k+1}\subseteq F\setminus F_k$, and therefore $A_k\subseteq A_{k+1}$.
		
		For reference, \cref{tab:notation} summarizes the notation used in the proof. The table is meant as a guide while reading the proof; the sets are introduced in context below.
		
		\begin{table}[t]
			\centering
			\footnotesize
			\renewcommand{\arraystretch}{1.7}
			\begin{tabularx}{0.98\textwidth}{@{}>{\centering\arraybackslash}p{0.14\textwidth}X@{}}
				\toprule
				\textbf{Set} & \textbf{Definition or defining property} \\
				\midrule
				$A_0,H,F_0$
				& $A_0=J\setminus\clo_1(F)$, $H=J\setminus\clo_2(F)$, and $F_0=\emptyset$. \\
				
				$F_{k+1},A_{k+1}$
				& $F_{k+1}=\widehat{\clo}_2(F,A_k)$ and $A_{k+1}=J\setminus\clo_1(F\setminus F_{k+1})$. \\
				
				$X_k$
				& Positive-measure sets added in the alternating chain: $X_0\subseteq A_0$, $X_n\subseteq A_n\cap H$, and $X_k\subseteq A_k\setminus(A_{k-1}\cup F)$ for $1\leq k<n$. \\
				
				$Y_k$
				& Sets removed from $F_k\setminus F_{k-1}$, with $\mu(Y_k)=\mu(X_k)$. \\
				
				$W_m$
				& $\left(F\cup\bigcup_{i=m}^n X_i\right)\setminus\bigcup_{i=m+1}^n Y_i$. \\
				
				$V_m$
				& $\left(F\cup\bigcup_{i=m}^n X_i\right)\setminus\left(F_m\cup\bigcup_{i=m+1}^n Y_i\right)$. \\
				
				$W_m^+$
				& $\left(F\cup\bigcup_{i=m}^n X_i\right)\setminus\bigcup_{i=m}^n Y_i$. \\
				
				$V_m^+$
				& $\left(F\cup\bigcup_{i=m}^n X_i\right)\setminus\left(F_{m-1}\cup\bigcup_{i=m}^n Y_i\right)$. \\
				
				$X'_k,Y'_k$
				& Thinned subsets of $X_k,Y_k$, respectively, each of measure $\eta$. \\
				
				$U_m$
				& $\left(F\cup\bigcup_{i=0}^{m}X'_i\cup\bigcup_{i=m+1}^nX_i\right)\setminus\left(\bigcup_{i=1}^{m}Y'_i\cup\bigcup_{i=m+1}^nY_i\right)$. \\
				
				$U_m^+$
				& $\left(F\cup\bigcup_{i=0}^{m}X'_i\cup\bigcup_{i=m+1}^nX_i\right)\setminus\left(\bigcup_{i=1}^{m+1}Y'_i\cup\bigcup_{i=m+2}^nY_i\right)$. \\
				
				$V_{n+1}^+$
				& $F\setminus F_n$, extending the sequence $(V_m^+)_{m=1}^n$. \\
				
				$\widehat F$
				& $\left(F\cup\bigcup_{i=0}^nX'_i\right)\setminus\bigcup_{i=1}^nY'_i$. \\
				\bottomrule
			\end{tabularx}
			\caption{Notation used in the proof of \cref{lem:short_augmentation}. The sets $A_k,F_k,X_k,Y_k$ describe the unthinned alternating chain; the primed sets and $U_m,U_m^+$ are used in the thinning step.}
			\label{tab:notation}
		\end{table}
		
		\begin{cla}\label{cla:layers}
			We have $\mu(F_{k+1}\setminus F_k)\geq\delta_F$ for every $k$ for which the construction has not yet stopped.
		\end{cla}
		
		\begin{claimproof}
			By definition, $J\setminus A_k=\clo_1(F\setminus F_k)$, and hence $r_1(J\setminus A_k)=\mu(F\setminus F_k)$. Since $r_1(J\setminus A_k)+r_2(A_k)\geq z$, we get
			\[
			r_2(A_k)\geq z-\mu(F\setminus F_k)=\mu(F_k)+\delta_F.
			\]
			As $F_{k+1}$ spans $A_k$ in $M_2$, it follows that $\mu(F_{k+1})\geq r_2(A_k)\geq\mu(F_k)+\delta_F$, proving the claim.
		\end{claimproof}
		
		Thus, the construction must stop after finitely many steps. Indeed, if $\mu(A_k\cap H)=0$ for every $k$, then \cref{cla:layers} gives $\mu(F_k)\geq k\delta_F$ for all $k$, which is impossible for $k>\mu(J)/\delta_F$. Let $n$ be the first index such that $\mu(A_n\cap H)>0$. Then $n\leq\mu(J)/\delta_F$.
		
		We now construct positive-measure sets $X_n,Y_n,X_{n-1},Y_{n-1},\dots,Y_1,X_0$ in this order, so that the following properties hold:
		\begin{enumerate}[label=\normalfont{(\arabic*)}, left=0pt, itemsep=0em]
			\item $X_n \subseteq A_n \cap H$, $X_0 \subseteq A_0$, and $X_k \subseteq A_k \setminus (A_{k-1}\cup F)$ for $1 \leq k < n$, \label{int:1}
			\item $Y_k \subseteq F_k \setminus F_{k-1}$ for $1 \leq k \leq n$, \label{int:2}
			\item $\mu(X_k)=\mu(Y_k)$ for $1 \leq k \leq n$, \label{int:3}
			\item $\left(F \cup \bigcup_{i=k}^n X_i\right) \setminus \bigcup_{i=k+1}^n Y_i \in \cF_2$ for $0 \leq k \leq n$, \label{int:4}
			\item $\left(F \cup \bigcup_{i=k}^n X_i\right) \setminus \left(F_k \cup \bigcup_{i=k+1}^n Y_i\right) \in \cF_1$ for $0 \leq k \leq n$. \label{int:5}
		\end{enumerate}
		Empty unions are interpreted as $\emptyset$. 
		
		To define $X_n$, observe that $A_n\cap H$ has positive measure and is disjoint from both $\clo_1(F\setminus F_n)$ and $\clo_2(F)$. Since $n\geq1$, the minimality of $n$ gives $\mu(A_{n-1}\cap H)=0$, and hence $(A_n\cap H)\setminus A_{n-1}$ has positive measure. Applying \cref{lem:extend_outside_closure} to $(A_n\cap H)\setminus A_{n-1}$ gives a set
		\[
		X_n\subseteq (A_n\cap H)\setminus A_{n-1}
		\]
		with positive measure such that $(F\cup X_n)\setminus F_n\in\cF_1$ and $F\cup X_n\in\cF_2$. Thus, \ref{int:1} holds for $X_n$, and \ref{int:4} and \ref{int:5} hold for $k=n$.
		
		Next, suppose that $X_n,Y_n,X_{n-1},Y_{n-1},\dots,Y_{m+1},X_m$ have already been defined for some $1\leq m\leq n$, and that all conditions involving only these sets hold. For brevity, let
		\[
		W_m\coloneqq \left(F \cup \bigcup_{i=m}^n X_i\right) \setminus \bigcup_{i=m+1}^n Y_i
		\]
		and
		\[
		V_m\coloneqq \left(F \cup \bigcup_{i=m}^n X_i\right) \setminus \left(F_m \cup \bigcup_{i=m+1}^n Y_i\right).
		\]
		By the induction hypothesis, $V_m\in\cF_1$ and $W_m\in\cF_2$. Since $V_m\subseteq J\setminus A_{m-1}=\clo_1(F\setminus F_{m-1})$, we have
		\[
		\mu(V_m)=r_1(V_m)\leq r_1(J\setminus A_{m-1})=\mu(F\setminus F_{m-1}).
		\]
		Applying \ref{ax:i3''} in $M_1$ to $V_m$ and $F\setminus F_{m-1}$, we obtain a set $Z\in\cF_1$ with $V_m\subseteq Z\subseteq V_m\cup(F\setminus F_{m-1})$ and $\mu(Z)\geq\mu(F\setminus F_{m-1})$.
		
		\begin{cla} \label{cla:first_exchange}
			The set $Z$ satisfies
			\[
			\mu(Z)=\mu(F\setminus F_{m-1})
			\quad\text{and}\quad
			Z\setminus V_m\subseteq F_m\setminus F_{m-1}.
			\]
		\end{cla}
		
		\begin{claimproof}
			Since $V_m\subseteq J\setminus A_{m-1}=\clo_1(F\setminus F_{m-1})$, and $F\setminus F_{m-1}$ also lies in this closed set, the set $Z$ lies in $J\setminus A_{m-1}$. Hence
			\[
			\mu(Z)\leq r_1(J\setminus A_{m-1})=\mu(F\setminus F_{m-1}).
			\]
			Together with the inequality given by \ref{ax:i3''}, this gives $\mu(Z)=\mu(F\setminus F_{m-1})$.
			
			It remains to locate the new part of $Z$. By \ref{int:3}, we have $\mu(V_m\setminus X_m)=\mu(F\setminus F_m)$. Moreover, $V_m\setminus X_m\subseteq J\setminus A_m=\clo_1(F\setminus F_m)$. Thus, $V_m\setminus X_m$ is a maximal independent set in the restriction $M_1|(J\setminus A_m)$. By maximality, $(Z\setminus V_m)\cap(J\setminus A_m)=\emptyset$. Since $Z\subseteq V_m\cup(F\setminus F_{m-1})$, it follows that $Z\setminus V_m\subseteq F_m\setminus F_{m-1}$.
		\end{claimproof}
		
		Define $Y_m\coloneqq (F_m\setminus F_{m-1})\setminus Z$. Then \ref{int:2} holds for $m$. Since $\mu(V_m)=\mu(F\setminus F_m)+\mu(X_m)$, \cref{cla:first_exchange} gives $\mu(Z\setminus V_m)=\mu(F_m\setminus F_{m-1})-\mu(X_m)$, and hence $\mu(Y_m)=\mu(X_m)$. Thus, \ref{int:3} holds for $m$.
		
		Define
		\[
		W^+_m\coloneqq \left(F \cup \bigcup_{i=m}^n X_i\right) \setminus \bigcup_{i=m}^n Y_i
		\]
		and
		\[
		V^+_m\coloneqq Z=\left(F \cup \bigcup_{i=m}^n X_i\right) \setminus \left(F_{m-1} \cup \bigcup_{i=m}^n Y_i\right).
		\]
		Clearly, $V^+_m\in\cF_1$, and $W^+_m\subseteq W_m\in\cF_2$, so $W^+_m\in\cF_2$ by \ref{ax:i2}.
		
		\begin{cla}\label{cla:available_region}
			There exists a positive-measure set $R$ which is disjoint from both $\clo_1(V^+_m)$ and $\clo_2(W^+_m)$, and such that $R\subseteq A_{m-1}\setminus(A_{m-2}\cup F)$ for $m\geq2$, while $R\subseteq A_0$ for $m=1$.
		\end{cla}
		
		\begin{claimproof}
			By \cref{cla:first_exchange}, we have $\mu(V^+_m)=\mu(F\setminus F_{m-1})$. Since $V^+_m=Z\subseteq V_m\cup(F\setminus F_{m-1})$ and both pieces lie in $J\setminus A_{m-1}$, we have $V^+_m\subseteq J\setminus A_{m-1}$. Hence, \cref{lem:closure_same} gives
			\[
			\clo_1(V^+_m)=\clo_1(F\setminus F_{m-1})=J\setminus A_{m-1}.
			\]
			
			We first show that $A_{m-1}$ is not spanned by $W^+_m$ in $M_2$. Suppose, to the contrary, that $A_{m-1}\subseteq\clo_2(W^+_m)$. By definition, $F_m$ also spans $A_{m-1}$ in $M_2$. Since $W^+_m\cup F_m=W_m\in\cF_2$, \cref{lem:closure_intersection} implies that $F_m\cap W^+_m$ spans $A_{m-1}$ in $M_2$. But $Y_m\subseteq F_m\setminus W^+_m$ and $\mu(Y_m)=\mu(X_m)>0$, so $F_m\cap W^+_m\subsetneq F_m$, contradicting the minimality of $F_m=\widehat{\clo}_2(F,A_{m-1})$. Thus, $A_{m-1}\not\subseteq\clo_2(W^+_m)$, and so $A_{m-1}\setminus\clo_2(W^+_m)$ has positive measure.
			
			On the other hand, $F_{m-1}\subseteq W^+_m$, and, when $m\geq2$, the definition of $F_{m-1}$ gives $A_{m-2}\subseteq\clo_2(F_{m-1})\subseteq\clo_2(W^+_m)$. Therefore, the part of $A_{m-1}\setminus\clo_2(W^+_m)$ lying in $A_{m-2}\cup F_{m-1}$ has measure zero. Moreover, $F\setminus F_{m-1}\subseteq\clo_1(F\setminus F_{m-1})=J\setminus A_{m-1}$, so $(F\setminus F_{m-1})\cap A_{m-1}$ has measure zero. It follows that the set
			\[
			R\coloneqq
			\begin{cases}
				A_{m-1}\setminus(A_{m-2}\cup F\cup\clo_2(W^+_m)), & \text{if } m\geq2,\\
				A_0\setminus\clo_2(W^+_1), & \text{if } m=1.
			\end{cases}
			\]
			has positive measure. By construction, $R$ is disjoint from $\clo_2(W^+_m)$; and since $R\subseteq A_{m-1}$, it is also disjoint from $\clo_1(V^+_m)=J\setminus A_{m-1}$.
		\end{claimproof}
		
		Applying \cref{lem:extend_outside_closure} to the set $R$ from \cref{cla:available_region}, with $V^+_m$ in $M_1$ and $W^+_m$ in $M_2$, we obtain a set $X_{m-1}\subseteq R$ with $\mu(X_{m-1})>0$ such that $V^+_m\cup X_{m-1}\in\cF_1$ and $W^+_m\cup X_{m-1}\in\cF_2$. This proves \ref{int:1}, \ref{int:4}, and \ref{int:5} for $m-1$.
		
		This completes the construction of the unthinned alternating chain. The set
		\[
		\left(F\cup\bigcup_{i=0}^n X_i\right)\setminus\bigcup_{i=1}^n Y_i
		\]
		is common independent and has measure larger than $\mu(F)$. 
		
		To obtain the bound in the statement of the lemma, we now thin the chain so that all exchanged pieces have the same measure. The thinning is carried out inductively. At each step, we first choose $Y'_{m+1}$ so that the resulting set remains independent in $M_2$, and then choose $X'_{m+1}$ so that common independence is restored.
		
		Choose $\eta>0$ such that $\eta\leq\mu(X_i)$ for $0\leq i\leq n$, and $\eta\leq\mu(Y_i)$ for $1\leq i\leq n$. We construct sets $X'_0,Y'_1,X'_1,\dots,Y'_n,X'_n$ in this order, so that the following properties hold:
		\begin{enumerate}[label=\normalfont{(\arabic*')}, left=0pt, itemsep=0em]
			\item $X'_k \subseteq X_k$ for $0 \leq k \leq n$, \label{int':1}
			\item $Y'_k \subseteq Y_k$ for $1 \leq k \leq n$, \label{int':2}
			\item $\mu(X'_0)=\eta$ and $\mu(X'_k)=\mu(Y'_k)=\eta$ for $1 \leq k \leq n$, \label{int':3}
			\item $\left(F \cup \bigcup_{i=0}^{k-1} X'_i \cup \bigcup_{i=k}^n X_i\right) \setminus \left(\bigcup_{i=1}^k Y'_i \cup \bigcup_{i=k+1}^n Y_i\right) \in \cF_2$ for $0 \leq k \leq n$, \label{int':4}
			\item $\left(F \cup \bigcup_{i=0}^{k} X'_i \cup \bigcup_{i=k+1}^n X_i\right) \setminus \left(\bigcup_{i=1}^k Y'_i \cup \bigcup_{i=k+1}^n Y_i\right) \in \cF_1 \cap \cF_2$ for $0 \leq k \leq n$. \label{int':5}
		\end{enumerate}
		
		Again, empty unions are interpreted as $\emptyset$. Choose $X'_0\subseteq X_0$ with $\mu(X'_0)=\eta$. Then \ref{int':1} and \ref{int':3} hold for $X'_0$. Moreover, \ref{int':4} holds for $k=0$, since the corresponding set is the common independent set obtained from the unthinned chain, and \ref{int':5} holds for $k=0$, since the corresponding set is a subset of it.
		
		Suppose that $X'_0,Y'_1,X'_1,\dots,Y'_m,X'_m$ have already been defined for some $0\leq m<n$, and that all relevant conditions above involving only these sets hold. Let
		\[
		U_m \coloneqq \left(F \cup \bigcup_{i=0}^{m} X'_i \cup \bigcup_{i=m+1}^n X_i\right) \setminus \left(\bigcup_{i=1}^m Y'_i \cup \bigcup_{i=m+1}^n Y_i\right).
		\]
		By the induction hypothesis, $U_m\in\cF_1\cap\cF_2$.
		
		\begin{cla} \label{cla:thin_y}
			There exists $Y'_{m+1}\subseteq Y_{m+1}$ with $\mu(Y'_{m+1})=\eta$ such that \ref{int':2}, \ref{int':3}, and \ref{int':4} hold for $m+1$.
		\end{cla}
		
		\begin{claimproof}
			Using \ref{int:3}, we have $\mu(W_{m+1})=\mu(F)+\mu(X_{m+1})$, while the already thinned equalities give $\mu(U_m)=\mu(F)+\eta$. Thus, $\mu(U_m)\leq\mu(W_{m+1})$. Applying \ref{ax:i3''} to $U_m$ and $W_{m+1}$ in $M_2$, we obtain $Z\in\cF_2$ with $U_m\subseteq Z\subseteq U_m\cup W_{m+1}$ and $\mu(Z)\geq\mu(W_{m+1})$. Since $F_{m+1}\subseteq W_{m+1}$ and $F_{m+1}$ spans $A_m$ in $M_2$, the set $W_{m+1}$ spans $A_m$ as well. The part of $U_m$ not already contained in $W_{m+1}$ is contained in $A_m$, and hence $U_m\subseteq\clo_2(W_{m+1})$. Thus, $Z\subseteq\clo_2(W_{m+1})$, which implies $\mu(Z)\leq\mu(W_{m+1})$. Therefore, $\mu(Z)=\mu(W_{m+1})$ and $\mu(Z\setminus U_m)=\mu(X_{m+1})-\eta$.
			
			We next show that $Z\setminus U_m\subseteq Y_{m+1}$. Indeed,
			\[
			Z\setminus U_m
			\subseteq
			W_{m+1}\setminus U_m
			=
			Y_{m+1}\cup\bigcup_{i=1}^mY_i'.
			\]
			If $m=0$, this proves the containment. Assume that $m>0$. Inside $F_m\cup A_{m-1}$, the set $U_m$ is obtained from $F_m$ by removing $Y'_1,\dots,Y'_m$ and adding $X'_0,\dots,X'_{m-1}$, while $X'_m$ and all later $X_i$ lie outside $F_m\cup A_{m-1}$. By \ref{int':3}, the removed and added parts both have total measure $m\eta$, and hence
			\[
			\mu\left(U_m\cap(F_m\cup A_{m-1})\right)=\mu(F_m).
			\]
			Since $F_m$ spans $A_{m-1}$ in $M_2$, we have $r_2(F_m\cup A_{m-1})=\mu(F_m)$. Thus, $U_m\cap(F_m\cup A_{m-1})$ is a maximal independent set of $M_2|(F_m\cup A_{m-1})$.
			
			Now $Z\cap(F_m\cup A_{m-1})$ is independent and contains $U_m\cap(F_m\cup A_{m-1})$. By maximality, the two sets are equal. Since
			\[
			\bigcup_{i=1}^mY_i'\subseteq F_m\subseteq F_m\cup A_{m-1}
			\]
			and $\bigcup_{i=1}^mY_i'$ is disjoint from $U_m$, it follows that
			\[
			\mu\left((Z\setminus U_m)\cap\bigcup_{i=1}^mY_i'\right)=0.
			\]
			Together with the displayed identity above, this gives $Z\setminus U_m\subseteq Y_{m+1}$.
			
			Define $Y'_{m+1}\coloneqq Y_{m+1}\setminus Z$. By definition, $Y'_{m+1}\subseteq Y_{m+1}$, proving \ref{int':2} for $m+1$. Since $Z\setminus U_m\subseteq Y_{m+1}$, the set in \ref{int':4} for $m+1$ is exactly $Z$, and hence belongs to $\cF_2$. Moreover, using the containment and \ref{int:3}, we get
			\[
			\mu(Y'_{m+1})=\mu(Y_{m+1})-\mu(Z\setminus U_m)=\mu(X_{m+1})-(\mu(X_{m+1})-\eta)=\eta,
			\]
			proving \ref{int':3} for $Y'_{m+1}$.
		\end{claimproof}
		
		We continue with defining $X'_{m+1}$. Let
		\[
		U^+_m \coloneqq \left(F \cup \bigcup_{i=0}^{m} X'_i \cup \bigcup_{i=m+1}^n X_i\right) \setminus \left(\bigcup_{i=1}^{m+1} Y'_i \cup \bigcup_{i=m+2}^n Y_i\right).
		\]
		
		\begin{cla} \label{cla:thin_cor_contraction}
			$U^+_m\setminus X_{m+1}\in\cF_1$.
		\end{cla}
		
		\begin{claimproof}
			It is enough to prove that
			\[
			T \coloneqq \left(F \cup \bigcup_{i=0}^{m} X'_i \cup \bigcup_{i=m+2}^n X_i\right) \setminus \left(\bigcup_{i=1}^{m} Y'_i \cup \bigcup_{i=m+2}^n Y_i\right)
			\]
			belongs to $\cF_1$, since $U^+_m\setminus X_{m+1}\subseteq T$.
			
			For convenience, extend the sequence $(V_j^+)_{j=1}^n$ by setting $V^+_{n+1}\coloneqq F\setminus F_n$. We apply \cref{lem:contraction} in $M_1$ with the closed sets $Z_0\coloneqq J\setminus A_{m+1}$ and $Z_1\coloneqq J\setminus A_m$. Since $A_m\subseteq A_{m+1}$, we have $Z_0\subseteq Z_1$. Let
			\[
			I\coloneqq V^+_{m+2},\quad
			K_1\coloneqq F_{m+1}\setminus F_m,\quad
			K_2\coloneqq V^+_{m+1}\setminus V^+_{m+2},
			\quad\text{and}\quad
			Q\coloneqq U_m\setminus V^+_{m+1}.
			\]
			By the first paragraph of the proof of \cref{cla:available_region}, $V_j^+$ is a maximal independent set in $M_1|(J\setminus A_{j-1})$ for every $1\leq j\leq n$. The same holds for $j=n+1$, since $J\setminus A_n=\clo_1(F\setminus F_n)$. Hence, $I=V^+_{m+2}$ is a maximal independent set in $M_1|Z_0$, while $I\cup K_2=V^+_{m+1}$ is a maximal independent set in $M_1|Z_1$.
			
			We first verify that $K_1,K_2\subseteq Z_1\setminus Z_0$. Since $I\cup K_2\subseteq Z_1$, we have $K_2\subseteq Z_1$. Moreover, $I\cup K_2$ is independent and contains the maximal independent set $I$ of $M_1|Z_0$, so $\mu(K_2\cap Z_0)=0$. Hence, $K_2\subseteq Z_1\setminus Z_0$.
			
			Similarly, $K_1\subseteq F\setminus F_m\subseteq Z_1$. The set $F\setminus F_{m+1}$ is maximal independent in $M_1|Z_0$, since $Z_0=\clo_1(F\setminus F_{m+1})$, and
			\[
			(F\setminus F_{m+1})\cup K_1=F\setminus F_m
			\]
			is independent. Therefore, $\mu(K_1\cap Z_0)=0$, and hence $K_1\subseteq Z_1\setminus Z_0$.
			
			We next show that $I\cup K_1$ is maximal independent in $M_1|Z_1$. Since $F\setminus F_{m+1}$ is maximal independent in $M_1|Z_0$ and $(F\setminus F_{m+1})\cup K_1$ is independent, \cref{prop:contraction_char} shows that $K_1$ is independent in $M_1/Z_0$. Applying \cref{prop:contraction_char} again, now with the maximal independent set $I$ of $M_1|Z_0$, gives $I\cup K_1\in\cF_1$. Furthermore,
			\[
			\mu(I\cup K_1)
			=
			\mu(F\setminus F_{m+1})+\mu(F_{m+1}\setminus F_m)
			=
			\mu(F\setminus F_m)
			=
			r_1(Z_1).
			\]
			Thus, $I\cup K_1$ is maximal independent in $M_1|Z_1$.
			
			Finally, $Q\subseteq J\setminus Z_1=A_m$, and $(I\cup K_2)\cup Q=U_m\in\cF_1$. We may therefore apply \cref{lem:contraction}, with $I$ playing the role of $I_0$ and $Q$ playing the role of $H$, to obtain
			\[
			I\cup K_1\cup Q
			=
			V^+_{m+2}\cup(F_{m+1}\setminus F_m)\cup(U_m\setminus V^+_{m+1})
			\in\cF_1.
			\]
			This set is precisely $T$, proving the claim.
		\end{claimproof}
		
		We have $\mu(U^+_m\setminus X_{m+1})=\mu(F)$ and $\mu(U_m)=\mu(F)+\eta$. Applying \ref{ax:i3''} to $U^+_m\setminus X_{m+1}$ and $U_m$ in $M_1$, we obtain a set $Z\in\cF_1$ with
		\[
		U^+_m\setminus X_{m+1}\subseteq Z\subseteq (U^+_m\setminus X_{m+1})\cup U_m
		\]
		and $\mu(Z)\geq\mu(U_m)$. Since $U_m\setminus(U^+_m\setminus X_{m+1})=X_{m+1}$, we have $Z\setminus(U^+_m\setminus X_{m+1})\subseteq X_{m+1}$ and this difference has measure at least $\eta$. Choose a subset $X'_{m+1}\subseteq Z\setminus(U^+_m\setminus X_{m+1})$ with $\mu(X'_{m+1})=\eta$. Then
		\[
		(U^+_m\setminus X_{m+1})\cup X'_{m+1}\subseteq Z,
		\]
		and hence this set belongs to $\cF_1$ by \ref{ax:i2}. Thus, \ref{int':1} and \ref{int':3} hold for $X'_{m+1}$.
		
		Set
		\[
		U_{m+1} \coloneqq \left(F \cup \bigcup_{i=0}^{m+1} X'_i \cup \bigcup_{i=m+2}^n X_i\right) \setminus \left(\bigcup_{i=1}^{m+1} Y'_i \cup \bigcup_{i=m+2}^n Y_i\right).
		\]
		By construction, $U_{m+1}\in\cF_1$, and $U_{m+1}\subseteq U^+_m\in\cF_2$. Thus, \ref{int':5} holds for $m+1$, and the induction continues.
		
		At the end, define
		\[
		\widehat F\coloneqq \left(F \cup \bigcup_{i=0}^n X'_i\right)\setminus \bigcup_{i=1}^n Y'_i.
		\]
		By \ref{int':5} with $k=n$, we have $\widehat F\in\cF_1\cap\cF_2$. The sets $X'_0,Y'_1,X'_1,\dots,Y'_n,X'_n$ are pairwise disjoint, while \ref{int':3} gives $\mu(X'_0)=\eta$ and $\mu(X'_i)=\mu(Y'_i)=\eta$ for $i\in[n]$. Hence, $\mu(\widehat F)=\mu(F)+\eta$. Moreover, the symmetric difference between $F$ and $\widehat F$ is contained in the union of these $2n+1$ sets, and therefore $\mu(\widehat F\triangle F)\leq(2n+1)\eta$.
	\end{proof}
	
	The next lemma identifies the limiting argument used both in the intersection theorem and in the attainment result. Its hypothesis says that, whenever the current gap is larger than $t$, the measure can be increased at a cost of at most $C(t)$ times the increase. Part~\ref{it:auga} shows that such improvements can be iterated until the gap is at most $t$, without losing this bound on the total change. Part~\ref{it:augb} then shows that the supremum is attained when these costs are summable along dyadic error bounds. In particular, the summability condition holds if $C(t)=O(1+\log(R/t))$ for sufficiently small $t$, since then $C(R2^{-i})=O(i)$.
	
	\begin{lem}\label{lem:summable_attainment}
		Let $\cA\subseteq\Sigma$ be a nonempty family, and let $R>0$ satisfy $\mu(F)\leq R$ for every $F\in\cA$. Suppose that $\cA$ is closed under taking limits, that is, whenever $(F_i)_{i\in\bN}$ is a sequence in $\cA$ for which $\lim_{i\to\infty} F_i$ exists, we have $\lim_{i\to\infty} F_i\in\cA$. Let $C\colon(0,R]\to\bR_{\geq0}$, and suppose that, for every $0<t\leq R$ and every $F\in\cA$ with $R-\mu(F)>t$, there exists $\widehat F\in\cA$ such that $\mu(\widehat F)>\mu(F)$ and
		\begin{equation}\label{eq:summable_local_augmentation}
			\mu(\widehat F\triangle F)\leq C(t)\left(\mu(\widehat F)-\mu(F)\right).
		\end{equation}
		Then the following statements hold.
		\begin{enumerate}[label=(\alph*)]\itemsep0em
			\item For every $0<t\leq R$ and every $F_0\in\cA$, there exists $G\in\cA$ such that $\mu(G)\geq\mu(F_0)$, $R-\mu(G)\leq t$, and $\mu(G\triangle F_0)\leq C(t)\left(\mu(G)-\mu(F_0)\right)$. \label{it:auga}
			\item If $\sum_{i=1}^{\infty}2^{-i}C(R2^{-i})<\infty$,
			then there exists $G\in\cA$ with $\mu(G)=R$. \label{it:augb}
		\end{enumerate}
	\end{lem}
	
	\begin{proof}
		Fix $0<t\leq R$ and $F_0\in\cA$. If $R-\mu(F_0)\leq t$, then we may take $G=F_0$. Otherwise, writing $\omega_1$ for the first uncountable ordinal, we construct a transfinite sequence $(F^\alpha)_{\alpha<\omega_1}$ in $\cA$, unless a set whose gap is at most $t$ is obtained earlier. We maintain
		\begin{equation}\label{eq:transfinite_change_bound}
			\mu(F^\gamma\triangle F^\beta)\leq C(t)\left(\mu(F^\gamma)-\mu(F^\beta)\right)
		\end{equation}
		whenever $\beta<\gamma$ and both sets have been defined. Start with $F^0=F_0$.
		
		Suppose that $F^\alpha$ has been defined and $R-\mu(F^\alpha)>t$. Choose $F^{\alpha+1}$ using the assumption. Then \eqref{eq:transfinite_change_bound} holds for $\beta=\alpha$ by \eqref{eq:summable_local_augmentation}. If $\beta<\alpha$, then the triangle inequality and the induction hypothesis give
		\begin{align*}
			\mu(F^{\alpha+1}\triangle F^\beta)
			&\leq \mu(F^{\alpha+1}\triangle F^\alpha)+\mu(F^\alpha\triangle F^\beta)\\
			&\leq C(t)\left(\mu(F^{\alpha+1})-\mu(F^\alpha)\right)+C(t)\left(\mu(F^\alpha)-\mu(F^\beta)\right)\\
			&=C(t)\left(\mu(F^{\alpha+1})-\mu(F^\beta)\right).
		\end{align*}
		Thus, \eqref{eq:transfinite_change_bound} is preserved at successor steps.
		
		Let $\lambda<\omega_1$ be a limit ordinal, and suppose that $F^\alpha$ has been defined for every $\alpha<\lambda$. Choose an increasing sequence $\alpha_1<\alpha_2<\dots$ cofinal in $\lambda$, that is, such that for every $\beta<\lambda$ there exists $j$ with $\beta<\alpha_j$. By \eqref{eq:transfinite_change_bound},
		\[
		\sum_{j=1}^{\infty}\mu(F^{\alpha_j}\triangle F^{\alpha_{j+1}})\leq C(t)\sum_{j=1}^{\infty}\left(\mu(F^{\alpha_{j+1}})-\mu(F^{\alpha_j})\right)\leq C(t)R<\infty.
		\]
		By \cref{cor:Borel_Cantelli}, the limit $F^\lambda\coloneqq\lim_{j\to\infty}F^{\alpha_j}$ exists and satisfies $\mu(F^\lambda)=\lim_{j\to\infty}\mu(F^{\alpha_j})$. Since $\cA$ is closed under taking limits, $F^\lambda\in\cA$.
		
		Moreover, as in the proof of \cref{cor:Borel_Cantelli},
		\[
		\mu(F^{\alpha_j}\triangle F^\lambda)\leq\sum_{\ell=j}^{\infty}\mu(F^{\alpha_\ell}\triangle F^{\alpha_{\ell+1}})\longrightarrow0.
		\]
		Fix $\beta<\lambda$. For all sufficiently large $j$, we have $\beta<\alpha_j$, and hence
		\[
		\mu(F^{\alpha_j}\triangle F^\beta)\leq C(t)\left(\mu(F^{\alpha_j})-\mu(F^\beta)\right).
		\]
		Therefore,
		\[
		\mu(F^\lambda\triangle F^\beta)\leq\mu(F^\lambda\triangle F^{\alpha_j})+C(t)\left(\mu(F^{\alpha_j})-\mu(F^\beta)\right),
		\]
		and letting $j\to\infty$ gives
		\[
		\mu(F^\lambda\triangle F^\beta)\leq C(t)\left(\mu(F^\lambda)-\mu(F^\beta)\right).
		\]
		Thus, \eqref{eq:transfinite_change_bound} is preserved at limit steps.
		
		The construction cannot continue through all countable ordinals. Indeed, the measure increases strictly at every successor step and is nondecreasing throughout the construction. Hence, for every $\alpha<\omega_1$, one could choose a rational number in the interval $\left(\mu(F^\alpha),\mu(F^{\alpha+1})\right)$. These intervals are pairwise disjoint, which would give an injection from $\omega_1$ into $\bQ$. Therefore, the construction stops with a set $G\in\cA$ satisfying $R-\mu(G)\leq t$. Taking $\beta=0$ in \eqref{eq:transfinite_change_bound} proves part~\ref{it:auga}.
		
		To prove part~\ref{it:augb}, set $t_i\coloneqq R2^{-i}$ for $i\in\bZ_{\geq0}$. Starting with an arbitrary $G_0\in\cA$, apply part~\ref{it:auga} successively to obtain $G_i\in\cA$ such that $R-\mu(G_i)\leq t_i$ and
		\begin{equation*}
			\mu(G_i\triangle G_{i-1})\leq C(t_i)\left(\mu(G_i)-\mu(G_{i-1})\right)
		\end{equation*}
		for every $i\in\bN$. Since $\mu(G_i)-\mu(G_{i-1})\leq R-\mu(G_{i-1})\leq t_{i-1}$, we have
		\begin{equation*}
			\sum_{i=1}^{\infty}\mu(G_i\triangle G_{i-1})\leq2R\sum_{i=1}^{\infty}2^{-i}C(R2^{-i})<\infty.
		\end{equation*}
		By \cref{cor:Borel_Cantelli}, the limit $G\coloneqq\lim_{i\to\infty}G_i$ exists and satisfies $\mu(G)=\lim_{i\to\infty}\mu(G_i)=R$. Since $\cA$ is closed under taking limits, $G\in\cA$.
	\end{proof}
	
	We now obtain the measurable analogue of Edmonds' matroid intersection theorem.
	
	\begin{thm}\label{thm:intersection}
		Let $M_1=(J,\Sigma,\mu,\cF_1)$ and $M_2=(J,\Sigma,\mu,\cF_2)$ be measurable matroids with rank functions $r_1$ and $r_2$. Then
		\begin{equation*}
			\sup\{\mu(F)\mid F\in\cF_1\cap\cF_2\}=\min_{X\in\Sigma}\{r_1(X)+r_2(J\setminus X)\}.
		\end{equation*}
	\end{thm}
	
	\begin{proof}
		Let $z\coloneqq\min_{X\in\Sigma}\{r_1(X)+r_2(J\setminus X)\}$, which is attained by \cref{lem:minimum}, since the function $X\mapsto r_1(X)+r_2(J\setminus X)$ is bounded, submodular, and continuous from above and below. For every $F\in\cF_1\cap\cF_2$ and every $X\in\Sigma$, we have
		\begin{equation*}
			\mu(F)=\mu(F\cap X)+\mu(F\cap(J\setminus X))\leq r_1(X)+r_2(J\setminus X).
		\end{equation*}
		Hence, $\mu(F)\leq z$, and the left-hand side is at most the right-hand side.
		
		For the converse, the assertion is immediate if $z=0$, so assume that $z>0$. Set $\cA\coloneqq\cF_1\cap\cF_2$ and
		\begin{equation*}
			C(t)\coloneqq1+\frac{2\mu(J)}{t}
		\end{equation*}
		for $0<t\leq z$. Every member of $\cA$ has measure at most $z$. By \cref{lem:liminf}, both $\cF_1$ and $\cF_2$ are closed under taking limits, and hence so is $\cA$.
		
		Let $0<t\leq z$ and $F\in\cA$ satisfy $z-\mu(F)>t$. By \cref{lem:short_augmentation}, there exist an integer
		\[
		n\leq\frac{\mu(J)}{z-\mu(F)}<\frac{\mu(J)}{t}
		\]
		and a set $\widehat F\in\cA$ such that $\mu(\widehat F)>\mu(F)$ and
		\[
		\mu(\widehat F\triangle F)\leq(2n+1)\left(\mu(\widehat F)-\mu(F)\right)\leq C(t)\left(\mu(\widehat F)-\mu(F)\right).
		\]
		Thus, \cref{lem:summable_attainment}\ref{it:auga}, applied with $R=z$ and $F_0=\emptyset$, gives common independent sets of measure arbitrarily close to $z$. Hence, the supremum on the left-hand side is at least $z$.
	\end{proof}
	
	\begin{rem}\label{rem:intersection_sup_not_attained}
		The supremum in \cref{thm:intersection} need not be attained. Let $J=[0,1]$ with the Lebesgue measure $\lambda$, and let $N^+=(J,\Sigma,\lambda,\cF^+)$ be the interval-halving matroid, where
		\[
		\cF^+ \coloneqq \left\{F\in\Sigma\,\middle|\, \lambda(F\cap[0,t])\leq \frac{t}{2} \text{ for every } t\in[0,1]\right\}.
		\]
		Let $N^-=(J,\Sigma,\lambda,\cF^-)$ be its reflected copy under the map $x\mapsto 1-x$; equivalently,
		\[
		\cF^- \coloneqq \left\{F\in\Sigma\,\middle|\, \lambda(F\cap[t,1])\leq \frac{1-t}{2} \text{ for every } t\in[0,1]\right\}.
		\]
		Both matroids have rank $1/2$. We show that $N^+$ and $N^-$ have common independent sets of measure arbitrarily close to $1/2$, but no common basis. For $n\geq2$, define
		\[
		F_n\coloneqq\bigcup_{k=1}^{n-1}\left[\frac{2k-1}{2n-1},\frac{2k}{2n-1}\right].
		\]
		Then
		\[
		\lambda(F_n)=\frac{n-1}{2n-1}=\frac12-\frac{1}{2(2n-1)}.
		\]
		A direct check gives
		\[
		0\leq\frac{t}{2}-\lambda(F_n\cap[0,t])\leq\frac{1}{2(2n-1)}
		\]
		for every $t\in[0,1]$, so $F_n\in\cF^+$. Moreover, $F_n$ is invariant under the reflection $x\mapsto1-x$, and hence $F_n\in\cF^-$. Thus, common independent sets can have measure arbitrarily close to $1/2$; see \cref{fig:halving3}.
		
		It remains to show that no common independent set of measure $1/2$ exists. Suppose that $B\in\cF^+\cap\cF^-$ and $\lambda(B)=1/2$. The constraint defining $\cF^+$ gives $\lambda(B\cap[0,t])\leq t/2$ for every $t\in[0,1]$. On the other hand, the constraint defining $\cF^-$ gives $\lambda(B\cap[t,1])\leq (1-t)/2$, which, since $\lambda(B)=1/2$, is equivalent to $\lambda(B\cap[0,t])\geq t/2$. Thus, $\lambda(B\cap[0,t])=t/2$ for every $t\in[0,1]$. As observed in \cref{lem:halving_measure}, no Borel set has this property. Therefore, $N^+$ and $N^-$ have no common basis, although their common independent sets have measures arbitrarily close to the common rank. 
		
		The same interval-halving obstruction also appears in the example from \cref{rem:unionnotm} showing that the raw measurable union need not be closed: in both cases, an exact solution would force every prefix constraint to be tight. The next subsection gives sufficient rank-expansion conditions for attainment.
	\end{rem}
	
	\begin{rem}
		Laczkovich's example~\cite{laczkovich1988closed} gives another nonattainment example. Let $M_S$ and $M_T$ be the two incidence partition matroids of the bipartite graphing $G$ from \cref{rem:laczkovich_matching}. Their common independent sets are precisely the measurable matchings. By \cref{rem:laczkovich_matching}, they have common independent sets of measure arbitrarily close to their common rank, but no common basis.
	\end{rem}
	
	%%%%%%%%%%%%%%%%%%%%%%%%%%%%%%%%
	\subsection{Expansion and Attainment}
	\label{sec:expansion_attainment}
	%%%%%%%%%%%%%%%%%%%%%%%%%%%%%%%%
	
	The intersection theorem determines the supremum of the measures of common independent sets, but the example in \cref{rem:intersection_sup_not_attained} shows that this supremum need not be attained. The distinction is particularly important for measurable matchings. Hall's condition yields matchings covering all but a set of arbitrarily small measure, whereas a measurable matching covering all vertices outside a null set may fail to exist~\cite{laczkovich1988closed,kun2021measurable}; see \cref{rem:laczkovich_matching} for a concrete example. Expansion is a standard additional hypothesis under which such matchings can be obtained~\cite{lyons2011perfect,grabowski2022measurable}. We give a rank-expansion criterion that guarantees the existence of a common basis.
	
	We first discuss the closure relations used in the proof. Let $F\in\cF_1\cap\cF_2$, and let $F_0,F_1,\dots$ and $A_0,A_1,\dots$ be the alternating layers constructed in the proof of \cref{lem:short_augmentation}.
	
	\begin{lem}\label{lem:alternating_layer_closures}
		For every index $k$ for which $F_{k+1}$ is defined, we have
		\[
		\clo_2(F\setminus F_{k+1})\setminus\clo_2(\emptyset)\subseteq\clo_1(F\setminus F_k)
		\qquad\text{and}\qquad
		\clo_1(F_k)\setminus\clo_1(\emptyset)\subseteq\clo_2(F_{k+1}).
		\]
	\end{lem}
	
	\begin{proof}
		Since $F_{k+1}$ and $F\setminus F_{k+1}$ are disjoint subsets of the $M_2$-independent set $F$, \cref{lem:closure_intersection} gives
		\[
		\clo_2(F\setminus F_{k+1})\setminus\clo_2(\emptyset)\subseteq J\setminus\clo_2(F_{k+1}).
		\]
		By the construction of $F_{k+1}$, we have $A_k\subseteq\clo_2(F_{k+1})$, and therefore
		\[
		J\setminus\clo_2(F_{k+1})\subseteq J\setminus A_k=\clo_1(F\setminus F_k),
		\]
		which proves the first inclusion. Similarly, $F_k$ and $F\setminus F_k$ are disjoint subsets of the $M_1$-independent set $F$, so \cref{lem:closure_intersection} gives
		\[
		\clo_1(F_k)\setminus\clo_1(\emptyset)\subseteq J\setminus\clo_1(F\setminus F_k)=A_k\subseteq\clo_2(F_{k+1}),
		\]
		which proves the second inclusion.
	\end{proof}
	
	We will also use the following consequence of \cref{lem:closure_intersection}.
	
	\begin{lem}\label{lem:complementary_closed_set}
		Let $M=(J,\Sigma,\mu,\cF)$ be a measurable matroid of rank $R$, with rank function $r$ and closure operator $\clo$. For every $M$-closed set $C$, there exists an $M$-closed set $D$ such that $C\cap D=\clo(\emptyset)$ and  $r(D)=R-r(C)$.
	\end{lem}
	
	\begin{proof}
		Let $B_1$ be a basis of $M|C$, and extend it to a basis $B_1\cup B_2$ of $M$. Set $D\coloneqq\clo(B_2)$. Since $C=\clo(B_1)$ and $B_1\cup B_2$ is independent, \cref{lem:closure_intersection} gives $C\cap D\subseteq\clo(\emptyset)$. The reverse inclusion holds because every closed set contains $\clo(\emptyset)$. Moreover, $r(D)=\mu(B_2)=R-\mu(B_1)=R-r(C)$.
	\end{proof}
	
	We introduce two forms of expansion, differing only in whether the rank comparison is imposed on closed sets or on their complements. Let $M_1$ and $M_2$ be measurable matroids on the same measure space, of common rank $R$, with rank functions $r_1,r_2$ and closure operator $\clo_1$ for $M_1$. Let $\varepsilon>0$ and $\alpha\in[0,1]$. The ordered pair $(M_1,M_2)$ is \emph{$(\varepsilon,\alpha)$-expanding} if, for every $M_1$-closed set $C$,
	\begin{equation}\label{eq:rank_expansion}
		r_2(C\setminus\clo_1(\emptyset))\geq\alpha R
		\qquad\text{or}\qquad
		r_2(C\setminus\clo_1(\emptyset))\geq(1+\varepsilon)r_1(C\setminus\clo_1(\emptyset)).
	\end{equation}
	It is \emph{$(\varepsilon,\alpha)$-co-expanding} if, for every $M_1$-closed set $C$,
	\begin{equation}\label{eq:complement_rank_expansion}
		r_2(J\setminus C)\geq\alpha R
		\qquad\text{or}\qquad
		r_2(J\setminus C)\geq(1+\varepsilon)r_1(J\setminus C).
	\end{equation}
	
	The proof of the next theorem strengthens the construction from \cref{lem:short_augmentation}. We first show that the min--max value in \cref{thm:intersection} is $R$. Then, starting from a common independent set $F$ with gap $\delta=R-\mu(F)$, we run the alternating-layer construction from \cref{lem:short_augmentation}. The expansion hypotheses force the forward layers to grow geometrically until they reach the threshold $\theta R$, after which the reverse expansion forces the remaining parts of $F$ to decay geometrically. This bounds the length of the augmentation by $O_\varepsilon(\log(R/\delta))$, and \cref{lem:summable_attainment} then yields a common basis.
	
	\begin{thm}\label{thm:intersection_attainment_expansion}
		Let $M_1=(J,\Sigma,\mu,\cF_1)$ and $M_2=(J,\Sigma,\mu,\cF_2)$ be measurable matroids of common rank $R$. Suppose that there exist $\varepsilon>0$ and $\theta\in[0,1]$ such that the following conditions hold.
		\begin{enumerate}[label=(\alph*)]\itemsep0em
			\item The ordered pair $(M_1,M_2)$ is either $(\varepsilon,\theta)$-expanding or $(\varepsilon,\theta)$-co-expanding.
			\item The ordered pair $(M_2,M_1)$ is either $(\varepsilon,1-\theta)$-expanding or $(\varepsilon,1-\theta)$-co-expanding.
		\end{enumerate}
		Then $M_1$ and $M_2$ have a common basis.
	\end{thm}
	
	\begin{proof}
		If $R=0$, then the assertion is immediate, so assume that $R>0$. We first show that
		\begin{equation}\label{eq:expansion_intersection_condition}
			r_1(X)+r_2(J\setminus X)\geq R
		\end{equation}
		for every $X\in\Sigma$. Suppose, to the contrary, that this inequality fails for some $X$. Then $r_2(J\setminus X)<\theta R$ or $r_1(X)<(1-\theta)R$. The two cases are symmetric under interchanging $M_1$ and $M_2$, replacing $X$ by $J\setminus X$, and replacing $\theta$ by $1-\theta$. We may therefore assume that $r_2(J\setminus X)<\theta R$. Set $C\coloneqq\clo_1(X)$ and $A\coloneqq J\setminus C$. Since closure preserves rank and $A\subseteq J\setminus X$, we have
		\begin{equation}\label{eq:expansion_obstruction}
			r_2(A)<\theta R
			\qquad\text{and}\qquad
			r_1(C)+r_2(A)<R.
		\end{equation}
		
		Suppose first that $(M_1,M_2)$ is expanding. Apply \cref{lem:complementary_closed_set} in $M_1$ to obtain an $M_1$-closed set $D$ such that $C\cap D=\clo_1(\emptyset)$ and $r_1(D)=R-r_1(C)$. Then $D\setminus\clo_1(\emptyset)\subseteq A$, and hence
		\[
		r_2(D\setminus\clo_1(\emptyset))\leq r_2(A)<\theta R.
		\]
		Moreover,
		\[
		r_1(D\setminus\clo_1(\emptyset))
		=
		R-r_1(C)
		>
		r_2(A)
		\geq
		r_2(D\setminus\clo_1(\emptyset)).
		\]
		Thus, both alternatives in \eqref{eq:rank_expansion} fail for $D$, a contradiction.
		
		Now suppose that $(M_1,M_2)$ is co-expanding. By submodularity of $r_1$,
		\[
		r_1(A)\geq R-r_1(C)>r_2(A).
		\]
		Together with $r_2(A)<\theta R$, this contradicts both alternatives in \eqref{eq:complement_rank_expansion} for the closed set $C$. This proves \eqref{eq:expansion_intersection_condition}.
		
		Since equality holds in \eqref{eq:expansion_intersection_condition} for $X=J$, \cref{thm:intersection} shows that the supremum of the measures of common independent sets is $R$. Set $\cA\coloneqq\cF_1\cap\cF_2$. By \cref{lem:liminf}, both $\cF_1$ and $\cF_2$ are closed under taking limits, and hence so is $\cA$.
		
		Let $F\in\cA$ with $\mu(F)<R$, and set $\delta\coloneqq R-\mu(F)$. Set
		\[
		A_0\coloneqq J\setminus\clo_1(F),
		\qquad
		H\coloneqq J\setminus\clo_2(F),
		\qquad
		F_0\coloneqq\emptyset,
		\]
		and run the alternating-layer construction from the proof of \cref{lem:short_augmentation}. Let $n$ be the first index for which $\mu(A_n\cap H)>0$. For $0\leq k\leq n$, set $f_k\coloneqq\mu(F_k)$ and $s_k\coloneqq\mu(F\setminus F_k)$. By \cref{cla:layers}, $f_1\geq\delta$ when $n\geq1$, and $s_k\geq\delta$ for every $k<n$.
		
		\begin{cla}\label{cla:help}
			For every $k<n$, we have $f_{k+1}\geq\theta R$ or $f_{k+1}\geq(1+\varepsilon)f_k$.
		\end{cla}
		
		\begin{claimproof}
			Suppose first that $(M_1,M_2)$ is expanding. Apply \eqref{eq:rank_expansion} to the $M_1$-closed set $\clo_1(F_k)$, and set $C_k\coloneqq\clo_1(F_k)\setminus\clo_1(\emptyset)$. By the loop-set observation, $r_1(C_k)=r_1(F_k)=f_k$. Moreover, \cref{lem:alternating_layer_closures} gives $C_k\subseteq\clo_2(F_{k+1})$, and hence $r_2(C_k)\leq r_2(F_{k+1})=f_{k+1}$. Thus, \eqref{eq:rank_expansion} gives either $f_{k+1}\geq\theta R$ or $f_{k+1}\geq r_2(C_k)\geq(1+\varepsilon)f_k$.
			
			Suppose now that $(M_1,M_2)$ is co-expanding. Apply \eqref{eq:complement_rank_expansion} to the $M_1$-closed set $J\setminus A_k=\clo_1(F\setminus F_k)$. Since $F\setminus F_k$ is $M_1$-independent, submodularity gives $r_1(A_k)\geq R-r_1(J\setminus A_k)=R-\mu(F\setminus F_k)=\delta+f_k$. By the construction of $F_{k+1}$, we have $A_k\subseteq\clo_2(F_{k+1})$, and hence $r_2(A_k)\leq r_2(F_{k+1})=f_{k+1}$. Thus, \eqref{eq:complement_rank_expansion} gives either $f_{k+1}\geq\theta R$ or $f_{k+1}\geq r_2(A_k)\geq(1+\varepsilon)r_1(A_k)\geq(1+\varepsilon)(\delta+f_k)\geq(1+\varepsilon)f_k$. This proves the claim.
		\end{claimproof}
		
		Define $L(\delta)\coloneqq\left\lceil\frac{\log(R/\delta)}{\log(1+\varepsilon)}\right\rceil+1$. If $n>L(\delta)$ and the first alternative in \cref{cla:help} does not occur among the first $L(\delta)$ layers, then $f_1\geq\delta$ and the second alternative give
		\begin{equation*}
			f_{L(\delta)}\geq(1+\varepsilon)^{L(\delta)-1}\delta\geq R,
		\end{equation*}
		contrary to $F_{L(\delta)}\subseteq F$ and $\mu(F)<R$. Thus, either $n\leq L(\delta)$, or there exists $m\leq L(\delta)$ such that $f_m\geq\theta R$.
		
		Assume that the latter case occurs. We claim that
		\begin{equation}\label{eq:backward_decay}
			s_k\geq(1+\varepsilon)s_{k+1}
		\end{equation}
		for every $m\leq k<n$. Since
		\[
		s_k\leq s_m=\mu(F)-f_m\leq\mu(F)-\theta R<(1-\theta)R
		\]
		for every $m\leq k<n$, we prove the claim separately in the expanding and co-expanding cases. 
		
		Suppose first that $(M_2,M_1)$ is expanding, and set $D_k\coloneqq\clo_2(F\setminus F_{k+1})$. By \cref{lem:alternating_layer_closures}, we have $r_1(D_k\setminus\clo_2(\emptyset))\leq s_k<(1-\theta)R$, while the loop-set observation gives $r_2(D_k\setminus\clo_2(\emptyset))=s_{k+1}$. Thus, the threshold alternative in \eqref{eq:rank_expansion} cannot hold, and the multiplicative alternative gives $s_k\geq r_1(D_k\setminus\clo_2(\emptyset))\geq(1+\varepsilon)r_2(D_k\setminus\clo_2(\emptyset))=(1+\varepsilon)s_{k+1}$.
		
		Suppose now that $(M_2,M_1)$ is co-expanding, and set $D_k\coloneqq\clo_2(F_{k+1})$. By the construction of $F_{k+1}$, we have $A_k\subseteq D_k$, so $J\setminus D_k\subseteq J\setminus A_k=\clo_1(F\setminus F_k)$, and hence $r_1(J\setminus D_k)\leq s_k<(1-\theta)R$. On the other hand, \cref{lem:closure_intersection} gives $\clo_2(F\setminus F_{k+1})\setminus\clo_2(\emptyset)\subseteq J\setminus D_k$, and therefore $r_2(J\setminus D_k)\geq s_{k+1}$. Again, the threshold alternative in \eqref{eq:complement_rank_expansion} cannot hold, and the multiplicative alternative gives $s_k\geq r_1(J\setminus D_k)\geq(1+\varepsilon)r_2(J\setminus D_k)\geq(1+\varepsilon)s_{k+1}$.
		
		If $m+L(\delta)<n$, then repeated application of \eqref{eq:backward_decay} gives
		\begin{equation*}
			s_{m+L(\delta)}\leq(1+\varepsilon)^{-L(\delta)}s_m<\delta,
		\end{equation*}
		contradicting \cref{cla:layers}. Therefore, $n\leq m+L(\delta)\leq2L(\delta)$. By \cref{lem:short_augmentation}, there exists $\widehat F\in\cA$ such that $\mu(\widehat F)>\mu(F)$ and
		\begin{equation}\label{eq:expansion_local_augmentation}
			\mu(\widehat F\triangle F)\leq\left(4L(\delta)+1\right)\left(\mu(\widehat F)-\mu(F)\right).
		\end{equation}
		For $0<t\leq R$, define
		\begin{equation*}
			C(t)\coloneqq4\left(\left\lceil\frac{\log(R/t)}{\log(1+\varepsilon)}\right\rceil+1\right)+1.
		\end{equation*}
		If $\delta>t$, then $L(\delta)\leq L(t)$, so \eqref{eq:expansion_local_augmentation} verifies the local augmentation hypothesis of \cref{lem:summable_attainment}. Moreover, $C(R2^{-i})=O_\varepsilon(i)$, and hence
		\begin{equation*}
			\sum_{i=1}^{\infty}2^{-i}C(R2^{-i})<\infty.
		\end{equation*}
		Therefore, \cref{lem:summable_attainment}\ref{it:augb} gives a set $B\in\cA$ with $\mu(B)=R$. Thus, $B$ is a common basis.
	\end{proof}
	
	For the two incidence partition matroids associated with a bipartite graphing, the usual bipartite-expander inequalities imply the co-expanding hypotheses of \cref{thm:intersection_attainment_expansion} with $\theta=1/2$. Thus, the theorem recovers the measurable perfect-matching theorem of Lyons and Nazarov~\cite{lyons2011perfect}; we give the details in \cref{sec:bipartite}, following the formulation of Grabowski, M\'ath\'e and Pikhurko~\cite[Equation~(2.5)]{grabowski2022measurable}.
	
	At $\theta=0$, the forward expansion condition is automatic. After truncation, this gives the following unequal-rank consequence.
	
	\begin{cor}\label{cor:one_sided_flat_expansion}
		Let $M_1=(J,\Sigma,\mu,\cF_1)$ and $M_2=(J,\Sigma,\mu,\cF_2)$ be measurable matroids with rank functions $r_1,r_2$ and closure operators $\clo_1,\clo_2$. Suppose that $\clo_2(\emptyset)\subseteq\clo_1(\emptyset)$ and that there exists $\varepsilon>0$ such that
		\begin{equation}\label{eq:one_sided_flat_expansion}
			r_1(C)\geq\min\{r_1(J),(1+\varepsilon)r_2(C)\}
		\end{equation}
		for every $M_2$-closed set $C$. Then there exists $F\in\cF_1\cap\cF_2$ such that
		\begin{equation*}
			\mu(F)=\min\{r_1(J),r_2(J)\}.
		\end{equation*}
	\end{cor}
	
	\begin{proof}
		Set $z\coloneqq\min\{r_1(J),r_2(J)\}$. If $z=0$, then the assertion is immediate. Let $M_i'\coloneqq(M_i)_z$, for $i=1,2$. By \cref{prop:truncation}, these matroids have common rank $z$, with rank functions $r_i'(X)=\min\{z,r_i(X)\}$. Since $z>0$, truncation leaves the loop set unchanged, and hence
		\begin{equation*}
			\clo_{M_2'}(\emptyset)=\clo_2(\emptyset)\subseteq\clo_1(\emptyset)=\clo_{M_1'}(\emptyset).
		\end{equation*}
		If $C$ is a proper $M_2'$-closed set, then $r_2'(C)<z$, and the rank formula for truncation shows that $C$ is $M_2$-closed. The only $M_2'$-closed set of rank $z$ is $J$. Consequently, \eqref{eq:one_sided_flat_expansion} gives
		\begin{equation}\label{eq:truncated_one_sided_expansion}
			r_1'(C)\geq\min\{z,(1+\varepsilon)r_2'(C)\}
		\end{equation}
		for every $M_2'$-closed set $C$. 
		
		Set $L_2\coloneqq\clo_{M_2'}(\emptyset)$. Since $L_2\subseteq\clo_{M_1'}(\emptyset)$, removing $L_2$ does not change either $r_2'$ or $r_1'$. Hence, \eqref{eq:truncated_one_sided_expansion} implies that, for every $M_2'$-closed set $C$,
		\[
		r_1'(C\setminus L_2)\geq z \quad\text{or}\quad r_1'(C\setminus L_2)\geq(1+\varepsilon)r_2'(C\setminus L_2).
		\]
		Thus, $(M_2',M_1')$ is $(\varepsilon,1)$-expanding. By the definition with threshold $0$, $(M_1',M_2')$ is automatically $(\varepsilon,0)$-expanding. Applying \cref{thm:intersection_attainment_expansion} with $\theta=0$ gives a common basis $F$ of $M_1'$ and $M_2'$. Therefore, $\mu(F)=z$, and, since truncation only removes independent sets, $F\in\cF_1\cap\cF_2$.
	\end{proof}
	
	%%%%%%%%%%%%%%%%%%%%%%%%%%%%%%%%
	\subsection{The Rank of the Union}
	\label{sec:union_rank}
	%%%%%%%%%%%%%%%%%%%%%%%%%%%%%%%%
	
	We now derive the rank formula for measurable matroid union and give sufficient conditions under which the union rank is attained before taking the upward closure. In the finite setting, matroid union can be reduced to matroid intersection in two standard ways. For any finite number of matroids, one can pass to disjoint copies of the ground set, take the direct sum of the matroids on the copies, and intersect it with the partition matroid which allows at most one copy of each original element to be chosen. This reduction provides a decomposition into independent sets of the original matroids and extends directly to measurable matroids.
	
	For two matroids on a finite ground set, duality gives another reduction to matroid intersection: the rank of $M_1\vee M_2$ is obtained by adding the rank of $M_2$ to the maximum size of a common independent set of $M_1$ and $M_2^*$. The same reduction works for measurable matroids, with the maximum size replaced by the supremum of the measures of common independent sets. Here we use the direct-sum and partition-matroid construction, which also keeps track of the decomposition into independent sets of the original matroids. We return to the dual formulation for two matroids in \cref{rem:two_matroid_dual_union}.
	
	Let $M_i=(J,\Sigma,\mu,\cF_i)$ be measurable matroids with rank functions $r_i$, for $i\in[k]$. Let
	\[
	M_\vee\coloneqq\bigvee_{i=1}^k M_i
	\]
	denote their measurable union, and let $r_\vee$ be its rank function.
	
	\begin{thm}\label{thm:union_rank_formula}
		For every $A\in\Sigma$,
		\begin{equation}\label{eq:union_rank_formula}
			r_\vee(A)=\min_{X\subseteq A}\left\{\mu(A\setminus X)+\sum_{i=1}^k r_i(X)\right\}.
		\end{equation}
	\end{thm}
	
	\begin{proof}
		It is enough to prove the formula for $A=J$. Indeed, for every $A\in\Sigma$, the measurable union of the restrictions $M_i|A$, for $i\in[k]$, is the restriction $M_\vee|A$ by \cref{prop:finite_union_raw_description}. Moreover, the rank function of $M_i|A$ is the restriction of $r_i$ to the measurable subsets of $A$. Thus, applying the ground-set formula to the matroids $M_i|A$, for $i\in[k]$, gives \eqref{eq:union_rank_formula} for $A$.
		
		For each $i\in[k]$, let $J_i$ be a disjoint copy of $J$, and let $\theta_i\colon J_i\to J$ be the natural bijection. Set $B\coloneqq J_1\cup\dots\cup J_k$. A set $Z\subseteq B$ is measurable if $\theta_i(Z\cap J_i)$ is measurable for every $i\in[k]$, and for such a set define
		\[
		\nu(Z)\coloneqq\sum_{i=1}^k\mu\left(\theta_i(Z\cap J_i)\right).
		\]
		Let $\theta\colon B\to J$ be the map whose restriction to $J_i$ is $\theta_i$, and let $N$ be the direct sum of the copies of $M_i$ on $J_i$, for $i\in[k]$. For $Z\subseteq B$, write $Z_i\coloneqq\theta_i(Z\cap J_i)$. By \cref{prop:direct_sum}, we have $r_N(Z)=\sum_{i=1}^k r_i(Z_i)$.
		
		Let $R$ be the finite Borel equivalence relation on $B$ defined by $xRy$ if and only if $\theta(x)=\theta(y)$. This relation is measure-preserving, so \cref{prop:finite_classed_partition} gives a measurable partition matroid $P$ whose independent sets meet almost every fiber of $\theta$ in at most one point. Its rank function is $r_P(Z)=\mu(\theta(Z))$. Indeed, the restriction of $\theta$ to every independent subset of $Z$ in $P$ is injective outside a null set, so its measure is at most $\mu(\theta(Z))$. Conversely, choosing for each point of $\theta(Z)$ the least-indexed copy in which it occurs gives a measurable independent subset of $Z$ whose $\nu$-measure is $\mu(\theta(Z))$.
		
		\begin{cla}\label{cla:rank}
			$r_\vee(J)=\sup\left\{\nu(Z)\,\middle|\,Z\in\cF_N\cap\cF_P\right\}$.
		\end{cla}
		\begin{claimproof}
			If $Z\in\cF_N\cap\cF_P$, then $Z_i\in\cF_i$ for every $i\in[k]$. Moreover, since $Z$ is independent in $P$, the sets $Z_1,\dots,Z_k$ are pairwise disjoint. Hence, $Z_1\cup\dots\cup Z_k\in\cF_1+\dots+\cF_k$, and therefore it is independent in $M_\vee$. Consequently,
			\[
			\nu(Z)=\mu\left(\bigcup_{i=1}^kZ_i\right)\leq r_\vee(J).
			\]
			
			Conversely, let $I$ be a basis of $M_\vee$. By \cref{prop:finite_union_raw_description}, for every $\delta>0$ there exist pairwise disjoint sets $F_i\in\cF_i$, for $i\in[k]$, such that $\bigcup_{i=1}^kF_i\subseteq I$ and $\mu\left(I\setminus\bigcup_{i=1}^kF_i\right)<\delta$. Then $F'\coloneqq\bigcup_{i=1}^k\theta_i^{-1}(F_i)$ is independent in both $N$ and $P$, and
			\[
			\nu(F')=\mu\left(\bigcup_{i=1}^kF_i\right)>r_\vee(J)-\delta.
			\]
			This proves the claim.
		\end{claimproof}
		
		Applying \Cref{cla:rank} and \cref{thm:intersection} to $N$ and $P$ gives
		\[
		r_\vee(J)=\min_{Z\subseteq B}\left\{r_N(Z)+r_P(B\setminus Z)\right\}.
		\]
		For $Z\subseteq B$, set $X_Z\coloneqq\bigcap_{i=1}^kZ_i$. Since $\theta(B\setminus Z)=J\setminus X_Z$, monotonicity of the rank functions gives
		\[
		r_N(Z)+r_P(B\setminus Z)
		=\sum_{i=1}^k r_i(Z_i)+\mu(J\setminus X_Z)
		\geq\sum_{i=1}^k r_i(X_Z)+\mu(J\setminus X_Z).
		\]
		
		Conversely, for $X\subseteq J$, set $Z_X\coloneqq\bigcup_{i=1}^k\theta_i^{-1}(X)$. Then
		\[
		r_N(Z_X)+r_P(B\setminus Z_X)=\sum_{i=1}^k r_i(X)+\mu(J\setminus X).
		\]
		Together, these give \eqref{eq:union_rank_formula} for $A=J$. Moreover, if $Z'\subseteq B$ attains the minimum given by the intersection theorem, then $X_{Z'}$ attains the minimum in \eqref{eq:union_rank_formula}.
	\end{proof}
	
	We next pass to a countable family. Let $M_i=(J,\Sigma,\mu,\cF_i)$ be measurable matroids with rank functions $r_i$, for $i\in\bN$. Let
	\[
	M_\vee\coloneqq\bigvee_{i\in\bN}M_i
	\]
	denote their countable measurable union, and let $r_\vee$ be its rank function.
	
	\begin{thm}\label{thm:countable_union_rank_formula}
		For every $A\in\Sigma$,
		\begin{equation}\label{eq:countable_union_rank_formula}
			r_\vee(A)=\min_{X\subseteq A}\left\{\mu(A\setminus X)+\sum_{i=1}^{\infty}r_i(X)\right\},
		\end{equation}
		where the series is understood in $[0,\infty]$.
	\end{thm}
	
	\begin{proof}
		It is enough to prove the formula for $A=J$. Indeed, for every $A\in\Sigma$, the countable measurable union of the restrictions $M_i|A$, for $i\in\bN$, is the restriction $M_\vee|A$ by \cref{thm:countable_union}. Moreover, the rank function of $M_i|A$ is the restriction of $r_i$ to the measurable subsets of $A$. Thus, applying the ground-set formula to the matroids $M_i|A$, for $i\in\bN$, gives \eqref{eq:countable_union_rank_formula} for $A$.
		
		For $n\in\bN$, let $M_\vee^{(n)}\coloneqq M_1\vee\dots\vee M_n$ and denote its rank function by $r_\vee^{(n)}$. By \cref{thm:countable_union,prop:increasing_limit_rank},
		\begin{equation}\label{eq:countable_union_rank_limit}
			r_\vee(J)=\lim_{n\to\infty}r_\vee^{(n)}(J).
		\end{equation}
		For $X\subseteq J$, put $f_n(X)\coloneqq\mu(J\setminus X)+\sum_{i=1}^n r_i(X)$ and $v_n\coloneqq\min_{X\subseteq J}f_n(X)$. By \cref{thm:union_rank_formula}, $v_n=r_\vee^{(n)}(J)$ and the minimum defining $v_n$ is attained. Hence, \eqref{eq:countable_union_rank_limit} gives $\lim_{n\to\infty}v_n=r_\vee(J)$.
		
		A countable sum of rank functions need not be continuous from above, so \cref{lem:minimum} cannot be applied directly. We instead choose a minimizer $X_n$ of $f_n$ for each $n$ so that $X_{n+1}\subseteq X_n$. Choose any minimizer $X_1$ of $f_1$. Suppose that $X_n$ has been chosen, and let $Y$ be a minimizer of $f_{n+1}$. Since $f_n$ is submodular, $r_{n+1}$ is monotone, and $f_{n+1}(X)=f_n(X)+r_{n+1}(X)$, we have
		\[
		f_n(X_n\cup Y)+f_{n+1}(X_n\cap Y)
		\leq f_n(X_n)+f_n(Y)+r_{n+1}(Y)
		=v_n+v_{n+1}.
		\]
		The two terms on the left are at least $v_n$ and $v_{n+1}$, respectively. Hence equality holds, and $X_n\cap Y$ is a minimizer of $f_{n+1}$. We may therefore take $X_{n+1}\coloneqq X_n\cap Y$.
		
		Set $X\coloneqq\bigcap_{n=1}^{\infty}X_n$. For fixed $m\in\bN$ and every $n\geq m$,
		\[
		\mu(J\setminus X_n)+\sum_{i=1}^m r_i(X_n)\leq f_n(X_n)=v_n.
		\]
		Since $(X_n)_{n\in\bN}$ is decreasing with intersection $X$, continuity of $\mu$ and continuity from above of the rank functions, given by \cref{cor:smooth}, yield
		\[
		\mu(J\setminus X)+\sum_{i=1}^m r_i(X)\leq r_\vee(J).
		\]
		Letting $m\to\infty$, we obtain
		\[
		\mu(J\setminus X)+\sum_{i=1}^{\infty}r_i(X)\leq r_\vee(J).
		\]
		
		On the other hand, for every $Y\subseteq J$ and every $n\in\bN$, we have $v_n\leq f_n(Y)$. Letting $n\to\infty$ gives
		\[
		r_\vee(J)\leq\mu(J\setminus Y)+\sum_{i=1}^{\infty}r_i(Y).
		\]
		Taking $Y=X$ proves equality and shows that $X$ attains the minimum in \eqref{eq:countable_union_rank_formula}.
	\end{proof}
	
	\begin{rem}\label{rem:two_matroid_dual_union}
		For two matroids, the union rank also admits a dual formulation in terms of matroid intersection. Let $M_1$ and $M_2$ be measurable matroids on $(J,\Sigma,\mu)$, and let $r_2^*$ be the rank function of the dual matroid $M_2^*$. Set $R\coloneqq r_{M_1\vee M_2}(J)$ and $z\coloneqq R-r_2(J)$. The dual rank formula and \cref{thm:union_rank_formula} give
		\[
		z=\min_{X\subseteq J}\left\{r_1(X)+r_2^*(J\setminus X)\right\}
		=\sup\left\{\mu(F)\mid F\text{ is independent in both }M_1\text{ and }M_2^*\right\},
		\]
		where the second equality follows from \cref{thm:intersection}. Thus, $R$ is obtained by adding $r_2(J)$ to the supremum of the measures of common independent sets of $M_1$ and $M_2^*$.
		
		Suppose that this supremum is attained by a set $F_1$. Extend $F_1$ to a basis $B^*$ of $M_2^*$, and set $F_2\coloneqq J\setminus B^*$. Then $F_1\in\cF_1$, the set $F_2$ is a basis of $M_2$, and $F_1\cap F_2=\emptyset$. Moreover,
		\[
		\mu(F_1\cup F_2)=z+r_2(J)=R.
		\]
		Hence, a common independent set of $M_1$ and $M_2^*$ attaining the supremum gives a set in $\cF_1+\cF_2$ whose measure is the rank of $M_1\vee M_2$. In particular, suppose that $z>0$, and let $N_1\coloneqq(M_1)_z$ and $N_2\coloneqq(M_2^*)_z$. Since $z$ is the supremum displayed above, we have $z\leq r_1(J)$ and $z\leq r_2^*(J)$. Thus, both $N_1$ and $N_2$ have rank $z$.
		
		If $N_1$ and $N_2$ satisfy the hypotheses of \cref{thm:intersection_attainment_expansion}, then they have a common basis of measure $z$, and the preceding construction gives a set in $\cF_1+\cF_2$ whose measure is $R$. If $z=0$, the same conclusion follows by taking $F_1=\emptyset$ and letting $F_2$ be a basis of $M_2$.
	\end{rem}
	
	For arbitrary finite families, a general union-attainment statement can be deduced from \cref{thm:intersection_attainment_expansion} by applying it to the direct-sum and partition matroids used above. Since translating its assumptions into conditions on the original matroids leads to a rather technical statement, we restrict ourselves to the following simpler consequence of \cref{cor:one_sided_flat_expansion}.
	
	\begin{thm}\label{thm:union_attainment_expansion}
		Let $M_i=(J,\Sigma,\mu,\cF_i)$ be measurable matroids with rank functions $r_i$, for $i\in[k]$, and let $A\in\Sigma$. Suppose that there exists $\varepsilon>0$ such that
		\begin{equation}\label{eq:union_attainment_expansion}
			\sum_{i=1}^k r_i(X)
			\geq
			\min\left\{
			\sum_{i=1}^k r_i(A),
			(1+\varepsilon)\mu(X)
			\right\}
		\end{equation}
		for every $X\subseteq A$. Then there exist pairwise disjoint sets $F_i\in\cF_i$ with $F_i\subseteq A$, for $i\in[k]$, such that
		\begin{equation*}
			\mu\left(\bigcup_{i=1}^kF_i\right)
			=
			\min\left\{
			\mu(A),
			\sum_{i=1}^k r_i(A)
			\right\}.
		\end{equation*}
		Consequently,
		\begin{equation*}
			r_\vee(A)
			=
			\min\left\{
			\mu(A),
			\sum_{i=1}^k r_i(A)
			\right\},
		\end{equation*}
		and $\bigcup_{i=1}^kF_i$ is a basis of $M_\vee|A$ which belongs to $\cF_1+\dots+\cF_k$.
	\end{thm}
	
	\begin{proof}
		It is enough to prove the statement for $A=J$. Indeed, the general case follows by applying the ground-set statement to the restrictions $M_i|A$, for $i\in[k]$, as in the proof of \cref{thm:union_rank_formula}. 
		
		Use the direct-sum and partition-matroid construction from the proof of \cref{thm:union_rank_formula}. Denote the copies of $J$ by $J_1,\dots,J_k$, the natural bijections by $\theta_i\colon J_i\to J$, and the measure on their disjoint union $B$ by $\nu$. Let $\theta\colon B\to J$ be the map whose restriction to $J_i$ is $\theta_i$. Thus, $N$ is the direct sum of the corresponding copies of $M_1,\dots,M_k$, and $P$ is the partition matroid whose independent sets meet almost every fiber of $\theta$ in at most one point. We have $r_N(B)=\sum_{i=1}^k r_i(J)$ and $r_P(B)=\mu(J)$.
		
		The matroid $P$ is loopless, so $\clo_P(\emptyset)=\emptyset\subseteq\clo_N(\emptyset)$. Moreover, since $r_P(Z)=\mu(\theta(Z))$, the closure of $Z$ in $P$ is $\theta^{-1}(\theta(Z))$. Thus, the $P$-closed sets are precisely the sets $\theta^{-1}(X)$ with $X\in\Sigma$. For every such set, $r_N(\theta^{-1}(X))=\sum_{i=1}^k r_i(X)$ and $r_P(\theta^{-1}(X))=\mu(X)$. Hence, \eqref{eq:union_attainment_expansion} is exactly the hypothesis of \cref{cor:one_sided_flat_expansion} for the ordered pair $(N,P)$. It follows that there exists a common independent set $Z$ of $N$ and $P$ such that
		\begin{equation*}
			\nu(Z)
			=
			\min\left\{
			\sum_{i=1}^k r_i(J),
			\mu(J)
			\right\}.
		\end{equation*}
		
		For $i\in[k]$, set $F_i\coloneqq\theta_i(Z\cap J_i)$. Independence in $N$ gives $F_i\in\cF_i$, while independence in $P$ implies that the sets $F_1,\dots,F_k$ are pairwise disjoint. Therefore,
		\begin{equation*}
			\mu\left(\bigcup_{i=1}^kF_i\right)
			=
			\nu(Z)
			=
			\min\left\{
			\sum_{i=1}^k r_i(J),
			\mu(J)
			\right\}.
		\end{equation*}
		Since $\bigcup_{i=1}^kF_i$ belongs to $\cF_1+\dots+\cF_k$, the preceding equality and the bound
		\begin{equation*}
			r_\vee(J)\leq\min\left\{\mu(J),\sum_{i=1}^k r_i(J)\right\}
		\end{equation*}
		show that equality holds for $r_\vee(J)$. Moreover, $\bigcup_{i=1}^kF_i$ has measure $r_\vee(J)$, so it is a basis of $M_\vee$.
	\end{proof}
	
	\begin{rem}
		The countable rank formula in \cref{thm:countable_union_rank_formula} does not yield an analogous attainment result by the same reduction. With countably many matroids, the corresponding direct sum would use countably many copies of $J$, whose total measure is infinite when $\mu(J)>0$, and the fibers of the partition matroid would be countably infinite. These two obstacles lie outside, respectively, the finite-measure framework of the intersection theorem and the finite-classed partition-matroid construction used above.
	\end{rem}
	
	Packing and covering problems are among the classical applications of matroid union to combinatorial decomposition. Edmonds and Fulkerson~\cite{edmonds1965transversals} characterized when a finite family of matroids on a common ground set admits pairwise disjoint bases, and when the ground set can be partitioned into sets that are independent in the respective matroids. When all the matroids are copies of the cycle matroid of a connected graph, these specialize to packing edge-disjoint spanning trees and partitioning the edge set into forests. In the finite setting, the corresponding rank conditions give exact decompositions, since the raw union is already a matroid. In the measurable setting, the same rank conditions determine the correct union rank but, in general, give only approximate decompositions through the upward closure. The additional $(1+\varepsilon)$ margin in \cref{thm:union_attainment_expansion} is a sufficient condition for exact attainment. The corresponding classical theorems of Nash-Williams and Tutte, together with their measurable extensions obtained from the union and attainment results, are discussed in \cref{sec:graphings}.
	
	\Cref{thm:union_attainment_expansion} immediately yields the following measurable packing and covering consequences.
	
	\begin{cor}\label{cor:union_packing_covering}
		Suppose that the hypothesis of \cref{thm:union_attainment_expansion} holds with $A=J$. Then the following hold.
		\begin{enumerate}[label=(\alph*),itemsep=0em]
			\item If $\sum_{i=1}^k r_i(J)\geq\mu(J)$, then there exist pairwise disjoint sets $F_i\in\cF_i$, for $i\in[k]$, which cover $J$.
			\item If $\sum_{i=1}^k r_i(J)\leq\mu(J)$, then the matroids $M_1,\dots,M_k$ admit pairwise disjoint bases.
		\end{enumerate}
		In particular, if $\sum_{i=1}^k r_i(J)=\mu(J)$, then $J$ admits a measurable partition into bases of $M_1,\dots,M_k$, up to null sets. 
	\end{cor}
	
	\begin{proof}
		Let $F_1,\dots,F_k$ be the sets given by \cref{thm:union_attainment_expansion}. If $\sum_{i=1}^k r_i(J)\geq\mu(J)$, then
		\begin{equation*}
			\mu\left(\bigcup_{i=1}^kF_i\right)=\mu(J),
		\end{equation*}
		so the sets cover $J$. If $\sum_{i=1}^k r_i(J)\leq\mu(J)$, then
		\begin{equation*}
			\sum_{i=1}^k\mu(F_i)=\sum_{i=1}^k r_i(J).
		\end{equation*}
		Since $\mu(F_i)\leq r_i(J)$ for every $i\in[k]$, equality holds for every $i$, and hence each $F_i$ is a basis of $M_i$.
	\end{proof}
	
	%%%%%%%%%%%%%%%%%%%%%%%%%%%%%%%%
	\section{Structural Applications}
	\label{sec:applications}
	%%%%%%%%%%%%%%%%%%%%%%%%%%%%%%%%
	
	This section presents several applications of the min--max theorems proved in \cref{sec:intun}. As a first application, we specialize partition matroids to bipartite graphings, obtaining exact Hall-deficiency formulas for measurable matchings and rank-expansion criteria for perfect and unbalanced matchings. We then apply the intersection theorem to two partition matroids associated with the coordinate projections of a product space, obtaining a continuous bipartite $b$-matching formula and results on prescribed cross-sections. Applying measurable matroid union and its attainment theorem to cycle matroids of graphings gives exact criteria for approximate forest packings and coverings, as well as exact decompositions under strengthened rank inequalities. Finally, we use the union theorem to derive basis-exchange results.
	
	%%%%%%%%%%%%%%%%%%%%%%%%%%%%%%%%
	\subsection{Bipartite Matchings}
	\label{sec:bipartite}
	%%%%%%%%%%%%%%%%%%%%%%%%%%%%%%%%
	
	In the finite case, matchings in a bipartite graph are the common independent sets of two partition matroids on the edge set: one imposes degree at most one on one side, and the other imposes degree at most one on the other side. Thus, the classical min-max theorems for bipartite matchings are special cases of matroid intersection. The aim of this section is to show that the same mechanism works for bipartite graphings and that the attainment criteria from \cref{sec:expansion_attainment} recover perfect and unbalanced matching results in the measure-preserving setting.
	
	Measurable matching problems in graphings have been studied from several directions; see, for example, the survey of Pikhurko~\cite{pikhurko2021borel}. Lyons and Nazarov~\cite{lyons2011perfect} proved the existence of measurable perfect matchings under an expansion condition. Bowen, Kun, and Sabok~\cite{bowen2021perfect} studied perfect matchings in hyperfinite graphings and, in particular, developed a method for rounding measurable fractional perfect matchings. On the other hand, the examples of Laczkovich~\cite{laczkovich1988closed} and Kun~\cite{kun2021measurable} show that Hall-type conditions alone do not imply the existence of a measurable perfect matching. More recently, Bernshteyn, Bowen, and Weilacher~\cite{bernshteyn2026measurable} obtained an unbalanced matching theorem for arbitrary Borel probability measures.
	
	We first recall the finite version in the form most closely related to matroid intersection~\cite{frobenius1917zerlegbare,konig1931graphok,hall1935representatives}. Let $G=(S\cup T,E)$ be a finite bipartite graph. For $Z\subseteq E$, let $\partial_S (Z)$ and $\partial_T (Z)$ denote the sets of vertices in $S$ and $T$, respectively, incident to an edge of $Z$. For $X\subseteq S$, let $N_G(X)$ denote the neighborhood of $X$ in $T$.
	
	\begin{thm}[Frobenius, Hall, Kőnig]\label{thm:finite_bipartite_matching_minmax}
		Let $G=(S\cup T,E)$ be a finite bipartite graph. Then
		\[
		\max\{|M|\mid M\text{ is a matching in }G\}
		=
		\min_{Z\subseteq E}\left\{|\partial_S (Z)|+|\partial_T(E\setminus Z)|\right\}.
		\]
		Equivalently,
		\[
		\max\{|M|\mid M\text{ is a matching in }G\}
		=
		|S|-\max_{X\subseteq S}\left\{|X|-|N_G(X)|\right\}.
		\]
		In particular, $G$ has a matching covering $S$ if and only if $|N_G(X)|\geq |X|$ for every $X\subseteq S$.
	\end{thm}
	
	The first formula is the direct partition-matroid-intersection form. The second is the usual Hall-deficiency form: it says that the maximum number of vertices of $S$ left unmatched is exactly the maximum Hall deficiency $|X|-|N_G(X)|$.
	
	We now turn to graphings. Let $G=(J,\Sigma_J,\nu,E)$ be a bipartite graphing with Borel bipartition $J=S\cup T$. Let $\Sigma_E$ denote the Borel $\sigma$-algebra of the unoriented edge space, let $\Sigma_S$ and $\Sigma_T$ denote the Borel $\sigma$-algebras induced on $S$ and $T$, respectively, and let $\eta\coloneqq\eta_G$ be the edge measure. We continue to write $\nu$ for the restrictions of the vertex measure to $S$ and $T$. Recall that a measurable matching is a Borel set $N\subseteq E$ such that every vertex is incident to at most one edge of $N$. Under our standing identification of
	measurable edge sets modulo null sets it is equivalent to requiring
	$\deg_N(x)\leq 1$ for $\nu$-almost every $x\in J$. For $Z\in\Sigma_E$, define
	$
	\partial_S (Z)\coloneqq\{s\in S\mid \deg_Z(s)>0\}
	$
	and
	$
	\partial_T (Z)\coloneqq\{t\in T\mid \deg_Z(t)>0\}.
	$
	For $X\in\Sigma_S$, let $N_G(X)$ denote the set of vertices in $T$ adjacent to a vertex of $X$. The sets $\partial_S(Z)$, $\partial_T(Z)$, and $N_G(X)$ are Borel by \cite[Theorem~18.2]{lovasz2012large}, since each is the neighborhood of a Borel set in a bounded-degree Borel graph.
	
	The one-sided partition matroid construction from \cref{prop:bipartite_partition_matroid} gives two measurable matroids on the edge space $(E,\eta)$. The first is $M_S=(E,\Sigma_E,\eta,\cF_S)$,
	where $$\cF_S\coloneqq\{F\in\Sigma_E\mid \deg_F(s)\leq1\text{ for $\nu$-almost every }\allowbreak s\in S\},$$
	and its rank function is $r_S(Z)=\nu(\partial_S (Z))$. Similarly, let $ M_T=(E,\Sigma_E,\eta,\cF_T)$,
	where $$\cF_T\coloneqq\{F\in\Sigma_E\mid \deg_F(t)\leq1\text{ for $\nu$-almost every }t\in T\}$$ with rank function $r_T(Z)=\nu(\partial_T (Z))$. Then a set of edges is independent in both $M_S$ and $M_T$ if and only if it is a measurable matching.
	
	The measurable bipartite matching theorem below is by now folklore. It follows from the local-im\-prove\-ment argument of Elek and Lippner~\cite[Proposition~1.1]{elek2009borel}, which gives, for every fixed $q$, a matching with no augmenting path containing at most $q$ matched edges. Applying the usual alternating-path argument to these matchings and letting $q\to\infty$ gives the result. Here we deduce it from \cref{thm:intersection}.
	
	\begin{thm}\label{thm:graphing_bipartite_matching_minmax}
		Let $G=(J,\Sigma_J,\nu,E)$ be a bipartite graphing with Borel bipartition $J=S\cup T$. Then
		\[
		\sup\{\eta(M)\mid M\text{ is a measurable matching in }G\}
		=
		\min_{Z\in\Sigma_E}
		\left\{
		\nu(\partial_S (Z))+\nu(\partial_T(E\setminus Z))
		\right\}.
		\]
		Equivalently,
		\[
		\sup\{\eta(M)\mid M\text{ is a measurable matching in }G\}
		=
		\nu(S)-\max_{X\in\Sigma_S}
		\left\{
		\nu(X)-\nu(N_G(X))
		\right\}.
		\]
	\end{thm}
	
	\begin{proof}
		Applying the measurable matroid intersection theorem to $M_S$ and $M_T$ gives
		\[
		\sup\{\eta(M)\mid M\in\cF_S\cap\cF_T\}
		=
		\min_{Z\subseteq E}
		\left\{
		r_S(Z)+r_T(E\setminus Z)
		\right\}.
		\]
		Since $\cF_S\cap\cF_T$ is precisely the family of measurable matchings, and since $r_S(Z)=\nu(\partial_S (Z))$ and $r_T(E\setminus Z)=\nu(\partial_T(E\setminus Z))$, this proves the first formula.
		
		We now rewrite the minimum in Hall-deficiency form. Let $Z\subseteq E$, and put $A\coloneqq S\setminus\partial_S (Z)$. Every edge incident to $A$ belongs to $E\setminus Z$, so $N_G(A)\subseteq\partial_T(E\setminus Z)$. Hence
		\[
		\nu(\partial_S (Z))+\nu(\partial_T(E\setminus Z))
		\geq
		\nu(S\setminus A)+\nu(N_G(A))
		=
		\nu(S)-\left(\nu(A)-\nu(N_G(A))\right).
		\]
		Taking the minimum over $Z$ gives one inequality. Conversely, fix $A\subseteq S$, and let $Z_A$ be the set of all edges whose endpoint in $S$ lies in $S\setminus A$. Then $\partial_S (Z_A)\subseteq S\setminus A$ and $\partial_T(E\setminus Z_A)=N_G(A)$. Therefore
		\[
		\nu(\partial_S (Z_A))+\nu(\partial_T(E\setminus Z_A))
		\leq
		\nu(S\setminus A)+\nu(N_G(A))
		=
		\nu(S)-\left(\nu(A)-\nu(N_G(A))\right).
		\]
		Taking the supremum over $A\subseteq S$ gives the reverse inequality, proving the second formula.
	\end{proof}
	
	The preceding theorem can be viewed as the deficiency form of the measurable Hall theorem: it determines the supremal size of a measurable matching also when Hall's condition fails. We mention separately the zero-deficiency case, which follows immediately from the existence of Borel matchings without short augmenting paths proved by Elek and Lippner~\cite[Proposition~1.1]{elek2009borel}.
	
	\begin{cor}\label{cor:graphing_approx_hall}
		Let $G=(J,\Sigma_J,\nu,E)$ be a bipartite graphing with Borel bipartition $J=S\cup T$. The following are equivalent.
		\begin{enumerate}[label=(\roman*)]\itemsep0em
			\item For every $\varepsilon>0$, there exists a measurable matching $M$ such that $\nu(S\setminus\partial_S(M))\leq\varepsilon$.
			\item $\nu(N_G(X))\geq\nu(X)$ for every $X\in\Sigma_S$.
		\end{enumerate}
	\end{cor}
	
	\begin{proof}
		If the Hall condition holds, then $\sup_{X\subseteq S}\{\nu(X)-\nu(N_G(X))\}=0$, since the value for $X=\emptyset$ is $0$. By \cref{thm:graphing_bipartite_matching_minmax}, the supremum of $\eta(M)$ over measurable matchings is $\nu(S)$. Since $\eta(M)=\nu(\partial_S(M))$ for every measurable matching $M$, there are measurable matchings covering subsets of $S$ of measure arbitrarily close to $\nu(S)$.
		
		Conversely, suppose that such approximate matchings exist, and fix $X\subseteq S$ and $\varepsilon>0$. Choose a measurable matching $M$ with
		\[
		\nu(S\setminus\partial_S(M))\leq\varepsilon.
		\]
		By \cref{lem:graphing_matching_tools}, the partner map $\tau_M$ is measure-preserving, and its restriction maps $X\cap\partial_S(M)$ injectively into $N_G(X)$.
		Therefore,
		\[
		\nu(X)
		\leq\varepsilon+\nu(X\cap\partial_S(M))
		\leq\varepsilon+\nu(N_G(X)).
		\]
		Since $\varepsilon>0$ was arbitrary, the Hall condition follows.
	\end{proof}
	
	An important special case is the regular bipartite case, which is often considered in the study of matchings in graphings.
	
	\begin{cor}\label{cor:regular_bipartite_graphing_approx_matching}
		Let $G=(J,\Sigma_J,\nu,E)$ be a $d$-regular bipartite graphing with Borel bipartition $J=S\cup T$, where $d\geq1$. Then, for every $\varepsilon>0$, there exists a measurable matching $M$ such that $\nu(S\setminus\partial_S (M))\leq\varepsilon$.
	\end{cor}
	
	\begin{proof}
		We verify the Hall condition. Let $X\in\Sigma_S$. By the graphing mass-transport identity~\cite{lovasz2012large}, 
		\[
		d \cdot \nu(X)
		=
		\int_X \deg(T,x)\,\diff\nu(x)
		=
		\int_{N_G(X)} \deg(X,y)\,\diff\nu(y)
		\leq
		d \cdot \nu(N_G(X)).
		\]
		Thus, $\nu(X)\leq\nu(N_G(X))$, and the result follows from \cref{cor:graphing_approx_hall}.
	\end{proof}
	
	\begin{rem}\label{rem:match}
		The conclusion of \cref{cor:graphing_approx_hall} cannot in general be strengthened to a matching covering all of $S$. Indeed, the graphing in \cref{rem:laczkovich_matching} is $2$-regular apart from a finite set, and hence satisfies the Hall condition by the same mass-transport argument as in \cref{cor:regular_bipartite_graphing_approx_matching}. However, it has no measurable perfect matching. More generally, Kun~\cite{kun2021measurable} constructed, for every $d\geq 2$, a $d$-regular measurably bipartite treeing with no measurable perfect matching.
	\end{rem}
	
	\Cref{rem:match} shows that Hall's condition alone does not force the supremum in \cref{thm:graphing_bipartite_matching_minmax} to be attained. In general, the measurable Hall condition guarantees only approximately perfect matchings. An important sufficient condition for exact attainment is expansion. The following theorem of Lyons and Nazarov~\cite{lyons2011perfect}, stated here in the graphing form used, for instance, by Grabowski, M\'ath\'e and Pikhurko~\cite{grabowski2022measurable}, is a fundamental example. We state the theorem in the form used in the literature. With the expansion convention of \cref{sec:expansion_attainment}, \cref{thm:intersection_attainment_expansion} also gives the same conclusion if the first inequality is weakened to $\nu(N_G(X))\geq\frac12\nu(S)$.
	
	\begin{thm}[Lyons--Nazarov]\label{thm:lyons_nazarov_expander_matching}
		Let $G=(J,\Sigma_J,\nu,E)$ be a locally finite bipartite graphing with Borel bipartition $J=S\cup T$ and $\nu(S)=\nu(T)<\infty$. Suppose that there exists $c>0$ such that, for every $X\in\Sigma_S\cup\Sigma_T$,
		\[
		\nu(N_G(X))>\frac{1}{2}\nu(S)
		\quad\text{or}\quad
		\nu(N_G(X))\geq (1+c)\nu(X).
		\]
		Then $G$ admits a measurable matching covering both $S$ and $T$.
	\end{thm}
	\begin{proof}
		Set $R\coloneqq\nu(S)=\nu(T)$. If $R=0$, then the assertion is immediate. The expansion assumption implies that the sets of isolated vertices in $S$ and $T$ are null. Hence, the measurable partition matroids $M_S$ and $M_T$ on the edge space have common rank $R$.
		
		Let $C$ be an $M_S$-closed set, and set $X\coloneqq S\setminus\partial_S(C)$. Since $C$ is a union of edge classes corresponding to vertices of $S$, the set $E\setminus C$ consists of all edges incident with $X$. Therefore, $r_S(E\setminus C)=\nu(X)$ and $r_T(E\setminus C)=\nu(N_G(X))$. Similarly, if $C$ is $M_T$-closed and $Y\coloneqq T\setminus\partial_T(C)$, then $r_T(E\setminus C)=\nu(Y)$ and $r_S(E\setminus C)=\nu(N_G(Y))$. Thus, $(M_S,M_T)$ and $(M_T,M_S)$ are both $(c,1/2)$-co-expanding. By \cref{thm:intersection_attainment_expansion}, the two matroids have a common basis. Such a common basis is a measurable matching covering both $S$ and $T$.
	\end{proof}
	
	The measure-preserving case of the following unbalanced matching result was already known from the Lyons--Nazarov method~\cite{lyons2011perfect}; see also~\cite[Section~2]{conley2018folner} for a detailed presentation. Here it follows directly from \cref{cor:one_sided_flat_expansion}. It is also worth mentioning that Bernshteyn, Bowen, and Weilacher~\cite[Section~1.2 and Theorem~1.4]{bernshteyn2026measurable} recently extended the result to arbitrary Borel probability measures.
	
	\begin{cor}
		\label{cor:unbalanced-graphing-matching}
		Let $G=(J,\Sigma_J,\nu,E)$ be a locally finite bipartite graphing with Borel bipartition $J=S\cup T$. Suppose that there are integers $a>b\ge0$ such that $\deg_G(s)\ge a$ and $\deg_G(t)\le b$ for every $s\in S$ and $t\in T$, respectively. Then $G$ admits a measurable matching covering $S$.
	\end{cor}
	
	\begin{proof}
		If $\nu(S)=0$, the empty matching suffices, so assume that $\nu(S)>0$. Then $b\ge1$, since otherwise $E=\emptyset$, contradicting $\deg_G(s)\ge a\ge1$ for every $s\in S$. Let $M_S$ and $M_T$ be the two partition matroids on the edge space, with rank functions $r_S$ and $r_T$. If $C\subseteq E$ is $M_S$-closed, then $C$ consists of all edges incident with a set $X\subseteq S$. Degree counting gives
		\begin{equation*}
			a\nu(X)\le\eta(C)\le b\nu(N_G(X)),
		\end{equation*}
		and hence
		\begin{equation*}
			r_T(C)=\nu(N_G(X))\ge\frac{a}{b}\nu(X)=\frac{a}{b}r_S(C).
		\end{equation*}
		Both partition matroids are loopless, $r_S(E)=\nu(S)$, and the same inequality with $C=E$ gives $r_T(E)\ge(a/b) \nu(S)>\nu(S)$. \cref{cor:one_sided_flat_expansion}, applied with $M_1=M_T$, $M_2=M_S$, and $\varepsilon=a/b-1$, therefore gives a measurable matching of measure $\nu(S)$, which covers $S$.
	\end{proof}
	
	%%%%%%%%%%%%%%%%%%%%%%%%%%%%%%%%
	\subsection{Continuous Bipartite \texorpdfstring{$b$}{b}-Matchings and Prescribed Cross-Sections}
	\label{sec:continuous_b_matching}
	%%%%%%%%%%%%%%%%%%%%%%%%%%%%%%%%
	
	The two fiberwise partition matroids associated with the coordinate projections give an atomless version of capacitated bipartite matching. This problem is closely related to capacity-constrained optimal transport. Indeed, identifying a measurable set $F\subseteq S\times T$ with its indicator function $1_F$, the section measures become marginal constraints, while $F\subseteq E$ is equivalent to the pointwise capacity constraint $0\leq 1_F\leq 1_E$. Replacing $1_F$ by an arbitrary density $h$ with $0\leq h\leq 1_E$ gives the corresponding fractional transport problem. General duality results for marginal problems were proved by Kellerer~\cite{kellerer1984duality}, and the fractional problem is also covered by the measurable flow and transshipment theory of Lovász~\cite{lovasz2021flows}. Our intersection theorem gives the min--max formula for the set-valued problem, while weak-star compactness and the Krein--Milman theorem~\cite[Chapter~V]{conway1990course}, together with the extreme-point characterization of Korman and McCann~\cite[Proposition~3.2]{korman2013insights}, give attainment.
	
	We begin with the classical theorem of Lorentz~\cite{lorentz1949plane} characterizing the possible cross-sections of a measurable subset of the unit square.
	
	\begin{thm}[Lorentz]
		\label{thm:lorentz_cross_sections}
		Let $b,c\colon[0,1]\to[0,1]$ be measurable, and let $b^\downarrow\colon[0,1]\to[0,1]$ denote the nonincreasing rearrangement of $b$, defined by $b^\downarrow(t)\coloneqq \inf\left\{s\in[0,1]\colon\lambda(\{x\in[0,1]\mid b(x)>s\})\leq t\right\}$. There exists a measurable set $F\subseteq[0,1]^2$ such that, writing $F_y\coloneqq\{x\in[0,1]\mid(x,y)\in F\}$ and $F^x\coloneqq\{y\in[0,1]\mid(x,y)\in F\}$,
		\[
		\lambda(F_y)=b(y)
		\quad\text{and}\quad
		\lambda(F^x)=c(x)
		\]
		for almost every $y\in[0,1]$ and almost every $x\in[0,1]$ if and only if $\int_0^1 b(y)\diff y=\int_0^1 c(x)\diff x$ and, for all $0\leq a\leq1$, 
		\[
		\int_0^a b^\downarrow(u)\diff u
		\leq
		\int_0^1\min\{a,c(x)\}\diff x.
		\]
	\end{thm}
	
	A second classical result generalized by the measurable formula below is the integral max-flow--min-cut theorem for capacitated bipartite networks. In the finite setting, it gives the following min--max formula for the maximum size of a bipartite matching subject to prescribed degree capacities; see, for example, Gale~\cite{gale1957flows}.
	
	\begin{thm}\label{thm:finite_capacitated_b_matching}
		Let $G=(S,T;E)$ be a finite bipartite graph, and let $b\colon T\to\bZ_{\geq 0}$ and $c\colon S\to\bZ_{\geq 0}$. For $A\subseteq T$ and $s\in S$, let $d_E(A,s)$ denote the number of edges joining $s$ to a vertex in $A$. Then the largest cardinality of an edge set $F\subseteq E$ satisfying $\deg_F(t)\leq b(t)$ for every $t\in T$ and $\deg_F(s)\leq c(s)$ for every $s\in S$ is
		\[
		\min_{A\subseteq T}
		\left\{
		\sum_{t\in T\setminus A}b(t)
		+
		\sum_{s\in S}\min\{c(s),d_E(A,s)\}
		\right\}.
		\]
	\end{thm}
	
	We now derive the corresponding atomless min--max formula from the intersection of the two fiberwise partition matroids. Its fractional form is also covered by measurable max-flow--min-cut and transshipment theory~\cite{lovasz2021flows}.
	
	Let $(S,\Sigma_S,\nu)$ and $(T,\Sigma_T,\eta)$ be standard measure spaces, and set $J\coloneqq S\times T$, $\Sigma\coloneqq\Sigma_S\otimes\Sigma_T$, and $\mu\coloneqq\nu\otimes\eta$. Let $b\colon T\to[0,\nu(S)]$ and $c\colon S\to[0,\eta(T)]$ be Borel measurable, and let $\cF_{\mathrm h}(b)$ and $\cF_{\mathrm v}(c)$ denote the families defined by the conditions $\nu(F_t)\leq b(t)$ for $\eta$-almost every $t\in T$ and $\eta(F^s)\leq c(s)$ for $\nu$-almost every $s\in S$, respectively. By \cref{prop:fiberwise_partition_matroid}, $M_{\mathrm h}(b)\coloneqq(J,\Sigma,\mu,\cF_{\mathrm h}(b))$ and $M_{\mathrm v}(c)\coloneqq(J,\Sigma,\mu,\cF_{\mathrm v}(c))$ are measurable matroids, with rank functions $r_{\mathrm h}(Z)=\int_T\min\{\nu(Z_t),b(t)\}\diff\eta(t)$ and $r_{\mathrm v}(Z)=\int_S\min\{\eta(Z^s),c(s)\}\diff\nu(s)$, respectively.
	
	Fix a measurable set $E\subseteq J$ of admissible pairs. For $A\subseteq T$ and $s\in S$, let $d_E(A,s)\coloneqq\eta(E^s\cap A)$, and set
	\begin{equation}
		m_E(b,c)\coloneqq
		\sup\{\mu(F)\mid F\subseteq E,\ F\in\cF_{\mathrm h}(b)\cap\cF_{\mathrm v}(c)\}.\label{eq:mebc}
	\end{equation}
	The finite and measurable settings are illustrated in \cref{fig:continuous_b_matching}.
	
	\begin{figure}[t]
		
		\begin{subfigure}[t]{0.48\textwidth}
			\centering
			
			\begin{tikzpicture}[scale=0.78]
				
				% Vertices
				\coordinate (s1) at (0,4.8);
				\coordinate (s2) at (0,3.2);
				\coordinate (s3) at (0,1.6);
				\coordinate (s4) at (0,0);
				
				\coordinate (t1) at (5.6,4.8);
				\coordinate (t2) at (5.6,3.2);
				\coordinate (t3) at (5.6,1.6);
				\coordinate (t4) at (5.6,0);
				
				% Admissible edges E
				\draw[gray!45, line width=0.7pt] (s1) -- (t1);
				\draw[gray!45, line width=0.7pt] (s1) -- (t2);
				\draw[gray!45, line width=0.7pt] (s1) -- (t3);
				
				\draw[gray!45, line width=0.7pt] (s2) -- (t1);
				\draw[gray!45, line width=0.7pt] (s2) -- (t3);
				\draw[gray!45, line width=0.7pt] (s2) -- (t4);
				
				\draw[gray!45, line width=0.7pt] (s3) -- (t1);
				\draw[gray!45, line width=0.7pt] (s3) -- (t2);
				\draw[gray!45, line width=0.7pt] (s3) -- (t4);
				
				\draw[gray!45, line width=0.7pt] (s4) -- (t2);
				\draw[gray!45, line width=0.7pt] (s4) -- (t3);
				\draw[gray!45, line width=0.7pt] (s4) -- (t4);
				
				% Feasible capacitated matching F
				\draw[black!85, line width=1.6pt] (s1) -- (t1);
				\draw[black!85, line width=1.6pt] (s1) -- (t2);
				\draw[black!85, line width=1.6pt] (s2) -- (t3);
				\draw[black!85, line width=1.6pt] (s3) -- (t2);
				\draw[black!85, line width=1.6pt] (s3) -- (t4);
				\draw[black!85, line width=1.6pt] (s4) -- (t3);
				
				% Vertex circles
				\foreach \p in {s1,s2,s3,s4,t1,t2,t3,t4}
				\fill[white] (\p) circle (0.13);
				
				\foreach \p in {s1,s2,s3,s4,t1,t2,t3,t4}
				\draw[black!85, line width=0.9pt] (\p) circle (0.13);
				
				% Vertex labels
				\node[font=\small, above right] at (s1) {$s_1$};
				\node[font=\small, above right] at (s2) {$s_2$};
				\node[font=\small, above right] at (s3) {$s_3$};
				\node[font=\small, above right] at (s4) {$s_4$};
				
				\node[font=\small, above left] at (t1) {$t_1$};
				\node[font=\small, above left] at (t2) {$t_2$};
				\node[font=\small, above left] at (t3) {$t_3$};
				\node[font=\small, above left] at (t4) {$t_4$};
				
				% Capacity labels
				\node[font=\small, left] at (-0.25,4.8) {$c=2$};
				\node[font=\small, left] at (-0.25,3.2) {$c=1$};
				\node[font=\small, left] at (-0.25,1.6) {$c=2$};
				\node[font=\small, left] at (-0.25,0) {$c=1$};
				
				\node[font=\small, right] at (5.85,4.8) {$b=1$};
				\node[font=\small, right] at (5.85,3.2) {$b=2$};
				\node[font=\small, right] at (5.85,1.6) {$b=2$};
				\node[font=\small, right] at (5.85,0) {$b=1$};
				
				% Side labels
				\node[font=\small] at (0,5.45) {$S$};
				\node[font=\small] at (5.6,5.45) {$T$};
			\end{tikzpicture}
			
			\caption{A finite bipartite graph with degree bounds. The light edges form $E$, while the dark edges form a feasible capacitated matching $F$.}
			\label{fig:continuous_b_matching_finite}
			
		\end{subfigure}
		\hfill
		\begin{subfigure}[t]{0.48\textwidth}
			\centering
			
			\begin{tikzpicture}[scale=0.78]
				
				% Paths used below
				\def\Epath{
					(0.45,0.65)
					.. controls (1.25,0.20) and (2.10,0.45) .. (2.75,0.80)
					.. controls (3.55,1.20) and (4.65,0.35) .. (5.45,0.70)
					.. controls (5.85,1.35) and (5.25,2.20) .. (5.45,2.90)
					.. controls (5.65,3.80) and (5.85,4.80) .. (5.25,5.45)
					.. controls (4.35,5.80) and (3.75,5.25) .. (3.00,5.05)
					.. controls (2.05,4.80) and (1.30,5.65) .. (0.65,5.20)
					.. controls (0.15,4.50) and (0.75,3.70) .. (0.65,2.85)
					.. controls (0.55,2.05) and (0.15,1.25) .. (0.45,0.65)
					-- cycle
				}
				
				\def\Flow{
					(2.60,0.85)
					.. controls (3.05,0.65) and (3.75,0.75) .. (4.20,1.05)
					.. controls (4.40,1.35) and (4.20,1.70) .. (3.85,1.90)
					.. controls (3.40,2.05) and (2.85,1.85) .. (2.55,1.55)
					.. controls (2.35,1.25) and (2.35,1.00) .. (2.60,0.85)
					-- cycle
				}
				
				\def\Fcentral{
					(1.15,2.65)
					.. controls (1.85,2.35) and (2.55,2.65) .. (3.10,2.45)
					.. controls (3.85,2.20) and (4.70,2.55) .. (4.85,3.10)
					.. controls (4.95,3.65) and (4.40,4.10) .. (3.75,4.15)
					.. controls (3.00,4.25) and (2.45,3.85) .. (1.75,4.05)
					.. controls (1.20,4.10) and (0.90,3.55) .. (1.15,2.65)
					-- cycle
				}
				
				\def\Fupper{
					(1.25,4.55)
					.. controls (1.70,4.35) and (2.15,4.45) .. (2.55,4.75)
					.. controls (2.75,5.00) and (2.45,5.30) .. (2.05,5.35)
					.. controls (1.55,5.40) and (1.10,5.15) .. (1.25,4.55)
					-- cycle
				}
				
				% Admissible region E
				\fill[gray!10] \Epath;
				
				% Feasible set F, clipped to E
				\begin{scope}
					\clip \Epath;
					\fill[gray!45] \Flow;
					\fill[gray!45] \Fcentral;
					\fill[gray!45] \Fupper;
				\end{scope}
				
				% Boundary of E
				\draw[black!55, line width=0.8pt] \Epath;
				
				% Chosen horizontal and vertical fibers
				\draw[gray!70, densely dotted, line width=0.8pt]
				(0,3.20) -- (6,3.20);
				
				\draw[gray!70, densely dotted, line width=0.8pt]
				(3.45,0) -- (3.45,6);
				
				% The thick portions of the fibers inside F are
				% precisely F_{t_0} and F^{s_0}.
				\begin{scope}
					\clip \Epath;
					
					% Intersection with the central component
					\begin{scope}
						\clip \Fcentral;
						\draw[black!90, line width=2.1pt]
						(0,3.20) -- (6,3.20);
						\draw[black!90, line width=2.1pt]
						(3.45,0) -- (3.45,6);
					\end{scope}
					
					% Lower component of the vertical section
					\begin{scope}
						\clip \Flow;
						\draw[black!90, line width=2.1pt]
						(3.45,0) -- (3.45,6);
					\end{scope}
				\end{scope}
				
				% Ground square
				\draw[very thick] (0,0) rectangle (6,6);
				
				% Coordinate spaces
				\node[label A] at (3,-0.65) {\small $S$};
				\node[label A] at (-0.65,2.7) {\small $T$};
				
				% Labels for E and F
				\node[font=\small] at (4.85,4.85) {$E$};
				\node[font=\small] at (2.15,3.55) {$F$};
				
				% Fiber labels
				\node[font=\small, left] at (-0.05,3.20) {$t_0$};
				\node[font=\small, below] at (3.45,-0.05) {$s_0$};
				
				% Section labels
				\node[font=\small, above] at (4.10,3.20) {$F_{t_0}$};
				\node[font=\small, right] at (3.45,2.18) {$F^{s_0}$};
				
			\end{tikzpicture}
			
			\caption{A feasible measurable set $F\subseteq E\subseteq S\times T$. The thick portions of the dotted fibers are the sections $F_{t_0}$ and $F^{s_0}$.}
			\label{fig:continuous_b_matching_measurable}
			
		\end{subfigure}
		
		\caption{Finite and continuous bipartite $b$-matchings.}
		\label{fig:continuous_b_matching}
		
	\end{figure}
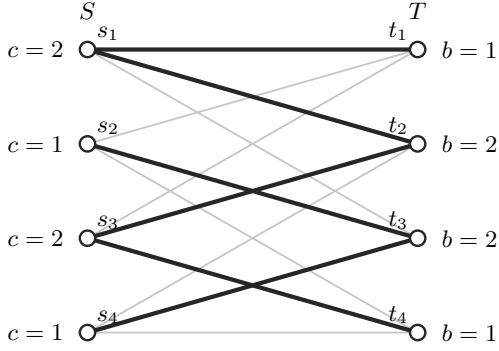
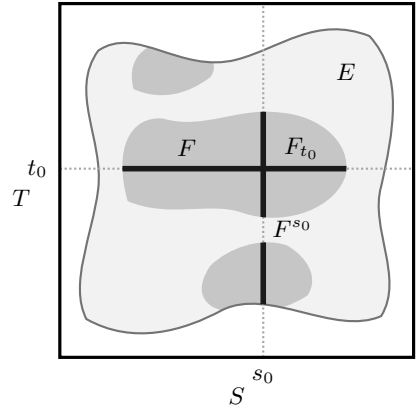
	
	\begin{thm}
		\label{thm:continuous_matching_minmax}
		We have
		\begin{equation}
			\label{eq:continuous_matching_hall}
			m_E(b,c)
			=
			\min_{A\subseteq T}
			\left\{
			\int_{T\setminus A} b(t)\diff\eta(t)
			+
			\int_S \min\{c(s),d_E(A,s)\}\diff\nu(s)
			\right\}.
		\end{equation}
	\end{thm}
	
	\begin{proof}
		Restrict $M_{\mathrm h}(b)$ and $M_{\mathrm v}(c)$ to $E$. The measurable matroid intersection theorem and the rank formulas give
		\[
		m_E(b,c)
		=
		\min_{Z\subseteq E}
		\left\{
		\int_T\min\{\nu(Z_t),b(t)\}\diff\eta(t)
		+
		\int_S\min\{\eta((E\setminus Z)^s),c(s)\}\diff\nu(s)
		\right\}.
		\]
		For nonnegative integrable measurable functions $u$ and $v$ on a standard measure space $(J',\Sigma',\rho)$, we have the elementary identity
		\[
		\int_{J'} \min\{u,v\}\diff\rho
		=
		\min_{C\in\Sigma'}
		\left\{
		\int_C u\diff\rho
		+
		\int_{J'\setminus C} v\diff\rho
		\right\}.
		\]
		Indeed, for every $C\in\Sigma'$, we have $\min\{u,v\}\leq u\mathbf{1}_C+v\mathbf{1}_{J'\setminus C}$ pointwise, with equality for $C=\{u\leq v\}$. Applying this identity to the first rank term with $C=A$, and to the second with $C=S\setminus B$, gives
		\[
		m_E(b,c)=\min_{\substack{A\subseteq T\\ B\subseteq S\\ Z\subseteq E}}
		\left\{
		\int_{T\setminus A} b\,\diff\eta
		+\mu\left(Z\cap(S\times A)\right)
		+\mu\left((E\setminus Z)\cap((S\setminus B)\times T)\right)
		+\int_B c\,\diff\nu
		\right\}.
		\]
		For fixed $A$ and $B$, the minimum over $Z$ is attained by
		\[
		Z=E\cap\left((S\setminus B)\times(T\setminus A)\right).
		\]
		The two middle terms then reduce to $\mu\left(E\cap((S\setminus B)\times A)\right)$, and hence
		\[
		m_E(b,c)=\min_{\substack{A\subseteq T\\ B\subseteq S}}
		\left\{
		\int_{T\setminus A} b\,\diff\eta
		+\mu\left(E\cap((S\setminus B)\times A)\right)
		+\int_B c\,\diff\nu
		\right\}.
		\]
		By Fubini's theorem, the middle term is $\int_{S\setminus B}d_E(A,s)\diff\nu(s)$. For fixed $A$, minimizing over $B$ pointwise gives \eqref{eq:continuous_matching_hall}.
	\end{proof}

	In contrast to general measurable matroid intersection, the supremum is always attained for the two fiberwise partition matroids considered here. We will rely on a result of Korman and McCann~\cite[Proposition~3.2]{korman2013insights} for joint densities on $\bR^d\times\bR^d$, which states that, among those with prescribed marginals satisfying $0\leq h\leq\bar h$, the extreme points are precisely the densities of the form $h=\mathbf{1}_W\bar h$ for some measurable set $W$.
	
	\begin{thm}
		\label{thm:continuous_matching_attainment}
		The supremum in \eqref{eq:mebc} is attained.
	\end{thm}
	
	\begin{proof}
		As usual, by the isomorphism theorem for atomless standard probability spaces \cite[Theorem~17.41]{kechris1995classical}, we may identify $S$ and $T$ with finite intervals equipped with Lebesgue measure.
		
		If $m_E(b,c)=0$, then the empty set attains the supremum. Suppose that $m_E(b,c)>0$, let $(F_i)_{i\in\bN}$ be a sequence of measurable subsets of $E$ belonging to $\cF_{\mathrm h}(b)\cap\cF_{\mathrm v}(c)$ such that $\lim_{i\to\infty}\mu(F_i)=m_E(b,c)$, and set $h_i\coloneqq\mathbf{1}_{F_i}$.
		Since $0\leq h_i\leq\mathbf{1}_E\leq\mathbf{1}_J$ for every $i\in\bN$, the weak-star sequential compactness in \cref{rem:weak_star_compactness} shows that, after passing to a subsequence, $(h_i)_{i\in\bN}$ converges weak-star to some $h\in L^\infty(\mu)$ with $0\leq h\leq\mathbf{1}_E$.
		
		Define the two marginals of $h$ by
		\[
		p(t)\coloneqq\int_S h(s,t)\diff\nu(s),
		\qquad
		q(s)\coloneqq\int_T h(s,t)\diff\eta(t).
		\]
		For every measurable set $A\subseteq T$, weak-star convergence and Fubini's theorem give
		\[
		\int_A p(t)\diff\eta(t)
		=
		\int_{S\times A}h\diff\mu
		=
		\lim_{i\to\infty}\int_{S\times A}h_i\diff\mu
		=
		\lim_{i\to\infty}\int_A\nu((F_i)_t)\diff\eta(t)
		\leq
		\int_A b(t)\diff\eta(t).
		\]
		It follows that $p\leq b$ almost everywhere. The same argument, applied to measurable subsets of $S$, shows that $q\leq c$ almost everywhere. Testing the weak-star convergence against the constant function $\mathbf{1}_J$ also gives
		\[
		\int_Tp\diff\eta
		=
		\int_Jh\diff\mu
		=
		\lim_{i\to\infty}\mu(F_i)
		=
		m_E(b,c),
		\]
		and, similarly, $\int_Sq\diff\nu=m_E(b,c)$.
		
		Let $\cK$ be the set of all functions $g\in L^\infty(\mu)$ satisfying $0\leq g\leq\mathbf{1}_E$,
		\[
		\int_S g(s,t)\diff\nu(s)=p(t)
		\text{ for almost every }t\in T
		\quad\text{and}\quad
		\int_T g(s,t)\diff\eta(t)=q(s)
		\text{ for almost every }s\in S.
		\]
		This set is nonempty, since $h\in\cK$, and it is convex. The constraint $0\leq g\leq\mathbf{1}_E$ is weak-star closed by the same argument as in \cref{rem:weak_star_compactness}. Moreover, by Fubini's theorem, the marginal constraints are equivalent to
		\[
		\int_{S\times A}g\diff\mu=\int_Ap\diff\eta
		\quad\text{for every }A\in\Sigma_T,
		\qquad
		\int_{B\times T}g\diff\mu=\int_Bq\diff\nu
		\quad\text{for every }B\in\Sigma_S.
		\]
		Indeed, equality of the integrals over every measurable set implies equality of the corresponding marginal functions almost everywhere. Each equality above defines a weak-star closed set, hence $\cK$ is weak-star closed, and hence is a weak-star compact convex subset of the unit ball of $L^\infty(\mu)$. By the Krein--Milman theorem~\cite[Chapter~V]{conway1990course}, $\cK$ has an extreme point, say $h^*$.
		
		The result of Korman and McCann~\cite[Proposition~3.2]{korman2013insights} applies and yields $h^*=\mathbf{1}_F$ almost everywhere for some measurable set $F \subseteq S \times T$.
		
		Replacing $F$ by a Borel representative modulo a null set, if necessary, does not change its marginals almost everywhere. Thus, $F\subseteq E$ has marginals $p$ and $q$, and
		\[
		\mu(F)=\int_Tp\diff\eta=m_E(b,c).
		\]
		Hence, $F$ attains the supremum.
		
	\end{proof}

	Combining \cref{thm:continuous_matching_minmax} and \cref{thm:continuous_matching_attainment} gives the exact atomless counterpart of \cref{thm:finite_capacitated_b_matching}: sums are replaced by integrals, and the number of available edges from $A$ to a vertex $s$ is replaced by the section measure $d_E(A,s)$. It also gives the exact deficiency formula, rather than only a feasibility criterion.
	
	\begin{cor}
		\label{cor:continuous_hall}
		There exists a measurable set $F\subseteq E$ such that $\nu(F_t)=b(t)$ for $\eta$-almost every $t\in T$ and $\eta(F^s)=c(s)$ for $\nu$-almost every $s\in S$ if and only if
		\[
		\int_Tb(t)\diff\eta(t)=\int_Sc(s)\diff\nu(s)
		\]
		and, for every measurable $A\subseteq T$,
		\[
		\int_A b(t)\diff\eta(t)
		\leq
		\int_S\min\{c(s),d_E(A,s)\}\diff\nu(s).
		\]
	\end{cor}
	
	\begin{proof}
		Necessity follows by integrating the two section functions. Indeed, for $\nu$-almost every $s\in S$, we have $\eta(F^s\cap A)\leq\min\{c(s),d_E(A,s)\}$. Conversely, the inequalities and \cref{thm:continuous_matching_minmax} give $m_E(b,c)=\int_Tb\diff\eta$. By \cref{thm:continuous_matching_attainment}, there is a feasible set $F$ of this measure. Since $\nu(F_t)\leq b(t)$ almost everywhere and the integrals agree, equality holds almost everywhere. Equality of the total capacities gives the same conclusion for the vertical sections.
	\end{proof}
	
	We finally specialize to the complete product. Let $b^\downarrow\colon[0,\eta(T)]\to\bR_{\geq 0}$ denote the nonincreasing rearrangement of $b$ with respect to $\eta$.
	
	\begin{cor}
		\label{cor:lorentz_deficiency}
		We have
		\[
		m_{S\times T}(b,c)
		=
		\int_Tb\diff\eta
		-
		\max_{0\leq a\leq\eta(T)}
		\left\{
		\int_0^ab^\downarrow(u)\diff u
		-
		\int_S\min\{a,c(s)\}\diff\nu(s)
		\right\}.
		\]
		In particular, exact cross-sections $b$ and $c$ can be realized if and only if their integrals agree and the expression inside the maximum is nonpositive for every $a\in[0,\eta(T)]$.
	\end{cor}
	
	\begin{proof}
		For $E=S\times T$, we have $d_E(A,s)=\eta(A)$. Since $T$ is atomless, the definition of $b^\downarrow$ implies the standard rearrangement identity
		\[
		\max\left\{\int_A b\diff\eta\,\middle|\,A\subseteq T,\ \eta(A)=a\right\}=\int_0^a b^\downarrow(u)\diff u
		\]
		for every $a\in[0,\eta(T)]$. By \eqref{eq:continuous_matching_hall},
		\begin{align*}
			m_{S\times T}(b,c)
			&=\min_{A\subseteq T}\left\{\int_{T\setminus A}b\diff\eta+\int_S\min\{\eta(A),c(s)\}\diff\nu(s)\right\}\\
			&=\int_Tb\diff\eta-\max_{A\subseteq T}\left\{\int_A b\diff\eta-\int_S\min\{\eta(A),c(s)\}\diff\nu(s)\right\}\\
			&=\int_Tb\diff\eta-\max_{0\leq a\leq\eta(T)}\left\{\int_0^a b^\downarrow(u)\diff u-\int_S\min\{a,c(s)\}\diff\nu(s)\right\}.
		\end{align*}
		The final assertion follows from the preceding corollary, since for each $a\in[0,\eta(T)]$ the left-hand side of its Hall-type inequality is maximized over sets $A\subseteq T$ with $\eta(A)=a$ by $\int_0^a b^\downarrow(u)\diff u$.
	\end{proof}
	
	For $S=T=[0,1]$, the feasibility part of \cref{cor:lorentz_deficiency} is precisely \cref{thm:lorentz_cross_sections}.

	%%%%%%%%%%%%%%%%%%%%%%%%%%%%%%%%
	\subsection{Measurable Nash-Williams--Tutte Theorems}
	\label{sec:graphings}
	%%%%%%%%%%%%%%%%%%%%%%%%%%%%%%%%
	
	Packing spanning trees and covering edges by forests are two fundamental decomposition problems in graph theory. The forest-covering theorem is due to Nash-Williams~\cite{nash1964decomposition}, while the spanning-tree packing theorem was proved independently by Nash--Williams and Tutte~\cite{nash1961edge,tutte1961problem}. These results give exact min--max characterizations for the two problems, determining the arboricity of a graph and the maximum number of pairwise edge-disjoint spanning trees it contains. Arboricity is a basic measure of graph sparsity and plays an important role in algorithms for sparse graphs, for example in the classical work of Chiba and Nishizeki~\cite{chiba1985arboricity} on listing triangles and cliques. The spanning-tree packing theorem implies that every $2k$-edge-connected graph contains $k$ pairwise edge-disjoint spanning trees, and such packings also appear in algorithms for minimum cuts~\cite{karger2000minimum} and in rigidity theory~\cite{tay1984rigidity}. Measurable analogues of the Nash--Williams theorem have previously been studied in connection with measurable arboricity; see, for example, \cite[p.~2]{aim2022descriptive}.
	
	The two theorems have a common matroidal formulation. Covering the edge set by $k$ forests is equivalent to the edge set being independent in the union of $k$ copies of the cycle matroid, while packing $k$ spanning trees is equivalent to the cycle matroid having $k$ pairwise disjoint bases. Using this connection, we apply the measurable matroid union theorem and its attainment theorem to cycle matroids of graphings. The union theorem gives exact characterizations of approximate packings and coverings, while the attainment theorem gives exact decompositions under strengthened rank inequalities.
	
	We first recall the finite statements in their graph-theoretic form. If $G$ is a finite graph and $U\subseteq V(G)$, let $i_G(U)$ denote the number of edges \emph{induced by $U$}. If $\cP$ is a partition of $V(G)$, let $e_G(\cP)$ denote the number of edges with \emph{endpoints in different classes of $\cP$}.
	
	\begin{thm}[Nash-Williams, Tutte]\label{thm:nash_williams_tutte_finite}
		Let $G$ be a finite graph and let $k\in\bN$.
		\begin{enumerate}[label=(\alph*)]\itemsep0em
			\item The edge set of $G$ can be covered by $k$ forests if and only if
			\[
			i_G(U)\leq k(|U|-1)
			\]
			for every nonempty set $U\subseteq V(G)$.\label{it:cover}
			\item If $G$ is connected, then $G$ contains $k$ pairwise edge-disjoint spanning trees if and only if
			\[
			e_G(\cP)\geq k(|\cP|-1)
			\]
			for every partition $\cP$ of $V(G)$.\label{it:pack}
		\end{enumerate}
	\end{thm}
	
	For graphings, we formulate these problems through the cycle matroid. As in \cref{sec:graphic}, we normalize $\nu(J)=1$ throughout this subsection. Let $G=(J,\Sigma_J,\nu,E)$ be a graphing, let $\Sigma_E$ denote the Borel $\sigma$-algebra of the edge space, and consider its cycle matroid on $(E,\Sigma_E,\eta_G)$, with rank function $\rho_G$. As $G$ will be fixed throughout \cref{sec:graphings}, we use $\eta \coloneqq \eta_G$ and $\rho \coloneqq \rho_G$ for brevity. Recall from \cref{sec:graphic} that the independent sets of this matroid are the hyperfinite forests, while its bases are the hyperfinite essential spanning forests. Thus, a covering by hyperfinite forests is equivalent to a covering by hyperfinite essential spanning forests, since every hyperfinite forest is contained in a hyperfinite essential spanning forest.
	
	We say that $G$ has an \emph{approximate covering} by $k$ hyperfinite essential spanning forests if, for every $\varepsilon>0$, there exist hyperfinite essential spanning forests $T_1,\dots,T_k\subseteq E$ such that $\eta(E\setminus (T_1\cup\dots\cup T_k))\leq\varepsilon$. We say that $G$ has an \emph{approximate packing} of $k$ hyperfinite essential spanning forests if, for every $\varepsilon>0$, there exist pairwise disjoint hyperfinite forests $F_1,\dots,F_k\subseteq E$ such that
	\[
	\sum_{i=1}^k \left(\rho(E)-\eta(F_i)\right)\leq\varepsilon.
	\]
	Equivalently, each $F_i$ can be extended to a hyperfinite essential spanning forest by adding edge sets of total measure at most $\varepsilon$. 
	
	For a measurable set $A\subseteq J$, let
	\[
	I(A)\coloneqq \left\{\{x,y\}\in E\mid x,y\in A\right\}.
	\]
	and, for $x\in A$, let $G[A]_x$ denote the connected component of the induced subgraph $G[A]$ containing $x$. Define
	\[
	\kappa(A)
	\coloneqq
	\int_A \frac{1}{|V(G[A]_x)|}\diff\nu(x)
	\]
	and
	$
	i(A)\coloneqq \eta(I(A)).
	$
	The integrand is Borel measurable  and
	\begin{equation}\label{eq:rank_induced_subgraph}
		\rho(I(A))
		=
		\nu(A)-\kappa(A)
	\end{equation}
	by the rank formula~\eqref{eq:cycle_matroid_rank}.
	In particular,
	\begin{equation}\label{eq:rank_whole_graphing}
		\rho(E)
		=
		1-\kappa(J).
	\end{equation}
	
	Let $\cP=(P_1,\ldots,P_m)$ be a finite measurable partition of a measurable set $A\subseteq J$. Set
	\[
	I(\cP)\coloneqq \bigcup_{i=1}^m I(P_i),
	\qquad
	\kappa(\cP)\coloneqq \sum_{i=1}^m\kappa(P_i).
	\]
	Define
	$
	i(\cP)\coloneqq \eta(I(\cP)).
	$
	If $\cP$ is a finite measurable partition of $J$, also define
	$
	e(\cP)
	\coloneqq
	\eta(E\setminus I(\cP)).
	$
	Again by the rank formula \eqref{eq:cycle_matroid_rank},
	\begin{equation}\label{eq:rank_partition_subgraph}
		\rho(I(\cP))
		=
		\nu(A)-\kappa(\cP).
	\end{equation}

	We will need the following lemma.
	
	\begin{lem}\label{lem:rank_preserving_measurable_partition}
		Let $A\subseteq J$ and $X\subseteq E$, where every edge in $X$ has both endpoints in $A$. Then there is a countable measurable partition $(P_n)_{n\in\bN}$ of $A$ such that, with
		\[
		Y\coloneqq \bigcup_{n\in\bN}I(P_n),
		\]
		we have $X\subseteq Y$ and $\rho(Y)=\rho(X)$.
	\end{lem}
	
	\begin{proof}
		Consider the Borel graph $(A,X)$. Let $A_{\mathrm{fin}}$ be the set of
		vertices belonging to finite connected components of $(A,X)$, and let
		$A_\infty\coloneqq A\setminus A_{\mathrm{fin}}$. A standard consequence of the Lusin--Novikov theorem~\cite[Theorem~18.10]{kechris1995classical} is that the set $A_{\mathrm{fin}}$ is Borel, and the connectedness relation of $(A,X)$
		restricted to $A_{\mathrm{fin}}$ is a finite Borel equivalence relation.
		Hence, \cref{lem:finite_class_selection}\ref{it:borel2} gives a Borel
		transversal $S\subseteq A_{\mathrm{fin}}$. Let
		$\pi\colon A_{\mathrm{fin}}\to S$ be the corresponding map that
		sends each vertex to the representative of its $X$-component, which is Borel by \cite[Theorem~14.12]{kechris1995classical}.
		
		Define an auxiliary graph $L$ on $S$ by joining distinct $s,t\in S$ if some edge of $G$ has one endpoint in the $X$-component represented by $s$ and the other endpoint in the $X$-component represented by $t$. Since the $X$-components represented in $S$ are finite and $G$ is locally finite, the graph $L$ is locally finite. Moreover, the map $(x,y)\mapsto(\pi(x),\pi(y))$ is finite-to-one on the relevant Borel set of edges, and hence $L$ is Borel by the Lusin--Novikov theorem~\cite[Theorem~18.10]{kechris1995classical}. By the standard Borel coloring theorem for locally finite Borel graphs
		\cite[Lemma~5.5]{pikhurko2021borel}, there is a countable Borel proper
		coloring $c\colon S\longrightarrow\ \bN \setminus \{1\}$.

		Put $P_1\coloneqq A_\infty$, and, for $n\geq2$, let
		\[
		P_n\coloneqq \pi^{-1}(c^{-1}(n)).
		\]
		These sets form a countable measurable partition of $A$. Moreover, $X\subseteq Y$.
		
		The properness of $c$ ensures that no edge of $G$ joins two distinct finite $X$-components contained in the same partition class. It follows that every finite $X$-component is also a connected component of $Y$. Every connected component of $Y$ intersecting $P_1$, on the other hand, contains an infinite $X$-component and is therefore infinite. Thus, the spanning graphs $(J,X)$ and $(J,Y)$ have exactly the same finite connected components, and hence $\rho(Y)=\rho(X)$.
	\end{proof}
	
	\begin{thm}\label{thm:graphing_packing_covering}
		Let $G=(J,\Sigma_J,\nu,E)$ be a graphing and let $k\in\bN$. Then the following hold.
		\begin{enumerate}[label=(\alph*)]\itemsep0em
			\item The graphing $G$ has an approximate covering by $k$ hyperfinite essential spanning forests if and only if
			\begin{equation}
				i(A)\leq k\left(\nu(A)-\kappa(A)\right)\label{eq:graphing_forest_covering_condition}
			\end{equation}
			for every measurable $A\subseteq J$.\label{it:mcov}
			\item The graphing $G$ has an approximate packing of $k$ hyperfinite essential spanning forests if and only if
			\begin{equation}
				e(\cP)\geq k\left(\kappa(\cP)-\kappa(J)\right)\label{eq:graphing_base_packing_condition}
			\end{equation}
			for every finite measurable partition $\cP$ of $J$.\label{it:mpack}
		\end{enumerate}
	\end{thm}
	
	\begin{proof}
		Let $M_\vee$ denote the measurable union of $k$ copies of the cycle matroid of $G$, and let $r_\vee$ be its rank function.
		
		We first prove \ref{it:mcov}. Suppose that $G$ has an approximate covering by $k$ hyperfinite essential spanning forests. Let $A\subseteq J$ be measurable, and fix $\varepsilon>0$. Choose hyperfinite essential spanning forests $T_1,\ldots,T_k$ such that
		\begin{equation*}
			\eta\left(E\setminus(T_1\cup\cdots\cup T_k)\right)\leq\varepsilon.
		\end{equation*}
		Then
		\[
		i(A)
		\leq
		\varepsilon
		+
		\sum_{j=1}^k \eta(T_j\cap I(A))
		\leq
		\varepsilon
		+
		k\rho(I(A))
		=
		\varepsilon
		+
		k\left(\nu(A)-\kappa(A)\right),
		\]
		where the last equality follows from \eqref{eq:rank_induced_subgraph}.
		Letting $\varepsilon$ tend to $0$ gives \eqref{eq:graphing_forest_covering_condition}.
		
		Conversely, suppose that \eqref{eq:graphing_forest_covering_condition} holds. Let $X\subseteq E$ be measurable, and let $A_X$ be the set of endpoints of the edges in $X$, which is Borel
		by \cite[Theorem~18.2]{lovasz2012large}. Apply \cref{lem:rank_preserving_measurable_partition} to $A_X$ and $X$. Write $(P_n)_{n\in\bN}$ for the resulting partition and set
		\[
		Y\coloneqq\bigcup_{n\in\bN}I(P_n).
		\]
		The edge sets $I(P_n)$ are pairwise disjoint, and hence
		\[
		\eta(X)
		\leq \eta(Y)
		=
		\sum_{n\in\bN}i(P_n)
		\leq
		k\sum_{n\in\bN}\left(\nu(P_n)-\kappa(P_n)\right)
		=
		k\rho(Y)
		=
		k\rho(X),
		\]
		where the penultimate equality follows from the rank formula and countable additivity, and the last equality follows from \cref{lem:rank_preserving_measurable_partition}. By \cref{thm:union_rank_formula},
		\[
		r_\vee(E)=\min_{X\subseteq E}\left\{\eta(E\setminus X)+k\rho(X)\right\}\geq\eta(E).
		\]
		Since $r_\vee(E)\leq\eta(E)$, it follows that $r_\vee(E)=\eta(E)$. Thus, $E$ is independent in $M_\vee$. By \cref{prop:finite_union_raw_description}, for every $\varepsilon>0$ there exist pairwise disjoint hyperfinite forests $F_1,\dots,F_k$ such that $\bigcup_{i=1}^kF_i\subseteq E$ and $\eta(E\setminus\bigcup_{i=1}^kF_i)<\varepsilon$. For every $i\in[k]$, extend $F_i$ to a hyperfinite essential spanning forest $T_i$. Then $\eta(E\setminus\bigcup_{i=1}^kT_i)<\varepsilon$, so $G$ has the required approximate covering.
		
		We now prove \ref{it:mpack}. Suppose that $G$ has an approximate packing of $k$ hyperfinite essential spanning forests. Let $\cP$ be a finite measurable partition of $J$, and fix $\varepsilon>0$. Choose pairwise disjoint hyperfinite forests $F_1,\ldots,F_k$ such that
		\begin{equation*}
			\sum_{j=1}^k \left(\rho(E)-\eta(F_j)\right)\leq\varepsilon.
		\end{equation*}
		Let $Y\coloneqq I(\cP)$. For every $j \in [k]$, the set $F_j\cap Y$ is a hyperfinite forest contained in $Y$. Hence,
		\begin{equation*}
			\eta(F_j\cap Y)\leq\rho(Y).
		\end{equation*}
		Since the sets $F_1,\ldots,F_k$ are pairwise disjoint, we have
		\[
		e(\cP)
		\geq
		\sum_{j=1}^k\eta(F_j\setminus Y)
		=
		\sum_{j=1}^k\left(\eta(F_j)-\eta(F_j\cap Y)\right)
		\geq
		k\rho(E)-\varepsilon-k\rho(Y)
		=
		k\left(\kappa(\cP)-\kappa(J)\right)-\varepsilon,
		\]
		where the last equality follows from \eqref{eq:rank_whole_graphing} and \eqref{eq:rank_partition_subgraph}. Since $\varepsilon>0$ was arbitrary, \eqref{eq:graphing_base_packing_condition} follows.
		
		Conversely, suppose that \eqref{eq:graphing_base_packing_condition} holds for every finite measurable partition of $J$. Let $X\subseteq E$ be measurable. Apply \cref{lem:rank_preserving_measurable_partition} with $A=J$. Let $(P_i)_{i\in\bN}$ be the resulting partition and put
		\[
		Y\coloneqq\bigcup_{i\in\bN}I(P_i).
		\]
		Thus, $X\subseteq Y$ and $\rho(Y)=\rho(X)$. For $n\in\bN$, set
		\[
		R_n\coloneqq\bigcup_{i>n}P_i,
		\qquad
		\cP_n\coloneqq(P_1,\ldots,P_n,R_n),
		\]
		and let $Y_n\coloneqq I(\cP_n)$. Then
		$
		Y_{n+1} \subseteq Y_n
		$ for all $n \in \bN$ and $\bigcap_{n\in\bN}Y_n=Y.
		$
		Using \eqref{eq:rank_whole_graphing}, \eqref{eq:rank_partition_subgraph} and \eqref{eq:graphing_base_packing_condition},
		\[
		\eta(E\setminus Y_n)
		\geq
		k\left(\kappa(\cP_n)-\kappa(J)\right)
		=
		k\left(\rho(E)-\rho(Y_n)\right).
		\]
		Continuity from below of $\eta$ and the continuity from above of $\rho$ proved in \cref{cor:smooth} give, upon letting $n$ tend to infinity,
		\[
		\eta(E\setminus Y)
		\geq
		k\left(\rho(E)-\rho(Y)\right).
		\]
		Since $X\subseteq Y$ and $\rho(Y)=\rho(X)$, it follows that
		\begin{equation*}
			\eta(E\setminus X)+k\rho(X)\geq k\rho(E).
		\end{equation*}
		By \cref{thm:union_rank_formula},
		\[
		r_\vee(E)=\min_{X\subseteq E}\left\{\eta(E\setminus X)+k\rho(X)\right\}\geq k\rho(E).
		\]
		Taking $X=E$ in the minimum gives the reverse inequality, so $r_\vee(E)=k\rho(E)$. Choose a basis $I$ of $M_\vee$. Then $\eta(I)=k\rho(E)$. By \cref{prop:finite_union_raw_description}, for every $\varepsilon>0$ there exist pairwise disjoint hyperfinite forests $F_1,\dots,F_k$ such that $\bigcup_{i=1}^kF_i\subseteq I$ and $\eta(I\setminus\bigcup_{i=1}^kF_i)<\varepsilon$. Consequently,
		\[
		\sum_{i=1}^k\left(\rho(E)-\eta(F_i)\right)=k\rho(E)-\eta\big(\bigcup_{i=1}^kF_i\big)=\eta\big(I\setminus\bigcup_{i=1}^kF_i\big)<\varepsilon.
		\]
		Thus, $G$ has the required approximate packing.
	\end{proof}
	
	\Cref{thm:graphing_packing_covering} characterizes approximate decompositions. Under a strengthening of the corresponding rank inequalities, \cref{thm:union_attainment_expansion} gives exact decompositions. For coverings, we apply it to $k$ copies of a suitable truncation of the cycle matroid, while for packings we apply it directly to $k$ copies of the cycle matroid.
	
	\begin{cor}\label{cor:graphing_exact_packing_covering}
		Let $G=(J,\Sigma_J,\nu,E)$ be a graphing and let $k\in\bN$. Then the following hold.
		\begin{enumerate}[label=(\alph*),itemsep=0em]
			\item Suppose that there exists $\varepsilon>0$ such that, for every $A\in\Sigma_J$ and every finite measurable partition $\cP$ of $A$,
			\[k\left(\nu(A)-\kappa(\cP)\right)\geq\min\left\{\eta(E),(1+\varepsilon)i(\cP)\right\}.\]
			Then $E$ can be partitioned, up to a null set, into $k$ hyperfinite forests. Consequently, $G$ admits a covering by $k$ hyperfinite essential spanning forests. \label{it:graphinga}
			\item Suppose that
			\[k\left(1-\kappa(J)\right)\leq\eta(E)\]
			and that there exists $\varepsilon>0$ such that, for every $A\in\Sigma_J$ and every finite measurable partition $\cP$ of $A$,
			\[k\left(\nu(A)-\kappa(\cP)\right)\geq\min\left\{k\left(1-\kappa(J)\right),(1+\varepsilon)i(\cP)\right\}.\]
			Then $G$ contains $k$ pairwise disjoint hyperfinite essential spanning forests. \label{it:graphingb}
		\end{enumerate}
	\end{cor}
	
	\begin{proof}
		Let $M$ denote the cycle matroid of $G$. We first record a consequence of \cref{lem:rank_preserving_measurable_partition} that will be used in both parts. Fix $\varepsilon>0$ and $C\geq0$, and suppose that
		\begin{equation}\label{eq:graphing_partition_expansion}
			k\left(\nu(A)-\kappa(\cP)\right)
			\geq
			\min\left\{C,(1+\varepsilon)i(\cP)\right\}
		\end{equation}
		for every measurable $A\subseteq J$ and every finite measurable partition $\cP$ of $A$. We claim that
		\begin{equation}\label{eq:graphing_edge_expansion}
			k\rho(X)
			\geq
			\min\left\{C,(1+\varepsilon)\eta(X)\right\}
		\end{equation}
		for every measurable $X\subseteq E$.
		
		Let $A_X$ be the set of endpoints of the edges in $X$. As above, $A_X$ is measurable. Apply \cref{lem:rank_preserving_measurable_partition} to $A_X$ and $X$. Let $(P_i)_{i\in\bN}$ be the resulting partition and put
		\[
		Y\coloneqq\bigcup_{i\in\bN}I(P_i).
		\]
		For $n\in\bN$, set
		\[
		R_n\coloneqq\bigcup_{i>n}P_i,
		\qquad
		\cP_n\coloneqq(P_1,\ldots,P_n,R_n),
		\]
		and let $Y_n\coloneqq I(\cP_n)$. Then
		$
		Y_{n+1} \subseteq Y_n
		$ for all $n \in \bN$ and $\bigcap_{n\in\bN}Y_n=Y.
		$ By \eqref{eq:rank_partition_subgraph} and \eqref{eq:graphing_partition_expansion},
		\[
		k\rho(Y_n)
		\geq
		\min\left\{C,(1+\varepsilon)\eta(Y_n)\right\}.
		\]
		Since $\rho$ is continuous from above by \cref{cor:smooth} and $\eta$ is
		also continuous from above, letting $n$ tend to infinity gives
		\[
		k\rho(Y)
		\geq
		\min\left\{C,(1+\varepsilon)\eta(Y)\right\}.
		\]
		Since $X\subseteq Y$ and $\rho(Y)=\rho(X)$, this implies \eqref{eq:graphing_edge_expansion}.
		
		Suppose first that the hypothesis in \ref{it:graphinga} holds. Applying the preceding observation with $C=\eta(E)$ gives
		\[
		k\rho(X)\geq\min\left\{\eta(E),(1+\varepsilon)\eta(X)\right\}
		\]
		for every $X\in\Sigma_E$. Taking $X=E$ gives
		$
		k\rho(E)\geq\eta(E).
		$
		Set $t\coloneqq\eta(E)/k$, and let $M'\coloneqq(M)_t$ be the $t$-truncation of $M$. Its rank function is $r'(X)=\min\{t,\rho(X)\}$, and its rank is $t$. For every $X\in\Sigma_E$, we have
		\[kr'(X)=\min\left\{\eta(E),k\rho(X)\right\}\geq\min\left\{\eta(E),(1+\varepsilon)\eta(X)\right\}.\]
		Thus, \cref{thm:union_attainment_expansion}, applied to $k$ copies of $M'$, gives pairwise disjoint sets $F_1,\dots,F_k$ that are independent in $M'$ and satisfy
		\[\eta\big(\bigcup_{i=1}^kF_i\big)=\eta(E).\]
		Each $F_i$ is a hyperfinite forest, and the sets cover $E$ up to a null set. Hence, they form the asserted partition. Extending each $F_i$ to a hyperfinite essential spanning forest gives the required covering.
		
		Now suppose that the hypotheses in \ref{it:graphingb} hold. By \eqref{eq:rank_whole_graphing}, the first hypothesis says that $k\rho(E)\leq\eta(E)$. Applying the preceding observation with $C=k\rho(E)$ gives
		\[k\rho(X)\geq\min\left\{k\rho(E),(1+\varepsilon)\eta(X)\right\}\]
		for every $X\in\Sigma_E$. Therefore, \cref{thm:union_attainment_expansion}, applied to $k$ copies of $M$, gives pairwise disjoint hyperfinite forests $F_1,\dots,F_k$ such that
		\[\eta\big(\bigcup_{i=1}^kF_i\big)=\min\left\{\eta(E),k\rho(E)\right\}=k\rho(E).\]
		Since $\eta(F_i)\leq\rho(E)$ for every $i\in[k]$, equality holds for every $i$. Thus, each $F_i$ is a basis of the cycle matroid and hence a hyperfinite essential spanning forest.
	\end{proof}
	
	\begin{rem}
		For two graphings $G$ and $H$ on the same probability space, write $G\sim H$ if they induce the same connectedness relation. The \emph{cost} of $G$ is defined by
		\[
		\cost(G)\coloneqq\inf\{\eta_H(E(H))\mid H\sim G\}.
		\]
		Thus, cost depends only on the measured equivalence relation generated by $G$, whereas the covering and packing thresholds in \cref{thm:graphing_packing_covering} depend on the chosen graphing. Cost is closely related to $\ell^2$-Betti numbers and measured isoperimetry; see~\cite{gaboriau2002invariants,pichot2010cout}.
		
		For every measurable edge set $F\subseteq E$, we have
		\[
		\rho_G(F)\leq\cost(G_F)\leq\eta_G(F),
		\]
		where $G_F$ denotes the subgraphing with edge set $F$; see~\cite{berczi2026cycle}. Consequently, an approximate covering of $E$ by $k$ hyperfinite essential spanning forests implies
		\[
		\cost(G)\leq\eta_G(E)\leq k\rho_G(E).
		\]
		If $G$ is a treeing, then $\cost(G)=\eta_G(E)$ by a result of Gaboriau~\cite{gaboriau2000cout}, so an approximate packing of $k$ hyperfinite essential spanning forests implies
		\[
		k\rho_G(E)\leq\cost(G).
		\]
	\end{rem}
	
	%%%%%%%%%%%%%%%%%%%%%%%%%%%%%%%%
	\subsection{Basis Exchanges}
	\label{sec:exchanges}
	%%%%%%%%%%%%%%%%%%%%%%%%%%%%%%%%
	
	We conclude this section with a measurable analogue of a theorem of Greene and Magnanti~\cite{greene1975some}. In the finite setting, the theorem is as follows.
	
	\begin{thm}[Greene and Magnanti]\label{thm:greene_magnanti_finite}
		Let $M$ be a finite matroid, and let $B_1,B_2$ be bases of $M$. If $B_1=X_1\cup\dots\cup X_q$ is a partition of $B_1$, then there exists a partition $B_2=Y_1\cup\dots\cup Y_q$ such that $(B_1\setminus X_i)\cup Y_i$ is a basis for every $i\in[q]$.
	\end{thm}
	
	In the measurable setting, the same argument gives an approximate statement. The loss comes from the upward closure in the definition of measurable matroid union.
	
	\begin{thm}\label{thm:measurable_greene_magnanti}
		Let $M=(J,\Sigma,\mu,\cF)$ be a measurable matroid, and let $B_1,B_2$ be bases of $M$. Let $B_1=X_1\cup\dots\cup X_q$ be a measurable partition of $B_1$. Then, for every $\varepsilon>0$, there exists a measurable partition $B_2=Y_0\cup Y_1\cup\dots\cup Y_q$ such that $\mu(Y_0)\leq\varepsilon$ and $(B_1\setminus X_i)\cup Y_i\in\cF$ for every $i\in[q]$. In particular, each set $(B_1\setminus X_i)\cup Y_i$ can be extended to a basis after adding a set of measure at most $\varepsilon$.
	\end{thm}
	
	\begin{proof}
		Let $r$ be the rank function of $M$, and set $R\coloneqq r(J)=\mu(B_1)=\mu(B_2)$. Let $C\coloneqq B_1\cap B_2$ and $D\coloneqq B_2\setminus B_1$. For each $i\in[q]$, define
		\[
		A_i\coloneqq (B_1\setminus X_i)\cup(X_i\cap B_2)=B_1\setminus(X_i\setminus B_2).
		\]
		Then $A_i$ is independent. For each $i\in[q]$, consider the contraction of $M$ by $A_i$, restricted to $D$: let $M_i\coloneqq(M/A_i)|D$, and let $r_i$ be its rank function. Since $A_i$ is independent, the contraction formula gives $r_i(Z)=r(A_i\cup Z)-\mu(A_i)$ for every $Z\subseteq D$.
		
		Consider the measurable union of the matroids $M_1,\dots,M_q$ on $D$, and let $r_\vee$ denote its rank function. By \Cref{thm:union_rank_formula},
		\[
		r_\vee(D)=\min_{Z\subseteq D}\left\{\mu(D\setminus Z)+\sum_{i=1}^q r_i(Z)\right\}.
		\]
		We claim that $\sum_{i=1}^q r_i(Z)\geq\mu(Z)$ for every $Z\subseteq D$. Fix $Z\subseteq D$, and set $U\coloneqq B_1\cup Z$. By \cref{cor:generalized_submodular_inequality},
		\[
		\sum_{i=1}^q r(A_i\cup Z)=\sum_{i=1}^q\widehat r(\mathbf 1_{A_i\cup Z})\geq\widehat r\left(\sum_{i=1}^q\mathbf 1_{A_i\cup Z}\right)=\widehat r\left((q-1)\mathbf 1_U+\mathbf 1_{C\cup Z}\right)=(q-1)r(U)+r(C\cup Z),
		\]
		since the sets $X_i\setminus B_2$ are pairwise disjoint and form a partition of $B_1\setminus B_2$, while $C\cup Z\subseteq U$. Since $B_1$ is a basis and $B_1\subseteq U$, we have $r(U)=R$. Moreover, $C\cup Z\subseteq B_2$, so $C\cup Z$ is independent, and hence $r(C\cup Z)=\mu(C)+\mu(Z)$. On the other hand,
		\[
		\sum_{i=1}^q\mu(A_i)=\sum_{i=1}^q\left(R-\mu(X_i\setminus B_2)\right)=qR-\mu(B_1\setminus B_2)=(q-1)R+\mu(C).
		\]
		Together with the contraction formula, this gives
		\[
		\sum_{i=1}^q r_i(Z)=\sum_{i=1}^q\left(r(A_i\cup Z)-\mu(A_i)\right)\geq\mu(Z).
		\]
		
		Hence, every term inside the minimum is at least $\mu(D\setminus Z)+\mu(Z)=\mu(D)$. The reverse inequality follows from $r_\vee(D)\leq\mu(D)$, so $r_\vee(D)=\mu(D)$. Thus, $D$ is independent in the union matroid. By \cref{prop:finite_union_raw_description}, there exist pairwise disjoint independent sets $Z_1,\dots,Z_q$ in $M_1,\dots,M_q$, respectively, such that $D'\coloneqq\bigcup_{i=1}^qZ_i\subseteq D$ and $\mu(D\setminus D')<\varepsilon$.
		
		Define $Y_0\coloneqq D\setminus D'$ and $Y_i\coloneqq(X_i\cap B_2)\cup Z_i$ for $i\in[q]$. Then $B_2=Y_0\cup Y_1\cup\dots\cup Y_q$ is a measurable partition and $\mu(Y_0)<\varepsilon$. Since each $Z_i$ is independent in $M_i=(M/A_i)|D$, we have $A_i\cup Z_i\in\cF$. Thus,
		\[
		(B_1\setminus X_i)\cup Y_i=(B_1\setminus X_i)\cup(X_i\cap B_2)\cup Z_i=A_i\cup Z_i\in\cF
		\]
		for every $i\in[q]$.
		
		It remains to prove the last assertion. Since $A_i\cup Z_i$ is independent, we have $\mu(Z_i)\leq R-\mu(A_i)=\mu(X_i\setminus B_2)$. Hence,
		\[
		\sum_{i=1}^q\left(\mu(X_i\setminus B_2)-\mu(Z_i)\right)=\mu(B_1\setminus B_2)-\mu(D')=\mu(D)-\mu(D')=\mu(Y_0)<\varepsilon.
		\]
		Therefore, each set $(B_1\setminus X_i)\cup Y_i$ is independent and has measure at least $R-\varepsilon$. Extending it to a maximal independent set gives a basis, and the added set has measure at most $\varepsilon$.
	\end{proof}
	
	The case $q=2$ of \cref{thm:greene_magnanti_finite} is the multiple symmetric basis exchange property, independently observed by Greene~\cite{greene1973multiple}, Brylawski~\cite{brylawski1973some}, and Woodall~\cite{woodall1974exchange}: if $A$ and $B$ are bases of a finite matroid, then for every $X\subseteq A$ there exists $Y\subseteq B$ such that both
	\[
	(A\setminus X)\cup Y
	\qquad\text{and}\qquad
	(B\setminus Y)\cup X
	\]
	are bases. Since exchanging individual elements is not meaningful in the measurable setting, we refer to the corresponding measurable statement simply as the \emph{symmetric basis exchange property}. We now show that this exact statement can fail for measurable matroids. In particular, the exceptional set $Y_0$ in \cref{thm:measurable_greene_magnanti} cannot in general be omitted, even when $q=2$.
	
	Let $I_1,I_2,I_3$ be three disjoint copies of $[0,1]$, and write $[0,t]_i$ for the copy of $[0,t]$ inside $I_i$. Set
	\[
	J\coloneqq I_1\cup I_2\cup I_3.
	\]
	Let $\mu$ be the measure on $J$ which is Lebesgue measure on $I_1$ and one half of Lebesgue measure on each of $I_2$ and $I_3$. Thus,
	\[
	\mu(I_1)=1
	\qquad\text{and}\qquad
	\mu(I_2)=\mu(I_3)=\frac{1}{2}.
	\]
	Define
	\begin{equation}\label{eq:symmetric_exchange_counterexample}
		\cF\coloneqq
		\left\{
		F\in\Sigma
		\,\middle|\,
		\lambda(F\cap [0,t]_1)
		+\tfrac{1}{2}\lambda(F\cap [0,t]_2)
		+\tfrac{1}{2}\lambda(F\cap [0,t]_3)
		\leq t
		\text{ for every } t\in[0,1]
		\right\}.
	\end{equation}
	
	\begin{prop}\label{prop:symmetric_exchange_fails}
		If $\cF$ is defined by \eqref{eq:symmetric_exchange_counterexample}, then $M=(J,\Sigma,\mu,\cF)$ is a measurable matroid for which the measurable symmetric basis exchange property fails.
	\end{prop}
	
	\begin{proof}
		Let $Z(t)\coloneqq [0,t]_1\cup[0,t]_2\cup[0,t]_3$. Then the defining condition for $\cF$ is precisely $\mu(F\cap Z(t))\leq t$ for every $t\in[0,1]$. Hence, $M$ is a measurable nested matroid of the type defined in \cref{sec:nested}.
		
		Set $A\coloneqq I_2\cup I_3$ and $B\coloneqq I_1$. Both $A$ and $B$ are independent, and $\mu(A)=\mu(B)=1$. Since every independent set has measure at most $1$, both $A$ and $B$ are bases.
		
		We claim that the symmetric basis exchange property fails for $X\coloneqq I_2\subseteq A$. Suppose that there exists a measurable set $Y\subseteq B$ such that both $(A\setminus X)\cup Y=I_3\cup Y$ and $(B\setminus Y)\cup X=(I_1\setminus Y)\cup I_2$ are bases. Since all bases have measure $1$, we must have $\mu(Y)=\mu(X)=1/2$.
		
		The independence of $I_3\cup Y$ and $(I_1\setminus Y)\cup I_2$ gives, for every $t\in[0,1]$,
		\[
		\lambda(Y\cap[0,t]_1)\leq\frac{t}{2}
		\qquad\text{and}\qquad
		\lambda((I_1\setminus Y)\cap[0,t]_1)\leq\frac{t}{2}.
		\]
		Adding the two inequalities, we get equality throughout, so $\lambda(Y\cap[0,t]_1)=t/2$ for every $t\in[0,1]$. Identifying $I_1$ with $[0,1]$, this contradicts \cref{lem:halving_measure}. This proves that no such $Y$ exists.
	\end{proof}

	%%%%%%%%%%%%%%%%%%%%%%%%%%%%%%%%
	\section{Further Directions}
	\label{sec:remarks}
	%%%%%%%%%%%%%%%%%%%%%%%%%%%%%%%%
	
	The present paper is the first step in a broader study of measurable matroids. One of our main ongoing directions is the development of a limit theory of finite matroids based on measurable matroids and quotient-convergence. The theory developed here is intended to provide the structural side of such a limit theory: rather than viewing a limit only through its rank function, one would like the limiting object to retain notions such as independence, bases, duality, intersection, and union. Such a theory could provide a bridge between finite matroid theory and measurable combinatorics analogous to the role played by graphons and graphings in graph limit theory. 
	
	We highlight three research directions that arise naturally from the results of this paper.
	
	\paragraph{Limits of finite matroids.}
	A basic problem is to understand more precisely the relation between quotient-convergence of finite matroid rank functions and measurable matroids. In particular, one would like to characterize which measurable matroids arise as quotient-limits of finite matroids, identify natural classes that admit finite approximations, and understand which structural and optimization properties are preserved under convergence. These questions are necessary steps toward a limit theory in which measurable matroids serve as genuine limit objects for finite matroids.
	
	\paragraph{Attainment and exact decompositions.}
	The measurable matroid intersection theorem determines the supremum of the measures of common independent sets, but the supremum need not be attained. The rank-expansion conditions in \cref{sec:intun} give sufficient conditions for attainment, and it would be interesting to understand what other natural assumptions imply the existence of an optimizer.
	
	A similar issue occurs for measurable union. An independent set of the measurable union can be approximated arbitrarily well by unions of independent sets of the original matroids, but need not itself admit such a decomposition. This is also why exceptional sets appear in the packing and basis-exchange applications. It would be useful to characterize classes of measurable matroids for which the raw union family is already $\sigma$-increasing, or for which every independent set of the measurable union admits an exact decomposition. Analogous questions arise for the transversal and matching constructions of \cref{sec:examples}.
	
	\paragraph{Minorizing measures.}
	The results of \cref{sec:extreme_points} show that the measures induced by independent sets are weak-star dense in $\mm(r)$, and those induced by bases are weak-star dense in $\bmm(r)$. Every base measure is an extreme point of $\bmm(r)$, but the converse fails, as shown by the interval-halving example. It remains to understand natural conditions under which every extreme point, or every exposed point, of $\bmm(r)$ is induced by a basis. Even for cycle matroids of graphings, it is not known whether every extreme point of $\bmm(r)$ is exposed.
	
	%%%%%%%%%%%%%%%%%%%%%%%%%%%%%%%%
	\section{Acknowledgements and Disclosures}
	%%%%%%%%%%%%%%%%%%%%%%%%%%%%%%%%
	
	\paragraph{Acknowledgements.}
	
	The authors are grateful to Márton Borbényi, Boglárka Gehér, Anett Kocsis, László Lovász, and László Márton Tóth for helpful discussions.
	
	András Imolay was supported by the EKÖP-KDP-25 University Research Scholarship Program, Cooperative Doctoral Program of the Ministry for Culture and Innovation, from the source of the National Research, Development and Innovation Fund. Ádám Schweitzer was supported by the Knut and Alice Wallenberg Foundation. The research received further support from the Lend\"ulet Programme of the Hungarian Academy of Sciences -- grant number LP2021-1/2021, from the Ministry of Innovation and Technology of Hungary from the National Research, Development and Innovation Fund -- grant numbers ADVANCED 150556 and 153096, and from the Dynasnet European Research Council Synergy project -- grant number ERC-2018-SYG 810115.
	
	\paragraph{Use of generative AI.}
	
	The main mathematical results and the framework of this paper were developed over a long period, before the authors began using advanced language models for mathematical research. In particular, the axiomatic theory of measurable matroids, the matroid intersection and union formulas, the expansion-based attainment theorem, and most of the applications were obtained by the authors independently of AI assistance. During later stages of the project, the authors used ChatGPT (OpenAI, including GPT-5.6 Sol and earlier versions) for grammar checking, language editing, formatting, polishing routine mathematical exposition, checking elementary arguments, and identifying minor errors. The only part of the paper in which AI assistance contributed substantively to the identification of mathematical results is \cref{sec:continuous_b_matching}. The direction studied there was proposed by the authors, while some of the formulations and results were first identified with the assistance of ChatGPT. The corresponding arguments were subsequently checked and rewritten by the authors before being included in the paper. All mathematical arguments and claims produced with AI assistance were verified by the authors, who take full responsibility for the paper.

	%%%%%%%%%%%%%%%%%%%%%%%%%%%%%%%%
	\bibliographystyle{abbrv}
	\bibliography{measurable}
	%%%%%%%%%%%%%%%%%%%%%%%%%%%%%%%%
	
\end{document}